\documentclass[12pt]{article}

\usepackage[T1]{fontenc}
\usepackage{lmodern}
\usepackage[margin=1in]{geometry}
\usepackage{amsmath,amssymb,amsthm,mathtools}
\usepackage{aliascnt}
\usepackage{etoolbox}
\usepackage{booktabs,array,longtable,tabularx}
\usepackage{pdflscape}
\usepackage{xspace,xcolor,seqsplit,enumitem,microtype,setspace,graphicx}
\usepackage{tikz}
\usetikzlibrary{arrows.meta,positioning,calc,fit,backgrounds,decorations.pathreplacing,matrix}
\usepackage[numbers,sort&compress]{natbib}
\usepackage[hypertexnames=false,colorlinks=true,linkcolor=blue!55!black,citecolor=blue!55!black,urlcolor=blue!55!black]{hyperref}
\usepackage[nameinlink,capitalise,noabbrev]{cleveref}
\usepackage{titlesec}
\usepackage{fancyhdr}

\definecolor{navy}{RGB}{29,54,92}
\definecolor{softblue}{RGB}{229,238,249}
\definecolor{softgray}{RGB}{245,246,248}
\definecolor{softgreen}{RGB}{229,244,234}
\definecolor{softamber}{RGB}{252,242,214}
\definecolor{softred}{RGB}{249,231,231}
\definecolor{readblue}{RGB}{37,78,128}
\definecolor{intuitiongreen}{RGB}{48,112,74}
\definecolor{checkpointamber}{RGB}{142,91,19}
\definecolor{usepurple}{RGB}{91,65,126}
\definecolor{auditgray}{RGB}{82,87,96}

\colorlet{figurehard}{red!24}
\colorlet{figurealgorithm}{green!20}
\colorlet{figuremixed}{blue!18}
\colorlet{figureopen}{orange!26}
\colorlet{figuregray}{black!7}

\newcommand{\versiondate}{September 11, 2026}
\newcommand{\currentunit}{Front matter}
\newcommand{\setrunningunit}[1]{\gdef\currentunit{#1}}

\allowdisplaybreaks
\numberwithin{equation}{section}
\newtheorem{theorem}{Theorem}
\makeatletter
\@addtoreset{theorem}{part}
\makeatother
\renewcommand{\thetheorem}{\Roman{part}.\arabic{theorem}}
\newaliascnt{lemma}{theorem}\newtheorem{lemma}[lemma]{Lemma}\aliascntresetthe{lemma}
\newaliascnt{proposition}{theorem}\newtheorem{proposition}[proposition]{Proposition}\aliascntresetthe{proposition}
\newaliascnt{corollary}{theorem}\newtheorem{corollary}[corollary]{Corollary}\aliascntresetthe{corollary}
\newaliascnt{claim}{theorem}\aliascntresetthe{claim}
\theoremstyle{definition}
\newaliascnt{definition}{theorem}\newtheorem{definition}[definition]{Definition}\aliascntresetthe{definition}
\newaliascnt{problem}{theorem}\aliascntresetthe{problem}
\newaliascnt{example}{theorem}\newtheorem{example}[example]{Example}\aliascntresetthe{example}
\theoremstyle{remark}
\newaliascnt{remark}{theorem}\newtheorem{remark}[remark]{Remark}\aliascntresetthe{remark}

\crefname{theorem}{theorem}{theorems}\Crefname{theorem}{Theorem}{Theorems}
\crefname{lemma}{lemma}{lemmas}\Crefname{lemma}{Lemma}{Lemmas}
\crefname{proposition}{proposition}{propositions}\Crefname{proposition}{Proposition}{Propositions}
\crefname{corollary}{corollary}{corollaries}\Crefname{corollary}{Corollary}{Corollaries}
\crefname{claim}{claim}{claims}\Crefname{claim}{Claim}{Claims}
\crefname{definition}{definition}{definitions}\Crefname{definition}{Definition}{Definitions}
\crefname{problem}{problem}{problems}\Crefname{problem}{Problem}{Problems}
\crefname{example}{example}{examples}\Crefname{example}{Example}{Examples}
\crefname{remark}{remark}{remarks}\Crefname{remark}{Remark}{Remarks}

\newcommand{\AWET}{\textnormal{\textsc{AWET}}\xspace}
\newcommand{\RXC}{\textnormal{\textsc{RXC3}}\xspace}
\newcommand{\RPE}{\textnormal{\textsc{RPE}}\xspace}
\newcommand{\PARTITION}{\textnormal{\textsc{Partition}}\xspace}

\newcommand{\T}{\mathcal T}\newcommand{\J}{\mathcal J}\newcommand{\B}{\mathcal B}

\newcommand{\1}{\mathbf 1}
\newcommand{\OPT}{\operatorname{OPT}}
\newcommand{\SDP}{\operatorname{SDP}}\newcommand{\one}{\mathbf 1}\newcommand{\Tr}{\operatorname{Tr}}\newcommand{\R}{\mathbb R}

\newcommand{\abs}[1]{\left|#1\right|}
\newcommand{\floor}[1]{\left\lfloor#1\right\rfloor}\newcommand{\ceil}[1]{\left\lceil#1\right\rceil}
\newcommand{\doi}[1]{\href{https://doi.org/#1}{doi:\,#1}}
\newcommand{\numstr}[1]{\begingroup\ttfamily\scriptsize\seqsplit{#1}\endgroup}
\newcommand{\smallnum}[1]{\begingroup\ttfamily\tiny\seqsplit{#1}\endgroup}
\newcommand{\mathnum}[2][0.36\textwidth]{\vcenter{\hbox{\parbox{#1}{\centering\ttfamily\scriptsize\seqsplit{#2}}}}}
\newcommand{\displaynum}[1]{\mbox{\ttfamily #1}}
\newcommand{\readnum}[1]{\vcenter{\hbox{\parbox{0.36\textwidth}{\centering\ttfamily\small\seqsplit{#1}}}}}
\newcommand{\mednum}[1]{\vcenter{\hbox{\parbox{0.36\textwidth}{\centering\ttfamily\footnotesize\seqsplit{#1}}}}}
\newcommand{\sgnvec}{\boldsymbol{\sigma}}

\newcommand{\bits}{\{0,1\}}

\newenvironment{intuition}{\begin{quote}\small\noindent\textcolor{intuitiongreen}{\textbf{Intuition.}}\ }{\end{quote}}
\newenvironment{howtoread}{\begin{quote}\small\noindent\textcolor{readblue}{\textbf{How to read this step.}}\ }{\end{quote}}
\newenvironment{howto}{\begin{quote}\small\noindent\textcolor{usepurple}{\textbf{How to use this result.}}\ }{\end{quote}}
\newenvironment{proofcheckpoint}{\begin{quote}\small\noindent\textcolor{checkpointamber}{\textbf{Proof checkpoint.}}\ }{\end{quote}}
\newenvironment{arithmeticguide}{\begin{quote}\small\noindent\textcolor{auditgray}{\textbf{Arithmetic guide.}}\ }{\end{quote}}
\newenvironment{finding}[1]{\begin{quote}\small\noindent\textcolor{readblue}{\textbf{#1.}}\ }{\end{quote}}
\newenvironment{caution}[1]{\begin{quote}\small\noindent\textcolor{checkpointamber}{\textbf{#1.}}\ }{\end{quote}}
\newenvironment{partposition}{\begin{quote}\small\noindent\textcolor{readblue}{\textbf{Part positioning.}}\ }{\end{quote}}
\newenvironment{notationreset}{\begin{quote}\small\noindent\textcolor{usepurple}{\textbf{Notation reset.}}\ }{\end{quote}}
\newenvironment{proofcontract}{\begin{quote}\small\noindent\textcolor{checkpointamber}{\textbf{Proof contract.}}\ }{\end{quote}}
\newenvironment{reusableidea}{\begin{quote}\small\noindent\textcolor{intuitiongreen}{\textbf{Reusable idea.}}\ }{\end{quote}}
\newenvironment{sectionexit}{\begin{quote}\small\noindent\textcolor{readblue}{\textbf{Section exit.}}\ }{\end{quote}}
\newcommand{\tablecert}[1]{\par\smallskip\begingroup\small\noindent #1\par\endgroup\smallskip}

\titleformat{\part}[display]
  {\centering\bfseries\color{navy}}
  {\Huge Part \Roman{part}}
  {0.8em}
  {\LARGE}
\titlespacing*{\part}{0pt}{1.2in}{1.0in}

\newcommand{\nextsectionhead}{}
\newcommand{\nextsubsectionhead}{}
\newcommand{\setshortsectionhead}[1]{\gdef\nextsectionhead{#1}}
\newcommand{\setshortsubsectionhead}[1]{\gdef\nextsubsectionhead{#1}}

\hypersetup{
 pdftitle={Asymmetric Weighted Earliness Tardiness: Scheduling with a Nonrestrictive Common Due Date},
 pdfauthor={Nicholas G. Hall; Hans Kellerer; Miao Song},
 pdfsubject={Exact complexity, constant-factor approximation, antithetical ordering, and separating-tree algorithms for AWET},
 pdfkeywords={common due date, earliness, tardiness, scheduling, strong NP-completeness, semidefinite programming, approximation, separable permutation, FPTAS}
}

\begin{document}
\setstretch{1.37931}
\setlength{\baselineskip}{20pt}

\begin{titlepage}
\thispagestyle{empty}
\centering
\vspace*{0.25in}
{\LARGE\bfseries Asymmetric Weighted Earliness-Tardiness:\par}
\vspace{0.22in}
{\Large\bfseries Scheduling with a Nonrestrictive\par}
\vspace{0.08in}
{\Large\bfseries Common Due Date\par}
\vspace{1.4in}
{\Large Nicholas G. Hall\par}
{\large The Ohio State University\par}
\vspace{0.16in}
{\Large Hans Kellerer\par}
{\large University of Graz, Austria\par}
\vspace{0.16in}
{\Large Miao Song\par}
{\large Hong Kong Polytechnic University, Hong Kong\par}
\vspace{2.2in}
{\large \today\par}
\end{titlepage}

\pagenumbering{arabic}
\setcounter{page}{1}
\thispagestyle{plain}
\begin{center}
{\LARGE\bfseries Results at a Glance\par}
\end{center}
\vspace{0.55em}

\begingroup
\small
\setlength{\fboxsep}{4pt}
\noindent\fcolorbox{navy}{softblue}{\parbox{0.93\textwidth}{
\textbf{Part I: Intractability.}
Positive-integer asymmetric \AWET at the boundary nonrestrictive due date is strongly NP-complete through a unary-polynomial exact-cover compiler.}}

\vspace{0.36em}
\noindent\fcolorbox{navy}{softblue}{\parbox{0.93\textwidth}{
\textbf{Part II: General approximation.}
An anchored SDP and deterministic marginal threshold give a polynomial-time $(3+2\sqrt2+\varepsilon)$-approximation for unrestricted \AWET.}}

\vspace{0.36em}
\noindent\fcolorbox{navy}{softblue}{\parbox{0.93\textwidth}{
\textbf{Part III: Antithetical ordering.}
Exact reversal is weakly NP-complete; a two-resource recurrence gives a pseudopolynomial exact algorithm.}}

\vspace{0.36em}
\noindent\fcolorbox{navy}{softblue}{\parbox{0.93\textwidth}{
\textbf{Part IV: Separating-tree analysis.}
For fixed total ratio-order refinements whose ratio permutation is separable, exact separating-tree convolution and geometric trimming give an FPTAS.}}

\vspace{0.36em}
\noindent\fcolorbox{navy}{softgreen}{\parbox{0.93\textwidth}{
\textbf{Part V: Global synthesis.}
The concluding Part integrates the exact-hardness, approximation, weak-hardness, and separating-tree regimes.}}

\vspace{0.42em}
\begin{finding}{Organizing thesis}
Once common-due-date schedule geometry is reduced to the canonical early-set objective, the computational regime is governed by the structure available in the two Smith-ratio orders.  Unrestricted order permits strong NP-completeness and still supports constant-factor approximation; exact reversal produces a weakly hard load-balancing core; and recursive direct/skew compatibility exposes the bounded interface needed for an FPTAS.  The four technical Parts are regimes of one quadratic model.
\end{finding}

\vspace{-0.25em}
\begin{finding}{Four unifying mechanisms}\par\smallskip
\begin{tabularx}{0.96\linewidth}{@{}X X@{}}
\textbf{Canonicalization:} canonical early-set choice $\to$ quadratic objective. &
\textbf{Certificates:} prefix/load squares and SDP pseudo-marginals.\\[2pt]
\textbf{Interfaces:} direct/skew blocks $\to$ four additive coordinates. &
\textbf{Controlled loss:} scale gaps, thresholds, and geometric boxes.\\
\end{tabularx}
\end{finding}
\endgroup

\clearpage
\begin{abstract}
Single-machine asymmetric weighted earliness--tardiness scheduling asks how to sequence jobs around a common synchronization date when early and late completion incur unrelated job-dependent penalties.  At the boundary nonrestrictive date $d=\sum_jp_j$, a compact V-shaped schedule reduces the continuous-time problem to a quadratic choice of a nonempty early set.  We establish four complementary results for this model.  First, the positive-integer problem is strongly NP-complete by a unary-polynomial reduction from Restricted Exact Cover by 3-Sets.  Second, unrestricted \AWET admits a polynomial-time $(3+2\sqrt2+\varepsilon)$-approximation based on an anchored semidefinite relaxation and deterministic marginal thresholding.  Third, when the earliness and tardiness ratio orders are strict reversals, the problem is weakly NP-complete but has an exact two-resource pseudopolynomial dynamic program.  Fourth, for fixed total refinements whose ratio permutation is separable, an exact separating-tree recurrence and coordinated geometric trimming yield an FPTAS.  The proofs use different manifestations of the same canonical objective: scale-separated prefix penalties, positive-semidefinite minimum-kernel covariance, a dominant completed load square, and a bounded four-coordinate decomposition interface.  Together, the results show that the decisive issue is not merely whether the two ratio orders agree, but whether their interaction can be controlled by a global certificate or compressed into a bounded constructive interface.
\end{abstract}

\begingroup\small


\endgroup

\clearpage
\section*{High-level contents}
\setrunningunit{Front matter}
\begin{center}
\small
\begin{tabularx}{0.96\textwidth}{>{\bfseries\raggedright\arraybackslash}p{0.24\textwidth} X r}
\toprule
Unit & Principal content & Page\\
\midrule
Front matter & Purpose, reading paths, literature context, contribution landscape, dependency map, and notation guide & \pageref{sec:global-purpose}\\
Part I: Intractability & Strong NP-completeness through prefix exactness, coded blocks, scale separation, and a global load lock & \pageref{part:intractability}\\
Part II: General Approximation & Anchored SDP, deterministic threshold rounding, and the finite-precision bit-model bridge & \pageref{part:approximation}\\
Part III: Antithetical Ordering & Weak NP-completeness at exact reversal and a pseudopolynomial exact recurrence & \pageref{part:antithetical}\\
Part IV: Separating-Tree Analysis & Exact separating-tree convolution and an FPTAS for fixed refinements with separable ratio permutation & \pageref{part:cotree}\\
Part V: Global Synthesis & Combined methodological landscape, modeling implications, and named open problems & \pageref{part:synthesis}\\
Appendix~\ref{app:algebra} & Canonical-objective algebra & \pageref{app:algebra}\\
Appendix~\ref{app:worked} & Fully worked yes/no certificates for the strong reduction & \pageref{app:worked}\\
Glossary and notation & Cross-Part terminology and symbol definitions & \pageref{app:glossary}\\
References & Bibliography & \pageref{app:references}\\
\bottomrule
\end{tabularx}
\end{center}

\clearpage
\renewcommand{\contentsname}{Detailed contents}
\tableofcontents
\clearpage

\section{Purpose, audience, and reading paths}
\label{sec:global-purpose}
This document contains four full theorem chains: unrestricted strong NP-completeness, a general constant-factor approximation algorithm, weak NP-completeness under strict antithetical ordering, and an FPTAS for separable ratio permutations.  

The intended audience includes researchers and graduate students in scheduling, combinatorial optimization, approximation algorithms, and computational complexity.  Familiarity with NP-completeness, dynamic programming, and elementary semidefinite programming is helpful but not required for the main proof architecture.  Definitions precede the mechanisms that use them, and the longer reductions and algorithms include intuition, proof checkpoints, and worked arithmetic.

There are five natural reading paths.
\begin{description}[style=nextline,leftmargin=4.0cm]
\item[Results path] Read this front matter, the principal theorem statements in Parts I--IV, and Part V.
\item[Hardness path] Read Part I for strong NP-completeness and Part III for the distinct completed-square reduction at exact ratio reversal.
\item[Algorithmic path] Read Part II for the unrestricted constant-factor algorithm and Part IV for the separating-tree FPTAS.  These Parts use different certificates: positive-semidefinite covariance in Part II and bounded additive interfaces in Part IV.
\item[Course path] Read the canonical objective in Section~2, one hardness chain (Part I or III), one algorithmic chain (Part II or IV), the corresponding worked example, and Part V.  This route supports a graduate module without requiring every arithmetic appendix.
\item[Implementation path] Read Sections~12--20 for the anchored-SDP algorithm, Theorem~III.11 for the two-resource recurrence, \cref{CT:sec:exactdp,CT:sec:trimming} for separating-tree tables and trimming, and the worked examples for exact data flow and reconstruction.
\end{description}

\begin{finding}{Reader-facing box contract}
The colored boxes have nonoverlapping roles.  \textcolor{readblue}{\textbf{How to read}} gives navigation or a safe skip; \textcolor{intuitiongreen}{\textbf{Intuition}} explains mechanism; \textcolor{checkpointamber}{\textbf{Proof checkpoint}} records what has been established and what remains; \textcolor{usepurple}{\textbf{How to use}} gives an algorithmic instruction; and \textcolor{auditgray}{\textbf{Arithmetic guide}} gives a reproducibility rule.  A box is used only when it performs one of these functions rather than paraphrasing the adjacent prose.
\end{finding}

\paragraph{Terminological convention.}
A job is \emph{early/on time} when it completes no later than the common due date and \emph{tardy} otherwise.  The superscript $L$ labels the tardiness channel.  The term \emph{strict antithetical} means that, after jobs are indexed by increasing early ratio, their tardiness ratios strictly decrease.  After fixed total tie refinements, the \emph{inversion graph} joins exactly the job pairs whose two ratio orders disagree.  The resulting ratio permutation is \emph{separable} precisely when this inversion graph is a cograph.

\paragraph{Style convention.}
Numbered structural units are capitalized as Parts and Appendices.  Compound modifiers are written consistently as \emph{positive-integer}, \emph{common-due-date}, \emph{ratio-order}, \emph{finite-precision}, and \emph{minimum-kernel}.  A slash is reserved for the intentional combined category \emph{early/on time}; elsewhere the prose uses ordinary conjunctions.

\section{Problem, canonical formulation, and literature context}
\label{sec:global-introduction}

\subsection*{Why the problem matters}
Many production and service systems are organized around a shared synchronization date rather than independent job-specific deadlines.  Components may need to be released together to assembly, orders may be coordinated for a common delivery, or preparatory activities may have to be completed near the start of a downstream stage.  Finishing too early can create holding, storage, financing, obsolescence, or premature-delivery costs; finishing too late can create service failures, contractual penalties, lost production, or disruption to dependent work.  The asymmetric model allows these consequences to differ both from job to job and between the two sides of the target date.  It is deliberately stylized, but it isolates the sequencing tradeoff created by heterogeneous early and tardy consequences; broad discussions of these interpretations appear in the common-due-date surveys~\cite{BakerScudder1990,GordonProthChu2002,RolimNagano2020}.

The boundary choice $d=\sum_jp_j$ is especially informative.  It gives enough aggregate time to process every job by the due date in a compact schedule starting at zero, but it does not make the optimization trivial.  A scheduler may still choose to place some processing after $d$ when the resulting reduction in earliness cost elsewhere outweighs the tardiness incurred.  Because total processing is fixed, every unit placed after the due date reappears as an equal initial offset before the early block.  This conservation relationship leaves a genuine combinatorial question---which jobs should lie on each side---while removing irrelevant calendar slack.

That balance makes the model a useful test case for complexity and approximation.  It is simple enough that exchange arguments determine the optimal order on each side of the due date, reducing the continuous-time problem to one canonical early-set objective.  Yet the remaining set choice is rich enough to display several sharply different regimes: strong NP-completeness without order restrictions, a universal constant-factor approximation, weak NP-completeness under exact reversal of the two ratio orders, and an FPTAS when the ratio interaction has a separable recursive decomposition.

Accordingly, the classifications below answer three complementary questions.  What prevents exact polynomial-time optimization in the general model?  What performance guarantee remains possible without structural assumptions?  Which recognizable forms of order structure restore substantially stronger algorithms?  The answers identify not only complexity boundaries but also the mathematical objects---global certificates and bounded interaction interfaces---that make those boundaries visible.

An \AWET instance consists of jobs $j\in\J$ with positive processing times $p_j$, positive earliness weights $w_j^E$, and positive tardiness weights $w_j^L$.  All jobs are available at time zero and are processed nonpreemptively on one machine.  The common due date is
\[
 d=P:=\sum_{j\in\J}p_j,
\]
and a schedule with completion times $C_j$ has objective
\[
 \sum_{j\in\J}\bigl(w_j^E(d-C_j)^+ + w_j^L(C_j-d)^+\bigr).
\]

The equality $d=\sum_jp_j$ is the \emph{boundary nonrestrictive due date}.  It permits a compact schedule to finish exactly at $d$, but it does not eliminate the early--tardy assignment decision: placing processing after $d$ creates the same amount of initial offset before the early block.  This conservation identity turns schedule geometry into the quadratic set objective used throughout the tutorial.

Write
\[
 r_j=\frac{w_j^E}{p_j},\qquad s_j=\frac{w_j^L}{p_j}.
\]
The canonical-schedule theorem says that an optimum may be chosen compact and V-shaped.  Consequently, for every nonempty early set $E$ and tardy complement $T=\J\setminus E$, the associated due-date-aligned canonical schedule has cost
\begin{equation}
\label{eq:front-canonical}
 \Phi(E)=
 \sum_{\{i,j\}\subseteq E}p_ip_j\min\{r_i,r_j\}
 +\sum_{\{i,j\}\subseteq T}p_ip_j\min\{s_i,s_j\}
 +\sum_{i\in T}p_i^2s_i.
\end{equation}
Moreover, the global optimum is $\OPT=\min_{\varnothing\ne E\subseteq\J}\Phi(E)$.  This formula is the common backbone of all four technical Parts.  Part I compiles an exact-cover signal into it; Part II relaxes and rounds it; Part III shapes it into an equal-load square; and Part IV decomposes it across separating-tree blocks.

\paragraph{Why the normalization \texorpdfstring{$d=\sum_jp_j$}{d = total processing time} is used.}
Put $P=\sum_jp_j$.  The familiar condition $d\ge P$ is a simple sufficient condition for the fixed common due date to be nonrestrictive~\cite{KellererRustogiStrusevich2020}: every compact V-shaped candidate has total processing span $P$, and subtracting $d-P$ from both the due date and all completion times preserves every earliness and tardiness value while keeping all start times nonnegative.  Thus every instance satisfying $d\ge P$ is equivalent, for the present objective, to the normalized case $d=P$.

The condition $d\ge P$ is sufficient rather than necessary.  Let $\Phi(E)$ be the canonical early-set objective in \eqref{eq:front-canonical}.  A sharper instance-dependent sufficient threshold is
\[
 d_{\mathrm{nr}}(I)
 :=\min\left\{\sum_{j\in E}p_j:
 E\in\arg\min_{\varnothing\ne A\subseteq\J}\Phi(A)\right\}.
\]
If $d\ge d_{\mathrm{nr}}(I)$, at least one optimal canonical early arm fits before time $d$, so the due date does not change the unconstrained optimum.  This criterion is less concise because it depends on the instance and on an optimal side assignment; the uniform normalization $d=P$ avoids carrying that threshold through every theorem.  When the due date is early enough to constrain placement of the V-shaped schedule, the problem is in the restrictive regime studied by Hall, Kubiak, and Sethi~\cite{HallKubiakSethi1991}.

\subsection{Common-due-date structure and exact complexity}
Early common-due-date models established the V-shaped structural viewpoint and the importance of distinguishing nonrestrictive and restrictive due-date regimes.  Kanet studied total deviation about a common due date, while Panwalkar, Smith, and Seidmann treated common due-date assignment~\cite{Kanet1981,PanwalkarSmithSeidmann1982}.  Hall and Posner analyzed weighted symmetric earliness--tardiness at a nonrestrictive common due date; Hall, Kubiak, and Sethi treated the companion restrictive setting~\cite{HallPosner1991,HallKubiakSethi1991}.  Broad reviews include Baker and Scudder, Gordon, Proth, and Chu, and Rolim and Nagano~\cite{BakerScudder1990,GordonProthChu2002,RolimNagano2020}.

For the weighted symmetric nonrestrictive model, Hall and Posner proved weak NP-hardness and gave a pseudopolynomial dynamic program~\cite{HallPosner1991}.  Hoogeveen and van de Velde recorded the strong-complexity boundary for general asymmetric weights as unresolved~\cite{HoogeveenVandeVelde1997}.  Part I resolves that published question in the exact formulation used here: positive-integer asymmetric weighted earliness--tardiness scheduling at the boundary due date $d=\sum_jp_j$ is strongly NP-complete.

\subsection{Pseudopolynomial algorithms, FPTAS results, and approximation}
Weak hardness in common-due-date scheduling has long coexisted with pseudopolynomial algorithms and approximation schemes in symmetric or otherwise structured settings.  The Hall--Posner recurrence is an early example; later FPTAS results include the schemes of Kovalyov and Kubiak and Kellerer, Rustogi, and Strusevich~\cite{HallPosner1991,KovalyovKubiak1999,KellererRustogiStrusevich2020}.  These positive results depend on structure absent from unrestricted asymmetric data.

Part II therefore addresses a different boundary.  It permits unrelated positive earliness and tardiness weights and obtains a universal constant-factor guarantee rather than an FPTAS.  Part III then identifies a strict reversed-order class whose hardness is weak and whose state space is pseudopolynomial, while Part IV obtains an FPTAS under a recursive order-decomposition condition.

\subsection{Separable-permutation decomposition and semidefinite methodology}
For fixed total refinements of the two ratio orders, Part IV represents their interaction by a ratio permutation.  Separable permutations admit recursive direct and skew sums and can be recognized, with a separating tree recovered, in linear time~\cite{BoseBussLubiw1998}.  Equivalently, their inversion graphs are cographs, whose union/join decomposition is represented by the corresponding cotree~\cite{CorneilPerlStewart1985,HabibPaul2010}.

Part II uses semidefinite optimization in a different way.  Polynomial-time weak optimization and rational bit-model control are standard parts of semidefinite programming~\cite{Alizadeh1995,GrotschelLovaszSchrijver1988,VandenbergheBoyd1996}.  Semidefinite relaxations are also a classical source of approximation algorithms in combinatorial optimization~\cite{GoemansWilliamson1995}.  The rounding here is not hyperplane rounding: it uses an anchored local-consistency relaxation, deterministic marginal thresholding, and a minimum-kernel covariance inequality tailored to the scheduling objective.

\begin{quote}\small
\noindent\textcolor{checkpointamber}{\textbf{Provenance note.}}
The tutorial proves every stated result.  Historical priority is claimed only where the comparison with published literature is documented explicitly; otherwise the theorem and its scope are stated without a priority claim over unpublished or independent work.
\end{quote}

\section{Contribution map and scope}
\label{sec:front-map}
\begin{center}
\scriptsize
\begin{tabularx}{0.99\textwidth}{>{\bfseries\raggedright\arraybackslash}p{0.045\textwidth} >{\raggedright\arraybackslash}p{0.205\textwidth} >{\raggedright\arraybackslash}p{0.225\textwidth} >{\raggedright\arraybackslash}p{0.245\textwidth} X}
\toprule
Part & Closest established boundary & New result here & Key mechanism & Not implied\\
\midrule
I & Symmetric nonrestrictive weights were weakly NP-hard and pseudopolynomially solvable; published tables left the general asymmetric strong boundary open. & Positive-integer \AWET at $d=\sum_jp_j$ is strongly NP-complete and unary NP-hard. & Effective-atom profile compression, separated prefix-square signal, and an anchor-coupled minimum-kernel load lock. & No APX-hardness or no-PTAS conclusion follows from strong NP-completeness alone.\\
II & FPTAS results were available in symmetric or structured common-due-date settings. & Unrestricted asymmetric \AWET has a polynomial-time $(3+2\sqrt2+\varepsilon)$-approximation. & Anchored SDP, deterministic marginal thresholding, local early-side consistency, and PSD tardy-side covariance payment. & The theorem does not establish a PTAS or optimality of the constant.\\
III & The unrestricted problem is strongly NP-complete, but exact reversal had not been classified by the Part I construction. & Strict antithetical \AWET is weakly NP-complete and has an $O(nW_EP)$ pseudopolynomial recurrence. & Binary pair tagging, a dominant completed load square, and a uniform residual budget. & Weak hardness does not rule out an FPTAS; exact reversal is separable and belongs to Part IV's structured class after fixed refinements.\\
IV & Separable permutations have recursive direct/skew decompositions and linear-time recognition. & For fixed total refinements of the two ratio orders, if the resulting ratio permutation is separable, \AWET admits an FPTAS. & A four-coordinate nonnegative bilinear interface, exact separating-tree convolution, and coordinated geometric trimming. & Separability is sufficient for this algorithm, not claimed necessary for every approximation scheme.\\
\bottomrule
\end{tabularx}
\end{center}
For the Part~III running-time bound, $n=|\J|$ is the number of jobs,
$W_E=\sum_{j\in\J}w_j^E$ is the total earliness weight, and
$P=\sum_{j\in\J}p_j$ is the total processing time.

\subsection*{Why these results belong together}
Each technical Part begins from the same canonical quadratic objective but asks a different question of it.  Part I asks how much information the objective can encode; Part II asks how much of its value can be certified by a relaxation; Part III asks what exact reversal does to both hardness and state-space dimension; and Part IV asks when order decomposition compresses all future interactions into a bounded interface.  This shared starting point is the main reason the four results are best read together.

\begin{figure}[ht]
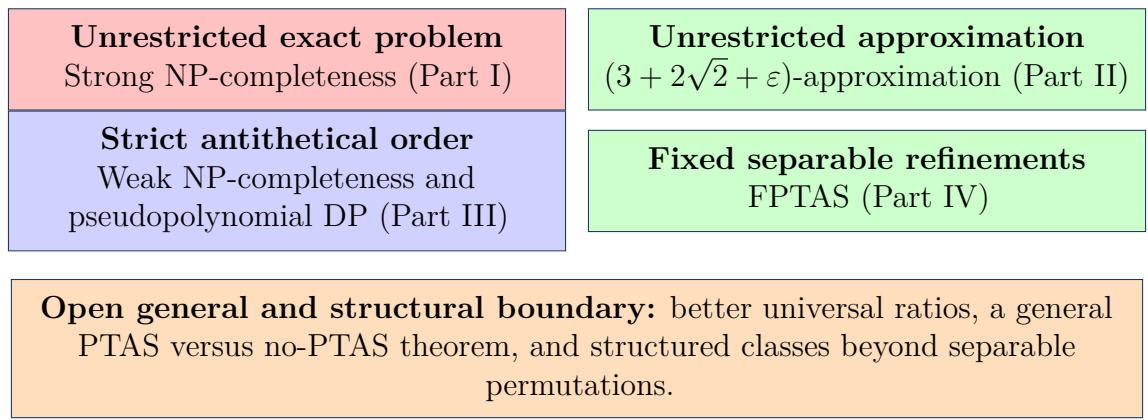

\centering
\setlength{\fboxsep}{6pt}
\begin{tabular}{@{}c@{\hspace{0.7em}}c@{}}
\fcolorbox{navy}{figurehard}{\parbox{0.42\textwidth}{\centering\textbf{Unrestricted exact problem}\\Strong NP-completeness (Part I)}} &
\fcolorbox{navy}{figurealgorithm}{\parbox{0.42\textwidth}{\centering\textbf{Unrestricted approximation}\\$(3+2\sqrt2+\varepsilon)$-approximation (Part II)}}\\[0.8em]
\fcolorbox{navy}{figuremixed}{\parbox{0.42\textwidth}{\centering\textbf{Strict antithetical order}\\Weak NP-completeness and pseudopolynomial DP (Part III)}} &
\fcolorbox{navy}{figurealgorithm}{\parbox{0.42\textwidth}{\centering\textbf{Fixed separable refinements}\\FPTAS (Part IV)}}
\end{tabular}

\vspace{0.8em}
\fcolorbox{navy}{figureopen}{\parbox{0.88\textwidth}{\centering
\textbf{Open general and structural boundary:} better universal ratios, a general PTAS versus no-PTAS theorem, and structured classes beyond separable permutations.}}
\caption{Front-matter landscape: four proved regimes of the canonical early-set objective and the principal open boundary.  As in the combined landscape, red marks exact intractability, green marks approximation algorithms, blue marks the mixed weak-hardness/pseudopolynomial regime, and amber marks unresolved boundaries.}
\label{fig:front-landscape}
\end{figure}

\clearpage
\subsection*{Dependency map}
\begin{figure}[ht]
\centering
\fcolorbox{navy}{figuremixed}{\parbox{0.90\textwidth}{\centering
\textbf{Shared formal prerequisite:} the canonical compact schedule and early-set objective~\eqref{eq:front-canonical}.}}

\vspace{0.8em}
\footnotesize
\begin{tabularx}{0.96\textwidth}{@{}>{\bfseries\raggedright\arraybackslash}p{0.10\textwidth} X@{}}
\toprule
Part & Logical dependency chain from the shared prerequisite\\
\midrule
I & $\Longrightarrow$ \RXC $\to$ \RPE (Lemmas~I.8--I.9) $\Longrightarrow$ coded blocks and scale separation $\Longrightarrow$ load lock (Lemma~I.15) $\Longrightarrow$ master decomposition (Lemma~I.16) $\Longrightarrow$ Theorem~I.2.\\[3pt]
II & $\Longrightarrow$ anchored SDP (Definition~II.9) $\Longrightarrow$ early and tardy payment (Lemmas~II.12--II.16) $\Longrightarrow$ rational realization (Lemma~II.20) $\Longrightarrow$ threshold algorithm $\Longrightarrow$ Theorem~II.3.\\[3pt]
III & $\Longrightarrow$ binary tagging (Lemma~III.3) $\Longrightarrow$ cost identity (Lemma~III.6) and residual bound $\Longrightarrow$ anchor and gap separation $\Longrightarrow$ weak NP-completeness; independently, the two-resource recurrence gives Theorem~III.11.\\[3pt]
IV & $\Longrightarrow$ fixed refinements and ordered separating tree $\Longrightarrow$ bilinear convolution (Lemma~IV.5) $\Longrightarrow$ exact DP (Theorem~IV.7) $\Longrightarrow$ representative invariant $\Longrightarrow$ FPTAS (Theorem~IV.11).\\
\bottomrule
\end{tabularx}
\caption{Logical dependency map.  The arrows denote formal dependence, not chronology.  \Cref{lem:separation} proves Part I's scale separation; Appendix~\ref{app:algebra} supplies the coefficient algebra, and Appendix~\ref{app:worked} the full yes/no arithmetic.}
\label{fig:dependency-map}
\end{figure}

\subsection*{Notation at a glance}
\begin{center}
\footnotesize
\begin{tabularx}{0.98\textwidth}{>{\raggedright\arraybackslash}p{0.17\textwidth} X >{\raggedright\arraybackslash}p{0.17\textwidth} >{\raggedright\arraybackslash}p{0.20\textwidth}}
\toprule
Symbol & Meaning & Scope & First definition\\
\midrule
$p_j,w_j^E,w_j^L$ & processing time, earliness weight, and tardiness weight & global & Section~2\\
$d$ & common due date; in the main results, the boundary nonrestrictive date $d=\sum_jp_j$ & global & Section~2\\
$r_j,s_j$ & early and tardiness Smith ratios & global & Section~2\\
$E,T$ (or $\mathcal T$) & early/on-time and tardy job sets & global; Part I uses $\mathcal T$ to avoid the scale $T$ & \eqref{eq:front-canonical}\\
$N,G,M,T,H,L$ & scale hierarchy, load-lock parameters, or kernel levels & reset in Part I & notation reset before Section~4\\
$W_E$ & total earliness weight $\sum_{j\in\J}w_j^E$ & Parts III--IV & Section~3\\
$A,B,\Gamma,R$ & master constant, decision threshold, residual budget, and residual & Part I & Section~8.1.9\\
$\mathcal A,\mathcal B,\mathcal U$ & SDP early-pair, tardy-pair, and self-cost accounts & Part II & \eqref{AA:eq:ABC}\\
$e_i,t_i,e_{ij},t_{ij}$ & SDP pseudo-marginals and pseudo-joint masses & Part II & Section~14\\
$Q,S,U,M,D$ & tagged total, selected load, residual budget, dominant scale, and due-date scale & Part III & notation reset before \cref{ANTI:sec:architecture}\\
$\prec_E,\prec_L,\pi$ & fixed ratio-order refinements and ratio permutation & Part IV & \cref{CT:sec:orders}\\
$z_M=(e_M,t_M,u_M,v_M)$ & additive separating-tree signature & Part IV & Definition~IV.4\\
$\lambda,H$ & geometric trimming ratio and separating-tree height & Part IV & \cref{CT:sec:trimming}\\
\bottomrule
\end{tabularx}
\end{center}

\begin{howtoread}
The contribution table separates theorem novelty from proof-technique novelty.  The dependency map shows which formal statements are needed, while the notation table shows where symbols are reset.  Readers may now enter any Part without treating the document as one linear 200-page proof.
\end{howtoread}

\clearpage
\part{Intractability}
\setrunningunit{Part I: Hardness}
\thispagestyle{plain}
\label{part:intractability}
\begin{center}
{\Large\bfseries Strong NP-Completeness of Single-Machine Asymmetric Weighted Earliness--Tardiness Scheduling with a Common Due Date\par}
\end{center}
\medskip
\begin{quote}\small
\noindent\textbf{Part synopsis.}
This Part proves strong, and hence unary, NP-completeness of positive-integer \AWET.  Restricted Exact Cover by 3-Sets is first converted to regular prefix exactness.  Each source set is then represented by an eight-job block with exactly two balanced states.  Separated early-ratio bands turn those states into squared prefix-cover errors; an anchor-coupled minimum-kernel lock enforces the block loads; and integer tardiness offsets correct the remaining logical fields.  The final objective is an exact constant plus a dominant prefix penalty and a bounded residual.
\end{quote}

\begin{partposition}
\textbf{Established boundary.}  The symmetric-weight nonrestrictive model was weakly NP-hard and pseudopolynomially solvable, while the published strong-complexity boundary for general asymmetric weights remained open.  \textbf{Result.}  Positive-integer \AWET at the boundary due date $d=\sum_jp_j$ is strongly NP-complete by a unary-polynomial reduction.  \textbf{Technique.}  Effective-atom profile compression and separated ratio gaps carry the logical signal, while an anchor-coupled minimum-kernel quadratic enforces all block loads simultaneously. 
\end{partposition}

\begin{notationreset}
Here $n$ is the size of the regular exact-cover and prefix-exactness source, $N=n+1$ controls the scale hierarchy, and $m$ denotes the number of constructed jobs when needed.  The global job data remain $p_j,w_j^E,w_j^L$, with ratios $r_j=w_j^E/p_j$ and $s_j=w_j^L/p_j$.  Scale symbols such as $G,M,T,H$, and $L$ are local to this Part.
\end{notationreset}

\begin{proofcontract}
\begin{itemize}[leftmargin=2em,itemsep=1pt,topsep=2pt]
\item \textbf{Input:} an $n$-set Restricted Exact Cover by 3-Sets instance.
\item \textbf{Output:} an \AWET\ instance with eight state jobs per source set and one anchor, together with a threshold $B$.
\item \textbf{Forward invariant:} an exact cover gives balanced block states, zero prefix penalty, and a canonical schedule of cost at most $B$.
\item \textbf{Reverse invariant:} every schedule of cost at most $B$ can be replaced without higher cost by an optimal canonical schedule; that canonical schedule has an early anchor, satisfies the load lock, decodes to block states, and annihilates all prefix-cover errors.
\item \textbf{Size obligation:} every constructed integer is bounded by a fixed polynomial (the written audit uses at most $N^{650}$), so unary length is polynomial.
\end{itemize}
\end{proofcontract}

\section*{Proof navigation for Part I}
Part I has two coordinated lanes.
\begin{center}
\small
\begin{tabularx}{0.96\textwidth}{>{\bfseries\raggedright\arraybackslash}p{0.17\textwidth} X X}
\toprule
Lane & Main chain & Audit support\\
\midrule
Logical lane & \RXC $\to$ prefix equations $\to$ two coded states per source set $\to$ zero prefix error $\to$ exact cover & Appendix~\ref{app:algebra} tracks the canonical-objective coefficients; Appendix~\ref{app:worked} instantiates both directions.\\
Numerical lane & scale hierarchy $\to$ ratio-order separation $\to$ anchor-coupled load lock $\to$ residual domination $\to$ integer decision gap & \Cref{lem:separation} proves the separation inequalities; Appendix~\ref{app:worked} reproduces every scale, field, and threshold value.\\
\bottomrule
\end{tabularx}
\end{center}
The two lanes meet in the master identity.  A theorem-level reader may follow the logical lane and consult only the stated numerical lemmas; an arithmetic-audit reader may reproduce the supporting appendices without changing the logical proof.  The two fully instantiated numerical companions are \cref{app:worked-yes,app:worked-no}: the first follows the forward construction through the 49-job schedule and threshold, while the second repeats the compiler for a closely related no-instance and audits the reverse contradiction.

\section{The \AWET\ problem and historical context}
\label{sec:problem}

\subsection{Formal definition}

A schedule specifies a permutation of the jobs and a nonnegative start time for each job, subject to nonoverlap and nonpreemption.  Because there is only one machine and no release dates, the only combinatorial decisions are the order, the placement of idle time, and which completion lies on each side of the common due date.

\begin{definition}[Asymmetric weighted earliness--tardiness scheduling]
An instance of \AWET\ consists of a set $\J=\{1,\ldots,m\}$ of jobs. Job $j$ has an integer processing time $p_j>0$, an integer earliness weight $w_j^E>0$, and an integer tardiness weight $w_j^L>0$. The jobs must be processed nonpreemptively on one continuously available machine, which can process at most one job at a time. The machine is available from time zero, and all start times must be nonnegative. The common due date is fixed at
\[
 d=\sum_{j\in\J}p_j.
\]
If $C_j$ is the completion time of job $j$, define
\[
 E_j=(d-C_j)^+,
 \qquad
 T_j=(C_j-d)^+,
\]
where $x^+=\max\{x,0\}$. The cost of a schedule is
\[
 F=\sum_{j\in\J}\bigl(w_j^E E_j+w_j^L T_j\bigr).
\]
The decision version of \AWET\ additionally contains an integer bound $B$ and asks whether a feasible schedule of cost at most $B$ exists.
\end{definition}

The terminology ``asymmetric'' refers to the absence of any required relation between $w_j^E$ and $w_j^L$.  A job with $C_j<d$ is \emph{early}; a job with $C_j=d$ has zero deviation and will be grouped with the early jobs; and a job with $C_j>d$ is \emph{tardy}.  We use \AWET\ throughout for both the optimization problem and, when a bound is present, its decision version. Baker and Scudder~\cite{BakerScudder1990} provide a broad review of single-machine sequencing with earliness and tardiness penalties.

The choice $d=\sum_jp_j$ is important.  It equals the length of a compact schedule containing all jobs.  Thus, if a compact sequence starts at time zero, its final job completes exactly at $d$; if some jobs are placed after $d$, the same total amount of processing must appear as an initial offset before the early block.  This conservation identity is used repeatedly in the set formulation and in the numerical schedule.

\begin{figure}[ht]
\centering
\begin{tikzpicture}[x=1cm,y=1cm,>=Latex,font=\small]
\draw[->,thick] (0,0) -- (12.2,0) node[right]{time};
\draw[very thick,navy] (1.0,0) -- (7.0,0);
\draw[very thick,checkpointamber] (7.0,0) -- (11.0,0);
\draw[thick] (7.0,-0.35) -- (7.0,0.55) node[above]{$d$};
\node[above=3pt] at (4.0,0) {early block};
\node[above=3pt] at (9.0,0) {tardy block};
\node[below=6pt,align=center] at (4.0,0) {chronological ratios $r_j$ increase\\toward the due date};
\node[below=6pt,align=center] at (9.0,0) {chronological ratios $s_j$\\decrease away from the due date};
\end{tikzpicture}
\caption{Geometry of a compact canonical schedule.  The early block ends at the common due date and the tardy block starts immediately afterward.}
\label{fig:schedule-geometry}
\end{figure}
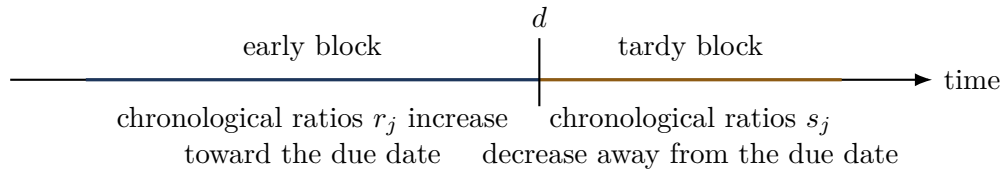

For complexity terminology, \emph{strong NP-hardness} means that NP-hardness persists even when every numerical value is bounded by a polynomial in the combinatorial size of the source instance.  A reduction with that property also remains polynomial when the target numbers are written in unary.  The proof below therefore establishes both strong NP-completeness and unary NP-hardness.

Hall and Posner~\cite{HallPosner1991} considered the symmetric-weight case $w_j^E=w_j^L$ with a nonrestrictive common due date. They proved weak NP-hardness and gave a pseudopolynomial dynamic program. Hoogeveen and van de Velde~\cite{HoogeveenVandeVelde1997} explicitly identified strong NP-hardness of the common-due-date problem with general weights as unresolved; their complexity table lists the general case as open. This tutorial resolves the published strong-complexity question for positive-integer asymmetric weighted earliness--tardiness scheduling at the boundary due date $d=\sum_jp_j$.

\subsection{Main result}

\begin{theorem}\label{thm:main}
The decision version of \AWET\ is strongly NP-complete. In particular, \AWET\ is unary NP-hard.
\end{theorem}

Membership in NP follows from \cref{lem:compact,prop:setobjective}.  If any schedule has cost at most the decision bound, then a globally optimal canonical schedule also has cost at most that bound.  A certificate may specify its nonempty early set and the canonical orders on the two sides of $d$; ratio comparisons use cross multiplication, and all start times, completion times, and objective terms are obtained by polynomially many integer additions and multiplications.  The remainder of Part~I proves strong NP-hardness.

\clearpage

\section{Overview of the proof}
\label{sec:overview}

The proof follows the two reductions in \cref{fig:chain}.  Reduction I rewrites Restricted Exact Cover by 3-Sets as an equivalent system of prefix equations.  Reduction II compiles that system into a scheduling instance.  Each source set becomes an eight-job block consisting of three occurrence pairs and one controller pair.  The two balanced orientations of the block encode one binary variable $x_v$.

The scheduling construction then performs three tasks.  First, an effective atom summarizes each complete two-job pair for cross-block calculations; its position places occurrence information in an element-specific early-ratio band, and the gaps between bands produce squared prefix errors.  Second, a field correction centers those squares at the required prefix values.  Third, an anchor and a positive-definite minimum-kernel penalty force every block of an optimal canonical assignment to use one of its two balanced orientations.  The decision threshold then separates zero prefix error from every nonzero prefix error.  Keeping these tasks separate is the main organizational principle of the proof.

\begin{finding}{Main-line map: logic and arithmetic}
\begin{description}[style=nextline,leftmargin=2.7cm,itemsep=2pt]
\item[Logic] source exact cover $\to$ prefix equalities $\to$ block spins $\to$ signed prefix totals $\to$ zero squared error.
\item[Arithmetic] positional digits $\to$ separated ratio bands $\to$ positive-definite load lock $\to$ bounded within-band residual $\to$ strict threshold gap.
\end{description}
\Cref{lem:separation} certifies the band and gap inequalities; Appendix~\ref{app:algebra} certifies the field and coupling algebra; Appendix~\ref{app:worked} certifies one full yes-instance and one full no-instance.
\end{finding}

\begin{figure}[t]
\centering
\resizebox{0.98\textwidth}{!}{%
\begin{tikzpicture}[
  node distance=4.2cm,
  every node/.style={font=\small},
  box/.style={draw,rounded corners,minimum width=3.35cm,minimum height=1.05cm,align=center,fill=figurehard},
  arr/.style={-{Latex[length=3mm]},thick}
]
\node[box] (rxc) {Restricted Exact Cover\\by 3-Sets (\RXC)};
\node[box,right=of rxc] (rpe) {Regular Prefix\\Exactness (\RPE)};
\node[box,right=of rpe] (awet) {Asymmetric weighted\\earliness--tardiness (\AWET)};
\draw[arr] (rxc) -- node[above,align=center]{prefix\\differences} (rpe);
\draw[arr] (rpe) -- node[above,align=center]{coded-block compiler} node[below,align=center]{load lock + prefix gaps} (awet);
\end{tikzpicture}%
}
\caption{Reduction chain. Both transformations are polynomial. Every numerical value in the second transformation is bounded by a fixed polynomial in the \RPE\ instance size.}
\label{fig:chain}
\end{figure}
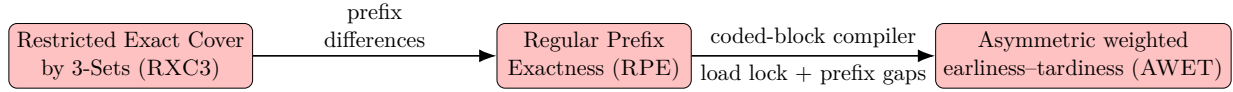

\section{A canonical set formulation of \AWET}
\label{sec:structure}

For each job, define its earliness and tardiness Smith ratios by
\[
 r_j=\frac{w_j^E}{p_j},
 \qquad
 s_j=\frac{w_j^L}{p_j}.
\]
These ratios determine the canonical order once the early-job set is fixed.  Integrality is not needed in this section; it is imposed by the reduction later.

\begin{lemma}[Compact canonical schedule]\label{lem:compact}
Every \AWET\ instance has an optimal schedule with the following properties:
\begin{enumerate}[label=(\roman*)]
\item there is no idle time between consecutive jobs;
\item at least one job completes exactly at $d$;
\item read outward from $d$, the jobs completing no later than $d$ are in nonincreasing order of $r_j$;
\item read outward from $d$, the tardy jobs are in nonincreasing order of $s_j$.
\end{enumerate}
\end{lemma}

\begin{proof}
We establish the four properties in order.

\par\medskip\noindent\textbf{Step 1: remove internal idle time.}\quad
Choose an optimal schedule with minimum idle time between its first and last jobs.  An idle interval wholly before $d$ can be closed by shifting the block immediately to its left rightward; every affected job remains early and becomes no more costly.  The symmetric left shift closes an interval wholly after $d$.  If an interval contains $d$, shift the adjacent blocks toward $d$.  Each operation weakly decreases cost, so an optimal schedule with minimum internal idle time has none.  This proves (i).

\par\medskip\noindent\textbf{Step 2: put a completion at the due date.}\quad
Fix a compact job sequence and let its first job start at $z\ge0$.  Because the sequence has total length $d$, each completion time has the form $C_j(z)=z+\widehat C_j$.  The total cost is a convex piecewise-linear function of $z$, being a sum of terms
\[
 w_j^E(d-C_j(z))^+ + w_j^L(C_j(z)-d)^+.
\]
A minimum on $[0,\infty)$ may be chosen at $z=0$ or at a breakpoint.  At $z=0$ the last job completes at $d$; at a breakpoint some job completes at $d$.  An endpoint of any flat minimum interval has the same property.  This proves (ii).

\par\medskip\noindent\textbf{Step 3: order the early jobs.}\quad
Let adjacent early job $i$ be farther from $d$ than adjacent early job $j$.  In the order $i,j$, the order-dependent contribution is
\[
 w_i^Ep_j=p_ip_jr_i;
\]
in the reverse order it is $p_ip_jr_j$.  The smaller ratio must therefore be farther from $d$.  Repeated adjacent exchanges prove (iii).

\par\medskip\noindent\textbf{Step 4: order the tardy jobs.}\quad
Let tardy job $i$ be closer to $d$ than tardy job $j$.  In the order $i,j$, the order-dependent contribution is
\[
 w_j^Lp_i=p_ip_js_j;
\]
in the reverse order it is $p_ip_js_i$.  Again the smaller ratio belongs farther from $d$, and adjacent exchanges prove (iv).
\end{proof}

\begin{intuition}
The lemma removes the continuous part of the scheduling decision.  Once the early-job set is chosen, both sides are packed around $d$ and ordered by their ratios.  The reduction may therefore optimize over job subsets rather than over arbitrary start times.
\end{intuition}

Let $E$ be the set of jobs completing no later than $d$.  To distinguish the tardy-job set from the numerical scale $T$ introduced in \cref{sec:red2}, write $\T=\J\setminus E$.  For a prescribed $E$, the ratio rules determine a due-date-aligned canonical schedule up to ties.  The structural lemma guarantees that at least one global optimum is of this form.

\begin{proposition}[Canonical set objective]\label{prop:setobjective}
For every nonempty $E\subseteq\J$, put $\T=\J\setminus E$.  Arrange the jobs of $E$ before $d$ in the early ratio order of \cref{lem:compact}, with the last early job completing at $d$, and arrange the jobs of $\T$ immediately after $d$ in the corresponding tardy ratio order.  The cost of this feasible due-date-aligned canonical schedule is
\begin{align}
 \Phi(E)
 ={}&\sum_{\{i,j\}\subseteq E}p_ip_j\min\{r_i,r_j\}
 \label{eq:setobj}\tag{SO}\\
 &+\sum_{\{i,j\}\subseteq \T}p_ip_j\min\{s_i,s_j\}
 +\sum_{i\in \T}p_i^2s_i.
 \nonumber
\end{align}
Moreover,
\begin{equation}
 \OPT=\min_{\varnothing\ne E\subseteq\J}\Phi(E).
 \label{eq:setopt}
\end{equation}
\end{proposition}

\begin{proof}
Fix a nonempty set $E$ and put $\T=\J\setminus E$.

\par\medskip\noindent\textbf{Feasibility.}\quad
Place the canonical early block so that its last job completes at $d$, followed immediately by the canonical tardy block.  The first start time is
\[
 d-\sum_{i\in E}p_i
 =\sum_{i\in \T}p_i\ge0,
\]
because $d=\sum_{j\in\J}p_j$.  Thus this due-date-aligned canonical schedule is feasible.

\par\medskip\noindent\textbf{Its cost.}\quad
For an unordered pair $\{i,j\}\subseteq E$, the processing time of the job closer to $d$ contributes to the earliness of the job farther from $d$.  The exchange order places the smaller early ratio farther away, so the pair contributes
\[
 p_ip_j\min\{r_i,r_j\}.
\]
There is no early self term because a job's own processing time ends at its completion.  An unordered tardy pair similarly contributes $p_ip_j\min\{s_i,s_j\}$.  Every tardy job also includes its own processing time in its tardiness, giving the self term
\[
 w_i^Lp_i=p_i^2s_i.
\]
Summing these contributions proves \eqref{eq:setobj}.

\par\medskip\noindent\textbf{Global optimality identity.}\quad
Let $V=\min_{\varnothing\ne E\subseteq\J}\Phi(E)$.  Every schedule just constructed is feasible, so $\OPT\le V$.  Conversely, \cref{lem:compact} supplies a globally optimal schedule with no internal idle time, at least one completion exactly at $d$, and the two canonical ratio orders.  Let $E^*$ be its jobs completing no later than $d$.  Because a completion occurs at $d$ and there is no internal idle time, the early block ends at $d$ and the tardy block starts immediately afterward; no job straddles the due date.  The preceding pair calculation therefore gives the schedule's cost as $\Phi(E^*)$.  Hence
\[
 \OPT=\Phi(E^*)\ge V.
\]
Together with $\OPT\le V$, this proves \eqref{eq:setopt}.
\end{proof}

\begin{proofcheckpoint}
Equation \eqref{eq:setobj} is the exact cost of each due-date-aligned canonical schedule, and \eqref{eq:setopt} says that minimizing these values gives the true scheduling optimum.  The first term depends only on which jobs are early, and the other two only on which jobs are tardy.  This is the point at which the scheduling instance becomes a discrete quadratic optimization problem.
\end{proofcheckpoint}

Encode the early--tardy choice by a spin
\[
 c_i=\begin{cases}
 +1,&i\in E,\\
 -1,&i\in \T.
 \end{cases}
\]
Then $\1\{i\in E\}=(1+c_i)/2$ and $\1\{i\in\T\}=(1-c_i)/2$.  A product of two indicators expands into a constant, two linear terms, and one quadratic term.  Below, these are called the constant term, logical fields, and logical couplings.

\section{Reduction I: \RXC\ to regular prefix exactness}
\label{sec:red1}

Gonzalez~\cite{Gonzalez1985} proved that Restricted Exact Cover by 3-Sets is NP-complete.

\begin{definition}[\RXC]\label{def:rxc}
An instance consists of a universe $X=\{e_1,\ldots,e_n\}$ and a family $\mathcal S=\{S_1,\ldots,S_n\}$.  Every $S_v$ has three elements, and every element belongs to three sets.  Its incidence matrix is
\[
 A^{\mathrm{inc}}=(a_{kv})\in\{0,1\}^{n\times n},
 \qquad
 a_{kv}=
 \begin{cases}
 1,&e_k\in S_v,\\
 0,&e_k\notin S_v.
 \end{cases}
\]
Rows represent elements and columns represent sets.  A vector $x\in\{0,1\}^n$ selects column $v$ when $x_v=1$.  The selected sets form an \emph{exact cover} exactly when
\[
 A^{\mathrm{inc}}x=\1,
 \qquad\text{equivalently}\qquad
 \sum_{v=1}^n a_{kv}x_v=1
 \quad(k=1,\ldots,n).
\]
The decision question is whether such a vector exists.
\end{definition}

Each row and each column contains three ones, so counting incidences gives $|\mathcal S|=|X|$; we use $n$ for both.  The intermediate problem keeps the same feasible vectors but replaces the individual row equations by cumulative prefix equations.  This form matches the intervals that will arise on the scheduling ratio axis.

\begin{definition}[Regular Prefix Exactness]\label{def:rpe}
An instance of \RPE\ is an ordered universe $X=\{e_1,\ldots,e_n\}$ together with the same three-uniform, three-regular family $\mathcal S=\{S_1,\ldots,S_n\}$.  Define
\[
 d_{v,k}=\sum_{\ell=1}^k a_{\ell v}
          =\abs{S_v\cap\{e_1,\ldots,e_k\}},
 \qquad
 C_k(x)=\sum_{v=1}^n d_{v,k}x_v,
\]
\RPE\ asks whether there is $x\in\{0,1\}^n$ such that
\[
 C_k(x)=k
 \qquad(k=1,\ldots,n).
\]
\end{definition}

Here $d_{v,k}$ counts the entries equal to one among the first $k$ rows of column $v$, and $C_k(x)$ counts selected incidences in those rows.  Thus $C_k(x)=k$ says that the first $k$ elements contain exactly $k$ selected incidences.

\begin{lemma}[Prefix-difference identity]\label{lem:prefixdifference}
For every binary vector $x$ and every $k=1,\ldots,n$, with $C_0(x)=0$,
\[
 C_k(x)-C_{k-1}(x)=\sum_{v:e_k\in S_v}x_v.
\]
\end{lemma}

\begin{proof}
Adding row $e_k$ changes the cumulative value of column $v$ by
\[
 d_{v,k}-d_{v,k-1}=a_{kv}.
\]
Hence
\[
\begin{aligned}
C_k(x)-C_{k-1}(x)
&=\sum_v\bigl(d_{v,k}-d_{v,k-1}\bigr)x_v\\
&=\sum_v a_{kv}x_v\\
&=\sum_{v:e_k\in S_v}x_v,
\end{aligned}
\]
which is the number of selected sets containing $e_k$.
\end{proof}

\subsection{Constructed instance}

Given an \RXC\ instance, choose any order $e_1,\ldots,e_n$ of its elements and retain the same family $S_1,\ldots,S_n$.

\subsection{The ``if'' proof}

\begin{lemma}\label{lem:red1if}
If the \RXC\ instance has an exact cover, then the constructed \RPE\ instance is a yes-instance.
\end{lemma}

\begin{proof}
Let $x_v=1$ exactly for the subset-columns in an exact cover.  Fix a prefix $\{e_1,\ldots,e_k\}$.  Each of its $k$ element rows is covered by exactly one selected subset-column, so there are exactly $k$ selected incidences in the prefix.  Counting the same incidences column by column gives
\[
\sum_v |S_v\cap\{e_1,\ldots,e_k\}|x_v
=\sum_vd_{v,k}x_v
=C_k(x).
\]
Hence $C_k(x)=k$ for every $k$, as required.
\end{proof}

\subsection{The ``only if'' proof}

\begin{lemma}\label{lem:red1onlyif}
If the constructed \RPE\ instance is a yes-instance, then the original \RXC\ instance has an exact cover.
\end{lemma}

\begin{proof}
Let $x$ satisfy $C_k(x)=k$ for all $k$, and set $C_0(x)=0$.  By \cref{lem:prefixdifference}, the number of selected subset-columns covering element row $e_k$ is
\[
\sum_{v:e_k\in S_v}x_v
=C_k(x)-C_{k-1}(x)
=k-(k-1)=1.
\]
Thus every element row is covered by exactly one selected subset-column.  Equivalently, $A^{\mathrm{inc}}x=\1$, so the selected columns form an exact cover of the universe.
\end{proof}

\begin{intuition}
The prefix equations contain exactly the same information as $A^{\mathrm{inc}}x=\1$: successive differences recover the row equations.  Their advantage is geometric--the scheduling construction can read cumulative counts from gaps on a one-dimensional ratio axis.
\end{intuition}

\begin{corollary}\label{cor:rpehard}
\RPE\ is NP-complete.
\end{corollary}

\begin{proof}
The transformation retains the same binary variables and computes all values $d_{v,k}$ using $O(n^2)$ additions, so it is polynomial.  \Cref{lem:red1if,lem:red1onlyif} prove feasibility equivalence with \RXC.  Membership in NP is immediate: a vector $x\in\{0,1\}^n$ is a certificate, and all $n$ prefix equations are checked in polynomial time.  Since \RXC\ is NP-complete, so is \RPE.
\end{proof}

\section{Reduction II: regular prefix exactness to \AWET}
\label{sec:red2}

Fix an \RPE\ instance $(X,\mathcal S)$ with $X=\{e_1,\ldots,e_n\}$ and $\mathcal S=\{S_1,\ldots,S_n\}$.  We may assume $n\ge3$, since smaller instances can be decided directly.  Set
\[
 N=n+1,
 \qquad
 \kappa(k)=n-k+1.
\]
Thus $N\ge4$, and earlier elements receive larger reverse levels.  The construction creates $8n+1$ jobs: eight \emph{state jobs} for each source set and one global anchor.  All job data are integers.

\begin{howtoread}
Read the construction in four layers.  First, arithmetic digits give each block exactly two balanced states.  Second, early ratios place the three occurrences of each set into the appropriate element bands.  Third, integer tardiness offsets supply the linear field needed to center the prefix squares.  Fourth, the anchor and minimum kernel force every optimal canonical assignment onto the two-state block manifold.  Each layer uses only properties already established by the preceding layers.
\end{howtoread}

\begin{finding}{Two-lane checkpoint}
The logical lane asks only what a balanced block means and how prefix errors decode.  The numerical lane proves that every low-cost schedule is balanced and that one nonzero prefix error dominates every residual term.  Readers following only the theorem architecture may treat the scale lemmas as certified interfaces and return to the proof of \cref{lem:separation} for the full separation arithmetic.
\end{finding}

\begin{table}[ht]
\centering
\caption{Construction blueprint for one source set $S_v$.}
\label{tab:construction-blueprint}
\begin{tabular}{@{}>{\raggedright\arraybackslash}p{0.22\textwidth} >{\raggedright\arraybackslash}p{0.33\textwidth} >{\raggedright\arraybackslash}p{0.34\textwidth}@{}}
\toprule
Component & Encoded information & Later use \\
\midrule
Three occurrence pairs & The three elements of $S_v$ & Their effective atoms appear in the three corresponding element bands. \\
One controller pair & Total signed mass $-3$ & It cancels the three $+1$ occurrence masses below the final element band. \\
Logical spin $\sigma_v$ & Select or reject $S_v$ & $x_v=(1+\sigma_v)/2$ becomes the exact-cover variable. \\
Target load $\beta_v$ & Balanced half of the block processing & The global load lock forces every optimum to one of the two logical states. \\
\bottomrule
\end{tabular}
\end{table}

\subsection{Constructed instance}
\label{subsec:construction}

\subsubsection{Numerical scales}

Use the fixed pair codes
\[
 Q_0=10,
 \quad Q_1=10^2,
 \quad Q_2=10^3,
 \quad Q_3=10^4,
 \quad q=Q_0+Q_1+Q_2+Q_3.
\]
Set
\[
 G=N^{20},
 \qquad
 M=N^{40},
 \qquad
 T=N^{50}.
\]
The ratio-space scale $G$ is distinct from the tardiness ratios $s_i$ and is used consistently throughout the construction and worked examples.  The exponents are deliberately loose; their only purpose is to make the two separations
\[
 T\gg nM\gg q\gg 6,
 \qquad
 GM\gg N^2T
\]
immediate.  Inside a block, the $T$ digit controls the number of unsplit pairs, the $M$ digit controls their aggregate element level, the codes $Q_h$ identify the pairs, and the final constants $1$ and $3$ determine the orientation.  In early-ratio space, the product $GM$ separates consecutive element bands from all block-index variation.

For each $v$, write
\[
 S_v=\{e_{k_{v1}},e_{k_{v2}},e_{k_{v3}}\},
 \qquad
 k_{v1}<k_{v2}<k_{v3}.
\]
Define four base processing values
\begin{align}
 P_{v0}&=T+Q_0,\label{eq:P0}\\
 P_{vh}&=T+M\kappa(k_{vh})+Q_h,
 \qquad h=1,2,3.\label{eq:Ph}
\end{align}
The controller has no $M$ digit.  Each occurrence base records, in decreasing order of scale, a common synchronizing value $T$, its element level $M\kappa(k_{vh})$, and its pair code $Q_h$.  The one-unit heavy--light difference is added when the two jobs are created.  A discrepancy at any scale is too large to be canceled by all smaller scales.

\subsubsection{The eight-job coded block}

For each set $S_v$, create the eight state jobs in \cref{tab:block}.  The fixed sign $\epsilon_i$ specifies one reference orientation; reversing all eight signs gives the other.

\begin{table}[t]
\centering
\caption{Jobs in block $v$. The occurrence pair $h$ represents element $e_{k_{vh}}$.}
\label{tab:block}
\begin{tabular}{cclcc}
\toprule
Pair & Job & Role & Processing time & $\epsilon$\\
\midrule
$0$ & $A_{v0}$ & controller, heavy & $P_{v0}+3$ & $-1$\\
$0$ & $B_{v0}$ & controller, light & $P_{v0}$ & $+1$\\
\midrule
$h\in\{1,2,3\}$ & $A_{vh}$ & occurrence, heavy & $P_{vh}+1$ & $+1$\\
$h\in\{1,2,3\}$ & $B_{vh}$ & occurrence, light & $P_{vh}$ & $-1$\\
\bottomrule
\end{tabular}
\end{table}

Let
\[
 \B_v=\{A_{v0},B_{v0},A_{v1},B_{v1},A_{v2},B_{v2},A_{v3},B_{v3}\}
\]
be the job set of block $v$.  Put
\[
 \beta_v=P_{v0}+P_{v1}+P_{v2}+P_{v3}+3.
\]
The block has total processing time $2\beta_v$.  Its assignment is \emph{canonical} when some $\sigma_v\in\{-1,+1\}$ satisfies
\[
 c_i=\sigma_v\epsilon_i
 \qquad(i\in\B_v).
\]
Thus $\sigma_v$ selects one of two complementary orientations.  The next lemma proves that these are exactly the zero-signed-load assignments.  In either orientation the tardy processing load is $\beta_v$.

The two orientations can be read directly from the signs:
\[
\begin{array}{c|c|c|c}
\sigma_v & \text{early controller job} & \text{early occurrence jobs} & \text{tardy load}\\ \hline
+1 & B_{v0} & A_{v1},A_{v2},A_{v3} & (P_{v0}+3)+P_{v1}+P_{v2}+P_{v3}=\beta_v\\
-1 & A_{v0} & B_{v1},B_{v2},B_{v3} & P_{v0}+\sum_{h=1}^3(P_{vh}+1)=\beta_v
\end{array}
\]
The controller difference $3$ balances the three occurrence differences $1+1+1$.  Consequently both orientations place exactly half of the block's processing on the tardy side.

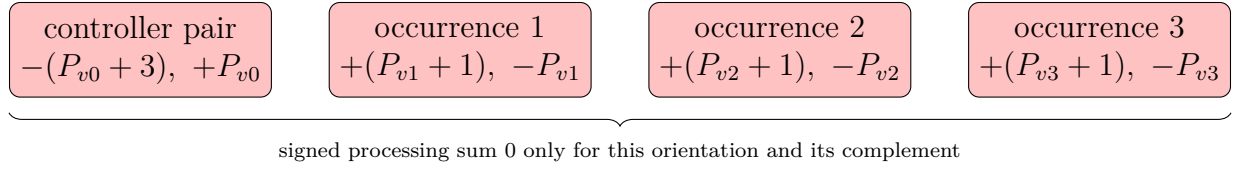
\begin{figure}[t]
\centering
\resizebox{0.98\textwidth}{!}{%
\begin{tikzpicture}[
  pair/.style={draw,rounded corners,minimum width=2.45cm,minimum height=1.25cm,align=center,fill=figurehard},
  lab/.style={font=\scriptsize,align=center},
  node distance=0.75cm
]
\node[pair] (c) {controller pair\\$-(P_{v0}+3),\ +P_{v0}$};
\node[pair,right=of c] (o1) {occurrence 1\\$+(P_{v1}+1),\ -P_{v1}$};
\node[pair,right=of o1] (o2) {occurrence 2\\$+(P_{v2}+1),\ -P_{v2}$};
\node[pair,right=of o2] (o3) {occurrence 3\\$+(P_{v3}+1),\ -P_{v3}$};
\draw[decorate,decoration={brace,amplitude=5pt,mirror}] ($(c.south west)+(0,-0.2)$) -- ($(o3.south east)+(0,-0.2)$)
 node[midway,below=7pt,lab]{signed processing sum $0$ only for this orientation and its complement};
\end{tikzpicture}%
}
\caption{Signed-processing representation of one coded block. The signs shown are $\epsilon_i$; multiplying every sign by $\sigma_v$ selects one of the two logical states.}
\label{fig:block}
\end{figure}

\subsubsection{Arithmetic synchronization}

\begin{lemma}[Two zero-load states]\label{lem:arithmetic}
For every sign vector $c\in\{-1,+1\}^{8}$ on block $v$,
\[
 \sum_{i\in\B_v}p_ic_i=0
 \quad\Longleftrightarrow\quad
 \text{$c_i=\sigma_v\epsilon_i$ for all $i\in\B_v$ and some
 $\sigma_v\in\{-1,+1\}$.}
\]
\end{lemma}

\begin{proof}
\par\medskip\noindent\textbf{Forward direction.}\quad
If $c_i=\sigma_v\epsilon_i$, each occurrence pair contributes
\[
\sigma_v\bigl((P_{vh}+1)-P_{vh}\bigr)=\sigma_v,
\]
while the controller contributes
\[
\sigma_v\bigl(-(P_{v0}+3)+P_{v0}\bigr)=-3\sigma_v.
\]
The three $+\sigma_v$ contributions cancel the one $-3\sigma_v$ contribution.

\par\medskip\noindent\textbf{Converse: expose the scale digits.}\quad
Suppose $\sum_{i\in\B_v}p_ic_i=0$, with arbitrary signs.  For pair $h\in\{0,1,2,3\}$ define
\[
 u_h=\frac{c_{A_{vh}}+c_{B_{vh}}}{2}\in\{-1,0,1\}.
\]
Thus $u_h=1,-1,0$ means, respectively, that the pair is all early, all tardy, or split.  Put
\[
 u_{\Sigma}=\sum_{h=0}^3u_h,
 \qquad
 z_{\kappa}=\sum_{h=1}^3\kappa(k_{vh})u_h.
\]
Expanding every processing time by its $T$-, $M$-, and $Q$-parts gives
\begin{align}
 \sum_{i\in\B_v}p_ic_i
 ={}&2T u_{\Sigma}+2M z_{\kappa}+2\sum_{h=0}^3Q_hu_h\label{eq:signedload}\\
 &+c_{A_{v1}}+c_{A_{v2}}+c_{A_{v3}}+3c_{A_{v0}}.\nonumber
\end{align}
The last line is a residual of absolute value at most $6$.

\par\medskip\noindent\textbf{The $T$ digit forces $u_{\Sigma}=0$.}\quad
If $u_{\Sigma}\ne0$, then $|2T u_{\Sigma}|\ge2T$.  The other terms have combined absolute value at most
\[
 2M|z_{\kappa}|+2q+6\le6nM+2q+6<T.
\]
They cannot cancel the $T$ term, contradicting zero signed load.  Hence $u_{\Sigma}=0$.

\par\medskip\noindent\textbf{The $M$ digit forces $z_{\kappa}=0$.}\quad
With $u_{\Sigma}=0$, a nonzero integer $z_{\kappa}$ would give $|2M z_{\kappa}|\ge2M$.  The remaining coded and residual terms have absolute value at most $2q+6<2M$.  Therefore $z_{\kappa}=0$.

\par\medskip\noindent\textbf{The pair codes force every $u_h=0$.}\quad
The sequence $10,100,1000,10000$ is superincreasing for coefficients in $\{-1,0,1\}$.  More explicitly, let $h^\star$ be the largest index with $u_{h^\star}\ne0$.  If $h^\star=0$, the coded sum has magnitude $10$; if $h^\star>0$, then
\[
 Q_{h^\star}-\sum_{h<h^\star}Q_h\ge90>10.
\]
Thus every nonzero vector $u$ satisfies $|\sum_hQ_hu_h|\ge10$.  Its doubled contribution in \eqref{eq:signedload} has magnitude at least $20$, too large to be canceled by the residual of magnitude at most $6$.  Consequently $u_h=0$ for all $h$, so the two jobs in each pair have opposite signs.

\par\medskip\noindent\textbf{The final residual chooses one orientation.}\quad
With all $u_h=0$, equation \eqref{eq:signedload} reduces to
\[
 c_{A_{v1}}+c_{A_{v2}}+c_{A_{v3}}+3c_{A_{v0}}=0.
\]
If $c_{A_{v0}}=-1$, all three occurrence signs must be $+1$; if $c_{A_{v0}}=+1$, all three must be $-1$.  Their partners have the opposite signs.  These are exactly $c=\epsilon$ and $c=-\epsilon$.
\end{proof}

\begin{intuition}
The proof decodes the signed load from largest scale to smallest.  The $T$ and $M$ digits remove aggregate imbalances, the superincreasing $Q_h$ digits force every pair to split, and the final six-unit residual selects one of the two complementary orientations.
\end{intuition}

Let $b_v$ be the total processing time of the tardy jobs in block $v$.  Because the block total is $2\beta_v$, its signed load is $2\beta_v-2b_v$.  Hence \cref{lem:arithmetic} gives the equivalence
\begin{equation}
 b_v=\beta_v
 \quad\Longleftrightarrow\quad
 c_i=\sigma_v\epsilon_i\text{ for every }i\in\B_v.
 \label{eq:loadiffcanonical}
\end{equation}
The global load lock will enforce the left side; the arithmetic lemma then supplies the two-state code on the right.

\subsubsection{Early ratios and the minimum-kernel field}

Define
\[
 \lambda_v=G+n-v+1,
 \qquad
 L_v=2\lambda_v,
\]
so $L_1>L_2>\cdots>L_n>0$.  The factor $2$ is chosen for the exact cancellation used in \cref{lem:loadlock}.  Let
\[
 K_{uv}=\min\{L_u,L_v\},
\]
and write
\[
 K=(K_{uv})_{u,v=1}^n,
 \qquad
 \beta=(\beta_1,\ldots,\beta_n)^\top,
 \qquad
 f=K\beta.
\]
Thus $K$ is the symmetric \emph{minimum kernel} and
$f_v=\sum_{u=1}^nK_{vu}\beta_u$.  This choice centers the load-lock quadratic at the target vector $\beta$.  Every state job $i\in\B_v$ receives early ratio
\begin{equation}
 r_i=f_v+\lambda_vp_i.
 \label{eq:earlyratio}
\end{equation}
Set $w_i^E=p_ir_i$.  The common term $f_v$ moves the whole block, whereas $\lambda_vp_i$ turns processing-time differences into ratio separation.  The same affine form permits the exact two-job compression used below.

\begin{lemma}[Kernel-field identities]\label{lem:fidentities}
For $v=1,\ldots,n-1$,
\[
 f_v-f_{v+1}=2\sum_{u=1}^v\beta_u.
\]
Moreover,
\[
 4T<\beta_v<5T
 \qquad\text{and}\qquad
 0<f_1-f_n<10N^2T.
\]
\end{lemma}

\begin{proof}
Because the levels $L_u$ decrease with $u$, compare rows $v$ and $v+1$ of $K$.  For $u\le v$,
\[
K_{vu}=L_v,
\qquad
K_{v+1,u}=L_{v+1},
\]
so their difference is $L_v-L_{v+1}=2$.  For $u\ge v+1$, both minima equal $L_u$, so their difference is zero.  Multiplying by $\beta_u$ and summing gives
\[
f_v-f_{v+1}=\sum_u(K_{vu}-K_{v+1,u})\beta_u
=2\sum_{u=1}^v\beta_u.
\]
In particular, the fields $f_v$ decrease strictly.

Each $\beta_v$ contains four copies of $T$ plus positive lower-order terms, so $\beta_v>4T$.  For the upper bound,
\[
 \beta_v
 \le4T+3nM+q+3<5T,
\]
because $nM,q+3<T/4$ under the chosen powers and $N\ge4$.  Finally, telescope the first identity:
\[
\begin{aligned}
f_1-f_n
&=2\sum_{v=1}^{n-1}\sum_{u=1}^v\beta_u\\
&<2n\sum_{u=1}^n\beta_u
<2n\cdot n\cdot5T
<10N^2T.
\end{aligned}
\]
\end{proof}

\begin{intuition}
The block shifts $f_v$ vary by only $O(N^2T)$, whereas consecutive element levels will be separated by more than $GM=N^{60}$.  Thus the kernel field is large enough for the load lock but too small to change the element-band order.
\end{intuition}

\subsubsection{Effective atoms: compressing complete pairs}

In this subsection, a \emph{two-job cluster} means one designated occurrence pair or the controller pair of a coded block.  Compression is applied only to complete clusters from different blocks.  Within-block terms are kept in the uncompressed objective: they contribute an assignment-independent constant and part of the baseline field $a_v^0$.  Here $a_v^0$ denotes the coefficient of $\sigma_v$ in the full early state-job cost on coded assignments; its explicit formula is given in \eqref{eq:a0}.

For this calculation, restrict to the coded assignments already characterized by \cref{lem:arithmetic}: $c_i=\sigma_v\epsilon_i$ for $i\in\B_v$.  Thus the early indicator of job $i$ is
\[
 \1\{i\in E\}=\frac{1+c_i}{2}
 =\frac{1+\sigma_v\epsilon_i}{2}.
\]
This is an algebraic restriction at this stage, not yet a claim about optimal assignments.  The global load lock and the subsequent choice of the anchor's tardiness ratio will justify the restriction for an optimal canonical schedule.

For a cluster $g$, define its signed processing mass and signed first ratio moment by
\[
 m_g=\sum_{i\in g}\epsilon_i p_i,
 \qquad
 \mu_g=\sum_{i\in g}\epsilon_i p_i r_i.
\]
Every designated pair has $m_g\ne0$, so define its \emph{effective ratio}
\[
 \rho_g=\frac{\mu_g}{m_g}.
\]
The bookkeeping pair $(m_g,\rho_g)$ is the \emph{effective atom} of cluster $g$: a signed point mass $m_g$ placed at ratio $\rho_g$.

The effective atom is not a new scheduling job and does not replace the two physical jobs inside their own block.  Its role is narrower: when two complete clusters are separated in ratio order, their four physical cross-job terms depend only on the signed mass $m_g$ and first moment $\mu_g=m_g\rho_g$.  The effective atom therefore preserves those cross-block interactions exactly while reducing the amount of algebra.

For an occurrence pair $(v,h)$, $h\in\{1,2,3\}$,
\[
\begin{aligned}
m_{vh}&=(P_{vh}+1)-P_{vh}=1,\\
\mu_{vh}
&=(P_{vh}+1)\bigl(f_v+\lambda_v(P_{vh}+1)\bigr)
-P_{vh}\bigl(f_v+\lambda_vP_{vh}\bigr)\\
&=f_v+\lambda_v(2P_{vh}+1).
\end{aligned}
\]
Hence
\[
 \rho_{vh}=f_v+\lambda_v(2P_{vh}+1).
\]
For the controller pair,
\[
\begin{aligned}
m_{v0}&=-(P_{v0}+3)+P_{v0}=-3,\\
\mu_{v0}
&=-(P_{v0}+3)\bigl(f_v+\lambda_v(P_{v0}+3)\bigr)
+P_{v0}\bigl(f_v+\lambda_vP_{v0}\bigr)\\
&=-3\bigl(f_v+\lambda_v(2P_{v0}+3)\bigr),
\end{aligned}
\]
so
\[
 \rho_{v0}=f_v+\lambda_v(2P_{v0}+3).
\]
Thus a block has three occurrence atoms of mass $+1$ and one controller atom of mass $-3$; its total effective mass is zero.

For complete clusters $g$ and $h$ from different blocks, define the physical signed cross-interaction
\[
 I^{\rm phys}(g,h)=
 \sum_{i\in g}\sum_{j\in h}
 \epsilon_i\epsilon_jp_ip_j\min\{r_i,r_j\},
\]
and its effective-atom counterpart
\[
 I^{\rm comp}(g,h)=m_gm_h\min\{\rho_g,\rho_h\}.
\]
These are signed coefficients arising after the early indicators are expanded; they are not separate scheduling costs.  In particular, a controller mass is negative even though every physical cost is nonnegative.

\begin{finding}{Symbolic micro-example: why a two-job cluster becomes one atom}
For a cluster $g=\{g^+,g^-\}$ with signed processing masses $+p^+$ and $-p^-$, the definitions give
\[
 m_g=p^+-p^-,\qquad \mu_g=p^+r_{g^+}-p^-r_{g^-}.
\]
If every ratio in $g$ exceeds every ratio in a second cluster $h$, each cross minimum selects the ratio in $h$.  Hence the four signed physical terms factor as
\[
 I^{\rm phys}(g,h)
 =\left(\sum_{i\in g}\epsilon_ip_i\right)
  \left(\sum_{j\in h}\epsilon_jp_jr_j\right)
 =m_g\mu_h=m_gm_h\rho_h.
\]
When the effective-ratio order agrees, this is $m_gm_h\min\{\rho_g,\rho_h\}$.  The mass records the net signed processing, while the first moment determines the effective location needed to preserve the interaction.
\end{finding}

\begin{lemma}[Exact pair compression]\label{lem:compression}
Let $g$ and $h$ be complete clusters from different blocks.  Suppose every physical early ratio in $g$ exceeds every physical early ratio in $h$.  Then
\[
 I^{\rm phys}(g,h)=m_gm_h\rho_h.
\]
Consequently, whenever the physical order of the two clusters agrees with their effective-ratio order,
\[
 I^{\rm phys}(g,h)=I^{\rm comp}(g,h).
\]
\end{lemma}

\begin{proof}
The lower ratio always belongs to $h$, so $\min\{r_i,r_j\}=r_j$ for all four cross-job pairs.  The sum therefore factors:
\[
\begin{aligned}
 I^{\rm phys}(g,h)
 &=\sum_{i\in g}\sum_{j\in h}\epsilon_i\epsilon_jp_ip_jr_j\\
 &=\left(\sum_{i\in g}\epsilon_ip_i\right)
   \left(\sum_{j\in h}\epsilon_jp_jr_j\right)\\
 &=m_g\mu_h
 =m_gm_h\rho_h.
\end{aligned}
\]
If the effective order agrees, then $\rho_h=\min\{\rho_g,\rho_h\}$ and the first case gives $I^{\rm comp}(g,h)$.  The symmetric case is identical.
\end{proof}

Let $\mathcal G_v$ be the four designated clusters of block $v$.  For jobs $i\in\B_u$ and $j\in\B_v$, with $u\ne v$, the coded early indicators satisfy
\[
 \1\{i,j\in E\}
 =\frac14\bigl(1+\sigma_u\epsilon_i+\sigma_v\epsilon_j
 +\sigma_u\sigma_v\epsilon_i\epsilon_j\bigr).
\]
Thus the coefficient of $\sigma_u\sigma_v$ records the interaction between the two block choices.  The $64$ physical cross-job terms can be grouped into $16$ cluster-pair interactions, so this coefficient is
\[
 \frac14\sum_{g\in\mathcal G_u}\sum_{h\in\mathcal G_v} I^{\rm phys}(g,h).
\]
When the physical and effective orders agree, \cref{lem:compression} replaces every $I^{\rm phys}(g,h)$ by $I^{\rm comp}(g,h)$.  The scale-separation lemma below verifies these order hypotheses for the construction.  The prefix-profile calculation then denotes the coefficient by $J_{uv}$ and rewrites it as a profile integral.  This is the purpose of effective atoms: they preserve the block-spin couplings exactly while converting the physical cross-block algebra into a one-dimensional signed-mass calculation.

\begin{intuition}
The physical pair is retained in the schedule, but its cross-block effect is summarized by one signed point mass.  Exactness comes from separation: every minimum in the four-term interaction selects the same lower cluster.
\end{intuition}

\paragraph{Ratio bands.}
View the physical ratios $r_i$ and effective ratios $\rho_g$ on a common real line, called early-ratio space.  For element $e_k$, let its \emph{effective element band} be the smallest closed interval containing all occurrence-atom ratios associated with $e_k$:
\[
 \mathcal R_k=[\underline\rho_k,\overline\rho_k].
\]
Define the controller band analogously:
\[
 \mathcal R_0=[\underline\rho_0,\overline\rho_0].
\]
The main-gap lengths are
\begin{align}
 \Delta_k&=\underline\rho_k-\overline\rho_{k+1},
 &&k=1,\ldots,n-1,\label{eq:deltak}\\
 \Delta_n&=\underline\rho_n-\overline\rho_0.\label{eq:deltan}
\end{align}
Let
\[
 W=\sum_{k=0}^n(\overline\rho_k-\underline\rho_k)
\]
be the total band width.  We write $\mathcal R_k>\mathcal R_\ell$ when every point of $\mathcal R_k$ exceeds every point of $\mathcal R_\ell$.  The open intervals between consecutive bands are the \emph{main gaps}.  The main gaps carry the exact-cover signal; all within-band intervals, whose total length is $W$, enter the residual.

\begin{lemma}[Scale separation]\label{lem:separation}
The physical ratio intervals of the two-job clusters are disjoint, and their order agrees with the effective-ratio order.  The bands satisfy
\[
 \mathcal R_1>\mathcal R_2>\cdots>\mathcal R_n>\mathcal R_0,
\]
where $\mathcal R_k=[\underline\rho_k,\overline\rho_k]$. Moreover,
\begin{equation}
 \Delta_k>GM\quad(k=1,\ldots,n),
 \qquad
 W<15N^3T.
 \label{eq:sepbounds}
\end{equation}
\end{lemma}

\begin{proof}
We give all separation and width bounds explicitly, so that the proof is self-contained and no asymptotic comparison is needed.

For an occurrence cluster $(v,h)$ representing element $e_k$, put $a=\kappa(k)$.  Its physical early-ratio interval and effective ratio are
\[
 \mathcal I_{vh}=
 [f_v+\lambda_v(T+Ma+Q_h),\ f_v+\lambda_v(T+Ma+Q_h+1)],
\]
\[
 \rho_{vh}=f_v+\lambda_v(2T+2Ma+2Q_h+1).
\]
For a controller cluster they are
\[
 \mathcal I_{v0}=
 [f_v+\lambda_v(T+Q_0),\ f_v+\lambda_v(T+Q_0+3)],
 \qquad
 \rho_{v0}=f_v+\lambda_v(2T+2Q_0+3).
\]
For $N\ge4$, the definitions $G=N^{20}$, $M=N^{40}$, $T=N^{50}$, and $q=11110$ imply
\begin{align*}
 n&<N, & G+N&<2G, & |\lambda_v-\lambda_w|&<N,\\
 N^2M&<N^2T, & 8G(q+3)&<N^2T, & GM&>20N^2T,\\
 4G(q+3)&<T, & 3nM+q+3&<T.
\end{align*}
For example, $GM>20N^2T$ is equivalent to $N^8>20$, and $8G(q+3)<N^2T$ is equivalent to $8(q+3)<N^{32}$; both hold for $N\ge4$.

\par\medskip\noindent\textbf{Different element levels.}\quad
Consider occurrence cluster $(v,h)$ for $e_k$ and occurrence cluster $(w,h')$ for $e_\ell$, where $k<\ell$.  Put $a=\kappa(k)$ and $b=\kappa(\ell)$, so $a-b\ge1$ and $b\le n$.  The lowest physical ratio of the first cluster minus the highest physical ratio of the second is
\begin{align*}
 D_{\rm phys}
 ={}&(f_v-f_w)+(\lambda_v-\lambda_w)T\\
 &+M\{\lambda_v(a-b)+(\lambda_v-\lambda_w)b\}
 +\lambda_vQ_h-\lambda_w(Q_{h'}+1).
\end{align*}
Using $f_v-f_w>-(f_1-f_n)>-10N^2T$, $\lambda_v\ge G+1$, $|\lambda_v-\lambda_w|<N$, $b<N$, and $\lambda_w<2G$, we obtain
\[
 D_{\rm phys}
 >(G+1)M-N^2M-NT-10N^2T-4G(q+3).
\]
The adverse terms are less than $13N^2T$, whereas $GM>20N^2T$; hence $D_{\rm phys}>0$.  Thus the full physical intervals are strictly ordered by element level.

For the effective points, the analogous subtraction gives
\[
 D_{\rm eff}
 >2(G+1)M-2N^2M-2NT-10N^2T-8G(q+3).
\]
The adverse terms are less than $15N^2T$, and therefore
\[
 D_{\rm eff}>2GM-15N^2T>GM.
\]
The same estimates, with the lower level's element digit set to zero and controller width $3$ absorbed in the code bound, show that every occurrence cluster lies above every controller cluster, physically and effectively, with effective separation greater than $GM$.

\par\medskip\noindent\textbf{The same element level.}\quad
Suppose two occurrence clusters represent the same element and $v<w$.  Then
\[
 \delta:=\lambda_v-\lambda_w=w-v\ge1,
 \qquad f_v-f_w>0.
\]
The lowest physical ratio in the $v$-cluster minus the highest in the $w$-cluster is at least
\[
 \delta T-4G(q+3)>0,
\]
and the effective-point difference is at least
\[
 2\delta T-8G(q+3)>0.
\]
Thus the physical and effective orders agree inside every element band.  The controller comparison is easier and follows from the same inequalities with the fixed code $Q_0$.

\par\medskip\noindent\textbf{Gap and width bounds.}\quad
The effective comparison between consecutive element levels, and between the last element level and the controller level, is greater than $GM$.  Hence $\Delta_k>GM$ for every $k$.

For a fixed element $e_k$, put $A_k=2T+2M\kappa(k)$.  For any two effective occurrence points in its band,
\begin{align*}
 |\rho_{vh}-\rho_{wh'}|
 &\le |f_v-f_w|
 +|\lambda_v-\lambda_w|(A_k+2q+1)
 +2\lambda_w|Q_h-Q_{h'}|\\
 &<10N^2T+N(2T+2nM+2q+1)+2(G+N)q\\
 &<10N^2T+3NT+N^2T
 <15N^2T.
\end{align*}
For the penultimate inequality, the scale definitions give
$2nM+2q+1<T$, so $N(2T+2nM+2q+1)<3NT$.  Moreover,
\[
 2(G+N)q<4Gq<8G(q+3)<N^2T,
\]
using the preliminary inequalities above.  Thus the final code-variation term is bounded by $N^2T$ itself.
The controller band satisfies the same upper bound.  Since there are $n+1=N$ bands,
\[
 W<15N^3T.
\]
All complete physical cluster intervals are therefore disjoint, their order agrees with the effective-ratio order, and the hypothesis of \cref{lem:compression} holds.
\end{proof}

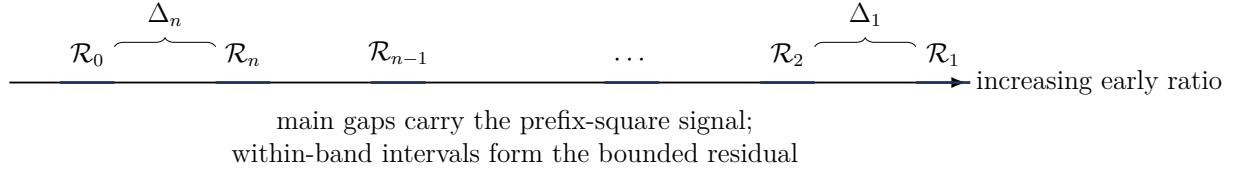
\begin{figure}[t]
\centering
\resizebox{0.98\textwidth}{!}{%
\begin{tikzpicture}[x=1.2cm,y=1cm,>=Latex]
\draw[->,thick] (0,0) -- (12.3,0) node[right] {increasing early ratio};
\foreach \x/\lab in {1/$\mathcal R_0$,3/$\mathcal R_n$,5/$\mathcal R_{n-1}$,8/$\cdots$,10/$\mathcal R_2$,12/$\mathcal R_1$}{
  \draw[very thick,navy] (\x-0.35,0) -- (\x+0.35,0);
  \node[above=3pt] at (\x,0) {\lab};
}
\draw[decorate,decoration={brace,amplitude=4pt}] (1.35,0.55) -- (2.65,0.55) node[midway,above=5pt] {$\Delta_n$};
\draw[decorate,decoration={brace,amplitude=4pt}] (10.35,0.55) -- (11.65,0.55) node[midway,above=5pt] {$\Delta_1$};
\node[below=8pt,align=center] at (6.5,0) {main gaps carry the prefix-square signal;\\within-band intervals form the bounded residual};
\end{tikzpicture}%
}
\caption{Effective-ratio bands. Element $e_1$ is highest, element $e_n$ is lowest, and the controller band lies below all element bands.}
\label{fig:bands}
\end{figure}

\subsubsection{Prefix profiles and the main quadratic signal}

Write $\1\{\cdot\}$ for an indicator.  For block $v$, define the signed effective-atom profile
\begin{equation}
 F_v(t)=\sum_{h=0}^3m_{vh}\1\{t\le\rho_{vh}\}.
 \label{eq:profile}
\end{equation}
Thus $F_v(t)$ is the signed mass at effective ratios at least $t$.  As $t$ moves downward, each occurrence atom adds $1$; after the controller atom is crossed, the mass $-3$ returns the profile to zero.  The controller therefore prevents a spurious tail below the last element band.

For distinct blocks $u,v$, define the cross-block quadratic coefficient generated by early interactions by the physical-job sum below.  By \cref{lem:compression,lem:separation}, it also has the effective-atom and profile representations
\begin{equation}
 \begin{aligned}
 J_{uv}
 &:=\frac14\sum_{i\in\B_u}\sum_{j\in\B_v}
       \epsilon_i\epsilon_jp_ip_j\min\{r_i,r_j\}\\
 &=\frac14\sum_{g=0}^3\sum_{h=0}^3
       m_{ug}m_{vh}\min\{\rho_{ug},\rho_{vh}\}\\
 &=\frac14\int_0^\infty F_u(t)F_v(t)\,dt.
 \end{aligned}
 \label{eq:profilekernel}
\end{equation}
The first line ranges over the eight physical jobs in each block.  The second ranges over the four effective atoms, whose signed masses are $m_{v0}=-3$ and $m_{vh}=1$ for $h=1,2,3$.

The final equality uses the exact minimum-kernel identity
\[
 \min\{a,b\}=\int_0^\infty\1\{t\le a\}\1\{t\le b\}\,dt.
\]
Equivalently, sort the effective ratios and sum the constant product $F_uF_v$ over consecutive intervals.  The integral notation makes the split into main gaps and within-band intervals explicit.

The factor $1/4$ comes from the product of the two early indicators: the coefficient of $\sigma_u\sigma_v$ for jobs $i\in\B_u$ and $j\in\B_v$ is $\epsilon_i\epsilon_j/4$.  Exact pair compression then gives \eqref{eq:profilekernel}.

On the main gap following element $e_k$, the active occurrence atoms of block $v$ are exactly those corresponding to $S_v\cap\{e_1,\ldots,e_k\}$, while the controller is inactive.  For $k<n$ this gap is $(\overline\rho_{k+1},\underline\rho_k)$; for $k=n$ it is $(\overline\rho_0,\underline\rho_n)$.  Hence
\begin{equation}
 F_v(t)=d_{v,k}
 \quad\text{throughout the $k$th main gap.}
 \label{eq:profileprefix}
\end{equation}
Because a block contains three $+1$ atoms and one $-3$ atom, $|F_v(t)|\le3$ everywhere.

Define the \emph{signed prefix total}
\[
 Y_k(\sigma)=\sum_{v=1}^n\sigma_vd_{v,k}.
\]
This notation is distinct from the source sets $S_v$.

\begin{finding}{Symbolic micro-example: why one ratio gap gives a square}
Fix the $k$th main gap, whose length is $\Delta_k$.  For two blocks $u,v$ alone, the signed profile total is $Y=\sigma_ud_{u,k}+\sigma_vd_{v,k}$, and their interaction on this gap is
\[
 \frac{\Delta_k}{4}\sigma_u\sigma_vd_{u,k}d_{v,k}
 =\frac{\Delta_k}{8}\bigl(Y^2-d_{u,k}^2-d_{v,k}^2\bigr).
\]
The subtracted terms do not depend on the spins.  Thus the gap length weights the square of the signed profile total; summing over all blocks gives the same interpretation for $Y_k(\sigma)$.
\end{finding}

The cross-block quadratic contribution of the $k$th main gap is
\[
\frac{\Delta_k}{4}\sum_{u<v}\sigma_u\sigma_vd_{u,k}d_{v,k}.
\]
Using
\[
Y_k(\sigma)^2
=\sum_vd_{v,k}^2+2\sum_{u<v}\sigma_u\sigma_vd_{u,k}d_{v,k},
\]
this equals, up to the diagonal term $-(\Delta_k/8)\sum_vd_{v,k}^2$ that is independent of $\sigma$,
\begin{equation}
 \frac{\Delta_k}{8}Y_k(\sigma)^2.
 \label{eq:maingapquadratic}
\end{equation}
Thus the gap length $\Delta_k$ weights the square of a signed prefix total.  If $x_v=(1+\sigma_v)/2$, regularity gives $\sum_vd_{v,k}=3k$, because each of the first $k$ elements belongs to three sets.  Therefore
\begin{align*}
Y_k(\sigma)
&=\sum_v(2x_v-1)d_{v,k}\\
&=2\sum_vd_{v,k}x_v-\sum_vd_{v,k}\\
&=2C_k(x)-3k,
\end{align*}
and hence
\begin{equation}
 Y_k(\sigma)+k=2(C_k(x)-k).
 \label{eq:prefixrelation}
\end{equation}
Thus completing the square in $Y_k+k$ produces exactly the prefix error $C_k(x)-k$.

\subsubsection{Logical-field correction}

The main gaps supply the quadratic term in $Y_k$.  To obtain a square centered at $C_k(x)=k$, we must also supply the matching linear term.  For state jobs $i,j$, write
\[
 \alpha_{ij}=p_ip_j\min\{r_i,r_j\}.
\]
Define $a_v^0$ as the coefficient of $\sigma_v$ in the early state-job cost after substituting $c_i=\sigma_v\epsilon_i$ in every block:
\begin{equation}
 a_v^0=
 \frac14\sum_{\{i,j\}\subseteq\B_v}\alpha_{ij}(\epsilon_i+\epsilon_j)
 +\frac14\sum_{\substack{i\in\B_v\\j\notin\B_v}}\alpha_{ij}\epsilon_i.
 \label{eq:a0}
\end{equation}
The first sum collects within-block pairs; the second collects pairs joining block $v$ to another block, with every physical pair counted once.  Thus $a_v^0$ is exactly the baseline coefficient of $\sigma_v$ and is computable in polynomial time.  \Cref{app:algebra} gives the indicator expansion.

The target coefficient that completes the main-gap squares is
\begin{equation}
 a_v^\star=\frac14\sum_{k=1}^n\Delta_k k d_{v,k}.
 \label{eq:astar}
\end{equation}
Indeed,
\[
\sum_va_v^\star\sigma_v
=\sum_{k=1}^n\frac{\Delta_k k}{4}Y_k(\sigma).
\]
Adding this to the main-gap square gives
\[
\frac{\Delta_k}{8}Y_k^2+\frac{\Delta_k k}{4}Y_k
=\frac{\Delta_k}{8}(Y_k+k)^2-\frac{\Delta_k}{8}k^2.
\]
The final term is constant, and \eqref{eq:prefixrelation} converts the square into $(\Delta_k/2)(C_k(x)-k)^2$.

To adjust the actual field $a_v^0$ toward $a_v^\star$, define the signed square-processing moment
\[
 D_v=\sum_{i\in\B_v}\epsilon_ip_i^2.
\]
For an occurrence pair, the signed square difference is
\[
(P_{vh}+1)^2-P_{vh}^2=2P_{vh}+1,
\]
whereas for the controller it is
\[
-(P_{v0}+3)^2+P_{v0}^2=-6P_{v0}-9.
\]
The six copies of $T$ cancel between the three occurrence pairs and the controller.  Substituting the definitions of the $P$ values gives
\begin{equation}
 D_v=
 2M\sum_{h=1}^3\kappa(k_{vh})
 +2(Q_1+Q_2+Q_3)-6Q_0-6>0.
 \label{eq:Dv}
\end{equation}
Positivity follows because the $M$ term is positive and dominates the fixed code terms.
Choose an integer
\begin{equation}
 \eta_v=\operatorname{round}\left(\frac{4(a_v^0-a_v^\star)}{D_v}\right),
 \label{eq:eta}
\end{equation}
where either nearest integer is permitted in a tie, and define
\begin{equation}
 e_v=a_v^0-\frac{\eta_vD_v}{4}-a_v^\star.
 \label{eq:ev}
\end{equation}
Then
\begin{equation}
 |e_v|\le\frac{D_v}{8},
 \qquad
 \sum_{v=1}^nD_v<7N^2M.
 \label{eq:ebound}
\end{equation}
The first inequality is just nearest-integer rounding.  The chosen integer $\eta_v$ differs from the ideal real value $4(a_v^0-a_v^\star)/D_v$ by at most $1/2$; multiplying that error by $D_v/4$ gives $D_v/8$.  Thus exact field cancellation is unnecessary: the remaining field error is explicitly bounded and will be absorbed into $\Gamma$.

\subsubsection{Tardiness ratios, the anchor, and the load-lock scale}

The tardiness ratio of every state job in block $v$ will be
\begin{equation}
 s_i=HL_v+\eta_v.
 \label{eq:lateratio}
\end{equation}
We now define $H$ so that the $HL_v$ layer forces the target block loads.
Recall that
\[
 b_v=\sum_{i\in\T\cap\B_v}p_i
\]
is the actual tardy processing load of block $v$, whereas $\beta_v$ is its prescribed half-load.  By \eqref{eq:loadiffcanonical}, $b_v=\beta_v$ holds exactly for the two coded states.

For a state job $i$, let $\operatorname{blk}(i)$ be its block index.  The worked examples later retain the shorthand $b(i)=\operatorname{blk}(i)$; here $\operatorname{blk}(i)$ avoids confusion with the block load $b_v$.  For state jobs $i,j$, define
\[
 \bar\eta_{ij}=
 \begin{cases}
 \eta_{\operatorname{blk}(i)},&\operatorname{blk}(i)=\operatorname{blk}(j),\\
 \eta_{\max\{\operatorname{blk}(i),\operatorname{blk}(j)\}},&\operatorname{blk}(i)\ne \operatorname{blk}(j).
 \end{cases}
\]
Let $\J_S$ be the set of all $8n$ state jobs, and define
\begin{align}
 U={}&1+
 \sum_{\{i,j\}\subseteq\J_S}p_ip_j\min\{r_i,r_j\}
 +\sum_{i\in\J_S}|\eta_{\operatorname{blk}(i)}|p_i^2\label{eq:U}\\
 &+\sum_{\{i,j\}\subseteq\J_S}|\bar\eta_{ij}|p_ip_j.
 \nonumber
\end{align}
The quantity $U$ is a deliberately crude upper bound on the oscillation of every assignment-dependent term that is not multiplied by $H$.  Here the oscillation of a function means its maximum value minus its minimum value over all early--tardy assignments.  Taking absolute values term by term is wasteful but safe; the added $1$ makes the later inequality strict.  Put
\begin{equation}
 H=U+2\max_v|\eta_v|+2.
 \label{eq:H}
\end{equation}
Add an anchor job $a$ with
\[
 p_a=H,
 \qquad
 r_a=1+\max_{i\in\J_S}r_i.
\]
The earliness weights of all jobs are $w_i^E=p_ir_i$.

With $H$ chosen as in \eqref{eq:H}, the state-job tardiness ratios \eqref{eq:lateratio} are positive and satisfy
\[
 s_i>s_j
 \quad\text{whenever }\operatorname{blk}(i)<\operatorname{blk}(j).
\]
To see this directly, let $\eta_{\max}=\max_v|\eta_v|$.  Positivity follows from $HL_v\ge2H>\eta_{\max}$.  For consecutive blocks,
\[
(HL_v+\eta_v)-(HL_{v+1}+\eta_{v+1})
\ge2H-2\eta_{\max}>0.
\]
Thus the dominant $HL_v$ layer fixes the tardy block order despite the field-correction offsets.  Their tardiness weights are $w_i^L=p_is_i$.

The proof below combines these tardy terms with the anchor's ordinary early-pair terms.  The choice $L_v=2\lambda_v$ makes the squared-processing terms cancel exactly.

\begin{lemma}[Global load lock]\label{lem:loadlock}
Among canonical early--tardy assignments with the anchor early, every optimum satisfies
\[
 b_v=\beta_v
 \qquad(v=1,\ldots,n).
\]
Consequently, every optimal block is in one of its two coded states.
\end{lemma}

\begin{proof}
The calculation has one purpose: isolate an assignment-dependent quadratic centered at the target load vector $\beta$, then show that every nonzero integer load error costs more than all remaining terms can save.  Let $q_v$ be the sum of squared processing times of the tardy jobs in block $v$, and let $Q_v=\sum_{i\in\B_v}p_i^2$.  The block square-sums $q_v,Q_v$ are unrelated to the code constants $q,Q_h$.  Write $b=(b_1,\ldots,b_n)^\top$.

\par\medskip\noindent\textbf{Dominant tardiness cost.}\quad
Ignore the offsets $\eta_v$ temporarily and retain only the $HL_v$ part of each tardiness ratio.  Within block $v$,
\[
\sum_{\{i,j\}\subseteq \T\cap\B_v}p_ip_j
=\frac12(b_v^2-q_v).
\]
Adding the tardy self terms $q_v$ gives $(b_v^2+q_v)/2$.  Between two blocks $u$ and $v$, all pair terms aggregate to $K_{uv}b_ub_v$.  Hence the full dominant tardiness contribution is
\[
 \frac H2\left(b^\top K b+\sum_{v=1}^nL_vq_v\right).
\]

\par\medskip\noindent\textbf{Interaction with the early anchor.}\quad
The anchor has processing time $H$ and early ratio larger than every state-job ratio.  Therefore, for every early state job $i$, the minimum in the early pair term is $r_i$.  The early jobs in block $v$ have total processing $2\beta_v-b_v$ and squared-processing sum $Q_v-q_v$.  Using $r_i=f_v+(L_v/2)p_i$, the anchor interaction with that block is
\[
 H\left[f_v(2\beta_v-b_v)+\frac{L_v}{2}(Q_v-q_v)\right].
\]
Summing these block contributions gives the full anchor--state-job interaction used by the construction.

\par\medskip\noindent\textbf{Cancellation and centering.}\quad
The $q_v$ terms in the two displayed expressions cancel exactly.  After discarding assignment-independent constants, the remaining dominant expression is
\[
\frac H2b^\top K b-Hf^\top b.
\]
Since $f=K\beta$ and $K$ is symmetric,
\[
\frac H2b^\top K b-H\beta^\top K b
=\frac H2(b-\beta)^\top K(b-\beta)-\frac H2\beta^\top K\beta.
\]
The final term is constant.  Thus the assignment-dependent dominant penalty is exactly
\begin{equation}
 \frac H2(b-\beta)^\top K(b-\beta).
 \label{eq:lockquadratic}
\end{equation}

\par\medskip\noindent\textbf{Positive definiteness of the minimum kernel.}\quad
Let $L_{n+1}=0$, and let $z^{(k)}$ have ones in positions $1,\ldots,k$ and zeros afterward.  Entrywise,
\[
K=\sum_{k=1}^n(L_k-L_{k+1})z^{(k)}(z^{(k)})^\top,
\]
because $K_{uv}=\sum_{k\ge\max\{u,v\}}(L_k-L_{k+1})$.  Consequently, for every vector $y$,
\begin{equation}
 y^\top K y=
 \sum_{k=1}^n(L_k-L_{k+1})
 \left(\sum_{v=1}^ky_v\right)^2.
 \label{eq:minkernelpd}
\end{equation}
All coefficients are at least $2$.  If $y$ is a nonzero integer vector, at least one integer prefix sum is nonzero, so $y^\top K y\ge2$.

\par\medskip\noindent\textbf{Dominance over all other terms.}\quad
Let $R_0$ denote the complete assignment-dependent remainder after the $H$-scaled expression and all assignment-independent constants have been removed.  Once the fixed tardiness block order is used, $R_0$ consists exactly of state--state early pair terms, tardiness-offset self terms, and tardiness-offset pair terms.  The oscillation of a coefficient times a $0$--$1$ indicator is at most the coefficient's absolute value; hence \eqref{eq:U} gives
\[
 \operatorname{osc}(R_0)\le U-1<U.
\]
For a wrong integer load vector $b\ne\beta$, equation \eqref{eq:lockquadratic} is at least $H$.  A correctly loaded coded assignment exists and has zero lock penalty.  Comparing any wrong-load assignment with any such correct-load assignment therefore gives
\[
 F_{\rm wrong}-F_{\rm correct}
 \ge H-\operatorname{osc}(R_0)
 \ge H-(U-1)>0.
\]
Thus no optimum among the canonical anchor-early assignments has a wrong load.  Every such optimum has $b=\beta$, and \eqref{eq:loadiffcanonical} together with \cref{lem:arithmetic} forces each block into one of its two coded states.
\end{proof}

\begin{intuition}
The anchor is not merely a large job.  Its early interactions supply the linear term $-H(K\beta)^\top b$ that shifts the minimum of the tardiness quadratic from zero load to the prescribed load $\beta$.  The minimum kernel then penalizes every nonzero integer load-error vector by a visible amount.
\end{intuition}

Choose any coded assignment, put the anchor early, and compute its schedule cost $C^\star$ using the state-job ratios already fixed. This cost does not depend on the anchor's tardiness ratio. Set
\begin{equation}
 s_a=1+\ceil{\frac{C^\star}{H^2}},
 \qquad
 w_a^L=Hs_a.
 \label{eq:anchorlateratio}
\end{equation}
If the anchor were tardy, its tardy self-cost alone would be
\[
p_a^2s_a=H^2\left(1+\left\lceil C^\star/H^2\right\rceil\right)>C^\star.
\]
All costs are nonnegative, while the chosen coded canonical schedule with an early anchor has cost $C^\star$.  Therefore no optimal canonical assignment can have the anchor tardy.  By \cref{lem:compact}, at least one globally optimal schedule is canonical; choosing such a schedule, its anchor is early.  This argument is independent of whether the chosen coded assignment represents an exact cover; it only supplies an explicit feasible benchmark.

Finally, set
\[
 d=\sum_{j\in\J}p_j.
\]
This completes the construction of the \AWET\ jobs.

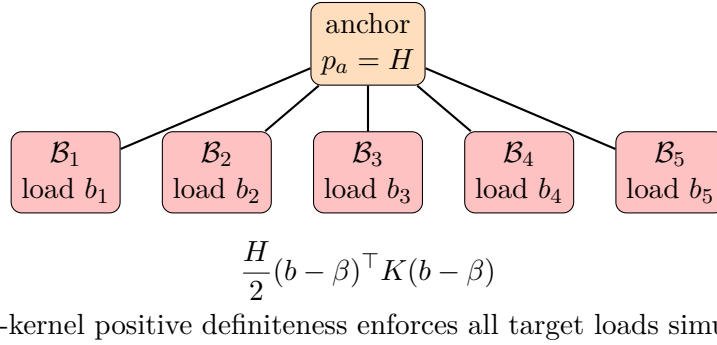
\begin{figure}[t]
\centering
\begin{tikzpicture}[
  state/.style={draw,rounded corners,minimum width=1.4cm,minimum height=0.75cm,fill=figurehard,align=center},
  anch/.style={draw,rounded corners,minimum width=1.5cm,minimum height=0.85cm,fill=figureopen,align=center},
  every node/.style={font=\small}
]
\node[anch] (a) {anchor\\$p_a=H$};
\foreach \x in {1,...,5}{
  \node[state] (b\x) at (2.0*\x-6, -1.7) {$\B_{\x}$\\load $b_{\x}$};
  \draw[thick] (a) -- (b\x);
}
\node at (0,-3.0) {$\displaystyle \frac H2(b-\beta)^\top K(b-\beta)$};
\node[align=center] at (0,-3.75) {minimum-kernel positive definiteness enforces all target loads simultaneously};
\end{tikzpicture}
\caption{Schematic of the anchor-coupled global load lock. The figure shows five blocks only; the construction uses $n$ blocks.}
\label{fig:loadlock}
\end{figure}

\subsubsection{Canonical objective, residual bound, and decision threshold}

On the coded state space, write $c_i=\sigma_v\epsilon_i$ in block $v$ and $x_v=(1+\sigma_v)/2$.  The dominant load-lock term is constant.  The blockwise tardiness offset $\eta_v$ changes the logical field of block $v$ by $-\eta_vD_v/4$; cross-block tardy costs are constant because every block has tardy load $\beta_v$.  Therefore the field coefficient is $a_v^\star+e_v$.

At this point every schedule that can be optimal is represented by the $n$ logical spins $\sigma$.  The remaining task is purely algebraic: separate the large, intended main-gap contribution from the small, unintended within-band and rounding contributions.

Expand the full canonical objective as
\begin{equation}
 \Phi_{\rm can}(\sigma)
 =C_0+\sum_{v=1}^n(a_v^\star+e_v)\sigma_v
 +\sum_{1\le u<v\le n}J_{uv}\sigma_u\sigma_v,
 \label{eq:canonicalexpansion}
\end{equation}
where $C_0$ is the computable constant coefficient and $J_{uv}$ is given by \eqref{eq:profilekernel}. Define the main-gap coefficient
\begin{equation}
 J_{uv}^{\rm main}
 =\frac14\sum_{k=1}^n\Delta_kd_{u,k}d_{v,k}.
 \label{eq:Jmain}
\end{equation}
Let
\begin{equation}
 \mathcal P(x)=
 \sum_{k=1}^n\frac{\Delta_k}{2}(C_k(x)-k)^2.
 \label{eq:prefixpenalty}
\end{equation}
Define
\begin{align}
 A={}&C_0-\frac18\sum_{k=1}^n\Delta_k
 \left(k^2+\sum_{v=1}^nd_{v,k}^2\right),\label{eq:Aconstant}\\
 R(\sigma)={}&\sum_{v=1}^ne_v\sigma_v
 +\sum_{1\le u<v\le n}(J_{uv}-J_{uv}^{\rm main})\sigma_u\sigma_v.
 \label{eq:Rremainder}
\end{align}
\begin{howtoread}
Split every $J_{uv}$ into its main-gap part and its within-band remainder. The main-gap quadratic terms contribute $\Delta_kY_k^2/8$ up to a constant, and the target fields contribute $\Delta_kkY_k/4$. Completing the square gives $\Delta_k(Y_k+k)^2/8$, while \eqref{eq:prefixrelation} converts this into $\Delta_k(C_k-k)^2/2$. The lemma below records this simultaneous calculation in full and identifies exactly where every constant and residual term goes.
\end{howtoread}

\begin{lemma}[Exact master decomposition]\label{lem:masterdecomposition}
For every coded state $\sigma\in\{-1,+1\}^n$ and
$x_v=(1+\sigma_v)/2$,
\begin{equation}
 \Phi_{\rm can}(\sigma)=A+\mathcal P(x)+R(\sigma).
 \label{eq:masteridentity}
\end{equation}
The identity is exact: $R(\sigma)$ contains the explicitly identified
within-band and rounding terms, not an omitted asymptotic error.
\end{lemma}

\begin{proof}
The purpose of the calculation is to assign every physical term to exactly one of three accounts: the intended prefix-square penalty, the explicitly bounded residual, or the state-independent constant.  The calculation has three parts. The main-gap couplings provide the quadratic
terms in the prefix sums, the target fields provide the linear terms needed
to complete the squares, and $A$ supplies the state-independent correction
for the omitted diagonal terms and the constants $k^2$.

First split each coupling into its main-gap and residual portions. By
\eqref{eq:canonicalexpansion} and the definition \eqref{eq:Rremainder},
\[
 \Phi_{\rm can}(\sigma)
 =C_0+\sum_{v=1}^na_v^\star\sigma_v
 +\sum_{1\le u<v\le n}J_{uv}^{\rm main}\sigma_u\sigma_v
 +R(\sigma).
\]

For a fixed prefix $k$, recall that
$Y_k(\sigma)=\sum_vd_{v,k}\sigma_v$. Since $\sigma_v^2=1$,
\[
 Y_k(\sigma)^2
 =\sum_{v=1}^nd_{v,k}^2
 +2\sum_{1\le u<v\le n}d_{u,k}d_{v,k}\sigma_u\sigma_v.
\]
Solving for the cross-block sum and using \eqref{eq:Jmain} yields
\[
 \sum_{1\le u<v\le n}J_{uv}^{\rm main}\sigma_u\sigma_v
 =\frac18\sum_{k=1}^n\Delta_kY_k(\sigma)^2
 -\frac18\sum_{k=1}^n\Delta_k\sum_{v=1}^nd_{v,k}^2.
\]
The second term is the diagonal correction: $J_{uv}^{\rm main}$ sums only
over $u<v$, whereas $Y_k^2$ also contains
$d_{v,k}^2\sigma_v^2=d_{v,k}^2$.

Next use \eqref{eq:astar} and interchange the two finite sums:
\[
 \sum_{v=1}^na_v^\star\sigma_v
 =\frac14\sum_{k=1}^n\Delta_k k
   \sum_{v=1}^nd_{v,k}\sigma_v
 =\frac14\sum_{k=1}^n\Delta_k kY_k(\sigma).
\]
Substitution gives
\[
 \begin{aligned}
 \Phi_{\rm can}(\sigma)
 ={}&C_0-\frac18\sum_{k=1}^n\Delta_k\sum_{v=1}^nd_{v,k}^2\\
 &+\sum_{k=1}^n
   \left[\frac{\Delta_k}{8}Y_k(\sigma)^2
   +\frac{\Delta_k k}{4}Y_k(\sigma)\right]
 +R(\sigma).
 \end{aligned}
\]
Complete the square separately for each prefix:
\[
 \frac{\Delta_k}{8}Y_k^2+
 \frac{\Delta_k k}{4}Y_k
 =\frac{\Delta_k}{8}(Y_k+k)^2-
  \frac{\Delta_k}{8}k^2.
\]
Consequently, the state-independent terms combine exactly into
\eqref{eq:Aconstant}, and
\[
 \Phi_{\rm can}(\sigma)
 =A+\frac18\sum_{k=1}^n\Delta_k
     \bigl(Y_k(\sigma)+k\bigr)^2+R(\sigma).
\]
Finally, \eqref{eq:prefixrelation} gives
$Y_k(\sigma)+k=2(C_k(x)-k)$. Therefore
\[
 \frac18\sum_{k=1}^n\Delta_k(Y_k+k)^2
 =\sum_{k=1}^n\frac{\Delta_k}{2}(C_k(x)-k)^2
 =\mathcal P(x)
\]
by \eqref{eq:prefixpenalty}. This proves \eqref{eq:masteridentity}.
\end{proof}

The coefficient difference $J_{uv}-J_{uv}^{\rm main}$ is the profile integral over the union of the element and controller bands. Since $|F_v(t)|\le3$, one block pair contributes at most $9W/4$ in absolute value.  There are fewer than $n^2/2$ block pairs, so
\[
 \abs{\sum_{u<v}(J_{uv}-J_{uv}^{\rm main})\sigma_u\sigma_v}
 \le\frac98n^2W.
\]
Together with \eqref{eq:ebound}, this gives
\begin{equation}
 |R(\sigma)|\le\Gamma,
 \qquad
 \Gamma=\frac98n^2W+\frac18\sum_{v=1}^nD_v.
 \label{eq:Gamma}
\end{equation}
By \cref{lem:separation} and \eqref{eq:ebound},
\[
 \Gamma<18N^5T+N^2M.
\]
Let $\Delta_{\min}=\min_{1\le k\le n}\Delta_k$.  By \cref{lem:separation},
$\Delta_{\min}>GM=N^{60}$.  From the displayed bound on $\Gamma$,
\[
2\Gamma+1<36N^5T+2N^2M+1
=36N^{55}+2N^{42}+1.
\]
For $N\ge4$, this is smaller than $N^{60}/2$. Hence
\begin{equation}
 \frac{\Delta_{\min}}2>2\Gamma+1.
 \label{eq:gapdominance}
\end{equation}
This is the decisive scale separation: one unit of squared prefix error costs more than the largest possible favorable swing of the residual on both sides of the comparison, with an additional unit to accommodate the integer floor.

The decision bound of the constructed \AWET\ instance is
\begin{equation}
 B=\floor{A+\Gamma}.
 \label{eq:threshold}
\end{equation}
All quantities in \eqref{eq:threshold} are exactly computable rationals with denominators dividing $8$.  Actual schedule costs are integers because processing times, weights, and canonical completion times are integers.  Therefore the floor causes no loss in the yes-direction and creates a clean integer threshold.

\subsection{The ``if'' proof}

\begin{lemma}\label{lem:red2if}
If the \RPE\ instance is a yes-instance, then the constructed \AWET\ instance has a schedule of cost at most $B$.
\end{lemma}

\begin{proof}
Let $x\in\{0,1\}^n$ satisfy $C_k(x)=k$ for every $k$, and set $\sigma_v=2x_v-1$.

\par\medskip\noindent\textbf{Build the early--tardy assignment.}\quad
In block $v$, choose the coded orientation $c_i=\sigma_v\epsilon_i$, and put the anchor early.  Every block then has tardy load $\beta_v$.  By \cref{prop:setobjective}, the resulting nonempty early set has a feasible compact canonical schedule with integral completion times.

\par\medskip\noindent\textbf{Evaluate its prefix penalty.}\quad
Every prefix equation is exact, so each square in \eqref{eq:prefixpenalty} is zero and $\mathcal P(x)=0$.

\par\medskip\noindent\textbf{Use the residual bound and integer threshold.}\quad
The master identity and $|R(\sigma)|\le\Gamma$ give
\[
\Phi_{\rm can}(\sigma)=A+R(\sigma)\le A+\Gamma.
\]
The left side is an integer.  Hence
\[
\Phi_{\rm can}(\sigma)\le\lfloor A+\Gamma\rfloor=B.
\]
Thus the constructed scheduling instance is a yes-instance.
\end{proof}

\subsection{The ``only if'' proof}

\begin{lemma}\label{lem:red2onlyif}
If the constructed \AWET\ instance has a schedule of cost at most $B$, then the \RPE\ instance is a yes-instance.
\end{lemma}

\begin{proof}
Assume some feasible schedule has cost at most $B$, and choose an optimal schedule.  Apply the structural lemma to make it compact and canonical without increasing cost.

\par\medskip\noindent\textbf{Gate 1: the anchor is early.}\quad
The schedule selected above is optimal and canonical.  By the choice of $s_a$ in \eqref{eq:anchorlateratio}, a tardy anchor in such a schedule would contribute its full self term, which exceeds the explicit feasible benchmark $C^\star$.  Hence this optimal canonical schedule has the anchor early.

\par\medskip\noindent\textbf{Gate 2: every block is coded.}\quad
With an early anchor, \cref{lem:loadlock} forces $b_v=\beta_v$ in every block.  By \cref{lem:arithmetic} and \eqref{eq:loadiffcanonical}, block $v$ has a unique logical spin $\sigma_v\in\{-1,+1\}$ and selection variable $x_v=(1+\sigma_v)/2$.

\par\medskip\noindent\textbf{Gate 3: every prefix equation is exact.}\quad
Suppose some prefix equation fails.  Since $C_k(x)-k$ is an integer, at least one squared error is at least $1$, and therefore
\[
\mathcal P(x)\ge\frac{\Delta_{\min}}2.
\]
The master identity, the lower residual bound $R(\sigma)\ge-\Gamma$, and \eqref{eq:gapdominance} imply
\[
\begin{aligned}
\Phi_{\rm can}(\sigma)
&\ge A+\frac{\Delta_{\min}}2-\Gamma\\
&>A+\Gamma+1\\
&>\lfloor A+\Gamma\rfloor=B,
\end{aligned}
\]
contradicting the assumed cost.  Thus $C_k(x)=k$ for all $k$, and the \RPE\ instance is a yes-instance.
\end{proof}

\begin{figure}[p]
\centering
\begin{tikzpicture}[
  node distance=2.00cm,
  font=\small,
  >=Latex,
  box/.style={draw,rounded corners,align=center,text width=0.40\textwidth,
              minimum height=0.92cm,fill=figurehard,inner sep=7pt},
  support/.style={font=\scriptsize,align=left,text width=0.40\textwidth,
                  draw=navy,rounded corners,fill=figuremixed,inner sep=4pt,
                  anchor=west},
  arr/.style={-{Latex[length=3mm]},very thick}
]
\node[box] (feasible) {a feasible schedule has cost at most $B$};
\node[box,below=of feasible] (canonical) {an optimal compact canonical schedule has cost at most $B$};
\node[box,below=of canonical] (anchor) {the anchor is early};
\node[box,below=of anchor] (loads) {$b_v=\beta_v$ for every block $v$};
\node[box,below=of loads] (coded) {each block has one logical spin $\sigma_v$ and $x_v=(1+\sigma_v)/2$};
\node[box,below=of coded] (prefix) {$C_k(x)=k$ for every prefix $k$};
\node[box,below=of prefix] (cover) {the selected subset-columns form an exact cover};

\draw[arr] (feasible) -- node[midway,xshift=9mm,support]
 {\textbf{Supporting result:} compact canonical schedule, \cref{lem:compact}.} (canonical);
\draw[arr] (canonical) -- node[midway,xshift=9mm,support]
 {\textbf{Supporting result:} anchor tardiness benchmark, \eqref{eq:anchorlateratio}.} (anchor);
\draw[arr] (anchor) -- node[midway,xshift=9mm,support]
 {\textbf{Supporting results:} global load lock, \cref{lem:loadlock}, using \cref{eq:lockquadratic,eq:minkernelpd}.} (loads);
\draw[arr] (loads) -- node[midway,xshift=9mm,support]
 {\textbf{Supporting results:} load/canonical equivalence \eqref{eq:loadiffcanonical} and \cref{lem:arithmetic}.} (coded);
\draw[arr] (coded) -- node[midway,xshift=9mm,support]
 {\textbf{Supporting results:} exact master decomposition \cref{lem:masterdecomposition}, residual bound \eqref{eq:Gamma}, gap dominance \eqref{eq:gapdominance}, and threshold \eqref{eq:threshold}.} (prefix);
\draw[arr] (prefix) -- node[midway,xshift=9mm,support]
 {\textbf{Supporting results:} prefix-difference identity \cref{lem:prefixdifference} and \cref{lem:red1onlyif}.} (cover);
\end{tikzpicture}
\caption{Expanded only-if implication chain.  The enlarged vertical gaps keep every supporting-result callout visually separate from the implication nodes.}
\label{fig:onlyif-gates}
\end{figure}
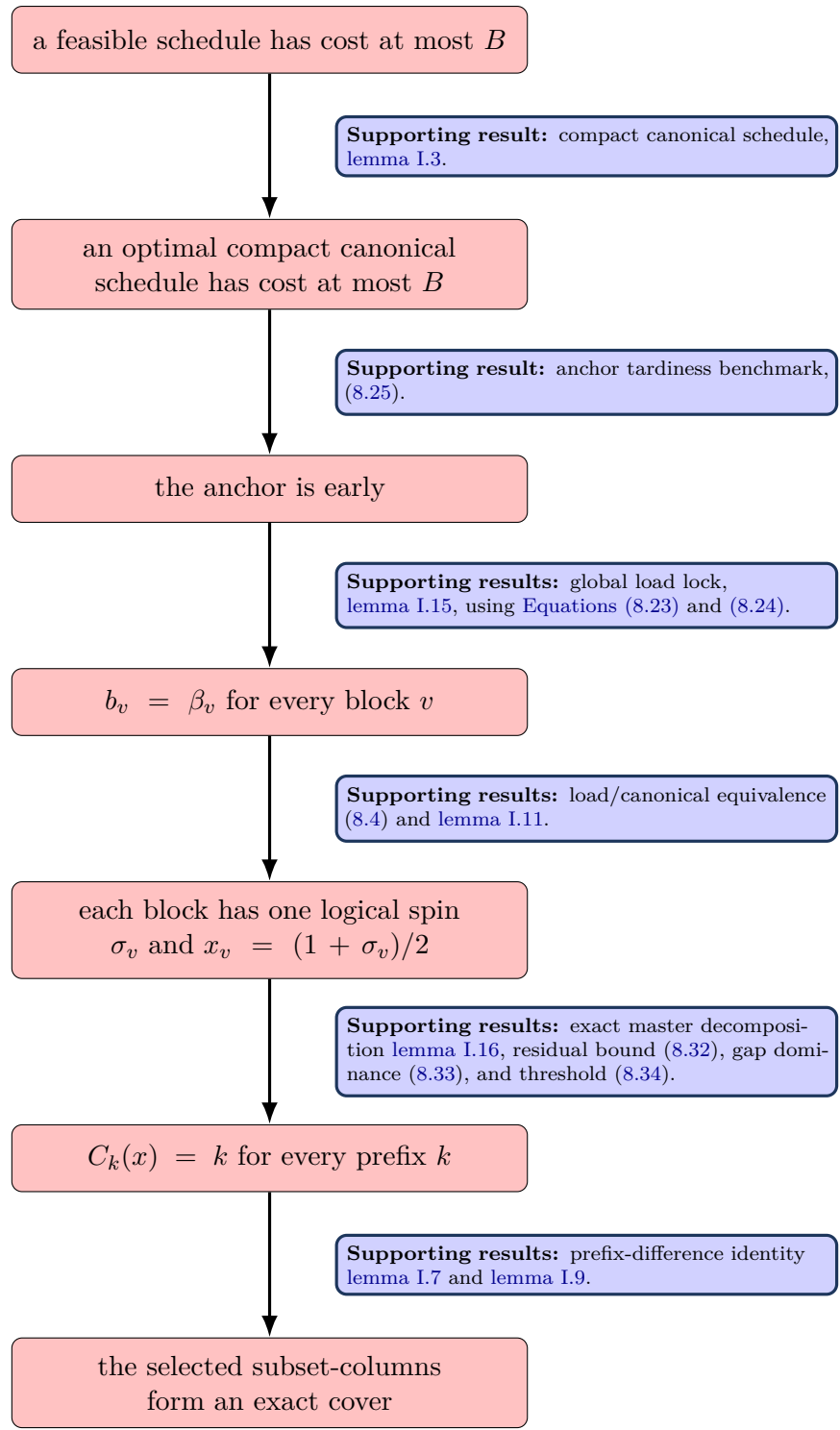
\clearpage

\begin{sectionexit}
Section~8 has established an exact reduction interface: every threshold-feasible schedule can be replaced without higher cost by a canonical schedule with an early anchor, one coded state per block, and zero prefix error, while every source solution constructs such a schedule.  Section~9 uses this interface only to audit encoding length and to verify that all output integers remain polynomially bounded even in unary.
\end{sectionexit}

\section{Polynomial and unary bounds}
\label{sec:unary}

\begin{lemma}\label{lem:polybound}
The transformation in \cref{sec:red2} can be performed in polynomial time,
and every integer in the resulting \AWET\ instance is bounded by $N^\nu$ for a
fixed constant $\nu$ independent of the source instance.
\end{lemma}

\begin{proof}
The point is not to optimize the exponent, but to verify the dependency chain:
each new scale is obtained from only polynomially many previously bounded
quantities. We give explicit coarse bounds. Since the source instances used in
the reduction have $N=n+1\ge4$, fixed numerical constants can be absorbed into
small extra powers of $N$.

\paragraph{Step 1: processing times and kernel data.}
The pair codes $Q_h$ and $q$ are fixed constants. Because
$\kappa(k)\le n<N$, $M=N^{40}$, and $T=N^{50}$, every base value $P_{vh}$ and
every state-job processing time satisfy
\[
 p_i<N^{51}.
\]
The bound $4T<\beta_v<5T$ from \cref{lem:fidentities} gives
$\beta_v<N^{52}$. Also
\[
 \lambda_v<N^{21},\qquad L_v<N^{21},\qquad K_{uv}<N^{21}.
\]
Therefore
\[
 f_v=\sum_{u=1}^nK_{vu}\beta_u
 <N\cdot N^{21}\cdot N^{52}=N^{74}.
\]
This is the first useful hierarchy: processing times are of order $N^{50}$,
whereas the additive kernel field $f_v$ is still only a fixed polynomial.

\paragraph{Step 2: early ratios, effective ratios, and early weights.}
Using \eqref{eq:earlyratio},
\[
 r_i=f_v+\lambda_vp_i<N^{75}.
\]
The explicit formulas for the effective atoms give the same bound
$|\rho_{vh}|<N^{75}$, and hence every gap $\Delta_k<N^{75}$. Each state-job
earliness weight consequently satisfies
\[
 w_i^E=p_ir_i<N^{126}.
\]
For later use, every early state-state pair coefficient obeys
\[
 \alpha_{ij}=p_ip_j\min\{r_i,r_j\}<N^{177}.
\]

\paragraph{Step 3: logical-field correction.}
For one fixed block, \eqref{eq:a0} contains fewer than $64N$ physical pair
terms. The preceding bound on $\alpha_{ij}$ therefore gives
\[
 |a_v^0|<N^{181}.
\]
Moreover, $d_{v,k}\le3$, $k<n+1=N$, and $\Delta_k<N^{75}$, so
\[
 |a_v^\star|<N^{78}.
\]
Equation \eqref{eq:Dv} gives the two-sided estimate
\[
 M=N^{40}<D_v<N^{43}.
\]
Nearest-integer rounding changes a number by at most $1/2$. Hence
\[
 \begin{aligned}
 |\eta_v|
 &\le \frac{4|a_v^0-a_v^\star|}{D_v}+\frac12\\
 &<N^{144}.
 \end{aligned}
\]
Thus the field-correction offsets are polynomially bounded even though they
are defined by a quotient: the denominator $D_v$ is itself at least
$N^{40}$.

\paragraph{Step 4: the load-lock scale and state-job tardiness data.}
There are $8n<8N<N^3$ state jobs and fewer than $N^6$ unordered state-job
pairs. In \eqref{eq:U}, the early pair sum is therefore less than $N^{183}$.
The terms involving $\eta_v$ have individual magnitude at most
\[
 |\eta_v|p_i^2<N^{144}N^{102}=N^{246}.
\]
After multiplying by the number of jobs or pairs, the largest contribution
to $U$ is less than $N^{252}$. Absorbing the three displayed sums and the
constant $1$ gives
\[
 U<N^{253},
 \qquad
 H=U+2\max_v|\eta_v|+2<N^{254}.
\]
It follows that the anchor processing time is $p_a=H<N^{254}$ and that every
state tardiness ratio and weight satisfy
\[
 s_i=HL_v+\eta_v<N^{276},
 \qquad
 w_i^L=p_is_i<N^{327}.
\]
The anchor early ratio is less than $N^{76}$, so
$w_a^E=p_ar_a<N^{330}$.

\paragraph{Step 5: the anchor tardiness weight and the decision threshold.}
To bound $C^\star$, use the canonical set objective rather than completion
times. There are fewer than $N^6$ state-state pairs. A tardy state-state pair
has cost coefficient at most
$p_ip_j\min\{s_i,s_j\}<N^{378}$, and an anchor-state early pair has coefficient
at most
\[
 p_ap_ir_i<N^{254}N^{51}N^{75}=N^{380}.
\]
Each state self-term has coefficient $p_i^2s_i<N^{378}$.  These bounds are independent of the coded signs.  There are fewer than $N^6$ state-state pairs, fewer than $N^3$ anchor-state pairs, and fewer than $N^3$ state self-terms.  Since the anchor is early, its tardiness self-term and tardiness pair terms are absent.  Consequently, every coded assignment with the anchor early satisfies the uniform bound
\[
 0\le\Phi_{\rm can}(\sigma)
 <N^{384}+N^{383}+N^{381}<N^{385}.
\]
In particular, the benchmark schedule satisfies $C^\star<N^{385}$.
Consequently,
\[
 s_a=1+\ceil{C^\star/H^2}<N^{386},
 \qquad
 w_a^L=Hs_a<N^{640}.
\]
The due date satisfies
\[
 d=p_a+\sum_{i\in\J_S}p_i<N^{255}.
\]
To bound $C_0$, average \eqref{eq:canonicalexpansion} over all sign vectors $\sigma\in\{-1,+1\}^n$.  Every linear term averages to zero, as does every quadratic term $\sigma_u\sigma_v$ with $u\ne v$.  Therefore
\[
 C_0=2^{-n}\sum_{\sigma\in\{-1,+1\}^n}\Phi_{\rm can}(\sigma),
 \qquad 0\le C_0<N^{385}.
\]
This averaging identity proves the bound; it is not an enumeration step in the reduction, since $C_0$ is computed directly by collecting the constant coefficients.  Since
$\Delta_k<N^{75}$ and $d_{v,k}\le3$, the correction in
\eqref{eq:Aconstant} is less than $N^{80}$, so $|A|<N^{386}$.
Furthermore, \eqref{eq:sepbounds} and \eqref{eq:ebound} give
$\Gamma<N^{58}$. Therefore the integer threshold
$B=\floor{A+\Gamma}$ has magnitude less than $N^{387}$.

We have thus obtained one completely explicit universal choice: every integer
written into the \AWET\ instance is smaller than $N^{650}$. The exponent is
deliberately loose; only its independence from the source instance matters.

\paragraph{Step 6: running time and unary size.}
The construction creates $8n+1=O(N)$ jobs. All displayed quantities are
computed using polynomially many additions, multiplications, comparisons,
minimum operations, and exact divisions or roundings of rationals whose
numerators and denominators have $O(\log N)$ bits. In particular, $U$ and
$C^\star$ can be evaluated by examining $O(n^2)$ job pairs, and
$\floor{A+\Gamma}$ can be computed exactly after multiplying by $8$, since
the relevant denominators divide $8$.

Hence the binary running time is polynomial in the source size. Under unary
encoding, the output contains $O(N)$ integers, each at most $N^{650}$, so its
total unary length is $O(N^{651})$. This is polynomial in $n$, which proves
the required strong reduction bound.
\end{proof}

\begin{proofcheckpoint}
The construction is numerically large but not pseudopolynomial.  ``Large'' means a high fixed power of $N$, not an exponent that depends on the source data.  Writing every output number in unary therefore still produces only polynomially many symbols.
\end{proofcheckpoint}

\begin{proof}[Proof of \cref{thm:main}]
By Gonzalez~\cite{Gonzalez1985}, \RXC\ is NP-complete. \Cref{lem:red1if,lem:red1onlyif} give a polynomial reduction from \RXC\ to \RPE. \Cref{lem:red2if,lem:red2onlyif} give a polynomial reduction from \RPE\ to \AWET. By \cref{lem:polybound}, every numerical value in the \AWET\ instance is polynomially bounded, so the reduction establishes strong NP-hardness in the standard sense of Garey and Johnson~\cite{GareyJohnson1979}. Membership in NP was noted after \cref{thm:main}. Hence \AWET\ is strongly NP-complete and, in particular, unary NP-hard.
\end{proof}

\section{Concluding remarks}

The reduction uses the asymmetry of the earliness and tardiness weights in two distinct ways. The early ratio order carries the exact-cover signal through independently positioned occurrence pairs, while the tardiness ratios and the early anchor jointly create a global positive-definite load lock. The controller pair closes each block's signed profile, preventing signal below the final element band. The field-correction offsets need not be exact: the polynomially large prefix gaps dominate the rounding error and all within-band interactions.

The proof can be summarized as four reusable design lessons.  First, convert schedule geometry into a set objective before introducing gadgets.  Second, give every local gadget a uniquely decodable arithmetic signature.  Third, separate logical signal intervals from residual intervals geometrically.  Fourth, use a positive-definite global penalty to enforce local feasibility simultaneously rather than trying to enforce each block with an isolated ad hoc penalty.

The result resolves the published strong-complexity question for positive-integer asymmetric weighted earliness--tardiness scheduling at the boundary due date $d=\sum_jp_j$~\cite{HoogeveenVandeVelde1997}.  It also shows that the difficulty is not caused by a large number of jobs: the construction uses only eight state jobs per source set and one global anchor.  The numerical exponents are deliberately loose; smaller scales may be possible without changing the proof architecture.  For complete numerical verification of both reduction directions, see the yes- and no-instances in \cref{app:worked-yes,app:worked-no}.

\begin{reusableidea}
A profile-and-gap compiler can turn cumulative combinatorial constraints into separated squared penalties.  Local arithmetic signatures identify the permitted states, a global positive-definite penalty enforces all local loads at once, and a residual budget converts approximate field correction into an exact decision gap.  This pattern is potentially useful whenever a pairwise minimum kernel can be represented by one-dimensional signed profiles.
\end{reusableidea}

\begin{finding}{Why Part II is next}
Part I has established exact intractability for unrestricted asymmetric data.  The remaining question is how much near-optimal structure survives despite that hardness.  Part II uses the same canonical objective but replaces scale separation by relaxation and rounding, producing a universal constant-factor guarantee.
\end{finding}

\clearpage

\part{General Approximation}
\setrunningunit{Part II: Approximation}
\thispagestyle{plain}
\label{part:approximation}
\begin{center}
{\Large\bfseries Anchored Semidefinite Relaxation and Deterministic Threshold Rounding\par}
\end{center}
\medskip
\begin{quote}\small
\noindent\textbf{Part synopsis.}
This Part proves a polynomial-time constant-factor approximation for unrestricted \AWET.  One early anchor is enumerated, a locally consistent semidefinite relaxation is solved, and jobs are rounded by a common marginal threshold.  Early pairs are paid by local joint consistency.  Tardy pairs and tardy self-costs form a positive-semidefinite minimum-kernel matrix, allowing a global covariance payment.  Threshold optimization gives $3+2\sqrt2+\varepsilon$.
\end{quote}

\begin{partposition}
\textbf{Established boundary.}  FPTAS results are known for symmetric or otherwise structured common-due-date settings, whereas unrestricted asymmetry remained outside those schemes.  \textbf{Result.}  For every fixed $\varepsilon>0$, unrestricted \AWET has a polynomial-time $(3+2\sqrt2+\varepsilon)$-approximation.  \textbf{Technique.}  An anchored local-consistency SDP is rounded by a deterministic marginal threshold; early pairs are paid locally, while the tardy self-cost completes a positive-semidefinite minimum kernel that pays all tardy pairs globally.
\end{partposition}

\begin{notationreset}
This Part uses $a_{ij}$ and $b_{ij}$ for early- and tardy-pair coefficients, $c_i$ for the tardy self-cost, $e_i,t_i$ for SDP marginals, and $e_{ij},t_{ij}$ for pseudo-joint masses.  The symbols $\mathcal A,\mathcal B,\mathcal U$ denote the three SDP objective accounts.  The threshold is $\theta=1-\delta$.
\end{notationreset}

\begin{proofcontract}
\begin{itemize}[leftmargin=2em,itemsep=1pt,topsep=2pt]
\item \textbf{Input:} an arbitrary positive-integer \AWET instance and a positive rational accuracy parameter $\varepsilon$.
\item \textbf{Relaxation:} enumerate an early anchor and solve one polynomial-size semidefinite program for each anchor.
\item \textbf{Rounding invariant:} the anchor remains early, early pairs are charged to pseudo-joint early mass, and tardy pairs are charged to the completed minimum kernel.
\item \textbf{Output:} the cheapest canonical schedule produced by the anchored threshold roundings.
\item \textbf{Precision obligation:} additive SDP error is absorbed into the requested $\varepsilon$ margin using only polynomially many bits.
\end{itemize}
\end{proofcontract}

\section*{Introduction and overview}
\phantomsection
\addcontentsline{toc}{section}{Introduction and overview}

This Part presents a fully specified approximation theorem for unrestricted \AWET.  The proof is constructive: it gives the SDP, the deterministic rounding rule, the schedule-recovery step, and the finite precision needed for polynomial running time.  The result complements earlier FPTAS work in symmetric or otherwise structured common-due-date models: here the two weight systems may be unrelated, so the guarantee is a universal constant rather than an approximation scheme~\cite{HallPosner1991,KovalyovKubiak1999,KellererRustogiStrusevich2020}.

\begin{howtoread}
The logical chain is
\begin{center}\small
canonical early--tardy partition $\longrightarrow$ anchored SDP
$\longrightarrow$ threshold rounding\\[3pt]
$\longrightarrow$ direct early-pair payment and PSD tardy-pair self-payment
$\longrightarrow$ constant ratio.
\end{center}
The main new inequality is \cref{AA:lem:selfpayment}.  Readers interested mainly in implementation may read \cref{AA:sec:model,AA:sec:sdp,AA:sec:algorithm} first.
\end{howtoread}

\begin{finding}{Algorithm at a glance}
For each possible early anchor $h$: (1) solve the anchored SDP; (2) read the early marginals $e_i$; (3) choose a rational threshold near $2-\sqrt2$ and put $i$ early when $e_i$ reaches it; (4) construct the canonical V-shaped schedule for that partition; and (5) evaluate its cost exactly.  Return the cheapest of the $n$ candidates.  Sections~14--19 prove why these steps are valid; Section~20 gives the implementation in full.
\end{finding}

\section{Problem definition and main theorem}
\label{AA:sec:model}

\begin{definition}[Asymmetric weighted earliness--tardiness scheduling]
An instance of \AWET consists of jobs $\J=\{1,\ldots,n\}$.  Job $i$ has positive integer processing time $p_i$, positive integer unit earliness weight $w_i^E$, and positive integer unit tardiness weight $w_i^L$.  All jobs are available at time zero and are processed nonpreemptively on one machine.  The common due date is
\begin{equation}
 d=P:=\sum_{i\in\J}p_i.
 \label{AA:eq:duedate}
\end{equation}
For completion time $C_i$, define
\[
 E_i=(d-C_i)^+,
 \qquad
 T_i=(C_i-d)^+.
\]
The objective is
\begin{equation}
 F=\sum_{i\in\J}\bigl(w_i^E E_i+w_i^L T_i\bigr).
 \label{AA:eq:objective}
\end{equation}
\end{definition}

Positive rational data can be scaled to integers without changing approximation ratios.  Define
\begin{equation}
 r_i=\frac{w_i^E}{p_i},
 \qquad
 s_i=\frac{w_i^L}{p_i}.
 \label{AA:eq:ratios}
\end{equation}

\begin{definition}[Approximation ratio]
A polynomial-time algorithm is a $C$-approximation for a minimization problem if it always returns a feasible solution of value at most $C\OPT$.
\end{definition}

\begin{theorem}[Main approximation theorem]
\label{AA:thm:main}
For every $\varepsilon>0$, unrestricted \AWET has a polynomial-time
\[
 \bigl(3+2\sqrt2+\varepsilon\bigr)
\]
-approximation.  When the accuracy parameter is supplied as a positive rational, the running time is polynomial in the input length and in $\log\!\bigl(1/\min\{\varepsilon,1\}\bigr)$.
\end{theorem}

Since $3+2\sqrt2\approx5.828427$, the theorem gives a universal constant ratio for the fully asymmetric problem.

\begin{corollary}[Convenient threshold]
\label{AA:cor:six}
The same method with threshold $\theta=7/12$ gives factor $6$ in the exact-SDP model and factor $6+\varepsilon$ in the ordinary finite-precision model.
\end{corollary}

\begin{intuition}
The early and tardy sides are deliberately treated differently.  There is no early self-cost, so early pairs are paid through a threshold greater than one half.  The tardy side does have a self-cost.  Those diagonal terms complete the tardy pair coefficients into a positive-semidefinite minimum-kernel matrix, which pays globally for all rounded tardy pairs.
\end{intuition}

\section{Canonical side assignment}
\label{AA:sec:canonical}

\begin{lemma}[Compact canonical schedule]
\label{AA:lem:compact}
There is an optimal schedule with no idle time between consecutive jobs, with at least one completion exactly at $d$, and with the following V-shaped orders: read outward from $d$, jobs completing no later than $d$ are in nonincreasing $r_i$ order, and tardy jobs are in nonincreasing $s_i$ order.
\end{lemma}

\begin{proof}
Choose an optimal schedule with minimum internal idle time.  An idle interval before $d$ can be shortened by shifting the block to its left toward $d$; every affected early completion moves toward its due date.  The symmetric operation removes an interval after $d$.  If an interval contains $d$, move both adjacent blocks toward $d$.  Thus an optimum can be chosen compact.

Fix a compact sequence.  Its length is $P=d$.  If the sequence begins at $z\ge0$, its cost is a convex piecewise-linear function of $z$.  A minimum may be chosen at $z=0$, when the final job completes at $d$, or at a breakpoint, when another completion equals $d$.

For adjacent early jobs $i,j$, with $j$ closer to $d$, the order-dependent contribution is $w_i^Ep_j=p_ip_jr_i$; reversing them gives $p_ip_jr_j$.  Thus the smaller ratio lies farther from $d$.  For tardy jobs, with $i$ closer to $d$ than $j$, the order-dependent term is $w_j^Lp_i=p_ip_js_j$, so the larger tardiness ratio lies closer to $d$.
\end{proof}

For $i\ne j$, define
\begin{equation}
 a_{ij}=p_ip_j\min\{r_i,r_j\},
 \label{AA:eq:aij}
\end{equation}
\begin{equation}
 b_{ij}=p_ip_j\min\{s_i,s_j\},
 \label{AA:eq:bij}
\end{equation}
and define the tardy self coefficient
\begin{equation}
 c_i=p_i^2s_i=p_iw_i^L.
 \label{AA:eq:ci}
\end{equation}
All these coefficients are positive integers, since
\[
 a_{ij}=\min\{w_i^Ep_j,p_iw_j^E\},
 \qquad
 b_{ij}=\min\{w_i^Lp_j,p_iw_j^L\}.
\]

\begin{proposition}[Canonical partition objective]
\label{AA:prop:canonical}
For every nonempty $E\subseteq\J$, let $T=\J\setminus E$.  The canonical schedule that orders $E$ by nondecreasing $r_i$ and ends it at $d$, then orders $T$ by nonincreasing $s_i$ and starts it at $d$, has cost
\begin{equation}
 \Phi(E)=
 \sum_{\{i,j\}\subseteq E}a_{ij}
 +\sum_{\{i,j\}\subseteq T}b_{ij}
 +\sum_{i\in T}c_i.
 \label{AA:eq:canonical}
\end{equation}
Every nonempty $E$ gives a feasible schedule of this form, and
\begin{equation}
 \OPT=\min_{\varnothing\ne E\subseteq\J}\Phi(E).
 \label{AA:eq:opt-set}
\end{equation}
\end{proposition}

\begin{proof}
The early block starts at
\[
 d-\sum_{i\in E}p_i=\sum_{i\in T}p_i\ge0,
\]
so the schedule is feasible.  For an early pair, the nondecreasing-$r$ order contributes $p_ip_j\min\{r_i,r_j\}=a_{ij}$.  For a tardy pair, the nonincreasing-$s$ order contributes $b_{ij}$.  Each tardy job contributes its own processing time to its delay, producing $c_i$.  This proves \eqref{AA:eq:canonical}.

By \cref{AA:lem:compact}, some optimal schedule has this canonical form for the nonempty set of jobs completing no later than $d$.  Conversely every nonempty set gives a feasible canonical schedule.  Therefore the minimum of \eqref{AA:eq:canonical} equals $\OPT$.
\end{proof}

\begin{howto}
Once $E$ has been selected, sort it by increasing $r_i$, begin it at $d-\sum_{i\in E}p_i$, and finish it at $d$.  Sort $T$ by decreasing $s_i$ and begin it at $d$.  Ties may be broken arbitrarily.  The cost can be evaluated exactly from \eqref{AA:eq:canonical}.
\end{howto}

\begin{proofcheckpoint}
All continuous-time scheduling choices have now disappeared.  The approximation task is exactly the binary minimization of \eqref{AA:eq:canonical} over nonempty $E$.
\end{proofcheckpoint}

\section{Why one early anchor is enumerated}
\label{AA:sec:anchor}

The rounding threshold will exceed $1/2$.  Without an extra condition, every early marginal could lie below that threshold, producing an inadmissible empty early set.  For every $h\in\J$, we therefore solve one relaxation with job $h$ forced early.

\begin{lemma}[Correct-anchor lower bound]
\label{AA:lem:correctanchor}
If $E^*$ is an optimal early set and $h\in E^*$, then the optimum of any valid relaxation with $h$ forced early is at most $\OPT$.
\end{lemma}

\begin{proof}
The incidence vector of $E^*$ is feasible for the anchored integral problem.  A relaxation cannot have larger optimum value.
\end{proof}

\begin{intuition}
The enumeration guesses only one certificate that the early side is nonempty, not the whole partition.  At least one of the $n$ guesses is compatible with the unknown optimum.
\end{intuition}

\section{The anchored semidefinite relaxation}
\label{AA:sec:sdp}

Use label $y_i=+1$ for early and $y_i=-1$ for tardy.  The SDP uses unit vectors $v_0,v_1,\ldots,v_n$ and sets
\begin{equation}
 m_i=\langle v_0,v_i\rangle,
 \qquad
 \rho_{ij}=\langle v_i,v_j\rangle.
 \label{AA:eq:m-rho}
\end{equation}
Define pseudo-marginals
\begin{equation}
 e_i=\frac{1+m_i}{2},
 \qquad
 t_i=\frac{1-m_i}{2}=1-e_i,
 \label{AA:eq:marginals}
\end{equation}
and pseudo-joint quantities
\begin{equation}
 e_{ij}=\frac{1+m_i+m_j+\rho_{ij}}4,
 \label{AA:eq:eij}
\end{equation}
\begin{equation}
 t_{ij}=\frac{1-m_i-m_j+\rho_{ij}}4.
 \label{AA:eq:tij}
\end{equation}
The two mixed local atoms are
\begin{equation}
 e_i-e_{ij}=\frac{1+m_i-m_j-\rho_{ij}}4,
 \qquad
 e_j-e_{ij}=\frac{1-m_i+m_j-\rho_{ij}}4.
 \label{AA:eq:mixedatoms}
\end{equation}
For integral vectors $v_i=y_iv_0$, these four quantities are the four ordinary joint-event indicators.

\begin{remark}[Probabilistic interpretation]
Suppose integral early--tardy partitions are combined with nonnegative weights that sum to one.  The weights may be viewed as probabilities of a random integral partition.  Then
\[
 m_i=\mathbb E[y_i],
 \qquad
 \rho_{ij}=\mathbb E[y_iy_j],
\]
and equations~\eqref{AA:eq:marginals}--\eqref{AA:eq:tij} have their ordinary probability meanings:
\[
 e_i=\Pr(i\in E),\qquad t_i=\Pr(i\in T),
 \qquad
 e_{ij}=\Pr(i,j\in E),\qquad t_{ij}=\Pr(i,j\in T).
\]
Indeed, for example,
\[
 \Pr(i\in E)=\mathbb E\!\left[\frac{1+y_i}{2}\right]
 =\frac{1+m_i}{2},
\]
and
\[
 \Pr(i,j\in E)
 =\mathbb E\!\left[\frac{(1+y_i)(1+y_j)}4\right]
 =\frac{1+m_i+m_j+\rho_{ij}}4.
\]
The tardy formulas follow in the same way from $(1-y_i)/2$.  A general feasible SDP point need not arise from a global probability distribution over integral partitions; this is why $e_i,t_i,e_{ij},t_{ij}$ are called pseudo-marginals and pseudo-joint quantities.  The probability interpretation will be used later to construct a transparent feasible point in the worked example; it does not make the rounding algorithm randomized.
\end{remark}

\begin{definition}[Anchored SDP]
\label{AA:def:sdp}
For anchor $h$, the relaxation $\SDP_h$ minimizes
\begin{equation}
 L_h=
 \sum_{i<j}a_{ij}e_{ij}
 +\sum_{i<j}b_{ij}t_{ij}
 +\sum_i c_it_i
 \label{AA:eq:sdpobjective}
\end{equation}
over unit-vector Gram matrices, subject to
\begin{align}
 e_{ij}&\ge0, & t_{ij}&\ge0,
 \label{AA:eq:local1}\\
 e_i-e_{ij}&\ge0, & e_j-e_{ij}&\ge0
 \label{AA:eq:local2}
\end{align}
for every $i<j$, and
\begin{equation}
 m_h=1.
 \label{AA:eq:anchor}
\end{equation}
Equivalently, use a matrix $X\succeq0$ with $X_{kk}=1$, $m_i=X_{0i}$, and $\rho_{ij}=X_{ij}$.
\end{definition}

Because $v_0$ and $v_h$ are unit vectors, $m_h=1$ implies $v_h=v_0$ and $e_h=1$.

\begin{proposition}[Validity]
\label{AA:prop:sdpvalid}
Every integral partition with $h$ early embeds in $\SDP_h$ with the same objective value.  Therefore, if $h$ belongs to an optimal early set,
\begin{equation}
 \SDP_h\le\OPT.
 \label{AA:eq:sdp-lower}
\end{equation}
\end{proposition}

\begin{proof}
Set $v_i=y_iv_0$.  Then $m_i=y_i$, $\rho_{ij}=y_iy_j$, and the local quantities become the four $0$--$1$ joint-event indicators.  The SDP objective is exactly \eqref{AA:eq:canonical}.  Apply \cref{AA:lem:correctanchor}.
\end{proof}

\begin{howto}
Use one symmetric matrix variable $X$ of order $n+1$.  Add $X\succeq0$, $X_{kk}=1$, the four pairwise local inequalities, and $X_{0h}=1$.  The objective is linear in $X$.  The SDP has $O(n^2)$ constraints.
\end{howto}

\begin{remark}[Why the local inequalities matter]
The Gram constraint alone need not make every off-diagonal pseudo-joint quantity nonnegative.  The early proof uses $t_{ij}\ge0$, and the tardy proof treats $\sum b_{ij}t_{ij}$ as a nonnegative objective component.  Local consistency is therefore essential.
\end{remark}

\section{Deterministic threshold rounding}
\label{AA:sec:rounding}

Choose
\begin{equation}
 \theta\in\left(\frac12,1\right),
 \qquad
 \delta=1-\theta\in\left(0,\frac12\right),
 \label{AA:eq:theta-delta}
\end{equation}
and round
\begin{equation}
 E_\theta=\{i:e_i\ge\theta\},
 \qquad
 T_\theta=\J\setminus E_\theta.
 \label{AA:eq:rounding}
\end{equation}
The anchor satisfies $e_h=1$, so $E_\theta$ is nonempty.

Split the SDP objective into
\begin{equation}
 \mathcal A=\sum_{i<j}a_{ij}e_{ij},
 \qquad
 \mathcal B=\sum_{i<j}b_{ij}t_{ij},
 \qquad
 \mathcal U=\sum_i c_it_i.
 \label{AA:eq:ABC}
\end{equation}
Thus $L_h=\mathcal A+\mathcal B+\mathcal U$.

\begin{intuition}
Two rounded-early jobs have early marginals whose sum exceeds one, forcing positive pseudo-joint early mass.  A rounded-tardy job has tardy marginal at least $\delta$.  These two certificates drive the two different charging arguments.
\end{intuition}

\section{Charging the early side}
\label{AA:sec:early}

The local quantities satisfy
\begin{equation}
 e_{ij}=e_i+e_j-1+t_{ij}.
 \label{AA:eq:jointidentity}
\end{equation}

\begin{lemma}[Early-pair threshold bound]
\label{AA:lem:early}
\begin{equation}
 \sum_{\{i,j\}\subseteq E_\theta}a_{ij}
 \le
 \frac{\mathcal A}{2\theta-1}
 =\frac{\mathcal A}{1-2\delta}.
 \label{AA:eq:earlybound}
\end{equation}
\end{lemma}

\begin{proof}
If $i,j\in E_\theta$, then by \eqref{AA:eq:jointidentity} and $t_{ij}\ge0$,
\[
 e_{ij}\ge e_i+e_j-1\ge2\theta-1>0.
\]
Thus $a_{ij}\le a_{ij}e_{ij}/(2\theta-1)$.  Sum over rounded-early pairs and enlarge the right side to all pairs.
\end{proof}

\begin{howtoread}
This explains why the threshold must exceed $1/2$.  At exactly $1/2$, two jobs can each have marginal $1/2$ while their joint-early mass is zero.
\end{howtoread}

\section{Tardy minimum-kernel self-payment}
\label{AA:sec:tardy}

\subsection{The completed tardy kernel}

Define the full matrix
\begin{equation}
 K^{\mathrm{tar}}_{ij}=p_ip_j\min\{s_i,s_j\}
 \qquad(i,j\in\J).
 \label{AA:eq:KT}
\end{equation}
Its off-diagonal entries are $b_{ij}$ and its diagonal entries are
\[
 K^{\mathrm{tar}}_{ii}=p_i^2s_i=c_i.
\]
Here the superscript $\mathrm{tar}$ labels the tardy kernel; throughout the tutorial, vector and matrix transposes are written with $^\top$.

\begin{lemma}[Minimum-kernel PSD]
\label{AA:lem:minpsd}
The matrix $K^{\mathrm{tar}}$ is positive semidefinite and entrywise nonnegative.
\end{lemma}

\begin{proof}
For $u\ge0$, define $q_i(u)=p_i\one\{u\le s_i\}$.  Then
\begin{equation}
 K^{\mathrm{tar}}_{ij}=\int_0^\infty q_i(u)q_j(u)\,du.
 \label{AA:eq:kernelintegral}
\end{equation}
Therefore, for every $x\in\R^n$,
\[
 x^\top K^{\mathrm{tar}} x
 =\int_0^\infty\left(\sum_i x_iq_i(u)\right)^2du\ge0.
\]
Entrywise nonnegativity is immediate.
\end{proof}

\begin{intuition}
At level $u$, only jobs with $s_i\ge u$ remain active.  The minimum kernel integrates the square of their weighted sum.  The tardy self-cost is precisely the diagonal needed to obtain this Gram matrix.
\end{intuition}

\subsection{Covariance domination}

Set
\begin{equation}
 z_i=\frac{v_0-v_i}{2}.
 \label{AA:eq:zi}
\end{equation}
Then
\begin{equation}
 \langle z_i,z_j\rangle=t_{ij},
 \qquad
 \|z_i\|^2=t_i,
 \qquad
 \langle z_i,v_0\rangle=t_i.
 \label{AA:eq:zidentities}
\end{equation}
Let $T^{\rm SDP}=(t_{ij})_{i,j}$, with $t_{ii}=t_i$, and let $t=(t_1,\ldots,t_n)^\top$.

\begin{lemma}[Covariance domination]
\label{AA:lem:covariance}
\begin{equation}
 T^{\rm SDP}-t t^\top\succeq0.
 \label{AA:eq:covariance}
\end{equation}
\end{lemma}

\begin{proof}
Let $u_i=z_i-t_iv_0$.  By \eqref{AA:eq:zidentities},
\[
 \langle u_i,u_j\rangle=t_{ij}-t_it_j.
\]
Thus $T^{\rm SDP}-t t^\top$ is the Gram matrix of the $u_i$.
\end{proof}

\begin{howtoread}
The matrix $T^{\rm SDP}-t t^\top$ behaves like a covariance matrix.  The SDP joint-tardy matrix dominates the rank-one product of its tardy marginals in positive-semidefinite order, even though entrywise domination need not hold.
\end{howtoread}

\subsection{The payment inequality}

For positive-semidefinite matrices $A,B$,
\begin{equation}
 \Tr(AB)\ge0,
 \label{AA:eq:tracepsd}
\end{equation}
because $\Tr(AB)=\Tr(A^{1/2}BA^{1/2})$.

\begin{lemma}[PSD tardy-pair self-payment]
\label{AA:lem:selfpayment}
\begin{equation}
 \sum_{\{i,j\}\subseteq T_\theta}b_{ij}
 \le
 \frac{\mathcal U+2\mathcal B}{2\delta^2}.
 \label{AA:eq:selfpayment}
\end{equation}
\end{lemma}

\begin{proof}
By \cref{AA:lem:minpsd,AA:lem:covariance} and \eqref{AA:eq:tracepsd},
\[
 \Tr\!\left(K^{\mathrm{tar}}\bigl(T^{\rm SDP}-t t^\top\bigr)\right)\ge0,
\]
so
\begin{equation}
 t^\top K^{\mathrm{tar}} t\le\Tr(K^{\mathrm{tar}} T^{\rm SDP}).
 \label{AA:eq:tracecompare}
\end{equation}
The right side equals
\begin{equation}
 \Tr(K^{\mathrm{tar}} T^{\rm SDP})
 =\sum_i c_it_i+2\sum_{i<j}b_{ij}t_{ij}
 =\mathcal U+2\mathcal B.
 \label{AA:eq:tracesdp}
\end{equation}
If $i\in T_\theta$, then $t_i=1-e_i>\delta$.  Since $K^{\mathrm{tar}}$ is entrywise nonnegative,
\begin{align}
 t^\top K^{\mathrm{tar}} t
 &\ge2\sum_{\{i,j\}\subseteq T_\theta}b_{ij}t_it_j
 \notag\\
 &\ge2\delta^2\sum_{\{i,j\}\subseteq T_\theta}b_{ij}.
 \label{AA:eq:trace-lower}
\end{align}
Combine \eqref{AA:eq:tracecompare}--\eqref{AA:eq:trace-lower}.
\end{proof}

\begin{proofcheckpoint}
The tardy self-cost does more than pay for individual tardy jobs.  It supplies the diagonal that turns the tardy pair matrix into the PSD kernel $K^{\mathrm{tar}}$.  The trace inequality then pays for all rounded tardy pairs globally, with no vertex-degree factor.
\end{proofcheckpoint}

\begin{lemma}[Tardy self-cost bound]
\label{AA:lem:unary}
\begin{equation}
 \sum_{i\in T_\theta}c_i
 \le\frac{\mathcal U}{\delta}.
 \label{AA:eq:unarybound}
\end{equation}
\end{lemma}

\begin{proof}
For $i\in T_\theta$, $t_i>\delta$, so $c_i\le c_it_i/\delta$.  Sum over $T_\theta$.
\end{proof}

\section{Approximation ratio and threshold optimization}
\label{AA:sec:ratio}

Combining \cref{AA:lem:early,AA:lem:selfpayment,AA:lem:unary} gives
\begin{equation}
 \Phi(E_\theta)
 \le
 \frac{\mathcal A}{1-2\delta}
 +\frac{\mathcal B}{\delta^2}
 +\left(\frac1\delta+\frac1{2\delta^2}\right)\mathcal U.
 \label{AA:eq:combined}
\end{equation}
Because $0<\delta<1/2$,
\begin{equation}
 \frac1\delta+\frac1{2\delta^2}
 \le\frac1{\delta^2}.
 \label{AA:eq:ucompare}
\end{equation}
Define
\begin{equation}
 C(\delta)=
 \max\left\{
 \frac1{1-2\delta},
 \frac1{\delta^2}
 \right\}.
 \label{AA:eq:Cdelta}
\end{equation}
Then
\begin{equation}
 \Phi(E_\theta)\le C(\delta)L_h.
 \label{AA:eq:rounded-vs-sdp}
\end{equation}

\begin{theorem}[Exact-SDP rounding theorem]
\label{AA:thm:exactrounding}
For every anchor $h$, every feasible point of $\SDP_h$ with objective value $L_h$, and every $\delta\in(0,1/2)$, threshold rounding with $\theta=1-\delta$ returns a nonempty early set satisfying
\[
 \Phi(E_\theta)\le C(\delta)L_h.
\]
In particular, rounding an optimal feasible point gives
\[
 \Phi(E_\theta)\le C(\delta)\SDP_h.
\]
For an anchor in an optimal early set, this further implies
\[
 \Phi(E_\theta)\le C(\delta)\OPT.
\]
\end{theorem}

\begin{proof}
The anchor makes the rounded early set nonempty, and the first inequality is \eqref{AA:eq:rounded-vs-sdp}.  The feasible region is compact, so an optimal feasible point exists; substituting $L_h=\SDP_h$ gives the second inequality.  The final inequality uses \eqref{AA:eq:sdp-lower}.
\end{proof}

The two terms in \eqref{AA:eq:Cdelta} are equal when
\begin{equation}
 \delta^2=1-2\delta.
 \label{AA:eq:balance}
\end{equation}
The root in $(0,1/2)$ is
\begin{equation}
 \delta^*=\sqrt2-1,
 \qquad
 \theta^*=2-\sqrt2.
 \label{AA:eq:optimal-threshold}
\end{equation}
At this value,
\begin{equation}
 C(\delta^*)=\frac1{(\sqrt2-1)^2}=3+2\sqrt2.
 \label{AA:eq:optimal-constant}
\end{equation}

\begin{intuition}
A smaller $\delta$ means a larger early threshold.  This improves the early denominator $1-2\delta$ but worsens the tardy denominator $\delta^2$.  The optimal threshold balances the two losses exactly.
\end{intuition}

\begin{corollary}[Optimized exact-SDP factor]
\label{AA:cor:optimized}
Enumerating all anchors and using $\theta^*=2-\sqrt2$ gives factor $3+2\sqrt2$ in the exact-SDP model.
\end{corollary}

\begin{proof}
Choose an anchor from an optimal early set and apply \cref{AA:thm:exactrounding}.  Returning the cheapest result over all anchors cannot be worse.
\end{proof}

For the rational threshold $\theta=7/12$, $\delta=5/12$, and
\[
 \frac1{1-2\delta}=6,
 \qquad
 \frac1{\delta^2}=\frac{144}{25}<6,
 \qquad
 \frac1\delta+\frac1{2\delta^2}=\frac{132}{25}<6.
\]
This proves \cref{AA:cor:six} in the exact-SDP model.

\begin{howto}
Use $\theta=7/12$ for the simplest proof and implementation.  Use a rational approximation to $2-\sqrt2$ when the best constant is desired.
\end{howto}

\section{Finite precision and polynomial running time}
\label{AA:sec:precision}

\paragraph{Precision convention.}
The Turing-model algorithm receives positive rational accuracy parameters.  We may assume $0<\varepsilon\le1$ and $0<\alpha\le1$: if a requested tolerance exceeds $1$, replacing it by $1$ only strengthens the resulting guarantee.  The usual statement for an arbitrary real $\varepsilon>0$ follows by supplying any smaller positive rational.  Thus all logarithmic precision bounds below are nonnegative.

A numerical SDP solver normally returns a point that is feasible only up to tolerance, whereas the threshold proof uses exact local inequalities and an exact Gram constraint.  The goal of this section is not to demand unrealistic numerical precision.  It is to convert a certified approximate point into an exactly feasible rational point while spending only the additive objective tolerance already reserved in the approximation analysis.  Exact facial reduction handles the anchor, an explicit rational interior point supplies uniform slack, and a small convex mixture repairs all residuals simultaneously.

\paragraph{Weak-optimization interface used below.}
We use the standard rational weak-optimization theorem for a well-bounded rational spectrahedron in affine coordinates; see \citet{GrotschelLovaszSchrijver1988}, together with the semidefinite-programming formulations in \citet{Alizadeh1995,VandenbergheBoyd1996}.  Suppose the feasible region lies in a known ball of radius $R$ and contains a known relative ball of radius $r>0$.  Given positive rational feasibility and objective tolerances $\eta_{\rm f}$ and $\eta_{\rm o}$, weak optimization returns a rational point, in time polynomial in the rational encoding length, $\log(R/r)$, $\log(1/\eta_{\rm f})$, and $\log(1/\eta_{\rm o})$, such that the affine matrix has least eigenvalue at least $-\eta_{\rm f}$, every specified affine slack is at least $-\eta_{\rm f}$, and the objective is at most the true optimum plus $\eta_{\rm o}$.  If the quoted weak-membership theorem is formulated with a Euclidean distance tolerance, divide that tolerance by a rational upper bound on the operator norms of the affine matrix map, the slack functionals, and the objective.  Those norms have polynomial encoding length here, so this rescaling preserves polynomial time.  The next lemma verifies the required inner and outer bounds for the reduced anchored SDP and then repairs the weak output to exact feasibility.

\begin{lemma}[Positive integral floor]
\label{AA:lem:floor}
If $n\ge2$, then $\OPT\ge1$.
\end{lemma}

\begin{proof}
If a feasible canonical partition has a tardy job, its objective contains a positive integer self-term $c_i$.  If all jobs are early, then $n\ge2$ and the objective contains a positive integer early-pair term $a_{ij}$.
\end{proof}

\begin{lemma}[Exact rational realization of an anchored SDP point]
\label{AA:lem:rational-realization}
Fix an anchor $h$ and a rational tolerance $0<\alpha\le1$.  For $n\ge2$, one can compute, in time polynomial in the binary input length and $\log(1/\alpha)$, a rational Gram matrix $X^{(h)}$ that satisfies every constraint of $\SDP_h$ exactly and whose objective value is at most
\[
 \SDP_h+\alpha.
\]
\end{lemma}

\begin{proof}
The proof first removes the anchor degeneracy, then obtains a near-optimal rational point, and finally repairs every feasibility residual without exceeding the reserved objective tolerance.  We give the feasibility repair explicitly.

\paragraph{Step 1: reduce the anchor face exactly.}
The Gram matrix $X$ is indexed by $\{0\}\cup\J$ and satisfies
\[
 X\succeq0,
 \qquad X_{00}=X_{hh}=1,
 \qquad X_{0h}=1.
\]
Writing $\mathbf e_k$ for the $k$th standard coordinate vector, consequently
\[
 (\mathbf e_0-\mathbf e_h)^\top X(\mathbf e_0-\mathbf e_h)=1+1-2=0.
\]
For a positive-semidefinite matrix, $z^\top X z=0$ implies $Xz=0$.  Hence rows and columns $0$ and $h$ are identical.  Delete the duplicate row and column $h$ and optimize over a reduced matrix
\[
 Y\in\mathbb S^n,
 \qquad \text{indexed by }\{0\}\cup(\J\setminus\{h\}).
\]
Conversely, duplicating row and column $0$ reconstructs the original anchored matrix.  We impose the unit diagonal identically by using only the off-diagonal entries of $Y$ as variables.

For an anchor pair $h,i$, duplication gives
\[
 e_{hi}=e_i,
 \qquad t_{hi}=0,
 \qquad e_i-e_{hi}=0,
 \qquad e_h-e_{hi}=t_i.
\]
The two zero relations are identities, and $e_i,t_i\ge0$ follow from $Y\succeq0$ and unit diagonal.  Thus only the local inequalities for pairs $i,j\ne h$ need be retained in the reduced formulation.

\paragraph{Step 2: exhibit a rational relative-interior point.}
Take
\[
 Y^\circ=I_n.
\]
For every pair $i,j\ne h$, one has $m_i=m_j=\rho_{ij}=0$, so all four local atoms equal $1/4$.  For an anchor pair, the two nontrivial marginals equal $1/2$.  Also $Y^\circ\succ0$ with smallest eigenvalue one.

For explicit boundedness, add the redundant rational inequalities
\[
 -1\le Y_{k\ell}\le1
 \qquad(k\ne\ell),
\]
which every positive-semidefinite unit-diagonal matrix already satisfies.  The point $I_n$ has unit slack in these inequalities.  Moreover, if $E$ is symmetric with zero diagonal and $\|E\|_F\le1/8$, then
\[
 \lambda_{\min}(I_n+E)\ge1-\|E\|_2\ge\frac78.
\]
Each nonanchor local atom changes by at most $3\|E\|_F/4\le3/32$, so it remains at least $5/32$; each anchor marginal changes by at most $\|E\|_F/2\le1/16$.  Hence, in the affine space of unit-diagonal symmetric matrices with the Frobenius norm, the reduced spectrahedron contains the relative ball of radius $1/8$ about $I_n$.  The box inequalities also give
\[
 \|Y-I_n\|_F\le\sqrt{n(n-1)}<n,
\]
so the same feasible region lies in the explicitly known outer ball of radius $n$ about $I_n$.

\paragraph{Step 3: obtain a certified weak solution.}
Put
\[
 W:=\sum_{i<j}(a_{ij}+b_{ij})+\sum_i c_i.
\]
Since $n\ge2$ and all coefficients are positive integers, $W\ge1$.  Define
\begin{equation}\label{AA:eq:gamma}
 \gamma:=\frac{\alpha}{6W},
 \qquad
 \tau:=\frac{\gamma}{8}.
\end{equation}
Because $0<\alpha\le1$, one has $0<\gamma\le1/6$.  Invoke the weak-optimization interface above on the reduced, explicitly bounded SDP, with the diagonal equations built into the parameterization, feasibility tolerance $\eta_{\rm f}=\tau$, and objective tolerance $\eta_{\rm o}=\alpha/2$.  It returns a rational matrix $\widetilde Y$ such that
\begin{align}
 \widetilde Y+\tau I_n&\succeq0,
 \label{AA:eq:weak-psd}\\
 s_\nu(\widetilde Y)&\ge-\tau
 \quad\text{for every retained local or box slack }s_\nu,
 \label{AA:eq:weak-slacks}\\
 \ell_h(\widetilde Y)&\le\SDP_h+\frac{\alpha}{2},
 \label{AA:eq:weak-objective}
\end{align}
where $\ell_h$ denotes the reduced anchored objective.  The relative inner radius $1/8$, outer radius $n$, and polynomial-bit coefficient norms established above are exactly the conditioning data required by the stated weak-optimization interface.

\paragraph{Step 4: mix with the interior point.}
Set
\[
 Y^r:=(1-\gamma)\widetilde Y+\gamma I_n.
\]
This matrix is rational and keeps the unit diagonal exactly.  From \eqref{AA:eq:weak-psd},
\[
 Y^r\succeq\bigl(\gamma-(1-\gamma)\tau\bigr)I_n
 \succeq(\gamma-\tau)I_n
 =\frac{7\gamma}{8}I_n\succ0.
\]
Every retained local inequality has slack at least $1/4$ at $I_n$, and therefore
\[
 s_\nu(Y^r)
 \ge-(1-\gamma)\tau+\frac\gamma4
 \ge\frac\gamma4-\tau
 =\frac\gamma8>0.
\]
The redundant box constraints are repaired by the same calculation.  Duplicating row and column $0$ now produces a rational matrix $X^{(h)}$ satisfying the original anchor equation, positive-semidefinite constraint, unit-diagonal equations, and all local inequalities exactly.

It remains to control the objective.  The approximate box constraints and $\tau\le1$ imply $|\widetilde Y_{k\ell}|\le2$.  Hence every pair pseudo-atom has absolute value at most $7/4$, every unary marginal $t_i$ has absolute value at most $3/2$, and
\[
 |\ell_h(\widetilde Y)|\le2W,
 \qquad
 0\le\ell_h(I_n)\le W.
\]
Using \eqref{AA:eq:weak-objective},
\begin{align*}
 \ell_h(Y^r)
 &\le \ell_h(\widetilde Y)
      +\gamma\left|\ell_h(I_n)-\ell_h(\widetilde Y)\right|\\
 &\le \SDP_h+\frac\alpha2+3\gamma W\\
 &\le \SDP_h+\alpha.
\end{align*}
By \eqref{AA:eq:gamma}, the last added term is exactly $3\gamma W=\alpha/2$.

\paragraph{Step 5: bit complexity and exact thresholding.}
The binary lengths of $W$, $\gamma$, and $\tau$ are polynomial in the input length and $\log(1/\alpha)$.  Thus the weak optimization and the rational convex combination use polynomially many bit operations.  The entries of $X^{(h)}$, and hence the marginals $e_i=(1+X^{(h)}_{0i})/2$, are exact rationals.  Comparison with a rational threshold $\theta$ is therefore exact, including the boundary case $e_i=\theta$.
\end{proof}

\begin{howtoread}
The weak solver is used only to approach the optimum.  Exact feasibility is restored afterward by an explicitly quantified convex mixture with the rational interior point $I_n$.  The mixture consumes only the prescribed additive tolerance, so every rounding inequality proved for feasible SDP points remains valid without an appeal to floating-point continuity.
\end{howtoread}

\paragraph{A rational near-optimal threshold.}
Let $0<\varepsilon\le1$, put
\[
 \eta:=\frac{\varepsilon}{512},
\]
and recall $\delta^*=\sqrt2-1$, the unique zero in $(0,1/2)$ of $g(x)=x^2+2x-1$.  Exact rational evaluation gives
\[
 g\!\left(\frac{53}{128}\right)=-\frac7{16384}<0,
 \qquad
 g\!\left(\frac{107}{256}\right)=\frac{697}{65536}>0.
\]
Thus $\delta^*$ lies in the dyadic interval $[53/128,107/256]\subset[0.4,0.42]$.  Bisect this interval using the sign of $g$ and exact rational arithmetic until its width is at most $2\eta$, and take the midpoint as $\delta$.  This requires $O(\bigl(\log(1/\varepsilon)\bigr)$ steps, keeps $\delta$ dyadic, and guarantees
\[
 |\delta-\delta^*|\le\eta.
\]
The numerator and denominator of $\delta$ consequently use $O(\bigl(\log(1/\varepsilon)\bigr)$ bits.  On $[0.4,0.42]$, for
\[
 f_1(x)=\frac1{1-2x},
 \qquad
 f_2(x)=\frac1{x^2},
\]
one has
\[
 |f_1'(x)|\le\frac{625}{8},
 \qquad
 |f_2'(x)|\le\frac{125}{4}.
\]
Since $f_1(\delta^*)=f_2(\delta^*)=3+2\sqrt2$, the mean-value theorem gives
\begin{align*}
 C(\delta)
 &\le C(\delta^*)+\frac{625}{8}|\delta-\delta^*|\\
 &\le3+2\sqrt2+\frac{625}{4096}\varepsilon\\
 &<3+2\sqrt2+\frac\varepsilon2.
\end{align*}
Hence both $\delta$ and $\theta=1-\delta$ are exact rationals of the required precision.

\begin{theorem}[Bit-model guarantee]
\label{AA:thm:bitmodel}
For every positive rational $\varepsilon$, rational threshold selection and polynomial-precision SDP optimization give a schedule of value at most
\[
 (3+2\sqrt2+\varepsilon)\OPT.
\]
After the precision normalization above, the running time is polynomial in the input length and $\log(1/\varepsilon)$.
\end{theorem}

\begin{proof}
Replace $\varepsilon$ by $\min\{\varepsilon,1\}$ if necessary; the stronger guarantee obtained for the smaller parameter also satisfies the original request.  The case $n=1$ is solved exactly by completing the sole job at $d$.  Assume $n\ge2$.  Use the dyadic bisection construction above to obtain a rational $\delta\in(0,1/2)$ satisfying
\[
 C(\delta)<3+2\sqrt2+\frac\varepsilon2,
\]
and write $\bar C=C(\delta)$.  For each anchor, apply \cref{AA:lem:rational-realization} with
\[
 \alpha=\frac{\varepsilon}{2\bar C}.
\]
For an anchor in an optimal early set, \cref{AA:prop:sdpvalid} and the realization lemma give an exactly feasible rational point of value at most $\OPT+\alpha$.  Since the rounding inequalities apply to every feasible SDP point,
\[
 \Phi(E_\theta)
 \le\bar C(\OPT+\alpha)
 =\bar C\OPT+\frac\varepsilon2.
\]
By \cref{AA:lem:floor}, this is at most
\[
 \left(\bar C+\frac\varepsilon2\right)\OPT
 \le(3+2\sqrt2+\varepsilon)\OPT.
\]
The bisection construction gives $\delta$ and $\theta$ with $O(\log(1/\varepsilon))$ bits, and $\alpha=\varepsilon/(2\bar C)$ has the same polynomial precision.  The realization lemma then uses only polynomially many additional bits and operations.
\end{proof}

\begin{remark}[Formal versus floating-point implementation]
The theorem is a Turing-model statement: a rational weak-optimization routine produces $\widetilde Y$, and the displayed convex mixture produces an exactly feasible rational Gram matrix before thresholding.  A practical floating-point solver may implement the same plan with certified residual bounds or rational reconstruction, but unverified solver tolerances are not used in the proof.
\end{remark}

\begin{sectionexit}
Section~19 has converted weak numerical optimization into an exactly feasible rational SDP point with polynomial bit length and controlled additive objective loss.  Section~20 uses that result only through Lemma~II.20: threshold comparisons are exact rational comparisons, and every returned schedule is evaluated from the original integer objective.
\end{sectionexit}

\section{Algorithm and implementation details}
\label{AA:sec:algorithm}

\begin{center}
\fbox{\begin{minipage}{0.92\textwidth}
\textbf{Algorithm: Anchored-SDP threshold approximation for \AWET}

\medskip
\textbf{Input:} Jobs $(p_i,w_i^E,w_i^L)$ and a positive rational accuracy parameter $\varepsilon$.

\begin{enumerate}[label=\arabic*.,leftmargin=2.2em]
\item If $n=1$, complete the unique job at $d$ and stop.
\item Replace $\varepsilon$ by $\min\{\varepsilon,1\}$, and compute $r_i,s_i$ and $a_{ij},b_{ij},c_i$.
\item Use the dyadic bisection construction of Section~19 to choose rational $\delta$ with $C(\delta)<3+2\sqrt2+\varepsilon/2$, write $\bar C=C(\delta)$, and set $\theta=1-\delta$.
\item For every anchor $h\in\J$:
  \begin{enumerate}[label=(\alph*),leftmargin=2em]
  \item use \cref{AA:lem:rational-realization} with
  $\alpha=\varepsilon/(2\bar C)$ to compute an exactly feasible rational
  point of $\SDP_h$ (\ref{AA:eq:sdpobjective}-\ref{AA:eq:anchor}) whose value is at most $\SDP_h+\alpha$;
  \item set $E_h=\{i:e_i\ge\theta\}$ and $T_h=\J\setminus E_h$;
  \item schedule $E_h$ by increasing $r_i$ ending at $d$ and $T_h$ by decreasing $s_i$ beginning at $d$;
  \item evaluate the exact cost from \eqref{AA:eq:canonical}.
  \end{enumerate}
\item Return the cheapest of the $n$ schedules.
\end{enumerate}
\end{minipage}}
\end{center}

\begin{theorem}[Algorithmic guarantee]
\label{AA:thm:algorithm}
The algorithm is feasible, runs in polynomial time, and satisfies \cref{AA:thm:main}.
\end{theorem}

\begin{proof}
Each rounded early set contains its anchor.  By \cref{AA:prop:canonical}, the two ratio sorts produce a feasible canonical schedule of the exact partition cost.  Let $E^*$ be optimal and choose $h\in E^*$.  This anchor is enumerated; its SDP optimum is at most $\OPT$ by \cref{AA:prop:sdpvalid}; and \cref{AA:thm:bitmodel} gives the desired ratio.  Returning the cheapest candidate cannot be worse.

Coefficient construction takes $O(n^2)$ arithmetic operations.  There are $n$ SDPs, each with a matrix of order $n+1$ and $O(n^2)$ linear constraints.  Standard weak SDP optimization, the two sorts, and exact candidate evaluation are polynomial.
\end{proof}

\begin{howto}
After the rational feasibility repair, only the first Gram-matrix row is needed: $e_i=(1+X_{0i})/2$.  Compare these exact rationals with $\theta$, sort the two sides, and evaluate the resulting schedule exactly.  Thus neither threshold membership nor the final candidate comparison depends on floating-point tie handling.
\end{howto}

\section{Worked example}
\label{AA:sec:example}

This example follows every line of the construction: instance data, the coefficients $a_{ij}$, $b_{ij}$, and $c_i$, a feasible anchored SDP moment point, all marginal and joint quantities, the three objective accounts $\mathcal A$, $\mathcal B$, and $\mathcal U$, threshold rounding, schedule recovery, and comparison with an optimal schedule.  The actual algorithm solves the anchored SDP.  For hand verification, this example instead chooses an explicit convex mixture of integral partitions, so every marginal and joint mass can be reproduced without a numerical solver.  The displayed point is pedagogical and is not asserted to be SDP-optimal.

\subsection*{Instance data and due date}

Take the following six jobs.
\begin{center}
\begingroup
\small
\setstretch{1.0}
\renewcommand{\arraystretch}{1.12}
\begin{tabular}{c r r r c c}
\toprule
job $i$ & $p_i$ & $w_i^E$ & $w_i^L$ & $r_i=w_i^E/p_i$ & $s_i=w_i^L/p_i$\\
\midrule
$1$ & $1$ & $1$  & $6$  & $1$ & $6$\\
$2$ & $2$ & $6$  & $4$  & $3$ & $2$\\
$3$ & $3$ & $6$  & $15$ & $2$ & $5$\\
$4$ & $4$ & $20$ & $4$  & $5$ & $1$\\
$5$ & $2$ & $8$  & $8$  & $4$ & $4$\\
$6$ & $5$ & $30$ & $15$ & $6$ & $3$\\
\bottomrule
\end{tabular}
\endgroup
\end{center}
The ratios in the last two columns use \eqref{AA:eq:ratios}.  For example,
$r_4=w_4^E/p_4=20/4=5$ and $s_4=w_4^L/p_4=4/4=1$.
The processing times are
\[
 (p_1,\ldots,p_6)=(1,2,3,4,2,5),
\]
so five different processing-time values occur.  Equation~\eqref{AA:eq:duedate} gives
\[
 d=P=1+2+3+4+2+5=17.
\]

\subsection*{The pair coefficients $a_{ij}$ and $b_{ij}$}

For every pair, the second column below substitutes the data into
\eqref{AA:eq:aij}, and the third column substitutes them into
\eqref{AA:eq:bij}.

\begingroup
\small
\setstretch{1.0}
\setlength{\tabcolsep}{5pt}
\renewcommand{\arraystretch}{1.08}
\begin{longtable}{c p{0.39\textwidth} p{0.39\textwidth}}
\toprule
pair & $a_{ij}=p_ip_j\min\{r_i,r_j\}$, from \eqref{AA:eq:aij}
     & $b_{ij}=p_ip_j\min\{s_i,s_j\}$, from \eqref{AA:eq:bij}\\
\midrule
\endfirsthead
\toprule
pair & $a_{ij}$ calculation & $b_{ij}$ calculation\\
\midrule
\endhead
$\{1,2\}$ & $1\cdot2\min\{1,3\}=2$ & $1\cdot2\min\{6,2\}=4$\\
$\{1,3\}$ & $1\cdot3\min\{1,2\}=3$ & $1\cdot3\min\{6,5\}=15$\\
$\{1,4\}$ & $1\cdot4\min\{1,5\}=4$ & $1\cdot4\min\{6,1\}=4$\\
$\{1,5\}$ & $1\cdot2\min\{1,4\}=2$ & $1\cdot2\min\{6,4\}=8$\\
$\{1,6\}$ & $1\cdot5\min\{1,6\}=5$ & $1\cdot5\min\{6,3\}=15$\\
$\{2,3\}$ & $2\cdot3\min\{3,2\}=12$ & $2\cdot3\min\{2,5\}=12$\\
$\{2,4\}$ & $2\cdot4\min\{3,5\}=24$ & $2\cdot4\min\{2,1\}=8$\\
$\{2,5\}$ & $2\cdot2\min\{3,4\}=12$ & $2\cdot2\min\{2,4\}=8$\\
$\{2,6\}$ & $2\cdot5\min\{3,6\}=30$ & $2\cdot5\min\{2,3\}=20$\\
$\{3,4\}$ & $3\cdot4\min\{2,5\}=24$ & $3\cdot4\min\{5,1\}=12$\\
$\{3,5\}$ & $3\cdot2\min\{2,4\}=12$ & $3\cdot2\min\{5,4\}=24$\\
$\{3,6\}$ & $3\cdot5\min\{2,6\}=30$ & $3\cdot5\min\{5,3\}=45$\\
$\{4,5\}$ & $4\cdot2\min\{5,4\}=32$ & $4\cdot2\min\{1,4\}=8$\\
$\{4,6\}$ & $4\cdot5\min\{5,6\}=100$ & $4\cdot5\min\{1,3\}=20$\\
$\{5,6\}$ & $2\cdot5\min\{4,6\}=40$ & $2\cdot5\min\{4,3\}=30$\\
\bottomrule
\end{longtable}
\endgroup

Equation~\eqref{AA:eq:ci} supplies one tardy self coefficient per job:
\begin{center}
\begingroup
\small
\setstretch{1.0}
\renewcommand{\arraystretch}{1.12}
\begin{tabular}{c c c}
\toprule
job & substitution into $c_i=p_i^2s_i=p_iw_i^L$ & $c_i$\\
\midrule
$1$ & $1^2\cdot6=1\cdot6$ & $6$\\
$2$ & $2^2\cdot2=2\cdot4$ & $8$\\
$3$ & $3^2\cdot5=3\cdot15$ & $45$\\
$4$ & $4^2\cdot1=4\cdot4$ & $16$\\
$5$ & $2^2\cdot4=2\cdot8$ & $16$\\
$6$ & $5^2\cdot3=5\cdot15$ & $75$\\
\bottomrule
\end{tabular}
\endgroup
\end{center}
Thus
\[
 c=(6,8,45,16,16,75).
\]

\subsection*{A feasible anchored SDP moment point}

Force job $1$ early.  Consider the following distribution over five integral early sets.  The numbers in the probability column are chosen convex-combination weights, not quantities calculated from the scheduling data.  Equivalently, for interpretation, one may imagine selecting one of the five integral partitions according to these probabilities.  The final column is obtained from the canonical partition formula \eqref{AA:eq:canonical}.
\begin{center}
\begingroup
\small
\setstretch{1.0}
\setlength{\tabcolsep}{4pt}
\renewcommand{\arraystretch}{1.12}
\begin{tabularx}{\textwidth}{c c >{\centering\arraybackslash}X >{\centering\arraybackslash}X c}
\toprule
scenario & probability & early set $E^{(q)}$ & tardy set $T^{(q)}$ & $\Phi(E^{(q)})$\\
\midrule
$S_1$ & $3/10$ & $\{1,2,3,5\}$ & $\{4,6\}$ & $154$\\
$S_2$ & $3/10$ & $\{1,2,5,6\}$ & $\{3,4\}$ & $164$\\
$S_3$ & $1/5$  & $\{1,2,6\}$   & $\{3,4,5\}$ & $158$\\
$S_4$ & $1/10$ & $\{1,3,5\}$   & $\{2,4,6\}$ & $164$\\
$S_5$ & $1/10$ & $\{1,3,4,5\}$ & $\{2,6\}$ & $180$\\
\bottomrule
\end{tabularx}
\endgroup
\end{center}
The probabilities sum to one, and job $1$ belongs to every early set.  Each row therefore gives an integral Gram point with $v_i=y_iv_0$ and anchor $m_1=1$.  The probability-weighted average of those Gram matrices is positive semidefinite and satisfies all local constraints.  Hence the distribution defines a feasible point of the SDP anchored at job $1$, as in \cref{AA:prop:sdpvalid}.

For a distribution over integral points, equations
\eqref{AA:eq:marginals}--\eqref{AA:eq:tij} have their ordinary probability meanings:
\[
 e_i=\Pr(i\in E),\qquad t_i=\Pr(i\in T),
\qquad
 e_{ij}=\Pr(i,j\in E),\qquad t_{ij}=\Pr(i,j\in T).
\]

\subsection*{The marginal quantities $e_i$ and $t_i$}

The early marginals are obtained by adding the probabilities of the scenarios in which the job is early.  The tardy marginal is then $t_i=1-e_i$, exactly as in \eqref{AA:eq:marginals}.
\begin{center}
\begingroup
\small
\setstretch{1.0}
\setlength{\tabcolsep}{5pt}
\renewcommand{\arraystretch}{1.13}
\begin{tabularx}{\textwidth}{c >{\raggedright\arraybackslash}X c c}
\toprule
job & calculation of $e_i=\Pr(i\in E)$ & $e_i$ & $t_i=1-e_i$\\
\midrule
$1$ & $3/10+3/10+1/5+1/10+1/10$ & $1$ & $0$\\
$2$ & $3/10+3/10+1/5$ & $4/5$ & $1/5$\\
$3$ & $3/10+1/10+1/10$ & $1/2$ & $1/2$\\
$4$ & $1/10$ & $1/10$ & $9/10$\\
$5$ & $3/10+3/10+1/10+1/10$ & $4/5$ & $1/5$\\
$6$ & $3/10+1/5$ & $1/2$ & $1/2$\\
\bottomrule
\end{tabularx}
\endgroup
\end{center}
Therefore
\[
 e=\left(1,\frac45,\frac12,\frac1{10},\frac45,\frac12\right),
 \qquad
 t=\left(0,\frac15,\frac12,\frac9{10},\frac15,\frac12\right).
\]

\subsection*{The joint quantities $e_{ij}$ and $t_{ij}$}

The next table performs the same probability addition for every pair.  By
\eqref{AA:eq:eij}, the middle column is $e_{ij}$; by
\eqref{AA:eq:tij}, the last column is $t_{ij}$.  For example, jobs $2$ and $6$ are both early in $S_2$ and $S_3$, so
$e_{26}=3/10+1/5=1/2$; they are both tardy in $S_4$ and $S_5$, so
$t_{26}=1/10+1/10=1/5$.

\begingroup
\small
\setstretch{1.0}
\setlength{\tabcolsep}{5pt}
\renewcommand{\arraystretch}{1.08}
\begin{longtable}{c p{0.39\textwidth} p{0.39\textwidth}}
\toprule
pair & $e_{ij}=\Pr(i,j\in E)$, from \eqref{AA:eq:eij}
     & $t_{ij}=\Pr(i,j\in T)$, from \eqref{AA:eq:tij}\\
\midrule
\endfirsthead
\toprule
pair & $e_{ij}$ calculation & $t_{ij}$ calculation\\
\midrule
\endhead
$\{1,2\}$ & $3/10+3/10+1/5=4/5$ & $0$\\
$\{1,3\}$ & $3/10+1/10+1/10=1/2$ & $0$\\
$\{1,4\}$ & $1/10$ & $0$\\
$\{1,5\}$ & $3/10+3/10+1/10+1/10=4/5$ & $0$\\
$\{1,6\}$ & $3/10+1/5=1/2$ & $0$\\
$\{2,3\}$ & $3/10$ & $0$\\
$\{2,4\}$ & $0$ & $1/10$\\
$\{2,5\}$ & $3/10+3/10=3/5$ & $0$\\
$\{2,6\}$ & $3/10+1/5=1/2$ & $1/10+1/10=1/5$\\
$\{3,4\}$ & $1/10$ & $3/10+1/5=1/2$\\
$\{3,5\}$ & $3/10+1/10+1/10=1/2$ & $1/5$\\
$\{3,6\}$ & $0$ & $0$\\
$\{4,5\}$ & $1/10$ & $1/5$\\
$\{4,6\}$ & $0$ & $3/10+1/10=2/5$\\
$\{5,6\}$ & $3/10$ & $0$\\
\bottomrule
\end{longtable}
\endgroup

\subsection*{The SDP accounts $\mathcal A$, $\mathcal B$, $\mathcal U$, and $L$}

Equation~\eqref{AA:eq:ABC} defines the early-pair account.  Substituting every nonzero $e_{ij}$ from the preceding table gives
\[
\begin{aligned}
\mathcal A
&=2\left(\frac45\right)
 +3\left(\frac12\right)
 +4\left(\frac1{10}\right)
 +2\left(\frac45\right)
 +5\left(\frac12\right)\\
&\quad
 +12\left(\frac3{10}\right)
 +12\left(\frac35\right)
 +30\left(\frac12\right)\\
&\quad
 +24\left(\frac1{10}\right)
 +12\left(\frac12\right)
 +32\left(\frac1{10}\right)
 +40\left(\frac3{10}\right)\\
&=57.
\end{aligned}
\]
The omitted pair terms have $e_{ij}=0$.

The tardy-pair account in the same equation uses every nonzero $t_{ij}$:
\[
\begin{aligned}
\mathcal B
&=8\left(\frac1{10}\right)
 +20\left(\frac15\right)
 +12\left(\frac12\right)
 +24\left(\frac15\right)\\
&\quad
 +8\left(\frac15\right)
 +20\left(\frac25\right)\\
&=\frac{126}{5}=25.2.
\end{aligned}
\]
These six terms correspond respectively to pairs
$\{2,4\}$, $\{2,6\}$, $\{3,4\}$, $\{3,5\}$,
$\{4,5\}$, and $\{4,6\}$.

Finally, the unary account in \eqref{AA:eq:ABC} is
\[
\begin{aligned}
\mathcal U
&=6(0)
 +8\left(\frac15\right)
 +45\left(\frac12\right)
 +16\left(\frac9{10}\right)
 +16\left(\frac15\right)
 +75\left(\frac12\right)\\
&=\frac{396}{5}=79.2.
\end{aligned}
\]
For this feasible point, write $L=L_1$.  Equation~\eqref{AA:eq:sdpobjective}, or equivalently the account split immediately following \eqref{AA:eq:ABC}, gives
\[
 L=\mathcal A+\mathcal B+\mathcal U
  =57+\frac{126}{5}+\frac{396}{5}
  =\frac{807}{5}=161.4.
\]
There is also a useful independent check.  Because this point is a probability mixture of integral partitions, its SDP objective is their expected canonical cost:
\[
\begin{aligned}
L
&=\frac3{10}(154)+\frac3{10}(164)+\frac15(158)
  +\frac1{10}(164)+\frac1{10}(180)\\
&=\frac{807}{5}.
\end{aligned}
\]

\begin{caution}{Pedagogical point versus algorithmic output}
The threshold rule below is the algorithmic rule, but it is being applied to the auditable convex-mixture point just constructed, not to an SDP optimum returned by a solver.  The comparison with the true scheduling optimum therefore illustrates the rounding mechanics rather than the worst-case certificate of the implemented algorithm.
\end{caution}

\subsection*{Threshold rounding}

Equation~\eqref{AA:eq:optimal-threshold} gives
\[
 \theta^*=2-\sqrt2\approx0.585786.
\]
Applying the rule \eqref{AA:eq:rounding} to the six early marginals gives
\[
\begin{array}{c|cccccc}
 i&1&2&3&4&5&6\\ \hline
 e_i&1&0.8&0.5&0.1&0.8&0.5\\
 e_i\ge\theta^*?&\text{yes}&\text{yes}&\text{no}&\text{no}&\text{yes}&\text{no}
\end{array}
\]
Hence
\[
 E_{\theta^*}=\{1,2,5\},
 \qquad
 T_{\theta^*}=\{3,4,6\}.
\]

\subsection*{Exact cost of the rounded canonical schedule}

First evaluate the partition directly from \eqref{AA:eq:canonical}.  The early-pair part is
\[
 a_{12}+a_{15}+a_{25}=2+2+12=16.
\]
The tardy-pair part is
\[
 b_{34}+b_{36}+b_{46}=12+45+20=77.
\]
The tardy self part is
\[
 c_3+c_4+c_6=45+16+75=136.
\]
Therefore the exact canonical cost is
\[
 \Phi(E_{\theta^*})=16+77+136=229.
\]

Now recover the actual schedule, as prescribed after
\eqref{AA:eq:opt-set}.  The early jobs are sorted by increasing ratio,
\[
 1\ (r_1=1),\qquad 2\ (r_2=3),\qquad 5\ (r_5=4),
\]
and their block starts at
\[
 d-(p_1+p_2+p_5)=17-(1+2+2)=12.
\]
The tardy jobs are sorted by decreasing ratio,
\[
 3\ (s_3=5),\qquad 6\ (s_6=3),\qquad 4\ (s_4=1),
\]
and their block starts at $d=17$.  The next table evaluates every job from the original objective \eqref{AA:eq:objective}.

\begin{center}
\begingroup
\small
\setstretch{1.0}
\setlength{\tabcolsep}{3pt}
\renewcommand{\arraystretch}{1.12}
\begin{tabular}{c c c c c c c}
\toprule
position & job & interval & $C_i$ & deviation & applicable weight & cost\\
\midrule
$1$ & $1$ & $[12,13]$ & $13$ & $E_1=17-13=4$ & $w_1^E=1$ & $1(4)=4$\\
$2$ & $2$ & $[13,15]$ & $15$ & $E_2=17-15=2$ & $w_2^E=6$ & $6(2)=12$\\
$3$ & $5$ & $[15,17]$ & $17$ & $E_5=0$ & $w_5^E=8$ & $8(0)=0$\\
$4$ & $3$ & $[17,20]$ & $20$ & $T_3=20-17=3$ & $w_3^L=15$ & $15(3)=45$\\
$5$ & $6$ & $[20,25]$ & $25$ & $T_6=25-17=8$ & $w_6^L=15$ & $15(8)=120$\\
$6$ & $4$ & $[25,29]$ & $29$ & $T_4=29-17=12$ & $w_4^L=4$ & $4(12)=48$\\
\midrule
\multicolumn{6}{r}{total} & $4+12+0+45+120+48=229$\\
\bottomrule
\end{tabular}
\endgroup
\end{center}
Thus the job-by-job cost from \eqref{AA:eq:objective} agrees exactly with the pair-and-self calculation from \eqref{AA:eq:canonical}.

\subsection*{The optimal solution and its cost}

Equation~\eqref{AA:eq:opt-set} permits exact enumeration because there are only $2^6-1=63$ nonempty early sets.  Evaluating \eqref{AA:eq:canonical} for every set and grouping the results by $|E|$ gives the following minimum in each cardinality class.
\begin{center}
\begingroup
\small
\setstretch{1.0}
\renewcommand{\arraystretch}{1.12}
\begin{tabular}{c c c}
\toprule
$|E|$ & minimizing early set within this class & minimum $\Phi(E)$\\
\midrule
$1$ & $\{6\}$ & $194$\\
$2$ & $\{3,6\}$ & $116$\\
$3$ & $\{1,3,6\}$ & $102$\\
$4$ & $\{1,2,3,6\}$ & $122$\\
$5$ & $\{1,2,3,5,6\}$ & $164$\\
$6$ & $\{1,2,3,4,5,6\}$ & $332$\\
\bottomrule
\end{tabular}
\endgroup
\end{center}
Since these six classes partition all nonempty early sets, the unique optimum is
\[
 E^*=\{1,3,6\},
 \qquad
 T^*=\{2,4,5\},
 \qquad
 \OPT=102.
\]
Its cost from \eqref{AA:eq:canonical} is
\[
\begin{aligned}
\Phi(E^*)
&=\underbrace{(a_{13}+a_{16}+a_{36})}_{\text{early pairs}}
 +\underbrace{(b_{24}+b_{25}+b_{45})}_{\text{tardy pairs}}
 +\underbrace{(c_2+c_4+c_5)}_{\text{tardy self terms}}\\
&=(3+5+30)+(8+8+8)+(8+16+16)\\
&=38+24+40=102.
\end{aligned}
\]

For completeness, the optimal canonical schedule can also be checked job by job.  The early order is $1,3,6$ because $1<2<6$ are their $r$-values, and its start time is
\[
 17-(p_1+p_3+p_6)=17-(1+3+5)=8.
\]
The tardy order is $5,2,4$ because $4>2>1$ are their $s$-values.
\begin{center}
\begingroup
\small
\setstretch{1.0}
\setlength{\tabcolsep}{3pt}
\renewcommand{\arraystretch}{1.12}
\begin{tabular}{c c c c c c c}
\toprule
position & job & interval & $C_i$ & deviation & applicable weight & cost\\
\midrule
$1$ & $1$ & $[8,9]$ & $9$ & $E_1=17-9=8$ & $w_1^E=1$ & $1(8)=8$\\
$2$ & $3$ & $[9,12]$ & $12$ & $E_3=17-12=5$ & $w_3^E=6$ & $6(5)=30$\\
$3$ & $6$ & $[12,17]$ & $17$ & $E_6=0$ & $w_6^E=30$ & $30(0)=0$\\
$4$ & $5$ & $[17,19]$ & $19$ & $T_5=19-17=2$ & $w_5^L=8$ & $8(2)=16$\\
$5$ & $2$ & $[19,21]$ & $21$ & $T_2=21-17=4$ & $w_2^L=4$ & $4(4)=16$\\
$6$ & $4$ & $[21,25]$ & $25$ & $T_4=25-17=8$ & $w_4^L=4$ & $4(8)=32$\\
\midrule
\multicolumn{6}{r}{total} & $8+30+0+16+16+32=102$\\
\bottomrule
\end{tabular}
\endgroup
\end{center}
The rounded schedule therefore has additive gap $229-102=127$ and comparison ratio
\[
 \frac{\Phi(E_{\theta^*})}{\OPT}
 =\frac{229}{102}
 \approx2.2451 < 3 + 2 \sqrt{2}.
\]

\begin{howtoread}
The example makes the global tardy-side argument visible.  Jobs $3$ and $6$ are both rounded tardy even though their pseudo-joint quantity is $t_{36}=0$, while their actual canonical pair cost is $b_{36}=45$.  That cost cannot be paid by an entrywise comparison with $t_{36}$.  It is covered only by the global PSD inequality, which combines the pair account $\mathcal B$ with the diagonal self account $\mathcal U$.
\end{howtoread}

\section{Why the proof works}
\label{AA:sec:interpretation}

\subsection{Early pairs use local consistency}
The early estimate needs no PSD property of the early kernel.  It uses only $e_{ij}=e_i+e_j-1+t_{ij}$ and $t_{ij}\ge0$.  The threshold replaces the missing early diagonal.

\subsection{Tardy pairs use the kernel diagonal}
The tardy self-costs are exactly the diagonal of $K^{\mathrm{tar}}$.  They complete the off-diagonal tardy pair matrix into a minimum-kernel Gram matrix.  This structure is stronger than a generic weighted binary constraint system.

\subsection{The rounding is deterministic}
No random hyperplane is needed.  A common marginal threshold avoids the linear-versus-quadratic mismatch that can make root-relative hyperplane rounding poor on unary tardiness terms.

\subsection{The anchor is a feasibility device}
The proof uses the anchor only to ensure $E_\theta\ne\varnothing$.  Enumerating all jobs guarantees one anchor in the optimal early set.

\begin{proofcheckpoint}
\[
\boxed{
\begin{minipage}{0.86\linewidth}
\centering
local joint consistency pays for early pairs\\[2pt]
$+$\\[-2pt]
minimum-kernel PSD with its diagonal pays for tardy pairs
\end{minipage}}
\]
Neither ingredient alone is sufficient for unrestricted asymmetric \AWET.
\end{proofcheckpoint}

\section{Concluding remarks}

The canonical side-assignment formulation of \AWET has two unrelated minimum kernels and an asymmetric diagonal term.  The correct approximation proof uses that asymmetry rather than trying to remove it.  A threshold above one half pays for early pairs through local pseudo-probability consistency.  The tardy self-cost completes the tardy minimum kernel into a positive-semidefinite matrix, and the SDP covariance matrix pays globally for all rounded tardy pairs.  Balancing the two threshold losses gives $3+2\sqrt2$.

The resulting algorithm is direct: enumerate an early anchor, solve a polynomial-size SDP, threshold the early marginals, and build the canonical V-shaped schedule.  This establishes a polynomial-time constant-factor approximation for unrestricted \AWET.

\begin{reusableidea}
When pair costs form a positive-semidefinite kernel only after diagonal completion, a unary or self-cost term can be an asset rather than a nuisance.  Combine that diagonal with a covariance domination inequality to pay all rounded pairs globally, and use a separate local-consistency argument on any side lacking a natural diagonal.
\end{reusableidea}

\begin{finding}{Why Part III is next}
Part II has established a universal approximation guarantee without restricting either ratio order.  It does not classify exact solvability on highly structured order patterns.  Part III moves to the maximally opposed strict order, where a completed load square gives weak hardness and a small two-resource state space gives pseudopolynomial solvability.
\end{finding}

\clearpage

\part{Antithetical Ordering}
\thispagestyle{plain}
\label{part:antithetical}
\setrunningunit{Part III: Antithetical}
\begin{center}
{\Large\bfseries Weak NP-Completeness at Exact Reversal\par}
\end{center}
\medskip
\begin{quote}\small
\noindent\textbf{Part synopsis.}
This Part classifies the strict antithetical ratio-order class.  A positional tagging lemma converts \PARTITION into an exact half-load question.  Affine opposite ratio systems and one anchor then produce a dominant completed square in the selected processing load, while an explicitly bounded cubic residual cannot erase a unit load gap.  The same reversed order yields a direct two-resource pseudopolynomial recurrence.  Exact reversal is also a separable permutation, so the weak hardness and pseudopolynomial algorithm proved here are compatible with the separating-tree FPTAS in Part~IV rather than in tension with it.
\end{quote}

\begin{partposition}
\textbf{Starting boundary.}  The unrestricted problem is strongly NP-complete, but that construction does not classify the exact-reversal subclass.  \textbf{Result.}  Strict antithetical \AWET is weakly NP-complete and has an exact pseudopolynomial algorithm.  \textbf{Technique.}  Binary positional tags create an exact half-load target; an anchor converts the dominant scheduling contribution into a completed load square, and a two-resource recurrence exploits the reversed order.  
\end{partposition}

\setshortsectionhead{Antithetical theorem}
\section{Introduction to the antithetical theorem}
\par\smallskip
The antithetical class orders the two Smith-ratio systems in opposite directions.  Unlike the agreeable class, it does not contain the symmetric Hall--Posner hard core, so a separate reduction is required.  The construction below starts from \PARTITION, encodes the source choice as an exact half-load condition, and then makes the dominant scheduling cost a convex square in that selected load.  A single anchor supplies the required linear field and is itself forced to the early side.

Exact reversal is simultaneously hard and highly structured.  Its ratio permutation is the reverse permutation, which is separable after the fixed strict refinements used here.  Consequently, the weak hardness and pseudopolynomial recurrence in this Part and the separating-tree FPTAS in Part~IV describe complementary exact and approximation properties of the same order extreme.

The canonical side-assignment formula is the only scheduling structural fact needed, and it is recorded first for completeness.

An \AWET\ instance consists of jobs $j\in\J=\{1,\ldots,n\}$ with positive integer processing time
$p_j$, earliness weight $w_j^E$, and tardiness weight $w_j^L$.  All jobs have the common due date
\[
 d=P:=\sum_{j\in\J}p_j.
\]
For completion time $C_j$, the objective is
\[
 F=\sum_{j\in\J}\left(w_j^E(d-C_j)^+ + w_j^L(C_j-d)^+\right).
\]
Put
\[
 r_j=\frac{w_j^E}{p_j},
 \qquad
 s_j=\frac{w_j^L}{p_j}.
\]

\begin{definition}[Strict antithetical ratio order]
An instance is \emph{strict antithetical} if its jobs can be indexed so that
\[
 r_1<r_2<\cdots<r_n
 \qquad\text{and}\qquad
 s_1>s_2>\cdots>s_n.
\]
Thus the two strict Smith-ratio orders are exact reversals.
\end{definition}

The standard canonical-schedule theorem converts the problem into a binary side assignment.  The
following form is the only scheduling fact used in the reductions.

\begin{proposition}[Canonical early-set objective]
\label{ANTI:prop:canonical}
For every nonempty $E\subseteq\J$, put $T=\J\setminus E$.  The compact canonical schedule that orders $E$ chronologically by nondecreasing $r_i$, ends its early block at $d$, orders $T$ chronologically by nonincreasing $s_i$, and begins its tardy block at $d$ has cost
\begin{equation}
\label{ANTI:eq:canonical}
 \Phi(E)=
 \sum_{\{i,j\}\subseteq E}p_ip_j\min\{r_i,r_j\}
 +\sum_{\{i,j\}\subseteq T}p_ip_j\min\{s_i,s_j\}
 +\sum_{i\in T}p_i^2s_i.
\end{equation}
Every nonempty $E$ gives a feasible canonical schedule of this form, and
\[
 \OPT=\min_{\varnothing\ne E\subseteq\J}\Phi(E).
\]
\end{proposition}
\begin{proof}
Fix a nonempty early set $E$ and let $T=\J\setminus E$.  Its canonical early block starts at
\[
 d-\sum_{i\in E}p_i=\sum_{i\in T}p_i\ge0,
\]
ends at $d$, and is followed immediately by the canonical tardy block.  Thus the displayed construction is feasible.

Read outward from the due date, the adjacent-exchange argument orders the early jobs by nonincreasing $r_j=w_j^E/p_j$; equivalently, their chronological order before $d$ is nondecreasing in $r_j$.  The tardy jobs are read both chronologically and outward from $d$, and are ordered by nonincreasing $s_j=w_j^L/p_j$.

To verify the pair coefficients, let early job $i$ lie farther from $d$ than early job $j$.  The processing of $j$ contributes $w_i^Ep_j=p_ip_jr_i$ to the earliness of $i$, and the exchange order places the smaller of $r_i,r_j$ farther from $d$.  Hence the pair contributes $p_ip_j\min\{r_i,r_j\}$.  Similarly, if tardy job $i$ is closer to $d$ than tardy job $j$, the processing of $i$ contributes $p_iw_j^L=p_ip_js_j$ to the tardiness of $j$, and the exchange order places the larger tardiness ratio closer to $d$.  The tardy pair accordingly contributes $p_ip_j\min\{s_i,s_j\}$.  Each tardy job also contributes the self term $p_i^2s_i$.  Summing these terms proves \eqref{ANTI:eq:canonical}.

Because every nonempty $E$ supplies a feasible canonical schedule, $\OPT\le\min_{\varnothing\ne E\subseteq\J}\Phi(E)$.  Conversely, by \cref{AA:lem:compact}, a global optimum can be chosen compact, with a nonempty early/on-time arm and the two canonical ratio orders.  If $E^*$ is its early set, the preceding calculation gives its cost as $\Phi(E^*)$.  Therefore $\OPT\ge\min_{\varnothing\ne E\subseteq\J}\Phi(E)$, proving equality.
\end{proof}

All quantities in \eqref{ANTI:eq:canonical} are integers when the weights and processing times are
integers, since
\[
 p_ip_j\min\{r_i,r_j\}
 =\min\{w_i^Ep_j,p_iw_j^E\},
\]
and similarly on the tardy side.  Hence whenever one canonical objective value is strictly larger
than another, it exceeds it by at least one.  This unit integrality fact is used in the separation below.

\begin{proofcheckpoint}
From this point onward, optimization and the reduction threshold can be analyzed through canonical binary side assignments: every assignment gives a feasible canonical schedule, and some global optimum is canonical.  The reduction therefore needs only to control the algebra of the canonical objective.
\end{proofcheckpoint}

\begin{notationreset}
The source instance consists of positive integers $b_1,\ldots,b_m$ with total $2B$.  Binary positional tags produce strictly increasing values $a_1,\ldots,a_N$ with total
$Q=\sum_{i=1}^N a_i$.  For an early-side vector $x\in\{0,1\}^N$, the selected tagged load is
\[
 S(x)=\sum_{i=1}^N a_ix_i.
\]
The later construction sets the residual budget $U=3Q^3$, the dominant scale $M=8Q^3$, and the common-due-date scale $D=M+Q$.  The symbols $C$ and $K$ denote, respectively, the all-variable-tardy baseline with the anchor early and the decision threshold.  These quantities are previewed here so that the proof architecture can be read before their formal construction.
\end{notationreset}

\setshortsectionhead{Hardness architecture}
\section{Antithetical hardness: proof architecture}
\label{ANTI:sec:architecture}
\begin{proofcontract}
\begin{itemize}[leftmargin=2em,itemsep=1pt,topsep=2pt]
\item \textbf{Input:} a binary \PARTITION instance with total sum $2B$.
\item \textbf{Output:} a positive-integer \AWET\ instance with strict reversed ratio orders and threshold $K$.
\item \textbf{Forward invariant:} a subset of sum $B$ becomes a tagged load $Q/2$ and eliminates the dominant square term.
\item \textbf{Reverse invariant:} if any schedule has cost at most $K$, a global optimum can be chosen canonical; that canonical optimum keeps the anchor early and has tagged load exactly $Q/2$.
\item \textbf{Size obligation:} the tag, scale, weights, and threshold have polynomial binary encoding length.
\end{itemize}
\end{proofcontract}
\setshortsectionhead{Reduction roadmap}
\section{Reduction roadmap}
\label{ANTI:sec:antireduction}

The antithetical case does not contain the symmetric hard core, so a separate reduction is needed.
The construction has four layers.

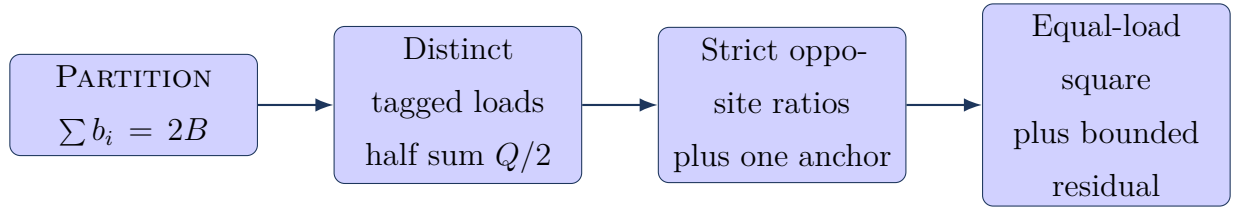
\begin{center}
\begin{tikzpicture}[
  node distance=1.05cm and 1.0cm,
  box/.style={draw=navy,rounded corners,fill=figuremixed,align=center,minimum height=1.05cm,text width=3.0cm},
  arr/.style={-{Latex[length=2.4mm]},thick,draw=navy}
]
\node[box] (p) {\PARTITION\\$\sum b_i=2B$};
\node[box,right=of p] (tag) {Distinct tagged loads\\half sum $Q/2$};
\node[box,right=of tag] (jobs) {Strict opposite ratios\\plus one anchor};
\node[box,right=of jobs] (gap) {Equal-load square\\plus bounded residual};
\draw[arr] (p) -- (tag);
\draw[arr] (tag) -- (jobs);
\draw[arr] (jobs) -- (gap);
\end{tikzpicture}
\end{center}

\begin{enumerate}[label=\textbf{Step \arabic*.},leftmargin=5em]
\item Binary positional tags force every half-sum subset to select exactly one number from each source
      pair, while the low-order part carries the original partition equation.
\item Each tagged number becomes a processing time.  Early ratios increase affinely with processing
      time and tardiness ratios decrease affinely, producing a strict antithetical chain.
\item A single final anchor supplies the linear field needed to make the dominant assignment cost
      $DM(S^2-QS)$.
\item The remaining cubic expression has absolute value at most $3Q^3$.  Taking $M=8Q^3$ leaves a
      strict gap between $S=Q/2$ and every other integer load.
\end{enumerate}

\begin{finding}{Scale ledger}
The four scales should be remembered by role rather than by magnitude alone.
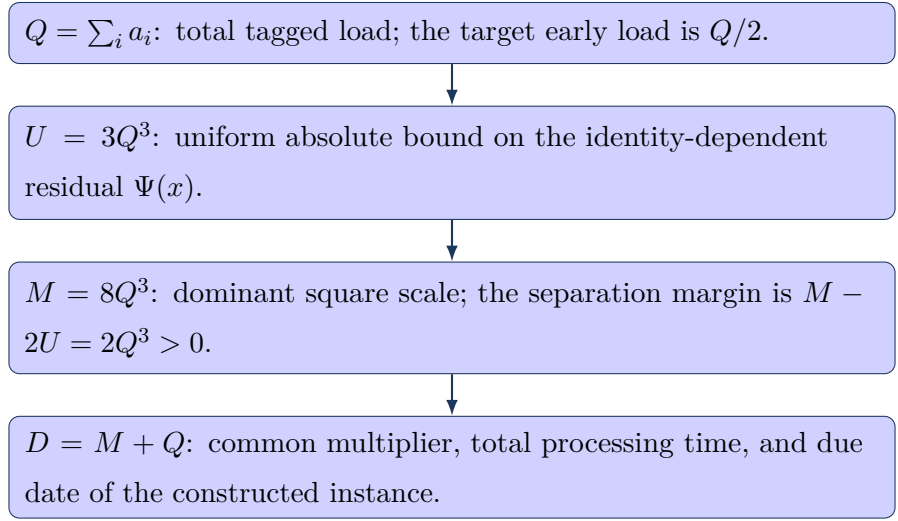
\begin{center}
\begin{tikzpicture}[
  node distance=0.55cm,
  scaleBox/.style={draw=navy,rounded corners,fill=figuremixed,align=left,text width=11.3cm,inner sep=6pt},
  arr/.style={-{Latex[length=2.1mm]},thick,draw=navy}
]
\node[scaleBox] (q) {$Q=\sum_i a_i$: total tagged load; the target early load is $Q/2$.};
\node[scaleBox,below=of q] (u) {$U=3Q^3$: uniform absolute bound on the identity-dependent residual $\Psi(x)$.};
\node[scaleBox,below=of u] (m) {$M=8Q^3$: dominant square scale; the separation margin is $M-2U=2Q^3>0$.};
\node[scaleBox,below=of m] (d) {$D=M+Q$: common multiplier, total processing time, and due date of the constructed instance.};
\draw[arr] (q) -- (u);
\draw[arr] (u) -- (m);
\draw[arr] (m) -- (d);
\end{tikzpicture}
\end{center}
The proof later uses the ledger only through $|\Psi|\le U$, the unit load gap, and $M>2U$.
\end{finding}

The reduction is developed below from source equivalence through the cost identity, anchor proof, gap separation, and encoding-length audit.

\setshortsectionhead{Half-sum compiler}
\section{A distinct half-sum compiler}
\label{ANTI:sec:tagging}

Begin with the standard \PARTITION problem on positive integers $b_1,\ldots,b_m$~\cite{GareyJohnson1979}.  If the total sum is odd, output the following fixed strict-antithetical no-instance of \AWET: two jobs
\[
 (p_1,w_1^E,w_1^L)=(1,1,2),
 \qquad
 (p_2,w_2^E,w_2^L)=(2,4,2),
\]
with due date $d=3$ and decision bound $K=1$.  Its ratios satisfy $r_1=1<2=r_2$ and $s_1=2>1=s_2$.  The three nonempty early sets have canonical costs $4$, $2$, and $2$, respectively, so this fixed target instance is a no-instance.

It therefore remains only to describe an even-total source instance.  Write
\[
 \sum_{i=1}^m b_i=2B,
\]
and ask whether some subcollection sums to $B$.  Put
\[
 \kappa=2B+1,
 \qquad
 L_i=\kappa 2^{i-1},
\]
and create the two numbers
\begin{equation}
\label{ANTI:eq:tags}
 a_{2i-1}=L_i,
 \qquad
 a_{2i}=L_i+b_i
 \qquad(i=1,\ldots,m).
\end{equation}
Let $N=2m$ and $Q=\sum_{j=1}^Na_j$.

\begin{arithmeticguide}
For any candidate subset of the tagged numbers, let $m_i\in\{0,1,2\}$ denote the number selected from the pair $\{a_{2i-1},a_{2i}\}$.  The positional argument uses the strict superincreasing inequality
\[
 2^{k-1}>\sum_{i=1}^{k-1}2^{i-1}.
\]
It does \emph{not} appeal to uniqueness of an arbitrary base-two representation with digits in
$\{0,1,2\}$.  What matters is more specific: the target coefficient at every position is $1$, while
each deviation $m_i-1$ lies in $\{-1,0,1\}$.  Consequently, the largest nonzero deviation cannot be
cancelled by all lower positions together.
\end{arithmeticguide}

\begin{lemma}[Binary pair tagging]
\label{ANTI:lem:tagging}
The numbers in \eqref{ANTI:eq:tags} are strictly increasing, $Q$ is even, and the following are equivalent:
\begin{enumerate}[label=(\roman*)]
\item some subset of $\{b_1,\ldots,b_m\}$ sums to $B$;
\item some subset of $\{a_1,\ldots,a_N\}$ sums to $Q/2$.
\end{enumerate}
Moreover, every half-sum subset chooses exactly one member of each pair
$\{a_{2i-1},a_{2i}\}$.
\end{lemma}

\begin{proof}
The total tagged sum is
\begin{align*}
 Q
 &=\sum_{i=1}^m\bigl(L_i+(L_i+b_i)\bigr)\\
 &=2\sum_{i=1}^mL_i+\sum_{i=1}^m b_i\\
 &=2\sum_{i=1}^mL_i+2B.
\end{align*}
Hence $Q$ is even and
\begin{equation}
\label{ANTI:eq:half-total}
 \frac Q2=\sum_{i=1}^mL_i+B.
\end{equation}

Consider an arbitrary subset of the $2m$ tagged numbers.  For pair $i$, let
$m_i\in\{0,1,2\}$ be the number of selected tags and let $z_i\in\{0,1\}$ indicate whether the second tag $L_i+b_i$ is selected.  Its sum is therefore
\begin{align}
 \sum_{i=1}^m m_iL_i+\sum_{i=1}^m z_ib_i
 &=\kappa\sum_{i=1}^m m_i2^{i-1}+\sum_{i=1}^m z_ib_i.
 \label{ANTI:eq:tagged-subset-sum}
\end{align}
Suppose this sum equals $Q/2$.  Substituting \eqref{ANTI:eq:half-total} and $L_i=\kappa2^{i-1}$ gives
\begin{equation}
 \kappa\sum_{i=1}^m m_i2^{i-1}+\sum_{i=1}^m z_ib_i
 =\kappa\sum_{i=1}^m2^{i-1}+B.
 \label{ANTI:eq:tagged-equality}
\end{equation}
Both $B$ and $\sum_i z_ib_i$ lie in $[0,2B]=[0,\kappa-1]$.  Reducing
\eqref{ANTI:eq:tagged-equality} modulo $\kappa$ therefore gives not merely a congruence but the exact equality
\begin{equation}
 \sum_{i=1}^m z_ib_i=B.
 \label{ANTI:eq:decoded-partition}
\end{equation}
Subtract \eqref{ANTI:eq:decoded-partition} from \eqref{ANTI:eq:tagged-equality} and divide by $\kappa$:
\begin{align*}
 \sum_{i=1}^m m_i2^{i-1}
 &=\sum_{i=1}^m2^{i-1},\\
 \sum_{i=1}^m (m_i-1)2^{i-1}
 &=0.
\end{align*}
Suppose, for contradiction, that some $m_i\ne1$, and let $k$ be the largest such index.  Then
$m_k-1\in\{-1,1\}$, and the preceding equality can be written line by line as
\begin{align*}
 0
 &=\sum_{i=1}^m(m_i-1)2^{i-1}\\
 &=(m_k-1)2^{k-1}+\sum_{i=1}^{k-1}(m_i-1)2^{i-1}.
\end{align*}
Moving the lower-position sum to the other side and taking absolute values gives
\begin{align*}
 2^{k-1}
 &=\abs{(m_k-1)2^{k-1}}\\
 &=\abs{\sum_{i=1}^{k-1}(m_i-1)2^{i-1}}\\
 &\le \sum_{i=1}^{k-1}\abs{m_i-1}2^{i-1}\\
 &\le \sum_{i=1}^{k-1}2^{i-1}\\
 &=2^{k-1}-1,
\end{align*}
a contradiction.  Hence $m_i=1$ for every $i$.  Every half-sum subset therefore chooses exactly one
tag from each pair, and \eqref{ANTI:eq:decoded-partition} decodes a source subset of sum $B$.

Conversely, suppose $I\subseteq\{1,\ldots,m\}$ satisfies $\sum_{i\in I}b_i=B$.  Select $L_i+b_i$ when $i\in I$ and select $L_i$ otherwise.  Exactly one tag is chosen from each pair, and its sum is
\[
 \sum_{i=1}^mL_i+\sum_{i\in I}b_i
 =\sum_{i=1}^mL_i+B
 =\frac Q2.
\]
This proves the equivalence.

Finally, $L_i<L_i+b_i$ because $b_i>0$.  For $i=1,\ldots,m-1$, one also has $b_i<\kappa$, so
\[
 L_i+b_i
 <\kappa2^{i-1}+\kappa
 \le\kappa2^i
 =L_{i+1}.
\]
The within-pair inequalities and these cross-pair inequalities together show that the tagged sequence is strictly increasing.
\end{proof}

\begin{example}[A small yes-instance]
\label{ANTI:ex:tagging}
For $(b_1,b_2,b_3)=(1,2,3)$, one has $B=3$, $\kappa=7$, and
\[
 (a_1,\ldots,a_6)=(7,8,14,16,28,31).
\]
The total is $Q=104$ and the half-total is $52$.  The source solution $1+2=3$ becomes
\[
 8+16+28=52,
\]
whereas the source solution $3=3$ becomes
\[
 7+14+31=52.
\]
Each tagged solution visibly selects one number from every pair.
\end{example}

\setshortsectionhead{Scheduling instance}
\section{The strict antithetical scheduling instance}
\label{ANTI:sec:construction}

Start with the strictly increasing tagged values
\[
 0<a_1<a_2<\cdots<a_N,
 \qquad
 Q=\sum_{i=1}^Na_i,
\]
where $Q$ is even.  Define
\begin{equation}
\label{ANTI:eq:scales}
 U=3Q^3,
 \qquad
 M=8Q^3,
 \qquad
 D=M+Q.
\end{equation}
For every variable job $i=1,\ldots,N$, set
\begin{align}
 p_i&=a_i,
 \label{ANTI:eq:pi}\\
 r_i&=MQ+Da_i,
 \label{ANTI:eq:ri}\\
 s_i&=M(Q+2M)-Da_i,
 \label{ANTI:eq:si}\\
 w_i^E&=p_ir_i,
 \qquad
 w_i^L=p_is_i.
 \label{ANTI:eq:wi}
\end{align}
Add one anchor job $g$ with
\begin{align}
 p_g&=M,
 \label{ANTI:eq:pg}\\
 r_g&=MQ+D(Q+1),
 \label{ANTI:eq:rg}\\
 s_g&=DM,
 \label{ANTI:eq:sg}\\
 w_g^E&=p_gr_g,
 \qquad
 w_g^L=p_gs_g.
 \label{ANTI:eq:wg}
\end{align}
The common due date is especially simple:
\[
 d=\sum_{i=1}^Np_i+p_g=Q+M=D.
\]

\begin{table}[ht]
\centering
\caption{The two job types in the antithetical reduction.}
\label{ANTI:tab:jobs}
\small
\begin{tabularx}{\textwidth}{>{\raggedright\arraybackslash}p{0.10\textwidth} >{\centering\arraybackslash}p{0.13\textwidth} >{\raggedright\arraybackslash}p{0.17\textwidth} >{\raggedright\arraybackslash}p{0.24\textwidth} >{\raggedright\arraybackslash}X}
\toprule
Job & Processing time & Early ratio & Tardiness ratio & Role \\
\midrule
Variable $i$ & $a_i$ & $MQ+Da_i$ & $M(Q+2M)-Da_i$ & carries the tagged subset load \\
Anchor $g$ & $M$ & $MQ+D(Q+1)$ & $DM$ & forces the side and creates the dominant linear field \\
\bottomrule
\end{tabularx}
\end{table}

\begin{lemma}[Strict antithetical order and positive data]
\label{ANTI:lem:order}
All constructed values are positive integers, all processing times are pairwise distinct, and
\[
 r_1<r_2<\cdots<r_N<r_g,
 \qquad
 s_1>s_2>\cdots>s_N>s_g.
\]
\end{lemma}

\begin{proof}
The tagged $a_i$ are strictly increasing, and $M>Q>a_N$.  Therefore
\[
 r_{i+1}-r_i=D(a_{i+1}-a_i)>0,
 \qquad
 s_i-s_{i+1}=D(a_{i+1}-a_i)>0.
\]
Also $r_g>MQ+DQ\ge r_N$.  For positivity,
\[
 s_N\ge M(Q+2M)-DQ=2M^2-Q^2>0.
\]
Finally,
\[
 s_N-s_g
 \ge M(Q+2M)-DQ-DM
 =M^2-MQ-Q^2>0,
\]
because $M > 2Q$.  The anchor processing time $M$ exceeds every $a_i$, so all processing times are
distinct as well.
\end{proof}

\setshortsectionhead{Equal-load square}
\section{The exact equal-load square}
\label{ANTI:sec:identity}

Assume temporarily that the anchor is early.  Let $x_i=1$ if variable job $i$ is early and $x_i=0$
if it is tardy, and define
\[
 S(x)=\sum_{i=1}^Na_ix_i,
 \qquad
 P_i=\sum_{h=1}^ia_h.
\]
Let $C$ be the cost with the anchor early and every variable tardy:
\begin{equation}
\label{ANTI:eq:C}
 C=\sum_{1\le i<j\le N}a_ia_js_j+\sum_{i=1}^Na_i^2s_i.
\end{equation}
Define the residual
\begin{align}
 \Psi(x)
 &:=\sum_{1\le i<j\le N}a_ia_j(a_i-a_j)x_ix_j
 \label{ANTI:eq:psi-pair}\\
 &\quad+\sum_{i=1}^N
 a_i\left(a_iP_i+\sum_{j>i}a_j^2\right)x_i.
 \label{ANTI:eq:psi-linear}
\end{align}

\begin{lemma}[Exact cost identity]
\label{ANTI:lem:identity}
For every $x\in\bits^N$ with the anchor early,
\begin{equation}
\label{ANTI:eq:identity}
 \Phi(x)=C+DM\bigl(S(x)^2-QS(x)\bigr)+D\Psi(x).
\end{equation}
\end{lemma}

\begin{proof}[Proof of \cref{ANTI:lem:identity}]
The calculation has three explicit purposes: write the canonical objective as a quadratic polynomial in the side bits, isolate its linear and pair coefficients by finite differences, and regroup the dominant terms into a function of the single aggregate load $S(x)$.  The residual $\Psi(x)$ will then contain every dependence on job identity and position.

Assume that the anchor job $g$ is early.  Recall that
\[
Q=\sum_{i=1}^{N}a_i,
\qquad
D=M+Q,
\]
and that variable job $i$ has
\[
p_i=a_i,
\qquad
r_i=MQ+Da_i,
\qquad
s_i=M(Q+2M)-Da_i.
\]
The anchor has processing time $p_g=M$ and early ratio
\[
r_g=MQ+D(Q+1)>r_i
\qquad (i=1,\ldots,N).
\]
Moreover, for $i<j$,
\[
r_i<r_j
\qquad\text{and}\qquad
s_i>s_j.
\]
Let $x_i=1$ when variable job $i$ is early and $x_i=0$ when it is tardy, and put
\[
S(x)=\sum_{i=1}^{N}a_ix_i,
\qquad
P_i=\sum_{h=1}^{i}a_h.
\]
The all-variable-tardy baseline, with the anchor still early, is
\begin{equation}
C=
\sum_{1\le i<j\le N}a_ia_js_j
+
\sum_{i=1}^{N}a_i^2s_i.
\label{ANTI:eq:III6-C}
\end{equation}
The residual appearing in the statement of the lemma is
\begin{equation}
\Psi(x)=
\sum_{1\le i<j\le N}a_ia_j(a_i-a_j)x_ix_j
+
\sum_{i=1}^{N}
 a_i\left(a_iP_i+\sum_{j>i}a_j^2\right)x_i.
\label{ANTI:eq:III6-Psi}
\end{equation}

We first write the full canonical cost as a function of the side vector $x$.  If variable job $i$ is early, its early pair with the anchor contributes
\[
p_gp_i\min\{r_g,r_i\}=Ma_ir_i.
\]
For $i<j$, two early variable jobs contribute $a_ia_jr_i$, because $r_i<r_j$, whereas two tardy variable jobs contribute $a_ia_js_j$, because $s_i>s_j$.  A tardy variable job $i$ also contributes its self term $a_i^2s_i$.  Hence
\begin{align}
\Phi(x)
={}&\sum_{i=1}^{N}Ma_ir_i x_i
 +\sum_{1\le i<j\le N}a_ia_jr_i x_ix_j \notag\\
&+\sum_{1\le i<j\le N}a_ia_js_j(1-x_i)(1-x_j)
 +\sum_{i=1}^{N}a_i^2s_i(1-x_i).
\label{ANTI:eq:III6-Phi}
\end{align}
Setting $x=\mathbf 0$ in \eqref{ANTI:eq:III6-Phi} gives exactly \eqref{ANTI:eq:III6-C}, so
\[
\Phi(\mathbf 0)=C.
\]

We next define the coefficients that will be used in the quadratic expansion.  Let $\mathbf e_i$ denote the binary vector whose only nonzero coordinate is coordinate $i$.  Define
\begin{equation}
h_i:=\Phi(\mathbf e_i)-\Phi(\mathbf 0),
\label{ANTI:eq:III6-hi-def}
\end{equation}
and, for $i<j$, define the two-variable interaction coefficient
\begin{equation}
q_{ij}:=
\Phi(\mathbf e_i+\mathbf e_j)
-\Phi(\mathbf e_i)
-\Phi(\mathbf e_j)
+\Phi(\mathbf 0).
\label{ANTI:eq:III6-qij-def}
\end{equation}
We now verify the coefficient representation directly.  Subtracting \eqref{ANTI:eq:III6-C} from \eqref{ANTI:eq:III6-Phi} and expanding $(1-x_i)(1-x_j)=1-x_i-x_j+x_ix_j$ gives
\begin{align*}
\Phi(x)-C
={}&\sum_{i=1}^{N}Ma_ir_ix_i
 +\sum_{1\le i<j\le N}a_ia_jr_ix_ix_j \\
&+\sum_{1\le i<j\le N}a_ia_js_j
 -\sum_{1\le i<j\le N}a_ia_js_jx_i
 -\sum_{1\le i<j\le N}a_ia_js_jx_j \\
&+\sum_{1\le i<j\le N}a_ia_js_jx_ix_j
 +\sum_{i=1}^{N}a_i^2s_i
 -\sum_{i=1}^{N}a_i^2s_ix_i \\
&-\sum_{1\le i<j\le N}a_ia_js_j
 -\sum_{i=1}^{N}a_i^2s_i.
\end{align*}
The two displayed pairs of constant sums cancel.  Hence
\begin{align*}
\Phi(x)-C
={}&\sum_{i=1}^{N}Ma_ir_ix_i
 -\sum_{i=1}^{N}a_i^2s_ix_i \\
&-\sum_{1\le i<j\le N}a_ia_js_jx_i
 -\sum_{1\le i<j\le N}a_ia_js_jx_j \\
&+\sum_{1\le i<j\le N}a_ia_j(r_i+s_j)x_ix_j.
\end{align*}
Grouping all linear terms that multiply a fixed $x_i$ yields
\begin{align*}
\Phi(x)-C
={}&\sum_{i=1}^{N}
\left(
 Ma_ir_i-a_i^2s_i
 -a_i\sum_{j=i+1}^{N}a_js_j
 -a_is_i\sum_{j=1}^{i-1}a_j
\right)x_i \\
&+\sum_{1\le i<j\le N}a_ia_j(r_i+s_j)x_ix_j.
\end{align*}
For a quadratic polynomial with zero constant term, evaluation at $\mathbf e_i$ isolates the coefficient of $x_i$, while the two-coordinate finite difference in \eqref{ANTI:eq:III6-qij-def} isolates the coefficient of $x_ix_j$.  Therefore the definitions \eqref{ANTI:eq:III6-hi-def}--\eqref{ANTI:eq:III6-qij-def} give the exact expansion
\begin{equation}
\Phi(x)-C
=
\sum_{i=1}^{N}h_ix_i
+
\sum_{1\le i<j\le N}q_{ij}x_ix_j.
\label{ANTI:eq:III6-quadratic-expansion}
\end{equation}
The following calculations recover these two displayed coefficients term by term and then simplify them using the constructed ratios.

\medskip
\noindent\emph{Calculation of $q_{ij}$.}
Fix $i<j$.  In the four values appearing in \eqref{ANTI:eq:III6-qij-def}, every term not involving both $i$ and $j$ cancels.  The pair $\{i,j\}$ itself has the following costs:
\[
\begin{array}{c|c}
\text{side vector for jobs $i,j$} & \text{cost of pair $\{i,j\}$}\\ \hline
(0,0) & a_ia_js_j\\
(1,0) & 0\\
(0,1) & 0\\
(1,1) & a_ia_jr_i.
\end{array}
\]
The mixed cases cost zero because the two jobs lie on opposite sides of the due date.  Substituting these four pair costs into \eqref{ANTI:eq:III6-qij-def} gives
\begin{equation}
q_{ij}=a_ia_j(r_i+s_j).
\label{ANTI:eq:III6-qij-basic}
\end{equation}
Using the definitions of $r_i$ and $s_j$ and then $D=M+Q$,
\begin{align}
q_{ij}
&=a_ia_j\bigl(MQ+Da_i+M(Q+2M)-Da_j\bigr) \notag\\
&=a_ia_j\bigl(2MQ+2M^2+D(a_i-a_j)\bigr) \notag\\
&=a_ia_j\bigl(2M(M+Q)+D(a_i-a_j)\bigr) \notag\\
&=D\,a_ia_j(2M+a_i-a_j) \notag\\
&=2DM\,a_ia_j+D\,a_ia_j(a_i-a_j).
\label{ANTI:eq:III6-qij-expanded}
\end{align}

\medskip
\noindent\emph{Calculation of $h_i$.}
Fix $i$.  By definition \eqref{ANTI:eq:III6-hi-def}, $h_i$ is the exact change in cost when job $i$ alone is changed from tardy to early, while the anchor remains early and every other variable job remains tardy.  The affected terms are as follows.
\[
\begin{array}{@{}lccc@{}}
\text{affected term}
& \text{cost at $\mathbf 0$}
& \text{cost at $\mathbf e_i$}
& \text{change}\\ \hline
\text{anchor--$i$ pair}
&0&Ma_ir_i&+Ma_ir_i\\[2pt]
\text{tardy self term of $i$}
&a_i^2s_i&0&-a_i^2s_i\\[2pt]
\text{pair $\{j,i\}$, $j<i$}
&a_ja_is_i&0&-a_ja_is_i\\[2pt]
\text{pair $\{i,j\}$, $j>i$}
&a_ia_js_j&0&-a_ia_js_j.
\end{array}
\]
For $j<i$, the all-tardy pair cost is $a_ja_i\min\{s_j,s_i\}=a_ja_is_i$, since $s_j>s_i$.  For $j>i$, it is $a_ia_j\min\{s_i,s_j\}=a_ia_js_j$, since $s_i>s_j$.  After job $i$ becomes early, each such variable pair is split across the two sides and therefore has zero pair cost.  Every term not listed in the table is unchanged and cancels in \eqref{ANTI:eq:III6-hi-def}.  Summing the displayed changes gives
\begin{equation}
h_i
=Ma_ir_i-a_i^2s_i
 -\sum_{j<i}a_ja_is_i
 -\sum_{j>i}a_ia_js_j.
\label{ANTI:eq:III6-hi-change}
\end{equation}
Because
\[
a_i^2s_i+\sum_{j<i}a_ja_is_i
=a_is_i\left(a_i+\sum_{j<i}a_j\right)
=a_iP_is_i,
\]
dividing \eqref{ANTI:eq:III6-hi-change} by $a_i>0$ yields
\begin{equation}
\frac{h_i}{a_i}
=Mr_i-P_is_i-\sum_{j>i}a_js_j.
\label{ANTI:eq:III6-hi-divided}
\end{equation}
We now substitute the constructed ratios into every term of \eqref{ANTI:eq:III6-hi-divided}:
\begin{align}
\frac{h_i}{a_i}
={}&M(MQ+Da_i)
-P_i\bigl(M(Q+2M)-Da_i\bigr) \notag\\
&-\sum_{j>i}a_j\bigl(M(Q+2M)-Da_j\bigr) \notag\\
={}&M^2Q+MDa_i
-M(Q+2M)P_i+Da_iP_i \notag\\
&-M(Q+2M)\sum_{j>i}a_j
+D\sum_{j>i}a_j^2 \notag\\
={}&M^2Q+MDa_i
-M(Q+2M)\left(P_i+\sum_{j>i}a_j\right) \notag\\
&+D\left(a_iP_i+\sum_{j>i}a_j^2\right).
\label{ANTI:eq:III6-hi-substitution}
\end{align}
Since $P_i$ contains precisely $a_1,\ldots,a_i$ and $Q=\sum_{j=1}^{N}a_j$,
\begin{equation}
P_i+\sum_{j>i}a_j=Q.
\label{ANTI:eq:III6-prefix-total}
\end{equation}
Applying \eqref{ANTI:eq:III6-prefix-total} to \eqref{ANTI:eq:III6-hi-substitution} gives
\begin{align}
\frac{h_i}{a_i}
={}&M^2Q+MDa_i-MQ(Q+2M)
+D\left(a_iP_i+\sum_{j>i}a_j^2\right).
\label{ANTI:eq:III6-hi-after-total}
\end{align}
The scalar part in \eqref{ANTI:eq:III6-hi-after-total} simplifies line by line as
\[
M^2Q-MQ(Q+2M)
=M^2Q-MQ^2-2M^2Q
=-MQ(M+Q)
=-MQD.
\]
Therefore
\begin{align}
\frac{h_i}{a_i}
&=MDa_i-MQD
+D\left(a_iP_i+\sum_{j>i}a_j^2\right) \notag\\
&=DM(a_i-Q)
+D\left(a_iP_i+\sum_{j>i}a_j^2\right).
\label{ANTI:eq:III6-hi-simplified}
\end{align}
Multiplying \eqref{ANTI:eq:III6-hi-simplified} by $a_i$ gives the required explicit linear coefficient:
\begin{equation}
h_i
=DM(a_i^2-Qa_i)
+D\,a_i\left(a_iP_i+\sum_{j>i}a_j^2\right).
\label{ANTI:eq:III6-hi-explicit}
\end{equation}

Finally, substitute \eqref{ANTI:eq:III6-qij-expanded} and \eqref{ANTI:eq:III6-hi-explicit} into the exact coefficient expansion \eqref{ANTI:eq:III6-quadratic-expansion}.  Grouping the terms multiplied by $DM$ separately from those multiplied only by $D$ gives
\begin{align}
\Phi(x)-C
={}&DM\left[
\sum_{i=1}^{N}(a_i^2-Qa_i)x_i
+2\sum_{1\le i<j\le N}a_ia_jx_ix_j
\right] \notag\\
&+D\left[
\sum_{1\le i<j\le N}a_ia_j(a_i-a_j)x_ix_j
+\sum_{i=1}^{N}a_i\left(a_iP_i+\sum_{j>i}a_j^2\right)x_i
\right].
\label{ANTI:eq:III6-grouped}
\end{align}
Because $x_i\in\{0,1\}$, one has $x_i^2=x_i$, and hence
\begin{align}
S(x)^2
&=\left(\sum_{i=1}^{N}a_ix_i\right)^2 \notag\\
&=\sum_{i=1}^{N}a_i^2x_i^2
+2\sum_{1\le i<j\le N}a_ia_jx_ix_j \notag\\
&=\sum_{i=1}^{N}a_i^2x_i
+2\sum_{1\le i<j\le N}a_ia_jx_ix_j.
\label{ANTI:eq:III6-S-square}
\end{align}
Since $S(x)=\sum_i a_ix_i$, the first bracket in \eqref{ANTI:eq:III6-grouped} is exactly
\[
S(x)^2-QS(x).
\]
By definition \eqref{ANTI:eq:III6-Psi}, the second bracket in \eqref{ANTI:eq:III6-grouped} is exactly $\Psi(x)$.  Consequently,
\[
\Phi(x)
=C+DM\bigl(S(x)^2-QS(x)\bigr)+D\Psi(x),
\]
which proves the lemma.
\end{proof}

\paragraph{How to read the identity.}
The scheduling construction does not merely approximate a partition objective.  The first term after
$C$ is an exact convex quadratic in the total selected processing load.  All dependence on the
identities and order positions of the selected jobs is isolated in the explicitly displayed residual
$\Psi$.  This separation makes it possible to compare a unit integer load gap with one uniform
worst-case residual bound.

\begin{lemma}[Uniform residual bound]
\label{ANTI:lem:residual}
For every $x\in\bits^N$,
\[
 |\Psi(x)|\le 3Q^3=U.
\]
\end{lemma}

\begin{proof}
Because $a_i<a_j$ for $i<j$,
\begin{align*}
 \left|\sum_{i<j}a_ia_j(a_i-a_j)x_ix_j\right|
 &\le\sum_{i<j}a_ia_j(a_j-a_i)\\
 &\le\sum_{i<j}a_ia_j^2
 \le Q\sum_j a_j^2
 \le Q^3.
\end{align*}
For the two positive parts of \eqref{ANTI:eq:psi-linear},
\[
 \sum_i a_i^2P_ix_i\le Q\sum_i a_i^2\le Q^3,
\]
and
\[
 \sum_i a_ix_i\sum_{j>i}a_j^2
 \le Q\sum_j a_j^2\le Q^3.
\]
The triangle inequality gives the claim.
\end{proof}

Completing the square uses the exact algebraic identity
\[
 S^2-QS=\left(S-\frac Q2\right)^2-\frac{Q^2}{4}.
\]
Multiplying by $DM$ gives the completed-square identity:
\begin{equation}
\label{ANTI:eq:complete-square}
 DM\bigl(S^2-QS\bigr)
 =DM\left(S-\frac Q2\right)^2-\frac{DMQ^2}{4}.
\end{equation}
Since $Q$ is even and $S$ is an integer, every non-half load has squared distance at least one.

\begin{sectionexit}
We have established the exact decomposition
$\Phi=C-DMQ^2/4+DM(S-Q/2)^2+D\Psi$, together with the uniform bound $|\Psi|\le U$.  The next section uses this result only through the unit integer load gap and the inequality $M>2U$ to force the anchor, separate yes- and no-instances, and define the decision threshold.
\end{sectionexit}

\setshortsectionhead{Anchor and separation}
\section{Forcing the anchor and separating yes from no}
\label{ANTI:sec:gap}

Define the integer decision threshold
\begin{equation}
\label{ANTI:eq:threshold}
 K=C-\frac{DMQ^2}{4}+DU.
\end{equation}
The next lemma rules out the only case not covered by \cref{ANTI:lem:identity}.

\begin{lemma}[The anchor is forced early]
\label{ANTI:lem:anchor}
Every canonical assignment of cost at most $K$ places the anchor in the early set.
\end{lemma}

\begin{proof}
We first show that the threshold is positive.  By \cref{ANTI:lem:order}, every variable tardiness ratio exceeds $s_g=DM$.  Hence each term in the baseline $C$ is strictly larger than the same processing product multiplied by $DM$, and
\begin{align*}
 C
 &>DM\left(\sum_{i<j}a_ia_j+\sum_i a_i^2\right)\\
 &=\frac{DM}{2}\left(Q^2+\sum_i a_i^2\right)\\
 &\ge\frac{DMQ^2}{2}.
\end{align*}
Consequently,
\begin{align*}
 K
 &=C-\frac{DMQ^2}{4}+DU\\
 &>\frac{DMQ^2}{4}+DU\\
 &>0.
\end{align*}

We next derive an upper bound on $K$.  Since $Q<M$ and $Da_i>0$,
\[
 s_i=M(Q+2M)-Da_i<M(Q+2M)<3M^2.
\]
Also
\[
 \sum_{i<j}a_ia_j+\sum_i a_i^2
 =\frac12\left(Q^2+\sum_i a_i^2\right)
 \le Q^2.
\]
Thus
\begin{equation}
 C<3M^2Q^2.
 \label{ANTI:eq:C-upper-expanded}
\end{equation}
Because $D=M+Q<2M$ and $U=3Q^3$,
\begin{align*}
 K
 &<C+DU\\
 &<3M^2Q^2+6MQ^3.
\end{align*}
Now $M=8Q^3$ and $Q\ge1$, so
\[
 Q^2\le Q^3=\frac M8,
 \qquad
 6MQ^3=\frac34M^2.
\]
Substitution gives
\begin{align*}
 K
 &<\frac38M^3+\frac34M^2\\
 &\le\frac38M^3+\frac3{32}M^3
   \qquad(M\ge8)\\
 &=\frac{15}{32}M^3\\
 &<M^3.
\end{align*}

However, if the anchor is tardy, its tardy self-cost alone is
\begin{align*}
 p_g^2s_g
 &=M^2(DM)\\
 &=DM^3\\
 &>M^3\\
 &>K.
\end{align*}
Therefore every canonical assignment of cost at most $K$ places the anchor early.
\end{proof}

\begin{theorem}[Correctness of the antithetical reduction]
\label{ANTI:thm:antihard}
The source \PARTITION instance is a yes-instance if and only if the constructed strict-antithetical
\AWET\ instance has a schedule of cost at most $K$.
\end{theorem}

\begin{proof}
Suppose first that the source is feasible.  By \cref{ANTI:lem:tagging}, there is a tagged side vector $x$ with
\[
 S(x)=\frac Q2.
\]
Place the anchor early and use the canonical schedule for this side assignment.  Combining the exact identity with the completed square gives
\begin{align*}
 \Phi(x)
 &=C+DM\left(S(x)-\frac Q2\right)^2
   -\frac{DMQ^2}{4}+D\Psi(x)\\
 &=C-\frac{DMQ^2}{4}+D\Psi(x).
\end{align*}
By \cref{ANTI:lem:residual}, $\Psi(x)\le U$, so
\begin{align*}
 \Phi(x)
 &\le C-\frac{DMQ^2}{4}+DU\\
 &=K.
\end{align*}
Thus a source partition produces a threshold-feasible schedule.

Conversely, suppose the constructed instance has a schedule of cost at most $K$.  Then its global optimum is at most $K$.  By \cref{AA:lem:compact}, choose an optimal schedule in compact canonical form; by \cref{ANTI:prop:canonical}, its cost is $\Phi(x)\le K$ for the corresponding variable-side vector $x$.  The anchor-forcing lemma \cref{ANTI:lem:anchor} places the anchor early.  The identity and the residual lower bound $\Psi(x)\ge-U$ therefore give
\begin{align*}
 \Phi(x)
 &=C-\frac{DMQ^2}{4}
   +DM\left(S(x)-\frac Q2\right)^2
   +D\Psi(x)\\
 &\ge C-\frac{DMQ^2}{4}
   +DM\left(S(x)-\frac Q2\right)^2
   -DU.
\end{align*}
If $S(x)\ne Q/2$, then $S(x)$ and $Q/2$ are integers, so
\[
 \left|S(x)-\frac Q2\right|\ge1
 \qquad\text{and}\qquad
 \left(S(x)-\frac Q2\right)^2\ge1.
\]
Hence
\begin{align*}
 \Phi(x)
 &\ge C-\frac{DMQ^2}{4}+DM-DU\\
 &=\left(C-\frac{DMQ^2}{4}+DU\right)+DM-2DU\\
 &=K+D(M-2U).
\end{align*}
Finally,
\[
 M-2U=8Q^3-6Q^3=2Q^3>0,
\]
so the last bound is strictly larger than $K$, a contradiction.  Therefore the chosen canonical optimum satisfies $S(x)=Q/2$.  The binary tagging lemma then selects exactly one tag from each pair and decodes a source subset of sum $B$.
\end{proof}

\begin{lemma}[Binary encoding length of the antithetical reduction]
\label{ANTI:lem:bitsize}
Let $L_{\rm src}$ be the binary input length of the source \PARTITION instance.  The constructed \AWET\ instance and threshold have size polynomial in $L_{\rm src}$ and can be computed in polynomial time.
\end{lemma}
\begin{proof}
There are $N=2m$ variable jobs and one anchor.  From \eqref{ANTI:eq:tags},
\[
 a_{2i-1}=\kappa2^{i-1},\qquad a_{2i}=\kappa2^{i-1}+b_i,
\]
so every tagged number, and hence $Q=\sum_i a_i$, has
$O(m+\log(B+1))=O(L_{\rm src})$ bits.  For $Q\ge1$,
\[
 M=8Q^3,\qquad U=3Q^3,\qquad D=M+Q\le9Q^3.
\]
A variable job satisfies
\[
 p_i\le Q,\qquad r_i=O(Q^4),\qquad s_i=O(Q^6),
\]
and therefore
\[
 w_i^E=O(Q^5),\qquad w_i^L=O(Q^7).
\]
For the anchor,
\[
 p_g=O(Q^3),\qquad w_g^E=O(Q^7),\qquad w_g^L=O(Q^9).
\]
The baseline $C$ is a sum of $O(m^2)$ products of integers having polynomially many bits.  The threshold
\[
 K=C-\frac{DMQ^2}{4}+DU
\]
has the same property; it is integral because $Q$ is even.  All displayed quantities are obtained by polynomially many integer operations on polynomial-bit operands.  Hence the instance map is a polynomial-time binary reduction.
\end{proof}

\begin{theorem}[Pseudopolynomial algorithm for strict antithetical order]
\label{ANTI:thm:pseudopoly}
Let jobs be indexed so that $r_1<\cdots<r_n$ and $s_1>\cdots>s_n$.  Put
\[
 W_E=\sum_{i=1}^n w_i^E,
 \qquad
 P=\sum_{i=1}^n p_i.
\]
Then the optimal value of \AWET can be computed exactly in
\[
 O\bigl(n(W_E+1)(P+1)\bigr)
\]
arithmetic operations and $O((W_E+1)(P+1))$ rolling memory.  An optimal early set and canonical schedule can be recovered by storing one predecessor per reachable state, using $O\bigl(n(W_E+1)(P+1)\bigr)$ memory, or by a standard recomputation scheme.
\end{theorem}
\begin{proof}
For $i<j$, the early ratio minimum is $r_i$ and the tardiness ratio minimum is $s_j$.  Thus the canonical objective for side bits $x_i\in\{0,1\}$ is
\begin{align*}
 \Phi(x)
 ={}&\sum_{i<j}w_i^Ep_jx_ix_j
 +\sum_{i<j}p_iw_j^L(1-x_i)(1-x_j)\\
 &+\sum_i p_iw_i^L(1-x_i).
\end{align*}
Process jobs in the index order.  After jobs $1,\ldots,j$ have been assigned, retain
\[
 A=\sum_{\substack{i\le j\\x_i=1}}w_i^E,
 \qquad
 T=\sum_{\substack{i\le j\\x_i=0}}p_i,
\]
and let $G_j(A,T)$ be the least cost of all terms whose later endpoint is at most $j$.  Initialize
\[
 G_0(0,0)=0
\]
and set all other entries to $+\infty$.

Suppose a state $(A,T)$ is present before job $j$ is inserted.  If job $j$ is early, every earlier early job $i$ contributes $w_i^Ep_j$, so the exact increment is
\[
 p_jA.
\]
The transition is
\[
 G_j(A+w_j^E,T)
 \leftarrow
 \min\bigl\{G_j(A+w_j^E,T),\;G_{j-1}(A,T)+p_jA\bigr\}.
\]
If job $j$ is tardy, every earlier tardy job contributes $p_iw_j^L$, totaling $Tw_j^L$, and job $j$ also contributes its self-cost $p_jw_j^L$.  The tardy transition is
\[
 G_j(A,T+p_j)
 \leftarrow
 \min\bigl\{G_j(A,T+p_j),\;G_{j-1}(A,T)+w_j^L(T+p_j)\bigr\}.
\]
Each early pair, tardy pair, and tardy self-term is charged exactly once, when its later-indexed job is inserted.  Conversely, every side assignment follows one recurrence path.  Because all earliness weights are positive, the nonempty-early condition is exactly $A>0$ at the terminal layer.  Therefore
\[
 \OPT=\min_{A>0,\,T}G_n(A,T).
\]
There are at most $(W_E+1)(P+1)$ states per layer and two constant-time transitions per state, proving the stated value-computation bounds.  Storing the minimizing predecessor for each reachable state permits direct reconstruction of an optimal side vector; alternatively, rolling tables can be recomputed during a backtracking pass.
\end{proof}

\begin{corollary}[Strict-antithetical NP-completeness and pseudopolynomial solvability]
\label{ANTI:cor:weaknp}
The decision version of positive-integer \AWET\ is NP-complete and pseudopolynomially solvable even when the early- and tardiness-ratio orders are both strict and are exact reversals of one another.  In the standard terminology for numerical scheduling problems, the restriction is therefore weakly NP-complete.
\end{corollary}
\begin{proof}
The correctness theorem and \cref{ANTI:lem:bitsize} give a polynomial-time binary reduction from \PARTITION, so the restricted problem is NP-hard.  It belongs to NP because an early set and its canonical orders form a polynomially checkable certificate whose cost is computable by exact integer arithmetic.  The value algorithm in \cref{ANTI:thm:pseudopoly} is pseudopolynomial.  Thus the restricted decision problem is NP-complete and pseudopolynomially solvable, which is the usual weak-NP-completeness classification.  In particular, it is not strongly NP-complete unless $\mathrm P=\mathrm{NP}$.
\end{proof}

\begin{corollary}[FPTAS for strict antithetical order]
	Positive-integer \AWET\ restricted to strict antithetical ratio orders
	admits an FPTAS.
\end{corollary}

\begin{proof}
	Under the strict inequalities \(r_1<\cdots<r_n\) and
	\(s_1>\cdots>s_n\), the two ratio orders are unique, and the induced
	ratio permutation is \(n(n-1)\cdots1\). This is the iterated skew sum
	of singleton permutations, hence is separable. Therefore the
	separating-tree FPTAS of \cref{CT:thm:main} applies.
\end{proof}

\begin{sectionexit}
The reduction is now complete: the constructed instance has a schedule of cost at most $K$ exactly when a compact canonical optimum has the anchor early and tagged load $Q/2$; binary tagging decodes that load to a source partition, and the same strict order admits the two-resource pseudopolynomial recurrence.  The worked examples use only the tag equivalence, the completed-square identity, and the certified margin $D(M-2U)>0$ to instantiate both directions numerically.
\end{sectionexit}

\clearpage
\setshortsectionhead{Worked examples}
\section{Two worked examples}
\begin{howtoread}
\textbf{Conceptual route.}  Compare the tagged half-load calculation and the closest non-half loads in the two examples.  \textbf{Audit route.}  Verify every job parameter, threshold component, and finite-state cost table.  The first route explains the square-gap mechanism; the second reproduces the full numerical certificate.
\end{howtoread}
\label{ANTI:sec:worked-examples}

The two source instances have four numbers each.  The first is a yes-instance of \PARTITION; the second changes only the final source number and is a no-instance.  Each source instance produces eight tagged variable jobs and one anchor job.  The comparison table predicts the distinction that the detailed calculations will verify.

\begin{center}
\footnotesize
\setlength{\tabcolsep}{3pt}
\begin{tabularx}{0.97\textwidth}{@{}>{\raggedright\arraybackslash}p{0.075\textwidth}>{\raggedright\arraybackslash}p{0.145\textwidth}>{\raggedright\arraybackslash}p{0.12\textwidth}>{\raggedright\arraybackslash}p{0.145\textwidth}>{\raggedright\arraybackslash}p{0.14\textwidth}>{\raggedright\arraybackslash}X@{}}
\toprule
case & source data & half-load target & attainable tagged load & least square displacement & threshold outcome\\
\midrule
yes & $(1,2,5,8)$, $B=8$ & $Q/2=263$ & $263$ & $(S-Q/2)^2=0$ & two canonical schedules have cost below $K$\\
no & $(1,2,5,10)$, $B=9$ & $Q/2=294$ & closest are $293,295$ & $(S-Q/2)^2\ge1$ & every schedule has cost above $K$\\
\bottomrule
\end{tabularx}
\end{center}

All arithmetic below is exact.

\setshortsubsectionhead{Yes instance}
\subsection{A ``yes'' instance}
\label{ANTI:worked-yes}

The source instance is
\[
 (b_1,b_2,b_3,b_4)=(1,2,5,8),
 \qquad \sum_i b_i=16=2B,
 \qquad B=8.
\]
It has two complementary-looking witnesses: $b_4=8$ and $b_1+b_2+b_3=1+2+5=8$.

\begin{howtoread}
The purpose of the example is to verify both directions of the reduction, not merely to exhibit one cheap schedule.  The ``if'' calculation starts with the source subset $\{1,2,3\}$, carries it through the binary tags, constructs the corresponding early set, and evaluates its exact canonical cost.  The ``only if'' calculation starts hypothetically from an arbitrary schedule of cost at most $K$, passes to a compact canonical global optimum, forces that optimum's anchor early and its selected tagged load to $Q/2$, and then decodes the original partition witness pair by pair.
\end{howtoread}

\paragraph{The source \PARTITION instance}
The exhaustive source-level subset audit is small enough to display.  It makes clear that the half-total $8$ is attained exactly by $\{4\}$ and $\{1,2,3\}$.
\begin{center}
\begin{tabular}{c r@{\qquad}c r}
\toprule
Subset & Sum & Subset & Sum\\
\midrule
$\varnothing$&0 & $\{4\}$&8\\
$\{1\}$&1 & $\{1,4\}$&9\\
$\{2\}$&2 & $\{2,4\}$&10\\
$\{1,2\}$&3 & $\{1,2,4\}$&11\\
$\{3\}$&5 & $\{3,4\}$&13\\
$\{1,3\}$&6 & $\{1,3,4\}$&14\\
$\{2,3\}$&7 & $\{2,3,4\}$&15\\
$\{1,2,3\}$&\textbf{8} & $\{1,2,3,4\}$&16\\
\bottomrule
\end{tabular}
\end{center}

\paragraph{Binary tagging and the exact half load}
Here $\kappa=2B+1=17$ and
\[
 (L_1,L_2,L_3,L_4)=(17,34,68,136).
\]
The four tagged pairs are therefore
\[
 \{17,18\},\qquad \{34,36\},\qquad \{68,73\},\qquad \{136,144\}.
\]
Thus
\[
 (a_1,\ldots,a_8)=(17,18,34,36,68,73,136,144),
 \qquad Q=526,
 \qquad \frac Q2=263.
\]
The source witness $\{1,2,3\}$ says ``take the second member'' of pairs 1, 2, and 3, and ``take the first member'' of pair 4.  Hence
\[
 x=(0,1,0,1,0,1,1,0),
 \qquad S(x)=18+36+73+136=263=Q/2.
\]
The other source witness $\{4\}$ gives the complementary tagged selection
\[
 x'=(1,0,1,0,1,0,0,1),
 \qquad S(x')=17+34+68+144=263.
\]
This concretely illustrates the pair-tagging lemma: every source choice is recorded only in the low-order increments $b_i$, whereas the positional terms $L_i$ force exactly one tag from each pair.

\paragraph{Numerical scales and the nine-job instance}
The scale values are
\[
 U=3Q^3=436594728,
 \qquad M=8Q^3=1164252608,
 \qquad D=M+Q=1164253134.
\]
The common due date is $d=D=1164253134$.  The table below lists every job parameter needed to reproduce the instance.  The displayed ratios are integers; the actual weights are $w_i^E=p_ir_i$ and $w_i^L=p_is_i$.

\begin{landscape}
\scriptsize
\setlength{\tabcolsep}{3pt}
\tablecert{\textbf{Purpose.}  The table lists the nine job parameters needed to reproduce the antithetical yes-instance.  \textbf{Verification rule.}  For each variable job, substitute $p_i=a_i$ into the affine ratio formulas and multiply $p_ir_i$ and $p_is_i$ to recover the two weights; for the anchor, use the four displayed anchor definitions.}
\begin{longtable}{c r r r r r}
\caption{Nine-job yes-instance: processing times, ratios, and physical weights.}\label{ANTI:tab:yes-jobs}\\
\toprule
Job & $p$ & $r=w^E/p$ & $s=w^L/p$ & $w^E$ & $w^L$\\
\midrule
\endfirsthead
\multicolumn{6}{c}{\tablename\ \thetable\ (continued)}\\
\toprule
Job & $p$ & $r$ & $s$ & $w^E$ & $w^L$\\
\midrule
\endhead
1&17&632189175086&2710968863074171858&10747215976462&46086470672260921586\\
2&18&633353428220&2710968861909918724&11400361707960&48797439514378537032\\
3&34&651981478364&2710968843281868580&22167370264376&92172940671583531720\\
4&36&654309984632&2710968840953362312&23555159446752&97594878274321043232\\
5&68&691566084920&2710968803697262024&47026493774560&184345878651413817632\\
6&73&697387350590&2710968797875996354&50909276593070&197900722244947733842\\
7&136&770735298032&2710968724528048912&104820000532352&368691746535814652032\\
8&144&780049323104&2710968715214023840&112327102526976&390379494990819432960\\
$g$&1164252608&1225958273426&1355484747631673472&1427325117135397595008&1578126652534397663180414976\\
\bottomrule
\end{longtable}
\end{landscape}

The strict antithetical inequalities can now be inspected numerically:
\[
 r_1<r_2<\cdots<r_8<r_g,
 \qquad
 s_1>s_2>\cdots>s_8>s_g.
\]
For example, $r_2-r_1=D(18-17)=1164253134>0$, while $s_1-s_2$ is the same positive quantity.  The same affine difference calculation applies to every adjacent pair.

\paragraph{Baseline, square term, residual allowance, and threshold}
For this instance the all-variable-tardy baseline with the anchor early is
\[
 C=445854051730378609479774.
\]
Also
\[
 DM=1355484747631673472,
 \qquad
 \frac{DMQ^2}{4}=93757524508935222384768,
\]
and
\[
 DU=508306780361877552.
\]
Consequently
\[
 K=C-\frac{DMQ^2}{4}+DU
 =352097035528223748972558.
\]
For the witness $x=(0,1,0,1,0,1,1,0)$, direct substitution in the residual formula gives
\[
 \Psi(x)=15322886.
\]
Because $S(x)=Q/2$, the completed-square term is at its exact minimum, so
\begin{align*}
 \Phi(x)
 &=C-\frac{DMQ^2}{4}+D\Psi(x)\\
 &=352096545061161434519730.
\end{align*}
Thus
\[
 K-\Phi(x)=490467062314452828>0.
\]
The witness schedule is therefore not merely feasible; it lies strictly below the decision threshold.

\paragraph{The ``if'' case: source witness $\Rightarrow$ threshold-feasible schedule}
We now trace the forward implication without suppressing any logical step.
\begin{enumerate}[label=\textbf{\arabic*.},leftmargin=4em]
\item The source subset $\{1,2,3\}$ has sum $1+2+5=8=B$.
\item For each selected source item choose $L_i+b_i$, and for each unselected item choose $L_i$.  This gives $18,36,73,136$.
\item Their tagged sum is $263=Q/2$.  The associated binary side vector is $x=(0,1,0,1,0,1,1,0)$.
\item Put the anchor $g$ on the early arm and put precisely the four tagged jobs selected by $x$ on the early arm.  Put the remaining four variable jobs on the tardy arm.  Proposition~\ref{ANTI:prop:canonical} supplies the compact canonical schedule for this side assignment.
\item The exact identity gives $\Phi=C+DM(S^2-QS)+D\Psi$.  Since $S=263$, the dominant term equals $-DMQ^2/4$.
\item The exact residual is $15322886<U=436594728$.  The inequality $D\Psi<DU$ and the numerical evaluation above gives $\Phi=352096545061161434519730<K$.
\end{enumerate}
This is the theorem's forward implication instantiated completely.

\paragraph{The ``only if'' case: threshold-feasible schedule $\Rightarrow$ source witness}
Now begin at the target end.  Suppose some schedule for this nine-job instance has cost at most $K$.
\begin{enumerate}[label=\textbf{\arabic*.},leftmargin=4em]
\item Its global optimum is at most $K$.  Choose a compact canonical optimal schedule, whose cost is therefore also at most $K$.
\item The anchor-forcing lemma applies to this canonical optimum because the anchor self-cost when tardy already exceeds $K$.  Hence $g$ must be early.
\item With the anchor early, the exact identity applies.  If $S\ne263$, integrality gives $|S-263|\ge1$, so the completed-square term increases by at least $DM$.
\item The residual can improve by at most $DU$ relative to zero and by at most $2DU$ when compared against the threshold allowance.  Every non-half load therefore satisfies
\[
 \Phi\ge K+D(M-2U).
\]
Here
\[
 D(M-2U)=338871186907918368>0,
\]
contradicting $\Phi\le K$.  Hence $S=263$.
\item The tagging lemma now forces one selected number from each pair.  Write $z_i=1$ when the second tag $L_i+b_i$ is selected.  Reducing the equality $S=Q/2$ modulo $\kappa=17$ gives
\[
 z_1(1)+z_2(2)+z_3(5)+z_4(8)=8.
\]
\item The source audit above shows that this equation has exactly two $0$--$1$ solutions: $(z_1,z_2,z_3,z_4)=(1,1,1,0)$ and $(0,0,0,1)$.  The chosen canonical optimum therefore decodes to the source partition $\{1,2,3\}$ or $\{4\}$.
\end{enumerate}
The reverse implication therefore recovers a genuine source witness rather than merely a numerically balanced tagged load.

\paragraph{An exhaustive half-load audit for the yes-instance}
There are $2^8=256$ variable-side vectors.  The binary-tagging lemma says that any vector with $S=263$ must choose exactly one tag from each pair, reducing the search to the $2^4=16$ source patterns already listed.  Exactly two of those patterns have low-order sum $8$.  Hence there are exactly two half-load vectors, the two vectors $x,x'$ displayed above.  Direct evaluation gives
\[
 \Phi(x)=352096545061161434519730,
\]
while for $x'=(1,0,1,0,1,0,0,1)$,
\[
 \Psi(x')=15680566,
 \qquad
 \Phi(x')=352096545477591495488850<K.
\]
Every other variable-side vector has $|S-263|\ge1$ and is excluded by the certified square-gap argument.  This completes both the constructive and the decoding audit.

\clearpage
\setshortsubsectionhead{No instance}
\subsection{A similar ``no'' instance}
\label{ANTI:worked-no}

Change only the last source number from $8$ to $10$:
\[
 (b_1,b_2,b_3,b_4)=(1,2,5,10),
 \qquad \sum_i b_i=18=2B,
 \qquad B=9.
\]
The first three numbers sum to only $8$, while the fourth number alone already exceeds $9$.  Thus the instance is intuitively close to the preceding one but crosses the partition boundary.

\begin{howtoread}
The forward direction is now demonstrated by showing precisely where the construction of a threshold-feasible schedule fails: there is no source subset of sum $B$, hence no tagged half load.  The reverse direction begins with a hypothetical target schedule of cost at most $K$, passes to a compact canonical global optimum, and derives the impossible source equation $z_1+2z_2+5z_3+10z_4=9$.
\end{howtoread}

\paragraph{The modified source instance and exhaustive subset audit}
The sixteen source subset sums are
\begin{center}
\begin{tabular}{c r@{\qquad}c r}
\toprule
Subset & Sum & Subset & Sum\\
\midrule
$\varnothing$&0 & $\{4\}$&10\\
$\{1\}$&1 & $\{1,4\}$&11\\
$\{2\}$&2 & $\{2,4\}$&12\\
$\{1,2\}$&3 & $\{1,2,4\}$&13\\
$\{3\}$&5 & $\{3,4\}$&15\\
$\{1,3\}$&6 & $\{1,3,4\}$&16\\
$\{2,3\}$&7 & $\{2,3,4\}$&17\\
$\{1,2,3\}$&8 & $\{1,2,3,4\}$&18\\
\bottomrule
\end{tabular}
\end{center}
No entry equals $B=9$.  Hence this is a no-instance of \PARTITION.

\paragraph{Binary tagging and the missing half load}
Now $\kappa=19$ and
\[
 (L_1,L_2,L_3,L_4)=(19,38,76,152).
\]
The tagged pairs are
\[
 \{19,20\},\qquad \{38,40\},\qquad \{76,81\},\qquad \{152,162\},
\]
so
\[
 (a_1,\ldots,a_8)=(19,20,38,40,76,81,152,162),
 \qquad Q=588,
 \qquad Q/2=294.
\]
If a tagged subset had sum $294$, the tagging lemma would force exactly one number per pair and, modulo $19$, would give
\[
 z_1+2z_2+5z_3+10z_4=9.
\]
The source table proves that no such binary vector $z$ exists.  Therefore no tagged subset has load $294$.

The two closest pair-respecting choices come from source sums $8$ and $10$.  They have tagged loads
\[
 293=294-1
 \quad\text{and}\quad
 295=294+1,
\]
respectively.  Explicitly,
\[
 x^-=(0,1,0,1,0,1,1,0),\qquad S(x^-)=20+40+81+152=293,
\]
and
\[
 x^+=(1,0,1,0,1,0,0,1),\qquad S(x^+)=19+38+76+162=295.
\]
Thus the dominant square is forced at least one unit away from its minimum.

\paragraph{Numerical scales and the nine-job no-instance}
The new scales are
\[
 U=609892416,
 \qquad M=1626379776,
 \qquad D=1626380364,
\]
and the common due date is $d=D=1626380364$.

\begin{landscape}
\scriptsize
\setlength{\tabcolsep}{3pt}
\tablecert{\textbf{Purpose.}  The table lists the nine job parameters needed to reproduce the antithetical no-instance.  \textbf{Verification rule.}  Recompute each ratio from the hatted scale formulas, multiply by the displayed processing time to recover each weight, and use adjacent differences to check strict antithetical order.}
\begin{longtable}{c r r r r r}
\caption{Nine-job no-instance: processing times, ratios, and physical weights.}\label{ANTI:tab:no-jobs}\\
\toprule
Job & $p$ & $r=w^E/p$ & $s=w^L/p$ & $w^E$ & $w^L$\\
\midrule
\endfirsthead
\multicolumn{6}{c}{\tablename\ \thetable\ (continued)}\\
\toprule
Job & $p$ & $r$ & $s$ & $w^E$ & $w^L$\\
\midrule
\endhead
1&19&987212535204&5290223276973701724&18757038168876&100514242262500332756\\
2&20&988838915568&5290223275347321360&19776778311360&105804465506946427200\\
3&38&1018113762120&5290223246072474808&38688322960560&201028483350754042704\\
4&40&1021366522848&5290223242819714080&40854660913920&211608929712788563200\\
5&76&1079916215952&5290223184270020976&82073632412352&402056962004521594176\\
6&81&1088048117772&5290223176138119156&88131897539532&428508077267187651636\\
7&152&1203521123616&5290223060665113312&182935210789632&804113905221097223424\\
8&162&1219784927256&5290223044401309672&197605158215472&857016133193012166864\\
$g$&1626379776&1914249342684&2645112132093118464&3113296417162551158784&4301956876888488418621784064\\
\bottomrule
\end{longtable}
\end{landscape}

Again the data visibly satisfy strict antithetical order.  The baseline and threshold components are
\[
 C=1087759799531923395684660,
\]
\begin{align*}
 DM&=2645112132093118464,\\
 \frac{DMQ^2}{4}&=228632912249600787554304,\\
 DU&=991917049534919424.
\end{align*}
and therefore
\[
 K=859127879199372143049780.
\]
The universal certified separation is
\[
 D(M-2U)=661278033023279616>0.
\]

\paragraph{The ``if'' route: exactly where the forward construction fails}
The forward proof of the theorem begins with a source subset summing to $B=9$.  The exhaustive source table shows that there is none.  Equivalently, the pair-respecting tagged loads nearest $Q/2=294$ are $293$ and $295$, not $294$.  Therefore the forward recipe cannot produce a vector $x$ with $S=Q/2$, and the dominant square cannot attain its minimum.

This failure is quantitative.  For the closest lower-load vector $x^-$,
\[
 \Psi(x^-)=21387022,
\]
so the exact identity gives
\[
 \Phi(x^-)=859129567177887326484828
 =K+1687978515183435048>K.
\]
For the closest upper-load vector $x^+$,
\[
 \Psi(x^+)=22087552,
\]
and
\[
 \Phi(x^+)=859129568317215562877748
 =K+1689117843419827968>K.
\]
Thus even the two load-nearest canonical candidates miss the threshold by more than $1.68\times10^{18}$.

\paragraph{The ``only if'' route: a hypothetical target witness is impossible}
Suppose, for contradiction, that some schedule for the no-instance has cost at most $K$.
\begin{enumerate}[label=\textbf{\arabic*.},leftmargin=4em]
\item Its global optimum is at most $K$.  Choose a compact canonical optimal schedule, whose cost is therefore also at most $K$.
\item The anchor-forcing lemma implies that $g$ is early in this canonical optimum.
\item Since the anchor is early, the exact identity and completed-square equation apply.
\item If $S\ne294$, integrality gives $|S-294|\ge1$.  The theorem's gap calculation then yields
\[
 \Phi\ge K+D(M-2U)=K+661278033023279616>K,
\]
a contradiction.  Hence the hypothetical canonical optimum would have to satisfy $S=294$.
\item By binary tagging, a half-load subset must choose exactly one member of each tag pair.  Let $z_i$ indicate use of the second member.  Reduction modulo $\kappa=19$ yields
\[
 z_1+2z_2+5z_3+10z_4=9.
\]
\item But the exhaustive sixteen-row source audit contains no subset sum $9$.  This is impossible.
\end{enumerate}
Therefore the target instance has no schedule of cost at most $K$.

\paragraph{Finite-state audit of the closest canonical assignments}
For completeness, all $2^8=256$ variable-side vectors were evaluated from the exact identity.  The table below lists the closest load levels and, for each displayed level, the best exact canonical cost found there.  The first row is the global minimum among anchor-early assignments.
\begin{center}
\small
\begin{tabular}{r r r r}
\toprule
$S$ & $S-Q/2$ & best $\Psi$ shown & $\Phi-K$\\
\midrule
293&$-1$&21387022&1687978515183435048\\
295&$+1$&22087552&1689117843419827968\\
292&$-2$&21341770&9623241314498558712\\
296&$+2$&22130722&9624524450539497240\\
291&$-3$&21273742&22848691335560748840\\
297&$+3$&22196980&22850192871715247472\\
290&$-4$&21228412&41364402536390677968\\
298&$+4$&22240072&41366047880349722208\\
\bottomrule
\end{tabular}
\end{center}
The numerical audit is stronger than the proof requires: the theoretical lower gap already excludes every non-half load uniformly, whereas the enumeration shows that the actual best no-instance cost exceeds $K$ by substantially more.

\paragraph{Side-by-side conclusion}
The two examples differ only in replacing the last source item $8$ by $10$.  In the yes-instance, the low-order equation has solutions $1+2+5=8$ and $8=8$, producing two tagged half-load assignments and two schedules below threshold.  In the no-instance, the low-order target becomes $9$; the first three source items stop at $8$ and the last starts at $10$.  Consequently the nearest tagged loads are $Q/2\pm1$, the square term pays at least one full unit of $DM$, and the residual is too small to erase the gap.  This is precisely the finite-arithmetic ``if and only if'' mechanism of the antithetical reduction.

\begin{reusableidea}
A hard binary choice can sometimes be isolated by making one aggregate quantity pay a dominant convex square while proving a uniform bound on every identity-dependent residual.  The reduction then needs only an integral unit gap and a scale choice larger than twice the residual budget; the detailed form of the residual may remain complicated.
\end{reusableidea}

\begin{finding}{Why Part IV is next}
Part III has classified exact reversal as weakly NP-complete and pseudopolynomially solvable.  It leaves open how far the algorithmic structure extends beyond one reverse permutation.  Part IV replaces the single load state by a four-coordinate separating-tree interface and uses recursive direct/skew order compatibility to obtain an FPTAS for a larger class.
\end{finding}

\clearpage

\part{Separating-Tree Analysis}
\thispagestyle{plain}
\label{part:cotree}
\setrunningunit{Part IV: Separating Trees}
\begin{center}
{\Large\bfseries An FPTAS for Separable Ratio Permutations\par}
\end{center}
\medskip
\begin{quote}\small
\noindent\textbf{Part synopsis.}
For fixed total refinements of the early- and tardiness-ratio orders, suppose the induced ratio permutation is separable.  A separating tree then recursively divides the jobs by direct and skew sums.  Every tree node determines a precise block of jobs, and a side assignment on that block is summarized by four nonnegative coordinates: early processing, tardy processing, early weight, and tardy weight.  The cost of pairs crossing two child blocks is an exact nonnegative bilinear form.  This gives an exact pseudopolynomial dynamic program for the original scheduling problem.  Keeping one minimum-cost candidate per coordinated geometric box converts that recurrence into an FPTAS.
\end{quote}

\begin{partposition}
\textbf{Starting boundary.}  The canonical side-assignment objective is quadratic, but a separable ratio permutation fixes the orientation of every pair crossing a direct or skew decomposition node.  \textbf{Result.}  For fixed total ratio-order refinements whose ratio permutation is separable, \AWET admits an FPTAS.  \textbf{Technique.}  An ordered separating tree compresses every future cross interaction into a four-coordinate nonnegative additive interface, and coordinated geometric trimming preserves one cheapest representative per box. 
\end{partposition}

\begin{notationreset}
The symbols $\prec_E$ and $\prec_L$ denote fixed total orders refining nondecreasing early and
tardiness ratios, and $\pi$ is their ratio permutation.  For a node $M$ of an ordered separating tree,
$J(M)$ denotes exactly the set of leaves below $M$; this set is called the \emph{separating-tree block}
of $M$.  For a side assignment on $J(M)$, the signature is $z_M=(e_M,t_M,u_M,v_M)$, where the four
entries are defined explicitly in \cref{CT:def:signature}.  The root height is $H$, the trimming ratio
is $\lambda=1+\varepsilon/(4H)$, and zero is always stored in a singleton bucket.
\end{notationreset}

\setshortsectionhead{Proof map}
\section{Introduction and proof map}

The argument has two layers.  First, an exact separating-tree dynamic program enumerates every side
assignment through a four-coordinate interface and returns an optimal schedule.  Its number of
numerical states is pseudopolynomial.  Second, coordinated geometric trimming stores only one
representative per four-dimensional box.  Nonnegativity and bilinearity ensure that the resulting
multiplicative loss can be charged by separating-tree height.

\begin{proofcontract}
\begin{itemize}[leftmargin=2em,itemsep=1pt,topsep=2pt]
\item \textbf{Input class:} positive-integer \AWET with $d=\sum_jp_j$, together with two fixed total
      ratio-order refinements whose ratio permutation is separable.
\item \textbf{Exact layer:} define a local canonical cost and a four-coordinate table at every
      separating-tree node; prove that the root recurrence equals the optimum of the original
      scheduling problem.
\item \textbf{Approximation layer:} replace the exact table by one least-cost candidate per geometric
      box while preserving zero coordinates exactly.
\item \textbf{Guarantee:} coordinate inflation is at most $\lambda$ per level and cost inflation at
      most $\lambda^2$ per level, giving $\lambda^{2H}\le1+\varepsilon$.
\item \textbf{Scope:} no distinct-ratio, balanced-tree, bounded-degree, or bounded-height assumption
      is imposed beyond $H\le n-1$ for a binary separating tree.
\end{itemize}
\end{proofcontract}

\setshortsectionhead{Ratio orders and separability}
\section{Ratio-order costs and separable permutations}
\label{CT:sec:orders}
Let
\[
 r_j=\frac{w_j^E}{p_j},
 \qquad
 s_j=\frac{w_j^L}{p_j}.
\]
Ratio comparisons are made exactly by cross multiplication.  Ties do not change any pair coefficient
in the canonical objective, but a permutation requires total orders.  We therefore fix one total order
$\prec_E$ refining nondecreasing $r_j$ and one total order $\prec_L$ refining nondecreasing $s_j$.
A deterministic job-index tie break is sufficient.

\begin{definition}[Ratio permutation and direct/skew sums]
\label{CT:def:orders}
List the jobs as
\[
 j_1\prec_E j_2\prec_E\cdots\prec_E j_n.
\]
For each position $k$, let $\pi(k)$ be the position of job $j_k$ in the order $\prec_L$.  The
permutation $\pi$ is the \emph{ratio permutation} of the chosen refinements.

Write $S_m$ for the set of permutations of $\{1,\ldots,m\}$.  For $\alpha\in S_p$ and
$\beta\in S_q$, their \emph{direct sum}
$\alpha\oplus\beta\in S_{p+q}$ is defined by
\[
 (\alpha\oplus\beta)(k)=
 \begin{cases}
  \alpha(k),&1\le k\le p,\\
  p+\beta(k-p),&p<k\le p+q,
 \end{cases}
\]
and their \emph{skew sum} $\alpha\ominus\beta\in S_{p+q}$ is defined by
\[
 (\alpha\ominus\beta)(k)=
 \begin{cases}
  q+\alpha(k),&1\le k\le p,\\
  \beta(k-p),&p<k\le p+q.
 \end{cases}
\]
Thus, in a direct sum, every entry of the left permutation is smaller than every entry of the right
permutation; in a skew sum, every entry of the left permutation is larger than every entry of the
right permutation.
\end{definition}

The ratio permutation depends on the chosen refinements when ratios tie.  This causes no objective
ambiguity: interchanging equal-ratio jobs leaves the relevant minimum coefficient unchanged.  The
results below apply to the fixed refinements actually used by the algorithm; they do not require every
possible tie refinement to have the same decomposition.

\begin{definition}[Separable permutation, separable instance, and separating tree]
\label{CT:def:compatible}
A \emph{separating tree} for a permutation $\pi\in S_n$ is a rooted plane binary tree whose leaves,
from left to right, are labeled
\[
 \pi(1),\pi(2),\ldots,\pi(n).
\]
For every tree node, the leaves below that node occupy consecutive leaf positions and their labels
form a consecutive set of integers.  Every internal node has one of two labels:
\begin{itemize}[leftmargin=2em,itemsep=1pt,topsep=2pt]
\item it is \emph{positive} or \emph{direct}, written $M=A\oplus B$, when every label below its left
      child $A$ is smaller than every label below its right child $B$;
\item it is \emph{negative} or \emph{skew}, written $M=A\ominus B$, when every label below its left
      child $A$ is larger than every label below its right child $B$.
\end{itemize}
A permutation is \emph{separable} when it has a separating tree.  Equivalently, it can be constructed
recursively from the singleton permutation $1$ by direct and skew sums.  An \AWET instance with fixed
refinements $(\prec_E,\prec_L)$ is \emph{separable} when its ratio permutation $\pi$ is separable.
\end{definition}

The leaves are the jobs in $\prec_E$ order, while their labels are the corresponding ranks in
$\prec_L$.  For example,
\[
 2143=(1\ominus1)\oplus(1\ominus1).
\]
Its separating tree has a positive root, two negative children, and leaves $2,1,4,3$ from left to
right:
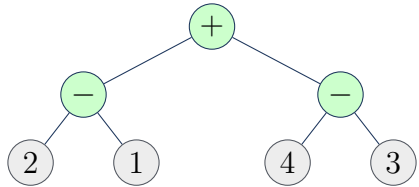
\begin{center}
\begin{tikzpicture}[
 level distance=9mm,
 level 1/.style={sibling distance=34mm},
 level 2/.style={sibling distance=14mm},
 inode/.style={draw=navy,circle,fill=figurealgorithm,minimum size=6mm,inner sep=1pt},
 lnode/.style={draw=auditgray,circle,fill=figuregray,minimum size=6mm,inner sep=1pt},
 edge from parent/.style={draw=navy}
]
\node[inode] {$+$}
 child {node[inode] {$-$}
   child {node[lnode] {$2$}}
   child {node[lnode] {$1$}}}
 child {node[inode] {$-$}
   child {node[lnode] {$4$}}
   child {node[lnode] {$3$}}};
\end{tikzpicture}
\end{center}
A separating tree need not be unique: associativity permits rotations among adjacent internal nodes
with the same sign.  Such rotations do not change the direct/skew relations used by the algorithm.

The strict antithetical ordering studied in Part~\ref{part:antithetical} has ratio permutation
$n(n-1)\cdots1$, an iterated skew sum of singleton permutations.  It is therefore separable, and the
FPTAS in \cref{CT:thm:main} applies to that class in addition to the exact pseudopolynomial algorithm
of Part~\ref{part:antithetical}.

Separable permutations can be recognized in linear time by a stack scan, and a separating tree can be
recovered from the same computation~\cite{BoseBussLubiw1998}.  This preprocessing is therefore no
bottleneck for the dynamic programs below.

For comparison with the graph terminology, the \emph{inversion graph} $G_\times$ joins two jobs when
their relative orders in $\prec_E$ and $\prec_L$ disagree.  A permutation is separable if and only if
its inversion graph is a cograph~\cite{BoseBussLubiw1998,CorneilPerlStewart1985}.  At a positive node,
the two child graphs are combined by disjoint union; at a negative node, they are combined by join.
Consequently, a suitably $\prec_E$-ordered binary refinement of the graph-theoretic cotree is
equivalent to a separating tree: the children are ordered by their left-to-right intervals in
$\prec_E$, disjoint-union nodes become positive/direct nodes, and join nodes become negative/skew
nodes.  The scheduling development below uses only the permutation and its separating tree.

\begin{definition}[Separating-tree block]
\label{CT:def:block}
For a separating-tree node $M$, let $J(M)$ be the set of leaf jobs below $M$.  We call $J(M)$ the
\emph{separating-tree block} of $M$.  It is a contiguous interval in $\prec_E$, and the consecutive
label condition in the separating tree makes it a contiguous interval in $\prec_L$ as well.
\end{definition}

\begin{finding}{Exact assumptions for the FPTAS}
The algorithm assumes only the standard positive-integer \AWET model with $d=\sum_jp_j$, fixed total
refinements of the two ratio orders, and separability of the resulting ratio permutation.  A
separating tree need not be supplied: it can be recognized and recovered in linear time.  Equal ratios
are allowed, and the binary separating tree may be completely unbalanced.  Separability is a
sufficient structural hypothesis for this algorithm, not a proved necessary condition for the
existence of an FPTAS by some other method.
\end{finding}

Put
\[
 P=\sum_{j\in\J}p_j,
 \qquad
 W_E=\sum_{j\in\J}w_j^E,
 \qquad
 W_L=\sum_{j\in\J}w_j^L.
\]

\setshortsectionhead{Four-coordinate signatures}
\section{Local side assignments and the four signature coordinates}
Fix a separating-tree node $M$.  A local side assignment is a subset $E_M\subseteq J(M)$ of jobs assigned to
the early/on-time side; its tardy complement is
\[
 T_M=J(M)\setminus E_M.
\]

\begin{definition}[Four-coordinate signature]
\label{CT:def:signature}
For the local side assignment $(E_M,T_M)$, define
\begin{align}
 e_M
 &=\sum_{j\in E_M}p_j,
 &&\text{early processing load},
 \label{CT:eq:eM}\\
 t_M
 &=\sum_{j\in T_M}p_j,
 &&\text{tardy processing load},
 \label{CT:eq:tM}\\
 u_M
 &=\sum_{j\in E_M}w_j^E,
 &&\text{earliness weight carried by early jobs},
 \label{CT:eq:uM}\\
 v_M
 &=\sum_{j\in T_M}w_j^L,
 &&\text{tardiness weight carried by tardy jobs}.
 \label{CT:eq:vM}
\end{align}
The signature is
\begin{equation}
 z_M=z_M(E_M):=(e_M,t_M,u_M,v_M)\in\mathbb Z_{\ge0}^4.
 \label{CT:eq:signature}
\end{equation}
Although $e_M+t_M=\sum_{j\in J(M)}p_j$, both processing coordinates are retained.  This makes every
coordinate nonnegative and additive, represents zero tardy load exactly, and avoids subtraction in
both the exact recurrence and the trimming proof.
\end{definition}

The entries $u_M$ and $v_M$ are not the total early and tardy weights of all jobs in the block.  They
are side-specific sums: $u_M$ includes $w_j^E$ only for jobs currently assigned early, whereas $v_M$
includes $w_j^L$ only for jobs currently assigned tardy.

For a local assignment $E_M$, define its \emph{internal canonical cost} by restricting the global
canonical objective to jobs and pairs lying wholly in $J(M)$:
\begin{align}
 \phi_M(E_M)
 :={}&
 \sum_{\{i,j\}\subseteq E_M}p_ip_j\min\{r_i,r_j\}
 \label{CT:eq:localcost-first}\\
 &+\sum_{\{i,j\}\subseteq T_M}p_ip_j\min\{s_i,s_j\}
 \notag\\
 &+\sum_{i\in T_M}p_iw_i^L.
 \label{CT:eq:localcost}
\end{align}
At a leaf containing job $j$, the two possible states are therefore
\begin{align}
 z_j^E
 &=(p_j,0,w_j^E,0),
 &\phi_j(\{j\})&=0,
 \label{CT:eq:leafE}\\
 z_j^{\mathrm{tar}}
 &=(0,p_j,0,w_j^L),
 &\phi_j(\varnothing)&=p_jw_j^L.
 \label{CT:eq:leafT}
\end{align}

\setshortsectionhead{Bilinear convolution}
\section{Cross costs and exact bilinear convolution}
Let $M$ have left child $A$ and right child $B$.  A \emph{cross pair for the split $(A,B)$} is an
unordered two-element set $\{i,j\}$ with $i\in J(A)$ and $j\in J(B)$.  The \emph{cross cost}
is the total canonical-objective contribution of those cross pairs $\{i,j\}$ for which $i$ and
$j$ are assigned to the same side.  For generic signatures
\[
 z_A=(e_A,t_A,u_A,v_A),
 \qquad
 z_B=(e_B,t_B,u_B,v_B),
\]
define
\begin{align}
 g_\oplus(z_A,z_B)
 &=u_Ae_B+v_At_B,
 \label{CT:eq:direct}\\
 g_\ominus(z_A,z_B)
 &=u_Ae_B+t_Av_B.
 \label{CT:eq:skew}
\end{align}
The first formula is used at a direct node and the second at a skew node.  The two cases are summarized below.

\begin{center}
\small
\begin{tabularx}{0.98\textwidth}{@{}p{0.12\textwidth}p{0.17\textwidth}p{0.18\textwidth}p{0.17\textwidth}X@{}}
\toprule
node & early cross relation & early term & tardy cross relation & tardy term and total\\
\midrule
direct $A\oplus B$ & $A\prec_E B$ & $u_Ae_B$ & $A\prec_L B$ & $v_At_B$; hence $g_\oplus=u_Ae_B+v_At_B$\\
skew $A\ominus B$ & $A\prec_E B$ & $u_Ae_B$ & $B\prec_L A$ & $t_Av_B$; hence $g_\ominus=u_Ae_B+t_Av_B$\\
\bottomrule
\end{tabularx}
\end{center}

\begin{lemma}[Exact nonnegative bilinear convolution]
\label{CT:lem:bilin}
Let $M=A\oplus B$ or $M=A\ominus B$.  For every side assignment on $J(M)$,
\begin{equation}
 z_M=z_A+z_B
 \label{CT:eq:additive-signature}
\end{equation}
and
\begin{equation}
 \phi_M(E_M)
 =\phi_A(E_A)+\phi_B(E_B)+g_\circ(z_A,z_B),
 \label{CT:eq:exact-convolution}
\end{equation}
where $\circ=\oplus$ at a direct node and $\circ=\ominus$ at a skew node.  Every term in
$g_\circ$ is nonnegative and is the product of one coordinate from each child.
\end{lemma}

\begin{proof}
Because $J(A)$ and $J(B)$ are disjoint, each signature coordinate splits into its two child sums.  For
example,
\begin{align*}
 e_M
 &=\sum_{j\in E_M}p_j\\
 &=\sum_{j\in E_A}p_j+\sum_{j\in E_B}p_j\\
 &=e_A+e_B.
\end{align*}
The same calculation gives $t_M=t_A+t_B$, $u_M=u_A+u_B$, and $v_M=v_A+v_B$, proving
\eqref{CT:eq:additive-signature}.

At either node type, every job of $A$ precedes every job of $B$ in the nondecreasing early-ratio order.
Thus $r_i\le r_j$ for $i\in J(A)$ and $j\in J(B)$.  The cross contribution of pairs that are both
early assigned is
\begin{align*}
 &\sum_{i\in E_A}\sum_{j\in E_B}p_ip_j\min\{r_i,r_j\}\\
 &\qquad=\sum_{i\in E_A}\sum_{j\in E_B}p_ip_jr_i\\
 &\qquad=\sum_{i\in E_A}\sum_{j\in E_B}w_i^Ep_j\\
 &\qquad=\left(\sum_{i\in E_A}w_i^E\right)
          \left(\sum_{j\in E_B}p_j\right)\\
 &\qquad=u_Ae_B.
\end{align*}

Suppose first that $M=A\oplus B$.  Then every job of $A$ also precedes every job of $B$ in the
nondecreasing tardiness-ratio order, so $s_i\le s_j$ for $i\in J(A)$ and $j\in J(B)$.  Hence the
cross contribution of pairs that are both tardy is
\begin{align*}
 &\sum_{i\in T_A}\sum_{j\in T_B}p_ip_j\min\{s_i,s_j\}\\
 &\qquad=\sum_{i\in T_A}\sum_{j\in T_B}p_ip_js_i\\
 &\qquad=\sum_{i\in T_A}\sum_{j\in T_B}w_i^Lp_j\\
 &\qquad=\left(\sum_{i\in T_A}w_i^L\right)
          \left(\sum_{j\in T_B}p_j\right)\\
 &\qquad=v_At_B.
\end{align*}
Adding the early and tardy cross contributions gives $g_\oplus(z_A,z_B)$.

Now suppose that $M=A\ominus B$.  Every job of $B$ precedes every job of $A$ in the nondecreasing
tardiness-ratio order, so $s_j\le s_i$ for $i\in J(A)$ and $j\in J(B)$.  The tardy cross contribution
is then
\begin{align*}
 &\sum_{i\in T_A}\sum_{j\in T_B}p_ip_j\min\{s_i,s_j\}\\
 &\qquad=\sum_{i\in T_A}\sum_{j\in T_B}p_ip_js_j\\
 &\qquad=\sum_{i\in T_A}\sum_{j\in T_B}p_iw_j^L\\
 &\qquad=\left(\sum_{i\in T_A}p_i\right)
          \left(\sum_{j\in T_B}w_j^L\right)\\
 &\qquad=t_Av_B.
\end{align*}
Adding the early cross term gives $g_\ominus(z_A,z_B)$.

Every tardy self-term belongs to a leaf, every within-child pair belongs to the corresponding child
cost, and every cross pair is charged at the unique lowest separating-tree node at which its two
jobs lie in different child blocks.  Thus no term is omitted or counted twice, and
\eqref{CT:eq:exact-convolution} follows.
\end{proof}

\setshortsectionhead{Exact dynamic program}
\section{The exact pseudopolynomial dynamic program}
\label{CT:sec:exactdp}
The exact recurrence is stated before trimming so that both its state space and its connection to the
original scheduling problem are explicit.

\begin{definition}[Exact separating-tree table]
\label{CT:def:exactdp}
For each node $M$ and integer vector $z\in\mathbb Z_{\ge0}^4$, define
\begin{equation}
 \mathcal D_M(z)
 :=\min\bigl\{\phi_M(E_M):E_M\subseteq J(M),\ z_M(E_M)=z\bigr\}.
 \label{CT:eq:exact-table}
\end{equation}
If no assignment has signature $z$, set $\mathcal D_M(z)=+\infty$.  With every finite entry, store a
backpointer to an assignment attaining it.
\end{definition}

At a leaf $j$, the only finite entries are
\begin{align}
 \mathcal D_j(p_j,0,w_j^E,0)
 &=0,
 \label{CT:eq:exact-leafE}\\
 \mathcal D_j(0,p_j,0,w_j^L)
 &=p_jw_j^L.
 \label{CT:eq:exact-leafT}
\end{align}
At an internal node $M=A\circ B$, where $\circ\in\{\oplus,\ominus\}$, use
\begin{align}
 \mathcal D_M(z)
 =\min\bigl\{
 &\mathcal D_A(z_A)+\mathcal D_B(z_B)+g_\circ(z_A,z_B):
 \label{CT:eq:exact-DP-first}\\
 &z_A,z_B\in\mathbb Z_{\ge0}^4,
 \quad z_A+z_B=z,
 \quad \mathcal D_A(z_A),\mathcal D_B(z_B)<+\infty
 \bigr\}.
 \label{CT:eq:exact-DP}
\end{align}

\begin{theorem}[Exact separating-tree dynamic program]
\label{CT:thm:exactdp}
For every separating-tree node $M$ and every signature $z$,
\begin{equation}
 \mathcal D_M(z)
 =\min\bigl\{\phi_M(E_M):E_M\subseteq J(M),\ z_M(E_M)=z\bigr\},
 \label{CT:eq:exact-induction}
\end{equation}
where the minimum of the empty set is $+\infty$.  Whenever $\mathcal D_M(z)<+\infty$, the stored
backpointers reconstruct an assignment attaining this value.  If $R$ is the root, then the optimum
value of the original scheduling instance is
\begin{equation}
 \OPT
 =\min\bigl\{\mathcal D_R(z):z=(e,t,u,v)\in\mathbb Z_{\ge0}^4,\ e>0\bigr\}.
 \label{CT:eq:root-opt}
\end{equation}
Backtracking from a minimizing root state gives an optimal early set and hence an optimal canonical
schedule.  The recurrence is pseudopolynomial.
\end{theorem}

\begin{proof}
We prove the following exact induction assertion $\mathcal I(M)$ for every separating-tree node $M$:
for every signature $z$, equality \eqref{CT:eq:exact-induction} holds, and every finite table entry has
a stored assignment attaining it.

At a leaf, \eqref{CT:eq:exact-leafE} and \eqref{CT:eq:exact-leafT} list the only two side assignments,
so $\mathcal I(M)$ holds.

Let $M=A\circ B$ be internal and assume $\mathcal I(A)$ and $\mathcal I(B)$.  Fix an arbitrary
assignment $E_M$.  It restricts uniquely to
\[
 E_A=E_M\cap J(A),
 \qquad
 E_B=E_M\cap J(B).
\]
Write $z_A=z_A(E_A)$ and $z_B=z_B(E_B)$.  By signature additivity,
\[
 z_M(E_M)=z_A+z_B.
\]
By the child induction assertions and \cref{CT:lem:bilin},
\begin{align*}
 \phi_M(E_M)
 &=\phi_A(E_A)+\phi_B(E_B)+g_\circ(z_A,z_B)\\
 &\ge \mathcal D_A(z_A)+\mathcal D_B(z_B)+g_\circ(z_A,z_B)\\
 &\ge \mathcal D_M(z_M(E_M)).
\end{align*}
The first inequality can be strict: the restrictions of an arbitrary parent assignment need not be
minimum-cost child assignments among all assignments having signatures $z_A$ and $z_B$.  Thus the
recurrence value is no larger than the cost of any assignment with the requested parent signature.

Conversely, let $z$ be any signature with $\mathcal D_M(z)<+\infty$, and take a pair
$(z_A,z_B)$ attaining the minimum in \eqref{CT:eq:exact-DP}.  By $\mathcal I(A)$ and
$\mathcal I(B)$, the stored child backpointers yield assignments $E_A$ and $E_B$
of costs $\mathcal D_A(z_A)$ and $\mathcal D_B(z_B)$.  Their union
\[
 E_M=E_A\cup E_B
\]
has signature $z_A+z_B=z$, and \cref{CT:lem:bilin} gives
\begin{align*}
 \phi_M(E_M)
 &=\mathcal D_A(z_A)+\mathcal D_B(z_B)+g_\circ(z_A,z_B)\\
 &=\mathcal D_M(z).
\end{align*}
Therefore the recurrence value is attainable and is exactly optimal for that signature.  For any
signature $z'$ with $\mathcal D_M(z')=+\infty$, the first half shows that no assignment can have
signature $z'$, because any such assignment would generate a finite recurrence candidate.  Storing
the two child backpointers proves the attainability part of $\mathcal I(M)$.

At the root, positive processing times imply
\[
 e_R>0
 \quad\Longleftrightarrow\quad
 E_R\ne\varnothing.
\]
By \cref{prop:setobjective}, every nonempty $E_R$ defines a feasible due-date-aligned canonical
schedule of cost $\phi_R(E_R)$, and the global optimum is the minimum of this expression over all
nonempty early sets.  Hence minimizing the exact root table over $e_R>0$ proves
\eqref{CT:eq:root-opt}.  Backpointers recover the side assignment.  Scheduling the early jobs in
nondecreasing $r_j$ order so that the early block ends at $d$, and scheduling the tardy jobs in reverse
$\prec_L$ order (equivalently, nonincreasing $s_j$ order) starting at $d$, then recovers an optimal
schedule for the original problem.

It remains to bound the exact state space.  Every reachable coordinate satisfies
\[
 0\le e,t\le P,
 \qquad
 0\le u\le W_E,
 \qquad
 0\le v\le W_L.
\]
Consequently, every node has at most
\begin{equation}
 S:=(P+1)^2(W_E+1)(W_L+1)
 \label{CT:eq:exact-state-bound}
\end{equation}
finite signatures.  A direct implementation scans at most $S^2$ child-state pairs at each of
$O(n)$ internal nodes, for at most $O(nS^2)$ exact-arithmetic operations.  All table costs and
backpointers have polynomial binary encoding length.  The running time is therefore polynomial in
$n,P,W_E,W_L$ and hence pseudopolynomial in the original binary input.
\end{proof}

\begin{howto}
The exact recurrence is already a self-contained algorithm for numerically small instances.  Construct the
ordered separating tree, fill the two leaf states, combine child tables by \eqref{CT:eq:exact-DP}, choose the
least root entry with $e>0$, and follow backpointers to recover the early/tardy assignment and its
canonical schedule.
\end{howto}

\setshortsectionhead{Geometric trimming}
\section{Geometric trimming of the exact recurrence}
\label{CT:sec:trimming}
A one-job instance is solved directly.  Hence assume the binary separating tree has root height $H\ge1$, where
leaves have height $0$.  Fix a rational accuracy parameter $0<\varepsilon\le1$ and put
\begin{equation}
 \lambda=1+\frac{\varepsilon}{4H}.
 \label{CT:eq:lambda}
\end{equation}

\begin{definition}[Geometric buckets and boxes]
\label{CT:def:boxes}
Zero has its own bucket, denoted $\bot$.  For an integer $x>0$, define
\[
 \beta_\lambda(x)=\floor{\log_\lambda x}.
\]
For a signature $z=(e,t,u,v)$, its box is
\begin{equation}
 \operatorname{box}_\lambda(z)
 =\bigl(\beta_\lambda(e),\beta_\lambda(t),
        \beta_\lambda(u),\beta_\lambda(v)\bigr),
 \label{CT:eq:box}
\end{equation}
where $\beta_\lambda(0)=\bot$.
\end{definition}

The bucket index need not be computed with floating-point logarithms.  Because $\varepsilon$ is
rational, write $\lambda=a/b>1$ in lowest terms; the integers $a$ and $b$ have polynomial binary
length.  For an integer coordinate $x>0$, the index $q=\beta_\lambda(x)$ is characterized exactly by
\[
 a^q\le xb^q,
 \qquad
 a^{q+1}>xb^{q+1}.
\]
A doubling search followed by binary search finds $q$ using exact integer comparisons, with powers
computed by repeated squaring.  For a coordinate bounded by $U$, the running-time estimate below gives
$q=O((H/\varepsilon)\log U)$.  The integers compared have bit length
$O(q(\log a+\log b)+\log U)$, which is polynomial in the binary input length and $1/\varepsilon$.
Thus box membership is computable in the ordinary Turing model without real-number arithmetic.

\paragraph{Trimmed recurrence.}
At a leaf, retain both exact states.  At an internal node $M=A\circ B$, generate one parent candidate
from every pair of retained child candidates.  Compute its signature and exact cost by
\begin{align*}
 z_M&=z_A+z_B,\\
 c_M&=c_A+c_B+g_\circ(z_A,z_B).
\end{align*}
Group the generated candidates by $\operatorname{box}_\lambda(z_M)$.  In every nonempty box, retain a
candidate of minimum exact cost, together with its child backpointers.  At the root, return the
minimum-cost retained candidate with $e>0$ and reconstruct its schedule.

\begin{lemma}[Same-box replacement]
\label{CT:lem:samebox}
Let candidate $(z',c')$ be discarded in favor of retained candidate $(\widehat z,\widehat c)$ from the
same box.  Then
\begin{align}
 \widehat c&\le c',
 \label{CT:eq:samebox-cost}\\
 \widehat z_\ell&=0
 \quad\Longleftrightarrow\quad
 z'_\ell=0,
 \label{CT:eq:samebox-zero}\\
 \widehat z_\ell&\le\lambda z'_\ell
 \qquad\text{for every coordinate }\ell.
 \label{CT:eq:samebox-coordinate}
\end{align}
\end{lemma}

\begin{proof}
The retained candidate minimizes exact cost in the box, proving \eqref{CT:eq:samebox-cost}.  The zero
bucket contains only zero, proving \eqref{CT:eq:samebox-zero}.  If a coordinate is positive and has
bucket index $q$, both coordinate values lie in
\[
 [\lambda^q,\lambda^{q+1}),
\]
so the retained value is strictly smaller than $\lambda$ times the discarded value.  This proves
\eqref{CT:eq:samebox-coordinate}.
\end{proof}

\begin{lemma}[Representative invariant]
\label{CT:lem:rep}
Let $M$ have height $h(M)$.  For every exact state $(z,c)$ represented in $\mathcal D_M$, the trimmed
table contains a retained state $(\widehat z,\widehat c)$ such that
\begin{align}
 \widehat z_\ell=0
 &\quad\Longleftrightarrow\quad
 z_\ell=0,
 \label{CT:eq:rep-zero}\\
 \widehat z_\ell
 &\le\lambda^{h(M)}z_\ell
 \qquad\text{for every coordinate }\ell,
 \label{CT:eq:rep-coordinate}\\
 \widehat c
 &\le\lambda^{2h(M)}c.
 \label{CT:eq:rep-cost}
\end{align}
\end{lemma}

\begin{proof}
The induction controls two quantities in the order needed by the recurrence.  First it bounds coordinate inflation before the parent is trimmed; then nonnegative bilinearity converts those coordinate bounds into a cost bound.  The singleton zero bucket is tracked separately so that root feasibility is never lost.

We use induction on $h(M)$.  A leaf has height $0$, and both leaf states are retained exactly, so all
three claims hold with equality.

Let $M=A\circ B$ have height $h\ge1$, and fix a finite exact entry $(z,c)$ of
$\mathcal D_M$, so $c=\mathcal D_M(z)$.  Choose child signatures $z_A,z_B$ attaining the minimum in
\eqref{CT:eq:exact-DP}, and put
\[
 c_A=\mathcal D_A(z_A),
 \qquad
 c_B=\mathcal D_B(z_B).
\]
Then
\[
 z=z_A+z_B,
 \qquad
 c=c_A+c_B+g_\circ(z_A,z_B).
\]
By induction, choose retained child representatives $(\widehat z_A,\widehat c_A)$ and
$(\widehat z_B,\widehat c_B)$.  Before trimming at $M$, combine them into
\begin{align}
 \widetilde z
 &=\widehat z_A+\widehat z_B,
 \label{CT:eq:tilde-z}\\
 \widetilde c
 &=\widehat c_A+\widehat c_B
   +g_\circ(\widehat z_A,\widehat z_B).
 \label{CT:eq:tilde-c}
\end{align}

Because $h(A),h(B)\le h-1$, coordinate additivity gives, for every coordinate $\ell$,
\begin{align*}
 \widetilde z_\ell
 &=\widehat z_{A,\ell}+\widehat z_{B,\ell}\\
 &\le\lambda^{h(A)}z_{A,\ell}
      +\lambda^{h(B)}z_{B,\ell}\\
 &\le\lambda^{h-1}\bigl(z_{A,\ell}+z_{B,\ell}\bigr)\\
 &=\lambda^{h-1}z_\ell.
\end{align*}

Every term in $g_\circ$ is a product of one nonnegative child coordinate from $A$ and one from
$B$.  For any such term,
\begin{align*}
 \widehat z_{A,q}\widehat z_{B,r}
 &\le\lambda^{h(A)}z_{A,q}\,\lambda^{h(B)}z_{B,r}\\
 &=\lambda^{h(A)+h(B)}z_{A,q}z_{B,r}\\
 &\le\lambda^{2(h-1)}z_{A,q}z_{B,r}.
\end{align*}
All three quantities $c_A$, $c_B$, and $g_\circ(z_A,z_B)$ are nonnegative.  Summing the
termwise bounds and using the child cost bounds therefore yields
\begin{align*}
 \widetilde c
 &=\widehat c_A+\widehat c_B
   +g_\circ(\widehat z_A,\widehat z_B)\\
 &\le\lambda^{2(h-1)}c_A
      +\lambda^{2(h-1)}c_B
      +\lambda^{2(h-1)}g_\circ(z_A,z_B)\\
 &=\lambda^{2(h-1)}
   \bigl(c_A+c_B+g_\circ(z_A,z_B)\bigr)\\
 &=\lambda^{2(h-1)}c.
\end{align*}

Let $(\widehat z,\widehat c)$ be the state retained in the box containing
$(\widetilde z,\widetilde c)$.  By \cref{CT:lem:samebox},
\begin{align*}
 \widehat z_\ell
 &\le\lambda\widetilde z_\ell\\
 &\le\lambda^h z_\ell,
\end{align*}
and
\begin{align*}
 \widehat c
 &\le\widetilde c\\
 &\le\lambda^{2(h-1)}c\\
 &\le\lambda^{2h}c.
\end{align*}
A coordinate sum is zero exactly when both child coordinates are zero.  The induction hypothesis
preserves those child zero patterns, and the singleton zero bucket preserves the parent zero pattern.
This proves \eqref{CT:eq:rep-zero}--\eqref{CT:eq:rep-cost}.
\end{proof}

\begin{theorem}[FPTAS for separable \AWET]
\label{CT:thm:main}
Fix total orders $\prec_E$ and $\prec_L$ refining nondecreasing early and tardiness ratios.  If the
resulting ratio permutation $\pi$ is separable, equivalently if a separating tree exists, then for
every rational $\varepsilon>0$ the trimmed separating-tree algorithm returns a feasible schedule of
cost at most
\[
 (1+\varepsilon)\OPT.
\]
Its running time is polynomial in the binary input length and $1/\varepsilon$.
\end{theorem}

\begin{proof}
If $\varepsilon>1$, run the construction with accuracy parameter $1$.  Its factor $2$ is at most
$1+\varepsilon$.  Hence it remains to consider $0<\varepsilon\le1$.

Let $R$ be the root, of height $H$, and let $(z^*,\OPT)$ be an optimal exact root state from
\cref{CT:thm:exactdp}.  Because the early set is nonempty, the first coordinate of $z^*$ is positive.
By \cref{CT:lem:rep}, the trimmed root table contains a representative
$(\widehat z,\widehat c)$ with the same zero pattern.  Its first coordinate is therefore also positive,
so the representative is feasible.  Its cost satisfies
\begin{align*}
 \widehat c
 &\le\lambda^{2H}\OPT\\
 &=\left(1+\frac{\varepsilon}{4H}\right)^{2H}\OPT\\
 &\le \exp(\varepsilon/2)\OPT\\
 &\le(1+\varepsilon)\OPT.
\end{align*}
For completeness, the last inequality follows because
\[
 f(\varepsilon)=\log(1+\varepsilon)-\frac{\varepsilon}{2}
\]
has $f(0)=0$ and
\[
 f'(\varepsilon)
 =\frac{1}{1+\varepsilon}-\frac12
 =\frac{1-\varepsilon}{2(1+\varepsilon)}
 \ge0
 \qquad(0\le\varepsilon\le1).
\]
The algorithm returns the cheapest feasible retained root state, so its cost is no larger than
$\widehat c$.

We next prove the fully polynomial bound.  For a positive coordinate bounded by $U$, the number of
positive buckets is at most $1+\ceil{\log_\lambda U}$; one additional bucket stores zero.  Define
\begin{align*}
 B_P
 &=2+\ceil{\log_\lambda\max\{1,P\}},\\
 B_E
 &=2+\ceil{\log_\lambda\max\{1,W_E\}},\\
 B_L
 &=2+\ceil{\log_\lambda\max\{1,W_L\}}.
\end{align*}
Every node retains at most
\begin{equation}
 B:=B_P^2B_EB_L
 \label{CT:eq:trimmed-state-bound}
\end{equation}
boxes.  Put $x=\varepsilon/(4H)$.  Since $0<x\le1/4$,
\[
 \log(1+x)\ge\frac{x}{2}.
\]
Consequently, for $U\ge1$,
\begin{align*}
 \log_\lambda U
 &=\frac{\log U}{\log(1+x)}\\
 &\le\frac{2\log U}{x}\\
 &=\frac{8H}{\varepsilon}\log U.
\end{align*}
Because $H\le n-1$, the bound $B$ is polynomial in $n$, $1/\varepsilon$, and the binary lengths of
$P,W_E,W_L$.

The ratio orders are obtained by exact sorting.  From the resulting permutation, separability can be
recognized and an ordered separating tree recovered in linear time~\cite{BoseBussLubiw1998}.
After binarization there are $O(n)$ internal nodes.  Each node examines at most $B^2$ pairs of retained
child states and performs constant-dimensional exact arithmetic on polynomial-bit integers and
rationals.  Exact bucket comparisons are polynomial as noted after \cref{CT:def:boxes}.  Therefore the
total running time is polynomial in the input length and $1/\varepsilon$.  Backpointers and the fixed
canonical orders reconstruct the returned schedule.  The one-job case was handled directly.
\end{proof}

\begin{sectionexit}
The exact separating-tree recurrence has now been converted into an FPTAS: every exact state has a retained representative with the same zero pattern, coordinate inflation at most $\lambda^h$, and cost inflation at most $\lambda^{2h}$; the root choice then gives $\lambda^{2H}\le1+\varepsilon$.  The worked example uses only the two cross forms, the same-box replacement rule, and root backtracking to show where one locally cheaper representative loses globally.
\end{sectionexit}

\setshortsectionhead{Worked example}
\section{Worked example}
This example uses six jobs, exactly four distinct processing times, an ordered separating tree containing both direct and skew nodes, and genuine geometric trimming at an internal node.  Most importantly, the trimmed dynamic program does \emph{not} return an exact optimum for the displayed accuracy parameter.  The example therefore shows both why trimming is useful and how its controlled information loss can affect the final schedule.

\setshortsubsectionhead{Instance and orders}
\subsection{Instance and ratio orders}

Let the six jobs have the following data.

\begin{center}
\begin{tabular}{c r r r r r}
\toprule
job $j$ & $p_j$ & $w_j^E$ & $w_j^L$ & $r_j=w_j^E/p_j$ & $s_j=w_j^L/p_j$\\
\midrule
1 & 1 & 140 & 315  & 140 & 315\\
2 & 1 & 141 & 321  & 141 & 321\\
3 & 1 & 145 & 322  & 145 & 322\\
4 & 2 & 292 & 200  & 146 & 100\\
5 & 3 & 441 & 600  & 147 & 200\\
6 & 4 & 592 & 1200 & 148 & 300\\
\bottomrule
\end{tabular}
\end{center}

Thus
\[
(p_1,\ldots,p_6)=(1,1,1,2,3,4),
\]
so the instance has exactly four distinct processing-time values, namely $1,2,3,4$.  The nonrestrictive common due date is
\[
d=P=1+1+1+2+3+4=12.
\]

The nondecreasing early-ratio order is
\[
1\prec_E2\prec_E3\prec_E4\prec_E5\prec_E6,
\]
whereas the nondecreasing tardiness-ratio order is
\[
4\prec_L5\prec_L6\prec_L1\prec_L2\prec_L3.
\]
Therefore the ratio permutation, written in the $\prec_E$ indexing, is
\[
\pi=(4,5,6,1,2,3)=123\ominus123.
\]
The displayed skew-sum decomposition proves that the permutation is separable.  Equivalently, its
inversion graph is the cograph $K_{3,3}$.  Use the ordered binary separating tree
\[
A=(1\oplus2)\oplus3,
\qquad
B=(4\oplus5)\oplus6,
\qquad
R=A\ominus B.
\]
Here and in the figure, the leaf labels $1,\ldots,6$ are job identifiers.  Under the formal convention
of \cref{CT:def:compatible}, their permutation labels---the corresponding $\prec_L$ ranks read from
left to right in $\prec_E$ order---are $4,5,6,1,2,3$.  The separating-tree height is $H=3$.

\begin{figure}[htbp]
\centering
\begin{tikzpicture}[
  level 1/.style={sibling distance=72mm,level distance=12mm},
  level 2/.style={sibling distance=34mm,level distance=12mm},
  level 3/.style={sibling distance=18mm,level distance=12mm},
  nd/.style={draw=navy,rounded corners,fill=figurealgorithm,minimum width=11mm,minimum height=7mm,align=center},
  leaf/.style={draw=auditgray,circle,fill=figuregray,minimum size=6mm,inner sep=1pt},
  note/.style={draw=checkpointamber,rounded corners,fill=figureopen,align=left,text width=5.2cm,inner sep=5pt},
  edge from parent/.style={draw=navy,thick}
]
\node[nd] (R) {$R:\ominus$}
 child {node[nd] (A) {$A:\oplus$}
   child {node[nd] (A12) {$\oplus$}
     child {node[leaf] {1}}
     child {node[leaf] {2}}}
   child {node[leaf] {3}}}
 child {node[nd] (B) {$B:\oplus$}
   child {node[nd] {$\oplus$}
     child {node[leaf] {4}}
     child {node[leaf] {5}}}
   child {node[leaf] {6}}};
\node[note,below=24mm of A12,text width=10.2cm,align=center] (trim) {At block $A$, states $E_A=\{2\}$ and $E_A=\{3\}$ enter the same box.  Local trimming keeps $\{3\}$ because $951<952$, but the global optimum requires the discarded state $\{2\}$.};
\draw[-{Latex[length=2mm]},thick,draw=checkpointamber] (trim.north) -- (A.south west);
\end{tikzpicture}
\caption{Job-labeled ordered separating tree for the six-job example.  Direct nodes form the two three-job blocks, the root is skew, and the annotation marks the left-block trimming collision that removes the globally optimal path.}
\label{CT:fig:worked-cotree}
\end{figure}
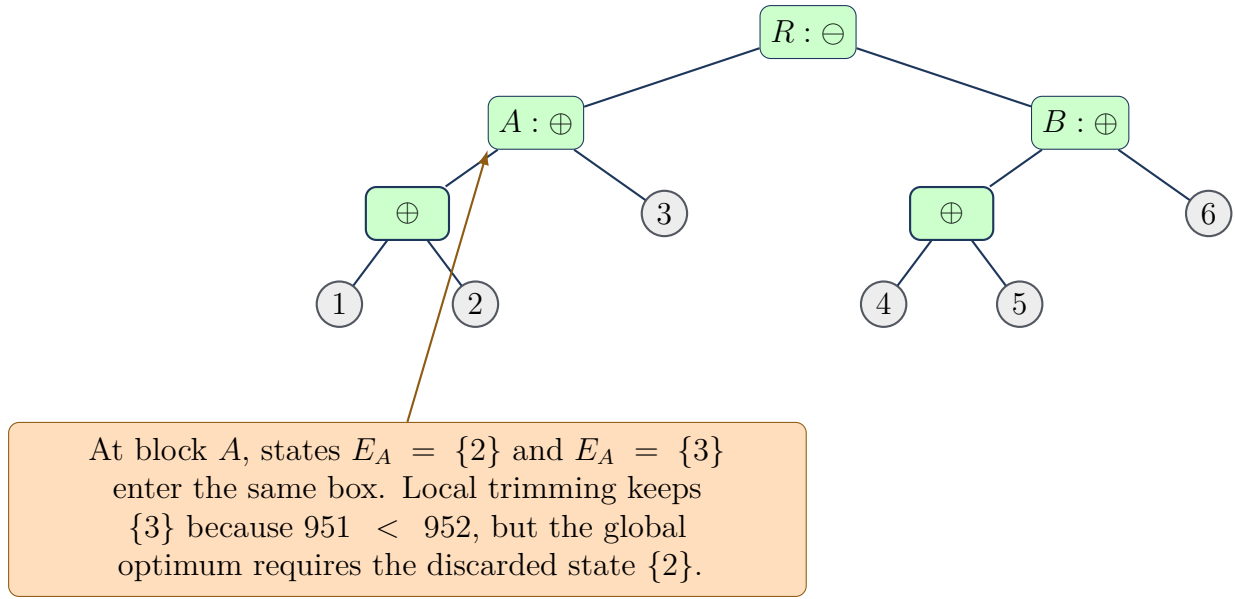

\setshortsubsectionhead{Trimming parameter}
\subsection{Trimming parameter}

All calculations below use the four-coordinate signature from \cref{CT:def:signature} and the two cross-cost formulas in \eqref{CT:eq:direct}--\eqref{CT:eq:skew}.  Every required coordinate and cross term is instantiated where it is used.

Choose
\[
\varepsilon=0.50.
\]
Since $H=3$,
\[
\lambda
=1+\frac{\varepsilon}{4H}
=1+\frac{0.50}{12}
=\frac{25}{24}
\approx1.0416667.
\]
For a positive integer $x$, put
\[
\beta(x)=\left\lfloor\log_{25/24}x\right\rfloor,
\]
and give zero its own bucket $\bot$.  A signature is assigned the four-dimensional box
\[
\operatorname{box}(e,t,u,v)=(\beta(e),\beta(t),\beta(u),\beta(v)).
\]
The bucket values used in the important left-subtree calculation are
\[
\begin{array}{c|rrrrrrrrrrrrrrrr}
x&1&2&3&140&141&145&281&285&286&315&321&322&426&636&637&643\\ \hline
\beta(x)&0&16&26&121&121&121&138&138&138&140&141&141&148&158&158&158
\end{array}
\]
and $\beta(958)=168$.

\setshortsubsectionhead{Left block A}
\subsection{The left block \texorpdfstring{$A$}{A}: trimming actually deletes states}

First combine jobs 1 and 2 at the direct node $1\oplus2$.  The four exact candidates are

\begin{center}
\begin{tabular}{c c c c}
\toprule
early set in $\{1,2\}$ & signature $(e,t,u,v)$ & exact cost & box\\
\midrule
$\varnothing$ & $(0,2,0,636)$ & $951$ & $(\bot,16,\bot,158)$\\
$\{1\}$       & $(1,1,140,321)$ & $321$ & $(0,0,121,141)$\\
$\{2\}$       & $(1,1,141,315)$ & $315$ & $(0,0,121,140)$\\
$\{1,2\}$     & $(2,0,281,0)$   & $140$ & $(16,\bot,138,\bot)$\\
\bottomrule
\end{tabular}
\end{center}

All four boxes are distinct, so no trimming occurs yet.

Now combine this table with job 3 at the direct node $A=(1\oplus2)\oplus3$.  Before trimming there are eight exact candidates:

\begin{center}
\begin{tabular}{c c r c}
\toprule
early set $E_A$ & signature $(e_A,t_A,u_A,v_A)$ & cost $c_A$ & box\\
\midrule
$\varnothing$ & $(0,3,0,958)$   & 1909 & $(\bot,26,\bot,168)$\\
$\{1\}$       & $(1,2,140,643)$ & 964  & $(0,16,121,158)$\\
$\{2\}$       & $(1,2,141,637)$ & 952  & $(0,16,121,158)$\\
$\{3\}$       & $(1,2,145,636)$ & 951  & $(0,16,121,158)$\\
$\{1,2\}$     & $(2,1,281,322)$ & 462  & $(16,0,138,141)$\\
$\{1,3\}$     & $(2,1,285,321)$ & 461  & $(16,0,138,141)$\\
$\{2,3\}$     & $(2,1,286,315)$ & 456  & $(16,0,138,140)$\\
$\{1,2,3\}$   & $(3,0,426,0)$   & 421  & $(26,\bot,148,\bot)$\\
\bottomrule
\end{tabular}
\end{center}

This is the first genuinely important trimming step.  The three single-early candidates
\[
E_A=\{1\},\qquad E_A=\{2\},\qquad E_A=\{3\}
\]
all lie in the same box $(0,16,121,158)$.  Their exact local costs are $964,952,951$, respectively.  The trimmed recurrence therefore keeps only
\[
E_A=\{3\},
\qquad
z_A=(1,2,145,636),
\qquad
c_A=951,
\]
and discards the candidate
\[
E_A=\{2\},
\qquad
z_A=(1,2,141,637),
\qquad
c_A=952.
\]
Notice the tradeoff: the retained state is locally cheaper by only $1$, but its early-weight coordinate is larger by $4$.

A second collision occurs between $E_A=\{1,2\}$ and $E_A=\{1,3\}$.  Both lie in box $(16,0,138,141)$, so the latter, with cost $461$, replaces the former, with cost $462$.  Thus the eight exact candidates at $A$ are compressed to five retained candidates.

The first collision is the one that will matter globally: the discarded state $E_A=\{2\}$ is part of the exact optimum of the six-job instance.

\setshortsubsectionhead{Right block B}
\subsection{The right block \texorpdfstring{$B$}{B}}

For $B=(4\oplus5)\oplus6$, all eight exact side assignments survive trimming for this value of $\lambda$.  The state that will be paired with the best left-side candidates is
\[
E_B=\{5,6\},
\qquad
T_B=\{4\}.
\]
Its signature is
\[
z_B=(7,2,1033,200),
\]
because
\[
e_B=3+4=7,
\qquad
t_B=2,
\qquad
u_B=441+592=1033,
\qquad
v_B=200.
\]
Its exact internal cost is
\[
\begin{aligned}
c_B
&=p_5p_6\min\{r_5,r_6\}+p_4w_4^L\\
&=(3)(4)(147)+(2)(200)\\
&=1764+400\\
&=2164.
\end{aligned}
\]

\setshortsubsectionhead{Root convolution}
\subsection{Root convolution and the returned solution}

All jobs of $A=\{1,2,3\}$ precede all jobs of $B=\{4,5,6\}$ in $\prec_E$, whereas all jobs of $B$ precede all jobs of $A$ in $\prec_L$.  Thus the root relation is $R=A\ominus B$, so the root is skew and
\[
g_{\ominus}(z_A,z_B)=u_Ae_B+t_Av_B.
\]
Using the retained left state $E_A=\{3\}$ and the right state $E_B=\{5,6\}$ gives
\[
\begin{aligned}
g_{\ominus}
&=(145)(7)+(2)(200)\\
&=1015+400\\
&=1415.
\end{aligned}
\]
Hence the corresponding root cost is
\[
\begin{aligned}
c_R
&=c_A+c_B+g_{\ominus}(z_A,z_B)\\
&=951+2164+1415\\
&=4530.
\end{aligned}
\]
The resulting early set is therefore
\[
\widehat E=\{3,5,6\}.
\]

The five retained states at $A$ and the eight retained states at $B$ generate $5\cdot8=40$ root candidates.  Root trimming compresses these to $35$ boxes.  The least-cost retained feasible root candidate is the one above, of cost $4530$, so the trimmed dynamic program returns $\widehat E=\{3,5,6\}$.

\setshortsubsectionhead{Trimming loss}
\subsection{What was lost by trimming}

Now reconsider the discarded left state
\[
E_A^*=\{2\},
\qquad
z_A^*=(1,2,141,637),
\qquad
c_A^*=952.
\]
Combine it with the \emph{same} right state $E_B=\{5,6\}$.  The skew cross cost would have been
\[
\begin{aligned}
g_{\ominus}(z_A^*,z_B)
&=(141)(7)+(2)(200)\\
&=987+400\\
&=1387.
\end{aligned}
\]
Thus this discarded path has total cost
\[
\begin{aligned}
c_R^*
&=952+2164+1387\\
&=4503.
\end{aligned}
\]
This calculation isolates the trimming effect exactly.  Replacing $E_A^*=\{2\}$ by $\{3\}$ saves
\[
952-951=1
\]
in local cost, but increases the root cross term by
\[
(145-141)e_B=4\cdot7=28.
\]
The net loss is therefore
\[
28-1=27,
\]
which is exactly
\[
4530-4503=27.
\]

\setshortsubsectionhead{Exact optimum}
\subsection{Exact optimum for comparison}

A direct enumeration of all $2^6-1=63$ nonempty early sets gives the five smallest exact objective values

\begin{center}
\begin{tabular}{c r}
\toprule
early set & exact cost\\
\midrule
$\{2,5,6\}$ & 4503\\
$\{1,5,6\}$ & 4508\\
$\{3,5,6\}$ & 4530\\
$\{5,6\}$   & 4673\\
$\{1,2,5,6\}$ & 4793\\
\bottomrule
\end{tabular}
\end{center}

Hence the exact optimum is unique and equals
\[
\OPT=4503,
\qquad
E^*=\{2,5,6\}.
\]
The trimmed algorithm returns
\[
\widehat\Phi=4530>4503=\OPT.
\]
Thus this six-job example is genuinely approximate rather than accidentally exact.

\setshortsubsectionhead{Direct cost checks}
\subsection{Canonical schedules and direct cost checks}

The common due date is $d=12$.

For the exact optimum $E^*=\{2,5,6\}$, the early processing load is $1+3+4=8$, so the early block starts at time $12-8=4$.  The early jobs are sequenced in nondecreasing $r_j$ order:
\[
2,5,6.
\]
Their completion times are $5,8,12$.  The tardy jobs are $\{1,3,4\}$ and are sequenced in nonincreasing $s_j$ order:
\[
3,1,4,
\]
with completion times $13,14,16$.  Therefore
\[
\begin{aligned}
\Phi(E^*)
&=141(12-5)+441(12-8)+592(12-12)\\
&\quad +322(13-12)+315(14-12)+200(16-12)\\
&=987+1764+0+322+630+800\\
&=4503.
\end{aligned}
\]

For the returned set $\widehat E=\{3,5,6\}$, the early block also has load $8$ and starts at time $4$.  The canonical early order is
\[
3,5,6,
\]
with completion times $5,8,12$.  The tardy jobs $\{1,2,4\}$ are ordered
\[
2,1,4,
\]
with completion times $13,14,16$.  Thus
\[
\begin{aligned}
\Phi(\widehat E)
&=145(12-5)+441(12-8)+592(12-12)\\
&\quad +321(13-12)+315(14-12)+200(16-12)\\
&=1015+1764+0+321+630+800\\
&=4530.
\end{aligned}
\]
This independently confirms the root-DP calculations.

\setshortsubsectionhead{Approximation audit}
\subsection{Approximation-factor audit and effect of tighter trimming}

For the displayed choice $\varepsilon=0.50$ and $H=3$,
\[
\lambda^{2H}
=\left(\frac{25}{24}\right)^6
\approx1.2775344
<1.50=1+\varepsilon.
\]
The theorem therefore guarantees a returned cost at most $1.5\OPT$.  The actual ratio in this example is much better:
\[
\frac{4530}{4503}
\approx1.005996.
\]

The example also illustrates how the accuracy parameter controls information loss.  If the same instance is run with
\[
\varepsilon=0.40,
\qquad
\lambda=1+\frac{0.40}{12}=\frac{31}{30},
\]
the important single-early states in block $A$ no longer collapse in the same way; the exact-optimal state $E_A=\{2\}$ survives, and the trimmed recurrence returns the exact optimum $4503$.  Thus the $\varepsilon=0.50$ run visibly demonstrates trimming loss, whereas a modestly tighter geometric grid restores the optimum on this instance.

\setshortsubsectionhead{Example lessons}
\subsection{What this example demonstrates}

This six-job instance separates the essential ideas of the separating-tree FPTAS.  The exact recurrence first generates distinct interface states.  Geometric trimming then identifies states whose four coordinates are close enough to occupy one box and keeps only the cheapest local representative.  Here that local choice is not globally exact: the state $E_A=\{3\}$ is cheaper inside block $A$ than $E_A=\{2\}$ by $1$, but its larger early-weight coordinate creates $28$ additional units of skew-root interaction.  The final error is therefore $27$, while remaining far inside the proved approximation bound.

The example also shows why the four-coordinate signature matters.  The trimmed recurrence is not merely remembering the number or processing load of early jobs.  It must retain enough information about both processing and side-specific weights to price future direct or skew interactions.  Geometric trimming compresses that information rather than deleting it indiscriminately, and the approximation theorem controls precisely the resulting loss.

\section*{Concluding synthesis}
\markright{Concluding synthesis}
The separating-tree result combines an exact decomposition with a controlled loss of numerical resolution.  The decomposition is structural: direct and skew nodes determine the orientation of every cross pair.  The approximation is numerical: one minimum-cost representative is retained in each four-dimensional geometric box.  Because the interface is additive, nonnegative, and bilinear, coordinate inflation compounds by at most $\lambda$ per level and cost inflation by at most $\lambda^2$ per level.

\begin{reusableidea}
When a recursive decomposition makes every cross interaction a nonnegative bilinear form of a small additive signature, approximate the interface rather than the underlying assignments.  Keeping one cheapest representative per coordinated geometric box preserves reconstructibility and supports a height-based multiplicative induction.
\end{reusableidea}

\begin{finding}{Why Part V is next}
Part IV has extended the structured algorithmic regime from exact reversal to every fixed refinement pair whose ratio permutation is separable.  What remains is to compare exact hardness, unrestricted approximation, weak hardness, and structured approximation without conflating their meanings.  Part V synthesizes the four certificates, identifies their transferable lessons, and states the next technical milestones.
\end{finding}

\clearpage

\part{Global Synthesis}
\thispagestyle{plain}
\label{part:synthesis}
\setrunningunit{Part V: Synthesis}
\begin{center}
{\Large\bfseries Exact Complexity, Approximation, and Order Structure\par}
\end{center}
\medskip
\begin{quote}\small
\noindent\textbf{Part synopsis.}
This concluding Part integrates the four theorem chains.  It separates exact from approximation classifications, explains why weak NP-completeness and an FPTAS coexist at exact reversal, compares the certificates and interfaces used in the proofs, and converts the remaining boundaries into named technical problems with concrete first milestones.
\end{quote}

\begin{partposition}
\textbf{Established in Parts I--IV.}  General positive-integer \AWET at the boundary due date is strongly NP-complete but has a universal constant-factor approximation; strict antithetical instances are weakly NP-complete and pseudopolynomially solvable; and fixed refinement pairs with separable ratio permutation admit an FPTAS.  \textbf{Result of this Part.}  The four classifications are placed in one landscape without treating exact and approximation complexity as interchangeable.  \textbf{Technique.}  The synthesis compares each structural obstacle with its certificate or interface, its loss or gap control, and its algorithmic consequence.  
\end{partposition}

\section{What the four technical Parts establish}
\label{sec:synthesis-results}
Part~\ref{part:intractability} settles the published strong-complexity question for positive-integer asymmetric weighted earliness--tardiness scheduling at $d=\sum_jp_j$.  Its theorem is strong NP-completeness.  Its reusable technique is a compiler that separates a prefix-square signal, an anchor-coupled minimum-kernel load lock, and a bounded residual while keeping every output number polynomially bounded.

Part~\ref{part:approximation} shows that exact intractability does not preclude a universal performance guarantee.  Its theorem is a polynomial-time $(3+2\sqrt2+\varepsilon)$-approximation for unrestricted \AWET.  Its reusable technique is asymmetric payment: local pseudo-probability consistency pays the early pairs, while the tardy self-cost completes a positive-semidefinite minimum kernel whose covariance inequality pays the tardy pairs globally.

Part~\ref{part:antithetical} classifies the exact-reversal extreme.  Its theorem is weak NP-completeness together with an exact pseudopolynomial recurrence.  Its reusable technique is a dominant completed square in one selected load, with every identity-dependent interaction confined to a uniformly bounded residual.  The same reversed order is separable, so the FPTAS of Part~\ref{part:cotree} also applies after the fixed strict refinements.

For fixed total ratio-order refinements, Part~\ref{part:cotree} treats instances whose ratio permutation is separable.  Its theorem is an FPTAS.  Its reusable technique is an ordered separating tree whose direct and skew nodes expose a four-coordinate nonnegative additive interface; the exact pseudopolynomial recurrence is then compressed by coordinated geometric trimming.

\section{Combined landscape}
The conclusions are summarized in \cref{fig:short-landscape}.  The horizontal direction changes the order promise; the vertical direction changes the required quality of solution.  No arrow should be read as an automatic transfer between exact and approximation complexity.

\begin{figure}[htbp]
\centering
\resizebox{0.98\textwidth}{!}{%
\begin{tikzpicture}[
  >=Latex,
  box/.style={draw=navy,rounded corners=3pt,very thick,align=left,text width=6.0cm,inner sep=9pt,font=\small},
  hard/.style={box,fill=figurehard},
  alg/.style={box,fill=figurealgorithm},
  mixed/.style={box,fill=figuremixed},
  open/.style={box,fill=figureopen},
  arrow/.style={->,thick,draw=navy}
]
\node[hard] (generalexact) at (0,0) {\textbf{Unrestricted exact problem}\\[2pt]
Strongly NP-complete for positive-integer data with $d=\sum_jp_j$ (Part~I).};
\node[alg] (generalapprox) at (8.0,0) {\textbf{Unrestricted approximation}\\[2pt]
Polynomial-time $(3+2\sqrt2+\varepsilon)$-approximation for every $\varepsilon>0$ (Part~II).};
\node[mixed] (antiexact) at (0,-4.5) {\textbf{Strict antithetical order}\\[2pt]
Weakly NP-complete and exactly solvable in pseudopolynomial time (Part~III).};
\node[alg] (cotreeapprox) at (8.0,-4.5) {\textbf{Fixed separable refinements}\\[2pt]
FPTAS by exact separating-tree convolution and coordinated geometric trimming (Part~IV).};
\node[open,text width=12.6cm] (openq) at (4.0,-8.35) {\textbf{Open general and structural boundaries.}\\[2pt]
The results do not decide whether unrestricted \AWET has a PTAS, whether a better universal constant is possible, or how far the FPTAS extends beyond separable ratio permutations.};
\draw[arrow] (generalexact) -- (generalapprox);
\draw[arrow] (antiexact) -- (cotreeapprox);
\draw[arrow] (generalexact) -- (antiexact);
\draw[arrow] (generalapprox) -- (cotreeapprox);
\draw[arrow] (antiexact) -- (openq.west);
\draw[arrow] (cotreeapprox) -- (openq.east);
\end{tikzpicture}%
}
\caption{Combined exact-complexity and approximation landscape.  Red marks exact intractability, green marks approximation algorithms, blue marks a mixed weak-hardness/pseudopolynomial regime, and amber marks unresolved boundaries.}
\label{fig:short-landscape}
\end{figure}
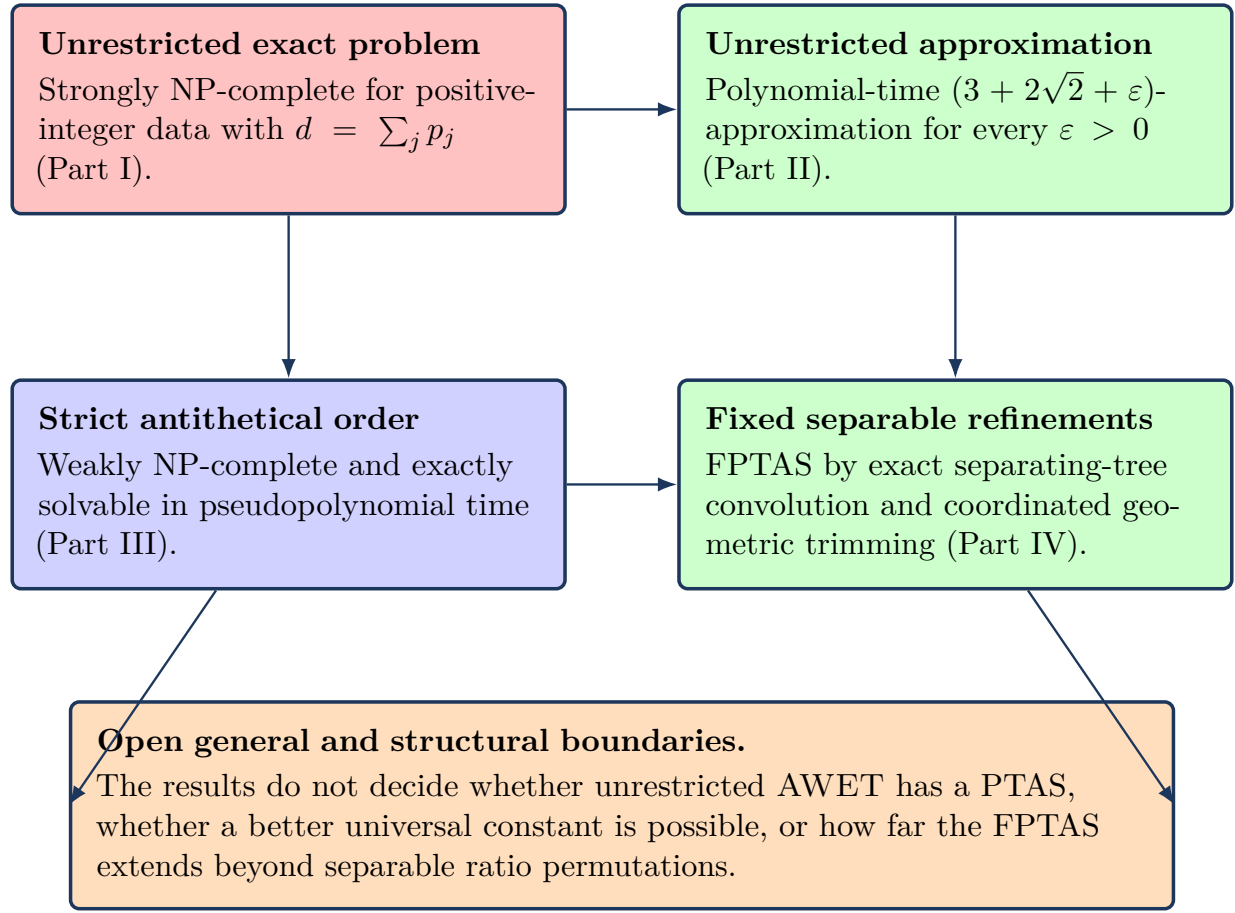

The antithetical subclass illustrates why the categories must remain distinct.  Its decision problem is NP-complete, so a polynomial exact algorithm is unlikely.  Its pseudopolynomial recurrence makes that hardness weak, and its reverse permutation is separable, so an FPTAS is available.  Weak NP-hardness, pseudopolynomial solvability, and an FPTAS are therefore mutually consistent, as in other knapsack-type regimes.

\section{Synthesis by obstacle, certificate, and control}
\label{sec:synthesis-table}
The following table compares the proof architectures rather than merely restating the theorems.

\begin{center}
\footnotesize
\setlength{\tabcolsep}{2.5pt}
\begin{tabularx}{0.98\textwidth}{@{}>{\raggedright\arraybackslash}p{0.045\textwidth}>{\raggedright\arraybackslash}p{0.18\textwidth}>{\raggedright\arraybackslash}p{0.205\textwidth}>{\raggedright\arraybackslash}p{0.19\textwidth}>{\raggedright\arraybackslash}X@{}}
\toprule
Part & structural obstacle & certificate or interface & loss or gap control & consequence\\
\midrule
I & many local exact-cover conditions must be encoded by nonnegative scheduling costs & effective-atom prefix profiles plus an anchor-coupled minimum-kernel load lock & main gaps dominate two residual swings and an integer unit: $\Delta_{\min}/2>2\Gamma+1$ & strong NP-completeness with polynomially bounded integer data\\
II & rounded tardy pairs may have zero pseudo-joint mass, so entrywise payment fails & anchored local SDP and the diagonal-completed tardy minimum kernel & threshold denominators are balanced; finite-precision loss is absorbed in the requested $\varepsilon$ & universal $(3+2\sqrt2+\varepsilon)$-approximation\\
III & exact reversal destroys the Part I compiler but leaves a global load choice & tagged half-load square plus residual $\Psi$ & the unit load gap pays $DM$, while both residual swings cost at most $2DU$ and $M>2U$ & weak NP-completeness and a two-resource pseudopolynomial algorithm\\
IV & future cross-block cost must be summarized without retaining every side assignment & four-coordinate nonnegative additive signature on an ordered separating tree & one factor $\lambda$ per coordinate level and at most $\lambda^2$ per cost level, with $\lambda^{2H}\le1+\varepsilon$ & exact pseudopolynomial recurrence and FPTAS\\
\bottomrule
\end{tabularx}
\end{center}

Two broad design choices recur.  A \emph{global certificate} is appropriate when one aggregate condition can dominate all unwanted interactions: the load lock of Part~I, the covariance inequality of Part~II, and the completed square of Part~III have this character.  A \emph{bounded constructive interface} is appropriate when the instance decomposes recursively: Part~IV stores exactly the information needed to price interactions with the unprocessed part of the tree.

This distinction is transferable beyond the present model.  In another two-sided sequencing or quadratic partition problem, the first audit should ask whether a global norm, kernel, or convex aggregate can certify feasibility or cost.  If not, the next question is whether an order or graph decomposition makes every cross interaction factor through a small nonnegative signature.  Only after those structural questions should one choose between scale-separated hardness, relaxation and rounding, pseudopolynomial enumeration, or interface trimming.

\setshortsectionhead{Implications}
\section{Methodological, modeling, and implementation implications}
\label{sec:model-implications}
\paragraph{Methodological implication.}
Canonicalization is the common entry point.  It removes continuous placement and local sequencing choices before any gadget, relaxation, or dynamic program is introduced.  This prevents proof machinery from paying repeatedly for decisions that an exchange argument already determines.

\paragraph{Model-structure implication.}
The ratio orders are informative data objects in their own right.  Two instances with similar processing times and weights can have different algorithmic structure if their early and tardy ratios induce different inversion graphs.  This observation is structural rather than managerial: it identifies which mathematical representation should be inspected before an algorithm is selected.

\paragraph{Implementation implication.}
A solver can use the following diagnostic workflow.
\begin{enumerate}[leftmargin=2.2em,itemsep=2pt]
\item Compute the two ratio orders exactly by cross multiplication and fix deterministic tie refinements.
\item Construct the ratio permutation and test whether it is separable.
\item If it is, use the separating-tree exact recurrence for small numerical ranges or the separating-tree FPTAS for general binary data.
\item Otherwise, use the unrestricted anchored-SDP approximation, while exact enumeration or the antithetical recurrence remains available when the order pattern or numerical range permits it.
\end{enumerate}
This workflow does not claim that nonseparable instances lack stronger algorithms.  It uses only the guarantees established in the tutorial and leaves room for future structural tests.

\paragraph{Hardness-design implication.}
Parts~I and III show two different uses of magnitude.  The strong reduction uses a hierarchy of polynomial scales to isolate many simultaneous signals; the antithetical reduction uses one dominant convex square and one residual budget.  The first architecture supports strong hardness, whereas the second remains compatible with a pseudopolynomial target algorithm.  The numerical form of a reduction should therefore be aligned with the exact complexity classification it seeks to prove.

\section{Named open problems and first milestones}
\label{sec:open-problems}
\paragraph{Open Problem A: the unrestricted PTAS boundary.}
Does unrestricted \AWET admit a PTAS, or can a no-PTAS theorem be proved by an approximation-preserving reduction whose relative gap survives the positive objective baseline?  A concrete first milestone is to prove or refute a bounded-core decomposition after conditioning on one early anchor, with a residual class that can be spliced at $O(\varepsilon)$ relative cost.

\paragraph{Open Problem B: a better universal approximation ratio.}
Can the factor $3+2\sqrt2$ be improved?  The current proof pays the early and tardy sides separately.  A first milestone is a relaxation and rounding inequality that couples the early local-consistency account with the tardy minimum-kernel account, thereby replacing the maximum of two threshold losses by one joint bound.

\paragraph{Open Problem C: structured classes beyond separable permutations.}
Which nonseparable ratio permutations still admit an FPTAS?  A first milestone is to handle one bounded-size prime node in a substitution decomposition while retaining separating-tree trimming around it.  A successful result would clarify whether the natural parameter is prime-node size, distance to separability, or another interface-width measure.

\paragraph{Open Problem D: practical acceleration with certificates.}
Can the $n$ anchored SDPs and the theoretically possible separating-tree boxes be reduced substantially in practice without weakening the guarantees?  A first milestone is to derive valid anchor-screening lower bounds and combine them with output-sensitive dictionaries that store only reachable signature boxes.  Warm starts and sparse Gram representations can then be assessed on top of those certified reductions.

\section{Final perspective}
This tutorial gives a deliberately balanced picture.  General positive-integer \AWET is strongly NP-complete, yet it has a universal constant-factor approximation.  Exact reversal remains NP-complete, but only weakly, and its order structure supports both a pseudopolynomial exact algorithm and an FPTAS through the broader separating-tree theorem.  The decisive theme is not simply whether the two ratio orders agree.  It is how much of the canonical quadratic interaction can be controlled by a dominant certificate or compressed into a bounded constructive interface.

The four theorem chains therefore support one final methodological claim: expose the canonical objective first, then match the proof device to the structure that remains.  Hardness requires a gap that dominates every residual; unrestricted approximation requires a global relaxation certificate; weakly hard order patterns may expose small pseudopolynomial states; and recursively decomposable orders may permit approximation of interfaces rather than assignments.
\clearpage
\appendix
\makeatletter
\@addtoreset{theorem}{section}
\makeatother
\setcounter{theorem}{0}
\renewcommand{\thetheorem}{\thesection.\arabic{theorem}}
\renewcommand{\thelemma}{\thetheorem}
\renewcommand{\theproposition}{\thetheorem}
\renewcommand{\thecorollary}{\thetheorem}
\renewcommand{\theclaim}{\thetheorem}
\renewcommand{\thedefinition}{\thetheorem}
\renewcommand{\theproblem}{\thetheorem}
\renewcommand{\theexample}{\thetheorem}
\renewcommand{\theremark}{\thetheorem}

\clearpage
\setrunningunit{Appendix A: Objective Algebra}
\setshortsectionhead{Canonical objective}
\section{Algebra of the canonical objective}
\label{app:algebra}

This appendix certifies the two coefficient formulas consumed by the Part~I master decomposition.  The early indicator expansion creates the baseline fields and couplings, while a common tardiness offset changes only one block's field on the coded manifold.  Keeping the expansions separate identifies the destination of every constant, linear, and quadratic term.

\subsection{Baseline early field}

For an early pair $i,j$ in the same block $v$,
\[
 \1\{i,j\in E\}
 =\frac14(1+\sigma_v\epsilon_i)(1+\sigma_v\epsilon_j)
 =\frac14\bigl(1+\epsilon_i\epsilon_j+\sigma_v(\epsilon_i+\epsilon_j)\bigr).
\]
For $i\in\B_u$ and $j\in\B_v$, $u\ne v$,
\[
 \1\{i,j\in E\}
 =\frac14\bigl(1+\sigma_u\epsilon_i+\sigma_v\epsilon_j
 +\sigma_u\sigma_v\epsilon_i\epsilon_j\bigr).
\]
Collecting the coefficients of $\sigma_v$ gives \eqref{eq:a0}; collecting the coefficients of $\sigma_u\sigma_v$ and applying \cref{lem:compression} gives \eqref{eq:profilekernel}.  The constant terms are retained inside $C_0$.  Hence every term in the physical early objective has a declared destination: constant, field, or coupling.

\subsection{Effect of a common tardiness offset}

On the coded state space, the tardy load of block $v$ is $\beta_v$. Therefore every cross-block tardy contribution is independent of $\sigma$. Within block $v$, the contribution of the offset $\eta_v$ is
\[
 \eta_v\left(
 \sum_{\{i,j\}\subseteq T\cap\B_v}p_ip_j
 +\sum_{i\in T\cap\B_v}p_i^2
 \right)
 =\frac{\eta_v}{2}(\beta_v^2+q_v),
\]
where
\[
 q_v=\sum_{i\in\B_v}p_i^2\frac{1-\sigma_v\epsilon_i}{2}
 =\frac12(Q_v-\sigma_vD_v).
\]
Thus the offset changes the logical field by $-\eta_vD_v/4$, as used in \eqref{eq:ev}.

\clearpage
\setrunningunit{Appendix B: Examples}
\section{Two worked examples}\label{app:worked}

This appendix provides reproducibility certificates for the Part~I reduction, not merely illustrative examples.  The yes-instance verifies the no-carry arithmetic, pair synchronization, strict ratio separation, load lock, master identity, decision threshold, and forward decoding chain.  A degree-preserving exchange then produces a no-instance and verifies the corresponding reverse contradiction.  Printing the full integer data makes every claimed inequality and schedule cost independently auditable without hidden computation.

\begin{howtoread}
\textbf{Theorem-verification lane.}  Read the source and prefix instances, the scale and gap summaries through Section~\ref{app:worked-loadlock}, the master-identity and threshold summary in Section~\ref{app:worked-threshold}, and the forward and reverse conclusions.  The landscape job catalogues and full schedule tables may be skipped without breaking the proof narrative.  \textbf{Full arithmetic-audit lane.}  Follow every defining recurrence and use each ``Purpose and verification'' statement to reproduce the displayed integer tables and all 64 coded-state evaluations.
\end{howtoread}

\begin{arithmeticguide}
Long integers are never meant to be compared visually as undifferentiated strings.  The certificate map and each table lead-in name a defining expression or subtraction.  Verify one row from that expression, then use only the property consumed later: equality, sign, strict order, stated lower bound, or final difference.  Full digit strings are retained as exact evidence; the defining expressions are the audit interface.
\end{arithmeticguide}

\begin{table}[htbp]
\centering
\small
\caption{Certificate map for Appendix~\ref{app:worked}.  The defining-expression column is the preferred audit route for the long integers.}
\label{tab:certificate-map}
\begin{tabularx}{0.98\textwidth}{@{}p{0.19\textwidth}p{0.25\textwidth}p{0.25\textwidth}X@{}}
\toprule
certificate & defining expression & claim verified & where consumed\\
\midrule
\cref{tab:worked-P,tab:no-P} & $P_{vh}=T+M\kappa(k_{vh})+Q_h$ and $\beta_v=\sum_hP_{vh}+3$ & integral bases and equal balanced loads & two-state synchronization and the global load lock\\
\cref{tab:worked-kernel,tab:no-kernel} & $f=K\beta$ and $f_v-f_{v+1}=2\sum_{u\le v}\beta_u$ & minimum-kernel fields and block order & early ratios and load-lock centering\\
\cref{tab:worked-atoms,tab:worked-bands,tab:no-bands} & $\rho=f_v+\lambda_v(2P+c)$; subtract adjacent endpoints & effective atoms, strict bands, main gaps, and total width & pair compression, scale separation, and $\Gamma$\\
\cref{tab:worked-couplings,tab:worked-fields,tab:no-couplings,tab:no-fields} & profile integrals and $4e_v=4a_v^0-\eta_vD_v-4a_v^\star$ & main/residual couplings and bounded field correction & exact master decomposition and residual bound\\
\cref{tab:worked-job-ratios,tab:worked-job-weights,tab:no-job-ratios,tab:no-job-weights} & $w_j^E=p_jr_j$, $w_j^L=p_js_j$ & the physical positive-integer \AWET instances & feasibility and exact input reproducibility\\
\cref{tab:worked-schedule,tab:worked-state-audit,tab:no-state-audit} & canonical completion times and $A+\mathcal P+R-B$ & yes schedule below threshold and every no state above threshold & forward and reverse correctness certificates\\
\bottomrule
\end{tabularx}
\end{table}

\subsection{A ``yes'' instance}\label{app:worked-yes}

This first example specializes the reduction to a six-element yes-instance.  No scale is compressed and no floating-point approximation is used.  The resulting scheduling instance has 48 state jobs and one anchor job.  Because the proof uses deliberately separated powers of $N$, its integers are long; the line-breaking typewriter format preserves every digit.

\begin{howtoread}
\textbf{Theorem-verification lane.}  Track the exact cover, prefix equations, two-state block property, band-gap certificate, load-lock slack, master threshold, and final schedule.  \textbf{Full arithmetic-audit lane.}  Reproduce every row from the formulas named in \cref{tab:certificate-map}.  In either lane, do not compare long integers visually; audit the defining recurrence or subtraction and retain only the equality, sign, or bound used later.
\end{howtoread}

\subsubsection{The restricted exact-cover instance}\label{app:worked-source}

Let \(X=\{e_1,e_2,e_3,e_4,e_5,e_6\}\).  The six sets are
\[
\begin{aligned}
S_1&=\{e_1,e_2,e_3\}, &
S_2&=\{e_1,e_2,e_5\}, &
S_3&=\{e_2,e_4,e_6\},\\
S_4&=\{e_1,e_3,e_6\}, &
S_5&=\{e_4,e_5,e_6\}, &
S_6&=\{e_3,e_4,e_5\}.
\end{aligned}
\]
With elements in the rows and subsets in the columns, its incidence matrix is
\begin{equation}\label{eq:worked-incidence}
A^{\mathrm{inc}}=
\begin{pmatrix}
1&1&0&1&0&0\\
1&1&1&0&0&0\\
1&0&0&1&0&1\\
0&0&1&0&1&1\\
0&1&0&0&1&1\\
0&0&1&1&1&0
\end{pmatrix}.
\end{equation}
Every element row and every subset column contains three ones, so this is a valid three-uniform, three-regular \RXC\ instance.  Since an exact cover of six elements by 3-sets must contain exactly two subset-columns, it is enough to inspect pairwise intersections.  The intersection-cardinality matrix is
\[
\begin{pmatrix}
3&2&1&2&0&1\\
2&3&1&1&1&1\\
1&1&3&1&2&1\\
2&1&1&3&1&1\\
0&1&2&1&3&2\\
1&1&1&1&2&3
\end{pmatrix}.
\]
Its only off-diagonal zero is the pair \((S_1,S_5)\).  Moreover, the corresponding incidence columns satisfy
\[
 \begin{pmatrix}1\\1\\1\\0\\0\\0\end{pmatrix}
 +
 \begin{pmatrix}0\\0\\0\\1\\1\\1\end{pmatrix}
 =
 \begin{pmatrix}1\\1\\1\\1\\1\\1\end{pmatrix}.
\]
Thus the subset-columns $S_1$ and $S_5$ form not merely an exact cover, but the unique exact cover.  The uniqueness is not required by the reduction, but it makes the final 64-state audit especially transparent: exactly one binary state should have zero prefix penalty.

\begin{figure}[ht]
\centering
\begin{tikzpicture}[font=\small]
\matrix (M) [matrix of nodes,
  nodes={draw,minimum width=8mm,minimum height=7mm,anchor=center},
  column 1/.style={nodes={fill=figurealgorithm}},
  column 5/.style={nodes={fill=figurealgorithm}},
  column sep=-\pgflinewidth,row sep=-\pgflinewidth] {
1&1&0&1&0&0\\
1&1&1&0&0&0\\
1&0&0&1&0&1\\
0&0&1&0&1&1\\
0&1&0&0&1&1\\
0&0&1&1&1&0\\
};
\foreach \i in {1,...,6}{
  \node[left=4mm of M-\i-1] {$e_{\i}$};
  \node[above=3mm of M-1-\i] {$S_{\i}$};
}
\node[draw,rounded corners,fill=figurealgorithm,below=8mm of M] {selected exact-cover columns: $S_1$ and $S_5$};
\end{tikzpicture}
\caption{The incidence matrix with elements in rows and subsets in columns.  The highlighted subset-columns sum rowwise to the all-ones vector.}
\label{fig:worked-incidence}
\end{figure}

\subsubsection{The associated prefix-exactness instance}\label{app:worked-rpe}

For each set \(S_v\), the prefix datum is \(d_{v,k}=|S_v\cap\{e_1,\ldots,e_k\}|\).  The full prefix matrix is shown in \cref{tab:worked-prefix}.  The exact-cover vector and spin vector are
\begin{equation}\label{eq:worked-x-sigma}
 x=(1,0,0,0,1,0),
 \qquad
 \sgnvec=2x-\mathbf 1=(+1,-1,-1,-1,+1,-1).
\end{equation}
\begin{table}[ht]
\centering
\caption{Prefix data with subset-columns preserved.  Row $k$ gives the cumulative counts through element row $e_k$; the last column evaluates \(C_k(x)=\sum_vd_{v,k}x_v\) at \(x=(1,0,0,0,1,0)\).}
\label{tab:worked-prefix}
\begin{tabular}{c c c c c c c c}
\toprule
Prefix & $S_1$ & $S_2$ & $S_3$ & $S_4$ & $S_5$ & $S_6$ & $C_k(x)$\\
\midrule
$k=1$ & 1 & 1 & 0 & 1 & 0 & 0 & 1 \\
$k=2$ & 2 & 2 & 1 & 1 & 0 & 0 & 2 \\
$k=3$ & 3 & 2 & 1 & 2 & 0 & 1 & 3 \\
$k=4$ & 3 & 2 & 2 & 2 & 1 & 2 & 4 \\
$k=5$ & 3 & 3 & 2 & 2 & 2 & 3 & 5 \\
$k=6$ & 3 & 3 & 3 & 3 & 3 & 3 & 6 \\
\bottomrule
\end{tabular}
\end{table}

The six equations are, explicitly,
\begin{align*}
C_1(x)&=1\cdot x_1+1\cdot x_2+0\cdot x_3+1\cdot x_4+0\cdot x_5+0\cdot x_6=1,\\
C_2(x)&=2x_1+2x_2+x_3+x_4=2,\\
C_3(x)&=3x_1+2x_2+x_3+2x_4+x_6=3,\\
C_4(x)&=3x_1+2x_2+2x_3+2x_4+x_5+2x_6=4,\\
C_5(x)&=3x_1+3x_2+2x_3+2x_4+2x_5+3x_6=5,\\
C_6(x)&=3(x_1+x_2+x_3+x_4+x_5+x_6)=6.
\end{align*}
Substitution of \eqref{eq:worked-x-sigma} gives \((C_1,\ldots,C_6)=(1,2,3,4,5,6)\).  Conversely, successive differences recover the element-cover equations:
\[
\begin{array}{lll}
C_1-C_0=x_1+x_2+x_4=1,&
C_2-C_1=x_1+x_2+x_3=1,&
C_3-C_2=x_1+x_4+x_6=1,\\
C_4-C_3=x_3+x_5+x_6=1,&
C_5-C_4=x_2+x_5+x_6=1,&
C_6-C_5=x_3+x_4+x_5=1.
\end{array}
\]
These are exactly the six element-row equations of \eqref{eq:worked-incidence}.  For instance, the $S_1$ column of the prefix table is $(1,2,3,3,3,3)^\top$ because the cumulative sum down incidence column $S_1$ increases in rows $e_1,e_2,e_3$ and then remains constant.  This illustrates concretely why successive prefix differences recover the original incidence rows.

\subsubsection{Numerical scales and base processing values}\label{app:worked-scales}

Here \(n=6\) and \(N=n+1=7\).  Therefore
\[
Q_0=10,\qquad Q_1=100,\qquad Q_2=1000,\qquad Q_3=10000,\qquad q=11110.
\]
\[
\begin{aligned}
G=7^{20}&=\displaynum{79792266297612001},\\
M=7^{40}&=\displaynum{6366805760909027985741435139224001},\\
T=7^{50}&=\displaynum{1798465042647412146620280340569649349251249}.
\end{aligned}
\]
The level map is
\[
(\kappa(1),\ldots,\kappa(6))=(6,5,4,3,2,1).
\]
The scale relationships are especially easy to see symbolically:
\[
M=G^2=7^{40},
\qquad
T=7^{10}M=7^{50},
\qquad
GM=7^{60}.
\]
Thus $T$ dominates all block-level $M$ terms, while the ratio-space gap $GM$ dominates block-field variation on the order of $N^2T=7^{52}$.
The common controller base is
\[
P_{v0}=T+10=\displaynum{1798465042647412146620280340569649349251259}
\qquad(v=1,\ldots,6).
\]
All other base processing values and block loads are displayed in \cref{tab:worked-P}.  For example, \(S_1=\{e_1,e_2,e_3\}\), so
\begin{align*}
P_{11}&=T+6M+100,\\
P_{12}&=T+5M+1000,\\
P_{13}&=T+4M+10000,\\
\beta_1&=P_{10}+P_{11}+P_{12}+P_{13}+3.
\end{align*}
Substitution gives the first row of \cref{tab:worked-P}.  The three coefficients $6,5,4$ are not arbitrary: they are the reversed element levels of $e_1,e_2,e_3$.  The pair codes $100,1000,10000$ identify occurrence positions $h=1,2,3$ even when two sets contain the same element.
\begin{landscape}
\scriptsize
\let\smallnum\numstr
\setlength{\tabcolsep}{1.5pt}
\setlength{\LTleft}{0pt}
\setlength{\LTright}{0pt}
\tablecert{\textbf{Purpose.}  The table certifies the integral base values and target load of every coded block.  \textbf{Verification rule.}  Form each $P_{vh}=T+M\kappa(k_{vh})+Q_h$, then add $P_{v0}+P_{v1}+P_{v2}+P_{v3}+3$; the two orientations must both sum to the displayed $\beta_v$.}
\begin{longtable}{c c c p{0.16\textheight} p{0.16\textheight} p{0.16\textheight} p{0.16\textheight} p{0.18\textheight}}
\caption{Exact base processing values and balanced block loads.}\label{tab:worked-P}\\
\toprule
$v$ & $S_v$ & $(\kappa_1,\kappa_2,\kappa_3)$ & $P_{v0}$ & $P_{v1}$ & $P_{v2}$ & $P_{v3}$ & $\beta_v$ \\
\midrule
\endfirsthead
\multicolumn{8}{c}{\tablename\ \thetable\ (continued)}\\
\toprule
$v$ & $S_v$ & $(\kappa_1,\kappa_2,\kappa_3)$ & $P_{v0}$ & $P_{v1}$ & $P_{v2}$ & $P_{v3}$ & $\beta_v$ \\
\midrule
\endhead
\midrule
\multicolumn{8}{r}{Continued on next page}\\
\endfoot
\bottomrule
\endlastfoot
1 & $\{1,2,3\}$ & $(6,5,4)$ & \smallnum{1798465042647412146620280340569649349251259} & \smallnum{1798465080848246712074448255018260184595355} & \smallnum{1798465074481440951165420269276825045372254} & \smallnum{1798465068114635190256392283535389906157253} & \smallnum{7193860266091735000116541148400124485376124} \\
2 & $\{1,2,5\}$ & $(6,5,2)$ & \smallnum{1798465042647412146620280340569649349251259} & \smallnum{1798465080848246712074448255018260184595355} & \smallnum{1798465074481440951165420269276825045372254} & \smallnum{1798465055381023668438336312052519627709251} & \smallnum{7193860253358123478298485176917254206928122} \\
3 & $\{2,4,6\}$ & $(5,3,1)$ & \smallnum{1798465042647412146620280340569649349251259} & \smallnum{1798465074481440951165420269276825045371354} & \smallnum{1798465061747829429347364297793954766924252} & \smallnum{1798465049014217907529308326311084488485250} & \smallnum{7193860227890900434662373233951513650032118} \\
4 & $\{1,3,6\}$ & $(6,4,1)$ & \smallnum{1798465042647412146620280340569649349251259} & \smallnum{1798465080848246712074448255018260184595355} & \smallnum{1798465068114635190256392283535389906148253} & \smallnum{1798465049014217907529308326311084488485250} & \smallnum{7193860240624511956480429205434383928480120} \\
5 & $\{4,5,6\}$ & $(3,2,1)$ & \smallnum{1798465042647412146620280340569649349251259} & \smallnum{1798465061747829429347364297793954766923352} & \smallnum{1798465055381023668438336312052519627700251} & \smallnum{1798465049014217907529308326311084488485250} & \smallnum{7193860208790483151935289276727208232360115} \\
6 & $\{3,4,5\}$ & $(4,3,2)$ & \smallnum{1798465042647412146620280340569649349251259} & \smallnum{1798465068114635190256392283535389906147353} & \smallnum{1798465061747829429347364297793954766924252} & \smallnum{1798465055381023668438336312052519627709251} & \smallnum{7193860227890900434662373233951513650032118} \\
\end{longtable}
\end{landscape}

\begin{arithmeticguide}
To audit a row of \cref{tab:worked-P}, first form each $P_{vh}$ from $T+M\kappa(k_{vh})+Q_h$ and then add the four bases plus $3$ to obtain $\beta_v$.  Rows 3 and 6 have the same $\beta$ because their three reversed levels have the same sum $5+3+1=4+3+2=9$; the fixed pair-code total is identical in every block.
\end{arithmeticguide}

\subsubsection{The 48 state jobs and arithmetic synchronization}\label{app:worked-blocks}

Block \(v\) contains the eight jobs specified in \cref{tab:block}.  To make the arithmetic visible, consider block 1.  In the \(\sigma_1=+1\) orientation its signed processing sum is
\begin{align*}
&-(P_{10}+3)+P_{10}
 +(P_{11}+1)-P_{11}
 +(P_{12}+1)-P_{12}
 +(P_{13}+1)-P_{13}\\
&\hspace{5em}=-3+1+1+1=0.
\end{align*}
The complementary \(\sigma_1=-1\) orientation also has sum zero.  For an arbitrary orientation, the decomposition \eqref{eq:signedload} becomes
\[
2T U+2M Z+20u_0+200u_1+2000u_2+20000u_3+R,
\qquad |R|\le6.
\]
The symbols $U$ and $Z$ have exactly the meanings used in the general proof: $U$ is the total pair imbalance and $Z$ is the reversed-element-level imbalance.
The exact inequalities are
\[
2T>6nM+2q+6,
\qquad
2M>2q+6,
\qquad
20>6.
\]
Thus $U=Z=0$, then $(u_0,u_1,u_2,u_3)=(0,0,0,0)$, and finally the residual equation forces the displayed state or its complement.  A direct exhaustive enumeration of all $2^8=256$ sign vectors in each of the six blocks gives the same conclusion.  In the job order
\[
(A_{v0},B_{v0},A_{v1},B_{v1},A_{v2},B_{v2},A_{v3},B_{v3}),
\]
the only zero-load vectors are
\[
(-,+,+,-,+,-,+,-)
\quad\text{and}\quad
(+,-,-,+,-,+,-,+).
\]
Hence all six blocks have exactly the required two balanced states.

\subsubsection{Kernel data, early ratios, and effective atoms}\label{app:worked-atoms}

The six values of \(\lambda_v=G+7-v\), the loads \(\beta_v\), and the fields \(f_v=\sum_u\min\{L_v,L_u\}\beta_u\) are in \cref{tab:worked-kernel}.  As one sample calculation, since \(L_1\) is the largest kernel level,
\[
 f_1=\sum_{u=1}^{6}L_u\beta_u.
\]
The displayed exact value agrees with direct multiplication and summation.  A useful local check is
\[
f_1-f_2=2\beta_1
=\readnum{14387720532183470000233082296800248970752248,}
\]
which is the $v=1$ case of \cref{lem:fidentities}.  Every state-job early ratio is then $r_i=f_v+\lambda_vp_i$.
\begin{landscape}
\scriptsize
\let\smallnum\numstr
\setlength{\tabcolsep}{1.5pt}
\setlength{\LTleft}{0pt}
\setlength{\LTright}{0pt}
\tablecert{\textbf{Purpose.}  The table certifies the block loads, kernel levels, and affine fields used by the load lock.  \textbf{Verification rule.}  Compute one row from $f=K\beta$, then audit successive rows by $f_v-f_{v+1}=2\sum_{u\le v}\beta_u$.}
\begin{longtable}{c p{0.19\textheight} p{0.13\textheight} p{0.13\textheight} p{0.30\textheight}}
\caption{Block loads and minimum-kernel data.}\label{tab:worked-kernel}\\
\toprule
$v$ & $\beta_v$ & $\lambda_v$ & $L_v=2\lambda_v$ & $f_v$ \\
\midrule
\endfirsthead
\multicolumn{5}{c}{\tablename\ \thetable\ (continued)}\\
\toprule
$v$ & $\beta_v$ & $\lambda_v$ & $L_v=2\lambda_v$ & $f_v$ \\
\midrule
\endhead
\midrule
\multicolumn{5}{r}{Continued on next page}\\
\endfoot
\bottomrule
\endlastfoot
1 & \smallnum{7193860266091735000116541148400124485376124} & \smallnum{79792266297612007} & \smallnum{159584532595224014} & \smallnum{6888172941284439601043621980215987825344444649279254568462502} \\
2 & \smallnum{7193860253358123478298485176917254206928122} & \smallnum{79792266297612006} & \smallnum{159584532595224012} & \smallnum{6888172941284439586655901448032517825111362352479005597710254} \\
3 & \smallnum{7193860227890900434662373233951513650032118} & \smallnum{79792266297612005} & \smallnum{159584532595224010} & \smallnum{6888172941284439557880460409132800868281309701844248213101762} \\
4 & \smallnum{7193860240624511956480429205434383928480120} & \smallnum{79792266297612004} & \smallnum{159584532595224008} & \smallnum{6888172941284439514717298914451283042126510583306463528429034} \\
5 & \smallnum{7193860208790483151935289276727208232360115} & \smallnum{79792266297612003} & \smallnum{159584532595224006} & \smallnum{6888172941284439457166416938520741303010853053899910986796066} \\
6 & \smallnum{7193860227890900434662373233951513650032118} & \smallnum{79792266297612002} & \smallnum{159584532595224004} & \smallnum{6888172941284439385227814545009233260024616971038941980442868} \\
\end{longtable}
\end{landscape}

\begin{arithmeticguide}
Each row of \cref{tab:worked-kernel} is generated from the previous one by subtracting twice the cumulative block load: $f_v-f_{v+1}=2\sum_{u\le v}\beta_u$.  This recurrence is much easier to audit than recomputing every long dot product independently.
\end{arithmeticguide}

The 24 compressed atoms are shown in \cref{tab:worked-atoms}.  For example,
\[
\rho_{11}=f_1+\lambda_1(2P_{11}+1)
=\mednum{7175180150600438820071889486974548885544592536491609034929479.}
\]
The controller atom in every block has mass $-3$; every occurrence atom has mass $+1$.  For an occurrence row, one checks the effective ratio by $f_v+\lambda_v(2P_{vh}+1)$; for a controller row, by $f_v+\lambda_v(2P_{v0}+3)$.  The expression $2P$ plus the small constant is the difference-of-squares calculation from the theory, not an independently chosen coordinate.
\begin{landscape}
\scriptsize
\let\smallnum\numstr
\setlength{\tabcolsep}{1.5pt}
\setlength{\LTleft}{0pt}
\setlength{\LTright}{0pt}
\tablecert{\textbf{Purpose.}  The table fixes every effective-atom mass, element label, and ratio used in the profile calculation.  \textbf{Verification rule.}  Use mass $+1$ and $\rho=f_v+\lambda_v(2P_{vh}+1)$ for occurrences, and mass $-3$ with $\rho=f_v+\lambda_v(2P_{v0}+3)$ for controllers.}
\begin{longtable}{c c c c c p{0.54\textheight}}
\caption{All effective atoms in the worked instance.}\label{tab:worked-atoms}\\
\toprule
block & pair & kind & element & mass & effective ratio $\rho$ \\
\midrule
\endfirsthead
\multicolumn{6}{c}{\tablename\ \thetable\ (continued)}\\
\toprule
block & pair & kind & element & mass & effective ratio $\rho$ \\
\midrule
\endhead
\midrule
\multicolumn{6}{r}{Continued on next page}\\
\endfoot
\bottomrule
\endlastfoot
1 & 0 & controller & -- & -3 & \smallnum{7175180144504176491196408644698275728267965338815663737832149} \\
1 & 1 & occurrence & $e_{1}$ & 1 & \smallnum{7175180150600438820071889486974548885544592536491609034929479} \\
1 & 2 & occurrence & $e_{2}$ & 1 & \smallnum{7175180149584395098592642679928503359331967330128854016182065} \\
1 & 3 & occurrence & $e_{3}$ & 1 & \smallnum{7175180148568351377113395872882457833120634758480120311948051} \\
2 & 0 & controller & -- & -3 & \smallnum{7175180144504176473211758027219981434794322360876116068577380} \\
2 & 1 & occurrence & $e_{1}$ & 1 & \smallnum{7175180150600438802087238793094585461162613729654839694986520} \\
2 & 2 & occurrence & $e_{2}$ & 1 & \smallnum{7175180149584395080607991998782151456768044494774954954685308} \\
2 & 3 & occurrence & $e_{5}$ & 1 & \smallnum{7175180146536263916170251615844849443585342172690650645057272} \\
3 & 0 & controller & -- & -3 & \smallnum{7175180144504176440839386903025440184723709029102059985466367} \\
3 & 1 & occurrence & $e_{2}$ & 1 & \smallnum{7175180149584395048235620810919552597607007679507211777723307} \\
3 & 2 & occurrence & $e_{4}$ & 1 & \smallnum{7175180147552307605277127247761907632453837526633847152404287} \\
3 & 3 & occurrence & $e_{6}$ & 1 & \smallnum{7175180145520220162318633684604262667301960008474503841566267} \\
4 & 0 & controller & -- & -3 & \smallnum{7175180144504176394079295323049098065328349229424976602291118} \\
4 & 1 & occurrence & $e_{1}$ & 1 & \smallnum{7175180150600438722954775936120363829879968940409256887323878} \\
4 & 2 & occurrence & $e_{3}$ & 1 & \smallnum{7175180148568351279996282398429941908362910730501632818899062} \\
4 & 3 & occurrence & $e_{6}$ & 1 & \smallnum{7175180145520220115558542091894309026088544237314550179923038} \\
5 & 0 & controller & -- & -3 & \smallnum{7175180144504176332931483261823732032972131018879125362155629} \\
5 & 1 & occurrence & $e_{4}$ & 1 & \smallnum{7175180147552307497369223530158530349793780061434355156796181} \\
5 & 2 & occurrence & $e_{5}$ & 1 & \smallnum{7175180146536263775889976774046930910853378741003081251833575} \\
5 & 3 & occurrence & $e_{6}$ & 1 & \smallnum{7175180145520220054410730017935331471914270055285828661319569} \\
6 & 0 & controller & -- & -3 & \smallnum{7175180144504176257395950783017399696745334254878857657299910} \\
6 & 1 & occurrence & $e_{3}$ & 1 & \smallnum{7175180148568351143312937756529351365235304358013215944716282} \\
6 & 2 & occurrence & $e_{4}$ & 1 & \smallnum{7175180147552307421833691013151363448112959009064812318199878} \\
6 & 3 & occurrence & $e_{5}$ & 1 & \smallnum{7175180146536263700354444269773375530991906294830430006115874} \\
\end{longtable}
\end{landscape}

\begin{howtoread}
In \cref{tab:worked-atoms}, compare only rows carrying the same element to form one band.  For example, $e_1$ occurs in blocks 1, 2, and 4; the maximum of their three ratios is the upper endpoint of $\mathcal R_1$, and the minimum is its lower endpoint.  The controller rows are grouped in $\mathcal R_0$.
\end{howtoread}

Taking minima and maxima within each element level gives \cref{tab:worked-bands}.  The bands are strictly ordered
\[
\mathcal R_1>\mathcal R_2>\mathcal R_3>\mathcal R_4>
\mathcal R_5>\mathcal R_6>\mathcal R_0.
\]
The narrowest main gap is \(\Delta_4\).
\begin{landscape}
\scriptsize
\let\smallnum\numstr
\setlength{\tabcolsep}{1.5pt}
\setlength{\LTleft}{0pt}
\setlength{\LTright}{0pt}
\tablecert{\textbf{Purpose.}  The table certifies the within-band widths and separating main gaps used in the scale-separation lemma.  \textbf{Verification rule.}  Take minima and maxima over atoms with the same element label, subtract endpoints to obtain widths and the ``gap below,'' and use only the sign and stated lower bound later.}
\begin{longtable}{c p{0.24\textheight} p{0.24\textheight} p{0.18\textheight} p{0.18\textheight}}
\caption{Effective bands, widths, and exact gaps below the bands.}\label{tab:worked-bands}\\
\toprule
band & lower endpoint & upper endpoint & width & gap below \\
\midrule
\endfirsthead
\multicolumn{5}{c}{\tablename\ \thetable\ (continued)}\\
\toprule
band & lower endpoint & upper endpoint & width & gap below \\
\midrule
\endhead
\midrule
\multicolumn{5}{r}{Continued on next page}\\
\endfoot
\bottomrule
\endlastfoot
$\mathcal R_{1}$ & \smallnum{7175180150600438722954775936120363829879968940409256887323878} & \smallnum{7175180150600438820071889486974548885544592536491609034929479} & \smallnum{97117113550854185055664623596082352147605601} & \smallnum{1016043624362133256191860470548001610280402871141813} \\
$\mathcal R_{2}$ & \smallnum{7175180149584395048235620810919552597607007679507211777723307} & \smallnum{7175180149584395098592642679928503359331967330128854016182065} & \smallnum{50357021869008950761724959650621642238458758} & \smallnum{1016043671122224938037094764486372921027091465775256} \\
$\mathcal R_{3}$ & \smallnum{7175180148568351143312937756529351365235304358013215944716282} & \smallnum{7175180148568351377113395872882457833120634758480120311948051} & \smallnum{233800458116353106467885330400466904367231769} & \smallnum{1016043538035810508767443732781466831379368792311995} \\
$\mathcal R_{4}$ & \smallnum{7175180147552307421833691013151363448112959009064812318199878} & \smallnum{7175180147552307605277127247761907632453837526633847152404287} & \smallnum{183443436234610544184340878517569034834204409} & \smallnum{1016043505663439397306514004527616836374161673142606} \\
$\mathcal R_{5}$ & \smallnum{7175180146536263700354444269773375530991906294830430006115874} & \smallnum{7175180146536263916170251615844849443585342172690650645057272} & \smallnum{215815807346071473912593435877860220638941398} & \smallnum{1016043538035810585169112863689946286355926164549607} \\
$\mathcal R_{6}$ & \smallnum{7175180145520220054410730017935331471914270055285828661319569} & \smallnum{7175180145520220162318633684604262667301960008474503841566267} & \smallnum{107907903666668931195387689953188675180246698} & \smallnum{1016043563214321373237055743646304716470164923487420} \\
$\mathcal R_{0}$ & \smallnum{7175180144504176257395950783017399696745334254878857657299910} & \smallnum{7175180144504176491196408644698275728267965338815663737832149} & \smallnum{233800457861680876031522631083936806080532239} & -- \\
\end{longtable}
\end{landscape}

The total within-band width and minimum gap are
\[
W=\readnum{1122242198645248067609119549079725635487220872,}
\qquad
\Delta_{\min}=\Delta_4=\readnum{1016043505663439397306514004527616836374161673142606.}
\]
For example, the first gap is computed as
\[
\Delta_1=\underline\rho_1-\overline\rho_2.
\]
The table's ``gap below'' column performs this subtraction for every band.  The important comparison is not the number of digits but that every gap is roughly six orders of magnitude longer than the sum of all within-band widths in this particular instance.

\subsubsection{Cross-block couplings and logical-field correction}\label{app:worked-fields}

For every block pair $u<v$, direct exact evaluation uses the 16 atom-pair terms in \eqref{eq:profilekernel} and the main-gap coefficient in \eqref{eq:Jmain}.  To avoid repeated quarters, \cref{tab:worked-couplings} lists \(4J_{uv}\), \(4J_{uv}^{\rm main}\), and their difference.  These are exact integers.
\begin{landscape}
\scriptsize
\let\smallnum\numstr
\setlength{\tabcolsep}{1.5pt}
\setlength{\LTleft}{0pt}
\setlength{\LTright}{0pt}
\tablecert{\textbf{Purpose.}  The table certifies every cross-block coupling and its split into main-gap signal and within-band residual.  \textbf{Verification rule.}  Evaluate the 16 atom-pair terms, multiply by four to avoid fractions, compute the main-gap sum from $d_{u,k}d_{v,k}\Delta_k$, and subtract the two columns.}
\begin{longtable}{c c p{0.28\textheight} p{0.28\textheight} p{0.28\textheight}}
\caption{Exact compressed cross-block coefficients, scaled by four.}\label{tab:worked-couplings}\\
\toprule
$u$ & $v$ & $4J_{uv}$ & $4J_{uv}^{\rm main}$ & $4(J_{uv}-J_{uv}^{\rm main})$ \\
\midrule
\endfirsthead
\multicolumn{5}{c}{\tablename\ \thetable\ (continued)}\\
\toprule
$u$ & $v$ & $4J_{uv}$ & $4J_{uv}^{\rm main}$ & $4(J_{uv}-J_{uv}^{\rm main})$ \\
\midrule
\endhead
\midrule
\multicolumn{5}{r}{Continued on next page}\\
\endfoot
\bottomrule
\endlastfoot
1 & 2 & \smallnum{35561530143865734261602374176616278035024822539117778} & \smallnum{35561524482297720070439503418374254326344771319303686} & \smallnum{5661568014191162870758242023708680051219814092} \\
1 & 3 & \smallnum{26417136355604242489535592373190354673478623208816986} & \smallnum{26417132287476273656363783629439267520804300646026555} & \smallnum{4068127968833171808743751087152674322562790431} \\
1 & 4 & \smallnum{30481310964557610157619596271275247883125806571430953} & \smallnum{30481306525945838438857975298331669625222809894104353} & \smallnum{4438611771718761620972943578257902996677326600} \\
1 & 5 & \smallnum{18288785562242113164137896643272598523828821569358634} & \smallnum{18288783814134074062067720888539270675489526318112240} & \smallnum{1748108039102070175754733327848339295251246394} \\
1 & 6 & \smallnum{27433178375735541316077789478608820936384401166606761} & \smallnum{27433175559339255535796932691536360537817896208124864} & \smallnum{2816396285780280856787072460398566504958481897} \\
2 & 3 & \smallnum{23369005367416078157885574436049353535201886057872113} & \smallnum{23369001738113584352983311887602567016676608507429348} & \smallnum{3629302493804902262548446786518525277550442765} \\
2 & 4 & \smallnum{26417136283665639994155192155661250705627384692494548} & \smallnum{26417132438547338626710059823713502289715748963195151} & \smallnum{3845118301367445132331947748415911635729299397} \\
2 & 5 & \smallnum{17272742142904996807128864378326947436541046080912646} & \smallnum{17272740308470634664761206884011653839115364644969634} & \smallnum{1834434362142367657494315293597425681435943012} \\
2 & 6 & \smallnum{24385047804791271335391422193694028716143603217098066} & \smallnum{24385045009976566232416460949699660033690204069527657} & \smallnum{2794814705102974961243994368682453399147570409} \\
3 & 4 & \smallnum{20320874242544569222535800521233764930250657255135288} & \smallnum{20320870990919738244607991395736301522937664712554878} & \smallnum{3251624830977927809125497463407312992542580410} \\
3 & 5 & \smallnum{15240655066833375904196480609127666525762826615153450} & \smallnum{15240653232399013494422981156631761266403512315870420} & \smallnum{1834434362409773499452495905259359314299283030} \\
3 & 6 & \smallnum{20320873512367754496724563295291371255942519401253571} & \smallnum{20320870857833323968141678625848354343243056783566841} & \smallnum{2654534430528582884669443016912699462617686730} \\
4 & 5 & \smallnum{15240655394154016938563652401259073180536669740444233} & \smallnum{15240653232399013494422981156631761266403512315870420} & \smallnum{2161755003444140671244627311914133157424573813} \\
4 & 6 & \smallnum{21336917330964114447149225178990828970935777657213120} & \smallnum{21336914395869134476909122358629821174622425575878836} & \smallnum{2935094979970240102820361007796313352081334284} \\
5 & 6 & \smallnum{17272743038540597328342912800590940138111547715453003} & \smallnum{17272740308470634664761206884011653839115364644969634} & \smallnum{2730069962663581705916579286298996183070483369} \\
\end{longtable}
\end{landscape}

\begin{howtoread}
The first two coefficient columns in \cref{tab:worked-couplings} are close because they differ only on intervals lying inside bands.  Their difference is the precise within-band residual for that block pair.  Summing the absolute values of such differences is unnecessary; the uniform profile bound $|F_v|\le3$ yields the simpler global estimate used in $\Gamma$.
\end{howtoread}

The field-correction arithmetic is summarized in \cref{tab:worked-fields}.  The columns are scaled so that
\begin{equation}\label{eq:worked-field-check}
4e_v=4a_v^0-\eta_vD_v-4a_v^\star
\end{equation}
can be checked by integer subtraction.  For block 1, for example,
\[
\eta_1=
\operatorname{round}\!\left(\frac{4a_1^0-4a_1^\star}{D_1}\right)
=-\mednum{121608653220394640131792521921827012932990685915755836745300153673135,}
\]
and the remainder in \eqref{eq:worked-field-check} is
\[
4e_1=\readnum{91993098399800992075465492602897143.}
\]
The bound $|e_v|\le D_v/8$ is verified in every row.  Because the table displays $4e_v$, the convenient integer check is $|4e_v|\le D_v/2$.  For block 1, the displayed residual $91993098399800992075465492602897143$ is smaller than $D_1/2$, as required.
\begin{landscape}
\scriptsize
\let\smallnum\numstr
\setlength{\tabcolsep}{1.5pt}
\setlength{\LTleft}{0pt}
\setlength{\LTright}{0pt}
\tablecert{\textbf{Purpose.}  The table certifies the integer field correction and the residual error in each block.  \textbf{Verification rule.}  Check $4e_v=4a_v^0-\eta_vD_v-4a_v^\star$ by integer subtraction and verify $|4e_v|\le D_v/2$.}
\begin{longtable}{c p{0.20\textheight} p{0.17\textheight} p{0.14\textheight} p{0.20\textheight} p{0.11\textheight}}
\caption{Exact logical-field correction.  All displayed field columns are multiplied by four.}\label{tab:worked-fields}\\
\toprule
$v$ & $4a_v^0$ & $4a_v^\star$ & $D_v$ & $\eta_v$ & $4e_v$ \\
\midrule
\endfirsthead
\multicolumn{6}{c}{\tablename\ \thetable\ (continued)}\\
\toprule
$v$ & $4a_v^0$ & $4a_v^\star$ & $D_v$ & $\eta_v$ & $4e_v$ \\
\midrule
\endhead
\midrule
\multicolumn{6}{r}{Continued on next page}\\
\endfoot
\bottomrule
\endlastfoot
1 & \smallnum{-23227760216999904400496634511560428002240935718792495585799850159117121591078392935076075583630791887268} & \smallnum{59946569427529543850729097518840776005094889033779729} & \smallnum{191004172827270839572243054176742164} & \smallnum{-121608653220394640131792521921827012932990685915755836745300153673135} & \smallnum{91993098399800992075465492602897143} \\
2 & \smallnum{-20646897884173472877318655544106810947660541066093452746135549401273460459000116801988589378943123704295} & \smallnum{52834264790768354735200710302385908165460135964273320} & \smallnum{165536949783634727629277313619846160} & \smallnum{-124726823293289057041984530587099588067121354739907606831762205114467} & \smallnum{44500147224446456465221426642419105} \\
3 & \smallnum{-16001345658772503665537712596128803649853182454880499211205269725821871149702350242447040798952925932636} & \smallnum{41657785519875287150786764786071028787207812961896975} & \smallnum{114602503696362503743345832506054152} & \smallnum{-139624747650957194248927671169972797731094030532099011391196785378386} & \smallnum{-22622512333360885622529828761470939} \\
4 & \smallnum{-19614552835556321709700007797235723545341044315144813932488116709618415424964229981711247984553556902256} & \smallnum{45721959758344851933280956454963430891626322209974773} & \smallnum{140069726739998615686311573062950156} & \smallnum{-140034205049642231928344418796909733781188573487259147666761466327671} & \smallnum{-67239674239640297051854277130310353} \\
5 & \smallnum{-13420483252853301239858135015841491156877958675776002961694190276450547110018236025673074055215771459374} & \smallnum{32513393540869648159184188040643415105518876960840054} & \smallnum{76401669130908335828897221670710146} & \smallnum{-175656937937551388015534071777519207387151471005203134839749376426987} & \smallnum{36483612940352395244995787676810674} \\
6 & \smallnum{-18582207673645379862578575837231175656770200985970200493102699693627551799352840020538131195353118101563} & \smallnum{44705915867809890200558139575548014376933260853094498} & \smallnum{114602503696362503743345832506054152} & \smallnum{-162144866598016402741684166715715899102730544112490726603040139389225} & \smallnum{29891598690064267239692162084116139} \\
\end{longtable}
\end{landscape}

\begin{arithmeticguide}
The quotient defining $\eta_v$ is not expected to be integral.  Rounding chooses the nearest integer multiplier of $D_v$; equation \eqref{eq:worked-field-check} then records the exact leftover.  The proof needs only the half-step bound, not a small decimal relative error.
\end{arithmeticguide}

\subsubsection{The load-lock scale, anchor, and due date}\label{app:worked-loadlock}

The four terms in \eqref{eq:U} evaluate to
\[
1=1,
\]
\[
\displaystyle\sum_{i<j}p_ip_j\min\{r_i,r_j\}
=
\mathnum[0.52\textwidth]{25654998533637856711218783177390546966600785528823415044693842834837405883779529535220745015765134014131401267942631440865121723230297964517549394530,}
\]
\[
\displaystyle\sum_i|\eta_{b(i)}|p_i^2
=
\mathnum[0.52\textwidth]{22351429418363254224931416082581126509365172927767481961232649263072711195064493397836417303196892711619529061808121836053790729275879854346002172010721524,}
\]
\[
\displaystyle\sum_{i<j}|\bar\eta_{ij}|p_ip_j
=
\mathnum[0.52\textwidth]{562093418052650571895673292558307306876005080517367319774850385852645843472540618343347088295348930077885807518366628326198584559884180249290425125667473428.}
\]
Therefore
\[
U=\mathnum{584444873126012359758461419859671610775917220045920330559498079809561389505010995520713040819290838554639350711576018104883816154281783333934391815227589483}
\]
and
\[
\max_v|\eta_v|=\mednum{175656937937551388015534071777519207387151471005203134839749376426987.}
\]
The load-lock scale is
\[
H=U+2\max_v|\eta_v|+2
=\mathnum{584444873126012359758461419859671610775917220045920330559498079809561389505010995520713392133166713657415381779719573143298590457223793740204071313980443459.}
\]
In particular,
\[
H-U=\mednum{351313875875102776031068143555038414774302942010406269679498752853976}>0.
\]
This positive difference is the numerical certificate for load-lock dominance.  A wrong integer load vector pays at least $H$, while every non-dominant term together can improve by less than $U$.
Every state job in block \(v\) receives the common tardiness ratio \(s_i=HL_v+\eta_v\).  All six values are positive and decrease strictly with \(v\), as visible in the full job table below.

The anchor data are
\[
\begin{aligned}
p_a&=H,\\
r_a&=1+\max_i r_i=\mednum{7031676545942439210557755733595268355444518632781564950501995,}\\
C^\star&=\mathnum{267886806177807976135173003247823899187721758959732956168790713347663219908907798227605566471005179125919095260487893414629014523117639278783931399668030172330838659432497804801216895928228214103907831843093389083896975601641119383754333813593128421093636747252,}\\
s_a&=1+\left\lceil C^\star/H^2\right\rceil=2.
\end{aligned}
\]
The computed ratio satisfies $0<C^\star/H^2\le1$, so the ceiling is $1$ and $s_a=2$.  Its tardy self-cost is
\[
p_a^2s_a=\mathnum{683151619446561368413615410444933623585099609900914566805231553229996594117852535253587426979235841601201985188673749202537372133835586671076633705422428293801024314335090374585925621334531619631706967943361532152450437040923960956249613431193297658282439578766783389222785684354764168571996115924105420591769362,}
\]
which exceeds \(C^\star\) by
\[
\mathnum{683151619446561368413615410444933623585099609900914298918425375422020458944849287429688239257476881868245816397960401539317463226037359065510162700243302374705763826441675745571402503695252835700307299913189201313791004543119159739353685202979193750450596485377699492247184043235380414238182522795684326955022110.}
\]
Thus no optimal canonical assignment can place the anchor tardy.  The common due date is the total processing time of all 49 jobs:
\[
d=\mathnum{584444873126012359758461419859671610775917220045920330559498079809561389505010995520713392133166713657415381779805899466147883766136104722754835310286860893.}
\]

\subsubsection{The constructed 49-job instance}\label{app:worked-jobs}

The decision instance consists of the 49 triples $(p_j,w_j^E,w_j^L)$, the due date $d$, and the threshold $B$.  \Cref{tab:worked-job-ratios,tab:worked-job-weights} lists every job.  The first table gives the intermediate ratios; the second gives the actual integer weights.  The equalities $w_j^E=p_jr_j$ and $w_j^L=p_js_j$ hold row by row.

\begin{howtoread}
The ratios are construction aids; the actual scheduling input contains only processing times and weights.  To audit one row, multiply the displayed $p_j$ by $r_j$ and $s_j$.  Jobs in the same block share a tardiness ratio, while their early ratios differ according to their processing values.  The anchor is recognizable by $p_a=H$ and $s_a=2$.
\end{howtoread}
\begin{landscape}
\tiny
\setlength{\tabcolsep}{1.5pt}
\setlength{\LTleft}{0pt}
\setlength{\LTright}{0pt}
\tablecert{\textbf{Purpose.}  The table records the ratio-level construction data for all 49 jobs of the yes-instance.  \textbf{Verification rule.}  Compute $r_j=f_v+\lambda_vp_j$ and $s_j=HL_v+\eta_v$ within each block; the anchor is checked from $p_a=H$, $r_a=1+\max r_j$, and $s_a=2$.}
\begin{longtable}{c c c c c p{0.18\textheight} p{0.27\textheight} p{0.32\textheight}}
\caption{All jobs: metadata, processing times, and early/tardiness ratios.}\label{tab:worked-job-ratios}\\
\toprule
job & block & pair & element & $\epsilon$ & $p_j$ & $r_j$ & $s_j$ \\
\midrule
\endfirsthead
\multicolumn{8}{c}{\tablename\ \thetable\ (continued)}\\
\toprule
job & block & pair & element & $\epsilon$ & $p_j$ & $r_j$ & $s_j$ \\
\midrule
\endhead
\midrule
\multicolumn{8}{r}{Continued on next page}\\
\endfoot
\bottomrule
\endlastfoot
$A_{1,0}$ & 1 & 0 & -- & -1 & \smallnum{1798465042647412146620280340569649349251262} & \smallnum{7031676542894308046120015312457131776806205113735858599565336} & \smallnum{93268361905489682802055503317404676472373959319656316709076762701138816358966377215617260857678937086056034256620059184940090822574587528772905120802618877369924761012351291} \\
$B_{1,0}$ & 1 & 0 & -- & 1 & \smallnum{1798465042647412146620280340569649349251259} & \smallnum{7031676542894308046120015312457131776806204874359059706729315} & \smallnum{93268361905489682802055503317404676472373959319656316709076762701138816358966377215617260857678937086056034256620059184940090822574587528772905120802618877369924761012351291} \\
$A_{1,1}$ & 1 & 1 & $e_{1}$ & 1 & \smallnum{1798465080848246712074448255018260184595356} & \smallnum{7031676545942439210557755733595268355444518632781564950501994} & \smallnum{93268361905489682802055503317404676472373959319656316709076762701138816358966377215617260857678937086056034256620059184940090822574587528772905120802618877369924761012351291} \\
$B_{1,1}$ & 1 & 1 & $e_{1}$ & -1 & \smallnum{1798465080848246712074448255018260184595355} & \smallnum{7031676545942439210557755733595268355444518552989298652889987} & \smallnum{93268361905489682802055503317404676472373959319656316709076762701138816358966377215617260857678937086056034256620059184940090822574587528772905120802618877369924761012351291} \\
$A_{1,2}$ & 1 & 2 & $e_{2}$ & 1 & \smallnum{1798465074481440951165420269276825045372255} & \smallnum{7031676545434417349818132330072245592338206029600187441128287} & \smallnum{93268361905489682802055503317404676472373959319656316709076762701138816358966377215617260857678937086056034256620059184940090822574587528772905120802618877369924761012351291} \\
\end{longtable}
\end{landscape}
\begin{landscape}
\tiny
\setlength{\tabcolsep}{1.5pt}
\setlength{\LTleft}{0pt}
\setlength{\LTright}{0pt}
\tablecert{\textbf{Purpose.}  The table records the physical positive-integer input triples for the 49-job yes-instance.  \textbf{Verification rule.}  Multiply each displayed $p_j$ by its ratio in \cref{tab:worked-job-ratios} to recover $w_j^E$ and $w_j^L$; no visual comparison of full digit strings is required.}
\begin{longtable}{c p{0.24\textheight} p{0.32\textheight} p{0.32\textheight}}
\caption{The 49 actual AWET job triples.  The processing-time column is repeated for completeness.}\label{tab:worked-job-weights}\\
\toprule
job & $p_j$ & $w_j^E=p_jr_j$ & $w_j^L=p_js_j$ \\
\midrule
\endfirsthead
\multicolumn{4}{c}{\tablename\ \thetable\ (continued)}\\
\toprule
job & $p_j$ & $w_j^E=p_jr_j$ & $w_j^L=p_js_j$ \\
\midrule
\endhead
\midrule
\multicolumn{4}{r}{Continued on next page}\\
\endfoot
\bottomrule
\endlastfoot
$A_{1,0}$ & \smallnum{1798465042647412146620280340569649349251262} & \smallnum{12646224453599219326700847020084580656881203566977028252810840500642549776005349984780184638810249454032} & \smallnum{167739888472010772804155684470915837396860058113042786420121542086130464472185935356531857830363261530701460745989355193681531979728334732921847323100907763626230075759021845176563827799336585636307469178944969079242} \\
$B_{1,0}$ & \smallnum{1798465042647412146620280340569649349251259} & \smallnum{12646224453599219326700847020084580656881182041436594740281366911416615681178165180884168191819035957585} & \smallnum{167739888472010772804155684470915837396859778307957069951073135919620512258156518234653898861413134300413357329540278294549885127945761696110589154998137903448675255486554121413977509083974177779675359404661932025369} \\
$A_{1,1}$ & \smallnum{1798465080848246712074448255018260184595356} & \smallnum{12646224727697089120677578060354606619030675665892566866192354528602991365727809470262888685477961139864} & \smallnum{167739892034940036146674719917218273705173013452656786985312077382621859929569231677192509127872212092091791937430192924315246415225397215089508226928894479275417299914762981123300691401579005619637691863566959204596} \\
$B_{1,1}$ & \smallnum{1798465080848246712074448255018260184595355} & \smallnum{12646224727697089120677578060354606619030668490712416265753534456713504391178923925670352193868926210385} & \smallnum{167739892034940036146674719917218273705172920184294881495629275327118542524892759303233189471555503015329090798613833957938030797964539536152422170894637859216232359823940406535771918496458203000760321938805946853305} \\
\end{longtable}
\end{landscape}

\subsubsection{The master identity and decision threshold}\label{app:worked-threshold}

For this instance,
\[
\sum_{v=1}^6D_v=\readnum{802217525874537526203420827542356930}
\]
and
\[
\Gamma=\frac98(6^2)W+\frac18\sum_vD_v
=\frac{\readnum{181803236180931295715614635714017263362700959729}}{4}.
\]
The exact constant coefficient and shifted constant are
\[
C_0=\frac{\mathnum{267886806177807976135173003247823899187721758959732956168790713347663219908907798227605566471005179125919095260487893414629014523117639278783931399668030172330838659432497804801216895928228214103907831843093234899276852755333078435488265487586345988979990076685}}{1},
\qquad
A=\frac{\mathnum{535773612355615952270346006495647798375443517919465912337581426695326439817815596455211132942010358251838190520975786829258029046235278557567862799336060344661677318864995609602433791856456428207815663686186408327918875608758123919221156633624554244928453356171}}{2}.
\]
The gap-dominance comparison is
\[
\frac{\Delta_{\min}}2=\readnum{508021752831719698653257002263808418187080836571303}
>\readnum{181803236180931295715614635714017263362700959731/2}=2\Gamma+1.
\]
The left side is the least possible main penalty for one nonzero integer prefix error.  The right side allows the residual to favor the yes-state by $\Gamma$, favor the no-state by another $\Gamma$, and still leaves one full unit for the integer floor.
The positive margin is
\[
\frac{\Delta_{\min}}2-(2\Gamma+1)
=\readnum{1015861702427258466010798389891902819110798972182875/2.}
\]
Finally,
\[
B=\left\lfloor A+\Gamma\right\rfloor
=\mathnum{267886806177807976135173003247823899187721758959732956168790713347663219908907798227605566471005179125919095260487893414629014523117639278783931399668030172330838659432497804801216895928228214103907831843093204164004888613424294783539481975740781438304901918017.}
\]

\subsubsection{The ``if'' case: constructing and evaluating the schedule}\label{app:worked-if}

Select \(S_1\) and \(S_5\), so \(x=(1,0,0,0,1,0)\) and \(\sgnvec=(+,-,-,-,+,-)\).  In a block with \(\sigma_v=+1\), the early jobs are
\[
B_{v0},A_{v1},A_{v2},A_{v3},
\]
and the tardy jobs are
\[
A_{v0},B_{v1},B_{v2},B_{v3}.
\]
For \(\sigma_v=-1\), the two lists are interchanged.  Thus the exact-cover assignment contains 24 early state jobs, 24 tardy state jobs, and the early anchor.  Every block has tardy load \(\beta_v\); for example, in selected block 1,
\[
(P_{10}+3)+P_{11}+P_{12}+P_{13}=\beta_1,
\]
whereas in unselected block 2,
\[
P_{20}+(P_{21}+1)+(P_{22}+1)+(P_{23}+1)=\beta_2.
\]

Because $C_k(x)=k$ for all $k$, every squared term in $\mathcal P(x)$ is zero, so the prefix penalty is exactly zero.  No cancellation among different prefixes is involved.  The residual is
\[
R(\sgnvec)=\frac{\readnum{-2343399975092575105438479690358125499641790995}}{2}.
\]
Consequently, the master identity gives the integer cost
\[
\Phi_{\rm cover}=A+R(\sgnvec)
=\mathnum{267886806177807976135173003247823899187721758959732956168790713347663219908907798227605566471005179125919095260487893414629014523117639278783931399668030172330838659432497804801216895928228214103907831843093204163958266104391515672057859076967098059714405782588.}
\]
This satisfies
\[
B-\Phi_{\rm cover}
=\readnum{46622509032779111481622898773683378590496135429}>0.
\]
Hence the constructed \AWET\ instance is a yes-instance.

For completeness, \cref{tab:worked-schedule} gives the entire compact canonical schedule.  The early jobs are listed in increasing $r_j$, which is chronological order before $d$: reading in the opposite direction, outward from $d$, gives the nonincreasing order of \cref{lem:compact}.  The anchor is last and completes at $d$.  The tardy jobs are listed chronologically in decreasing $s_j$.  The early block starts at
\[
\readnum{43163161424646654456155491275381998153208717,}
\]
which is precisely the total tardy processing load.  The final completion time exceeds $d$ by
\[
\readnum{43163161424646654456155491275381998153208717.}
\]
The equality of the initial start offset and final tardiness span is forced by $d=\sum_jp_j$: both equal the total processing load assigned to the tardy side.
\begin{landscape}
\tiny
\setlength{\tabcolsep}{1.5pt}
\setlength{\LTleft}{0pt}
\setlength{\LTright}{0pt}
\tablecert{\textbf{Purpose.}  The table certifies feasibility and the job-by-job objective value of the canonical schedule induced by the exact cover.  \textbf{Verification rule.}  Accumulate each completion time from the previous row, subtract from or add to $d$ according to the side, and multiply the deviation by the applicable physical weight.}
\begin{longtable}{r c c p{0.17\textheight} p{0.19\textheight} p{0.17\textheight} p{0.30\textheight}}
\caption{Canonical schedule induced by the exact cover.  Deviation means earliness on the early side and tardiness on the tardy side.}\label{tab:worked-schedule}\\
\toprule
pos. & job & side & start & completion & deviation & cost contribution \\
\midrule
\endfirsthead
\multicolumn{7}{c}{\tablename\ \thetable\ (continued)}\\
\toprule
pos. & job & side & start & completion & deviation & cost contribution \\
\midrule
\endhead
\midrule
\multicolumn{7}{r}{Continued on next page}\\
\endfoot
\bottomrule
\endlastfoot
1 & $A_{6,0}$ & early & \smallnum{43163161424646654456155491275381998153208717} & \smallnum{44961626467294066602775771615951647502459979} & \smallnum{584444873126012359758461419859671610775917220045920330559498079809561389505010995520713392133166713657415381779760937839680589699533328951138883662784400914} & \smallnum{7391021046306870481429364970463123228457244497618896383901032528237088047629893164689181897860752608959047450883215758443528948469840071326777014864694963664687301485681804321466959178132008213979305692781739749787876687494823254429411827840326437762426011456} \\
2 & $B_{5,0}$ & early & \smallnum{44961626467294066602775771615951647502459979} & \smallnum{46760091509941478749396051956521296851711238} & \smallnum{584444873126012359758461419859671610775917220045920330559498079809561389505010995520713392133166713657415381779759139374637942287386708670798314013435149655} & \smallnum{7391021046306870558934667438382867269939611419281542455547013701249375542744933875977893946565584608550892722427785307581119802805101252581892289203327690780743541920376191743636320184664671355893492628719973144171478097105588215938053021839793671487807403735} \\
3 & $A_{4,0}$ & early & \smallnum{46760091509941478749396051956521296851711238} & \smallnum{48558556552588890896016332297090946200962500} & \smallnum{584444873126012359758461419859671610775917220045920330559498079809561389505010995520713392133166713657415381779757340909595294875240088390457744364085898393} & \smallnum{7391021046306870621316984119622864377533200616381788211454711174880266143032127674642432652378116896760806303289884192300196669316446837582898672360297860693362951355794093334298058091392947021566113011008942959529078423578249687296422919248200774764755349212} \\
\end{longtable}
\end{landscape}

\begin{howtoread}
For each row of \cref{tab:worked-schedule}, completion equals start plus processing time.  On the early side, deviation is $d-C_j$; on the tardy side it is $C_j-d$.  Multiplying the deviation by the relevant weight gives the final column.  Consecutive rows also verify compactness because each start equals the preceding completion.
\end{howtoread}

\subsubsection{The ``only if'' case: numerical enforcement and exhaustive audit}\label{app:worked-onlyif}

Suppose a schedule has cost at most \(B\).  Choose an optimal schedule and replace it, using \cref{lem:compact}, by an optimal canonical schedule.  The three mechanical enforcement steps specialize as follows.

\par\medskip\noindent\textbf{Anchor enforcement.}\quad
If the anchor in this canonical schedule were tardy, its full self-cost would exceed \(C^\star\), as calculated in \cref{app:worked-loadlock}.  Therefore the selected optimal canonical schedule keeps it early.

\par\medskip\noindent\textbf{Load enforcement.}\quad
For a nonzero integer load-error vector \(y=b-\beta\), the kernel identity gives \(y^\top K y\ge2\).  The dominant loss is therefore at least \(H\), whereas all remaining assignment-dependent oscillation is less than \(U\).  Numerically,
\[
H-U=\mednum{351313875875102776031068143555038414774302942010406269679498752853976}>0.
\]
Thus every optimal block has load \(\beta_v\), and the arithmetic synchronization calculation restricts it to one of the two coded states.

\par\medskip\noindent\textbf{Prefix enforcement.}\quad
If a coded state is not prefix-exact, one residual \(C_k(x)-k\) is a nonzero integer and
\[
\mathcal P(x)\ge\frac{\Delta_{\min}}2.
\]
Consequently,
\[
\Phi_{\rm can}(\sgnvec)
\ge A+\frac{\Delta_{\min}}2-\Gamma
=B+\readnum{2031723404854516932021596779783805638221597944365757/4}
>B.
\]
Thus a schedule meeting the bound must be prefix-exact.  Successive differences then give exactly one selected subset-column in every element row, so the selected subset-columns form an exact cover.

The logical space has only \(2^6=64\) states.  \Cref{tab:worked-state-audit} evaluates every one using exact arithmetic.  The column \(\mathcal P\) is the weighted prefix penalty, \(R\) is the bounded residual, and the last column is \(\Phi_{\rm can}-B\).  Exactly one row is prefix-exact: \(x=100010\), corresponding to \(S_1,S_5\).  The least expensive non-prefix-exact row is \(x=010001\), with
\[
(C_1,\ldots,C_6)=(1,2,3,4,6,6),
\qquad
(C_k-k)_{k=1}^6=(0,0,0,0,1,0),
\]
and it still exceeds \(B\) by
\[
\readnum{507974865948322853769201117934529122335558302158437.}
\]
This exhaustive audit is not needed by the general proof, but it independently demonstrates the full ``only if'' mechanism on the instance.  It also shows why the residual cannot create a false positive: the closest violating state has only one unit of prefix error, yet its main-gap penalty still leaves an enormous positive margin over $B$.
\begin{landscape}
\tiny
\setlength{\tabcolsep}{1.5pt}
\setlength{\LTleft}{0pt}
\setlength{\LTright}{0pt}
\tablecert{\textbf{Purpose.}  The table exhausts the 64 coded block states and certifies the master identity and decision comparison.  \textbf{Verification rule.}  From each bit vector compute the prefix totals, $\mathcal P(x)$, $R(2x-\mathbf1)$, and finally $A+\mathcal P+R-B$; only its sign determines the threshold outcome.}
\begin{longtable}{c p{0.10\textheight} p{0.11\textheight} c p{0.18\textheight} p{0.16\textheight} p{0.24\textheight}}
\caption{Exact audit of all 64 coded states.}\label{tab:worked-state-audit}\\
\toprule
$x$ & $C(x)$ & $C(x)-(1,\ldots,6)$ & RPE? & $\mathcal P(x)$ & $R(\sigma)$ & $\Phi_{\rm can}-B$ \\
\midrule
\endfirsthead
\multicolumn{7}{c}{\tablename\ \thetable\ (continued)}\\
\toprule
$x$ & $C(x)$ & $C(x)-(1,\ldots,6)$ & RPE? & $\mathcal P(x)$ & $R(\sigma)$ & $\Phi_{\rm can}-B$ \\
\midrule
\endhead
\midrule
\multicolumn{7}{r}{Continued on next page}\\
\endfoot
\bottomrule
\endlastfoot
\texttt{000000} & \smallnum{0,0,0,0,0,0} & \smallnum{-1,-2,-3,-4,-5,-6} & no & \smallnum{92459962968399192009913285559484191110613765994619783/2} & \smallnum{11600998877981094437295696934454575715863598716} & \smallnum{46229947634389428753227151171780101505566758185668676} \\
\texttt{000001} & \smallnum{0,0,1,2,3,3} & \smallnum{-1,-2,-2,-2,-2,-3} & no & \smallnum{26417132704720167332446023625306355559058079565646449/2} & \smallnum{4635543695973209980922003998029907345967878441} & \smallnum{13208525537094734406609063830998247305120545075461734} \\
\texttt{000010} & \smallnum{0,0,0,1,2,3} & \smallnum{-1,-2,-3,-3,-3,-3} & no & \smallnum{41657785613395469788661376630301505329605362715667489/2} & \smallnum{6446598013118369388834497569617946793961239244} & \smallnum{20828853802486702779876148245989393778433634643833057} \\
\texttt{000011} & \smallnum{0,0,1,3,5,6} & \smallnum{-1,-2,-2,-1,0,0} & no & \smallnum{10160435966657714440716528464146977456280405576633423/2} & \smallnum{2211212793774066638377383919492274607136002338} & \smallnum{5080174743732605761600973705798479716098969249079118} \\
\texttt{000100} & \smallnum{1,1,2,2,2,3} & \smallnum{0,-1,-1,-2,-3,-3} & no & \smallnum{12192522571531490330843055990702283061668963371495459} & \smallnum{6569792868416433863856376462065723020141772575/2} & \smallnum{12192480405618879306236059015231585590214632767141815} \\
\texttt{000101} & \smallnum{1,1,3,4,5,6} & \smallnum{0,-1,0,0,0,0} & no & \smallnum{508021835561112469018547382243186460513545732887628} & \smallnum{-1490927535658854843250287395190987015486999407/2} & \smallnum{507975639288299406767196853440560360704197314147993} \\
\texttt{000110} & \smallnum{1,1,2,3,4,6} & \smallnum{0,-1,-1,-1,-1,0} & no & \smallnum{2032087126428642714640082682742701437568274047889732} & \smallnum{584501145579265109423232356220731491186201257/2} & \smallnum{2032041967870170271448708490699951043618178965750429} \\
\texttt{000111} & \smallnum{1,1,3,5,7,9} & \smallnum{0,-1,0,1,2,3} & no & \smallnum{7620326698928899517576780958295258675528221054251535} & \smallnum{-2016079333168860185850272928437986178301603987/2} & \smallnum{7620280240080187700322759129499865952219291228209610} \\
\texttt{001000} & \smallnum{0,1,1,2,2,3} & \smallnum{-1,-1,-2,-2,-3,-3} & no & \smallnum{14224609690766272722090151825148484113878217995534358} & \smallnum{7763973669385208097718398724036050233392943127/2} & \smallnum{14224568121944062181870271780688917627587494016765990} \\
\texttt{001001} & \smallnum{0,1,2,4,5,6} & \smallnum{-1,-1,-1,0,0,0} & no & \smallnum{1524065416760084351498199483907920681343431564614532} & \smallnum{-857867833573395045690101114987887581163123963/2} & \smallnum{1524019537017122331976747735198434683083800307812619} \\
\texttt{001010} & \smallnum{0,1,1,3,4,6} & \smallnum{-1,-1,-2,-1,-1,0} & no & \smallnum{4064174245663425105887178517188902489777528671928631} & \smallnum{1124040664479304999700991804881511018186790243/2} & \smallnum{4064129356874712112715749464025876426217197090083821} \\
\texttt{001011} & \smallnum{0,1,2,5,7,9} & \smallnum{-1,-1,-1,1,2,3} & no & \smallnum{8636370280127871400056433059959992896358106885978439} & \smallnum{-2037660913152134731874349461544434430228310109/2} & \smallnum{8636323810488369591165138219126333619825051096583453} \\
\texttt{001100} & \smallnum{1,2,3,4,4,6} & \smallnum{0,0,0,0,-1,0} & no & \smallnum{1016043538035810585169112863689946286355926164549607/2} & \smallnum{-1182490778102345647382811877996376096553660455} & \smallnum{507975135718081957414980145374166642486025853374417} \\
\texttt{001101} & \smallnum{1,2,4,6,7,9} & \smallnum{0,0,1,2,2,3} & no & \smallnum{18288783781761702797803452898468461770531204454467627/2} & \smallnum{-2558316549611407116266700789712031651750359716} & \smallnum{9144343881755256554670681278874512668918109801634166} \\
\texttt{001110} & \smallnum{1,2,3,5,6,9} & \smallnum{0,0,0,1,1,3} & no & \smallnum{11176479112628142341609128561034305570961572149078993/2} & \smallnum{-2340702277111156525146888025659512546732163084} & \smallnum{5588191764802748826824110229970198621652398667136481} \\
\texttt{001111} & \smallnum{1,2,4,7,9,12} & \smallnum{0,0,1,3,4,6} & no & \smallnum{62994699973295303883765882363836128733367579728936281/2} & \smallnum{-986458085956636288114197651076171918858378976} & \smallnum{31497303549380520752422724164061484786196030330849233} \\
\texttt{010000} & \smallnum{1,2,2,2,3,3} & \smallnum{0,0,-1,-2,-2,-3} & no & \smallnum{18288783781761702797803452898468461770531204454467627/2} & \smallnum{2718379939698918068719447052225688406643273431} & \smallnum{9144349158451745864995866265022354606638168195267313} \\
\texttt{010001} & \smallnum{1,2,3,4,6,6} & \smallnum{0,0,0,0,1,0} & no & \smallnum{1016043538035810585169112863689946286355926164549607/2} & \smallnum{-1452260537205991426410251515516526564104876435} & \smallnum{507974865948322853769201117934529122335558302158437} \\
\texttt{010010} & \smallnum{1,2,2,3,5,6} & \smallnum{0,0,-1,-1,0,0} & no & \smallnum{2032087043699249906073957737309083667753530465454601/2} & \smallnum{-601586563021439322247437019013514833823143029} & \smallnum{1015997469454016698773727717558594316046090734344340} \\
\texttt{010011} & \smallnum{1,2,3,5,8,9} & \smallnum{0,0,0,1,3,3} & no & \smallnum{19304827416914627022962031470553875861808981465475849/2} & \smallnum{-2042157077262767111460556300456733621500809526} & \smallnum{9652366215491191015889975371061708969855028556688467} \\
\texttt{010100} & \smallnum{2,3,4,4,5,6} & \smallnum{1,1,1,0,0,0} & no & \smallnum{1524065416760084351498199483907920681343431564614532} & \smallnum{-3505208405413028608632227805560228326840279201/2} & \smallnum{1524018213346836412159966264135089396913427469235000} \\
\texttt{010101} & \smallnum{2,3,5,6,8,9} & \smallnum{1,1,2,2,3,3} & no & \smallnum{14224609690766272722090151825148484113878217995534358} & \smallnum{-5976299399282367393250902925452031564173910365/2} & \smallnum{14224561251807527848082526296038092883546595233339244} \\
\texttt{010110} & \smallnum{2,3,4,5,7,9} & \smallnum{1,1,1,1,2,3} & no & \smallnum{8636370280127871400056433059959992896358106885978439} & \smallnum{-5821631403965462048076741324210368492923964495/2} & \smallnum{8636321918503124184501480117930402286858019748756260} \\
\texttt{010111} & \smallnum{2,3,5,7,10,12} & \smallnum{1,1,2,3,5,6} & no & \smallnum{38609654862604694435409592285212210168008257961867899} & \smallnum{-2832582472507637420862257871504179364116628921/2} & \smallnum{38609607995504412948766952950424345911602735228313507} \\
\texttt{011000} & \smallnum{1,3,3,4,5,6} & \smallnum{0,1,0,0,0,0} & no & \smallnum{508021835561112469018547382243186460513545732887628} & \smallnum{-2742659219569340114337207467384673829946821913/2} & \smallnum{507975013422457451524561309980524263860790084236740} \\
\texttt{011001} & \smallnum{1,3,4,6,8,9} & \smallnum{0,1,1,2,3,3} & no & \smallnum{12192522571531490330843055990702283061668963371495459} & \smallnum{-5774871312321993335257718569043704846207748185/2} & \smallnum{12192474233286788937022459458184070035500699592381435} \\
\texttt{011010} & \smallnum{1,3,3,5,7,9} & \smallnum{0,1,0,1,2,3} & no & \smallnum{7620326698928899517576780958295258675528221054251535} & \smallnum{-5713723500190507897365983799344361682281088773/2} & \smallnum{7620278391258104189498903371644430499031539238467217} \\
\texttt{011011} & \smallnum{1,3,4,7,10,12} & \smallnum{0,1,1,3,5,6} & no & \smallnum{36577567743369912044162496450766009115799003337829000} & \smallnum{-3285795667615997706453336328405400332401048307/2} & \smallnum{36577520649663033003339714320438916408782996462064915} \\
\texttt{011100} & \smallnum{2,4,5,6,7,9} & \smallnum{1,2,2,2,2,3} & no & \smallnum{26417132704720167332446023625306355559058079565646449/2} & \smallnum{-2590688921212174621078667225290826492494243578} & \smallnum{13208518310862117221224461830327023984386706613339715} \\
\texttt{011101} & \smallnum{2,4,6,8,10,12} & \smallnum{1,2,3,4,5,6} & no & \smallnum{92459962968399192009913285559484191110613765994619783/2} & \smallnum{-1171699987618261128718561768324028648543372430} & \smallnum{46229934861690563153871585157521398726962393778697530} \\
\texttt{011110} & \smallnum{2,4,5,7,9,12} & \smallnum{1,2,2,3,4,6} & no & \smallnum{71123048896253768418408453090674022521894454840115103/2} & \smallnum{-1914466058078617841348428079356537261236803195} & \smallnum{35561477082851780897762456293250003400094125508014425} \\
\texttt{011111} & \smallnum{2,4,6,9,12,15} & \smallnum{1,2,3,5,7,9} & no & \smallnum{171711359776874062425398128792875165751680870559027705/2} & \smallnum{2234592838178877356928256663909256765784551322} & \smallnum{85855636672220824158752492421035318280781360388825243} \\
\texttt{100000} & \smallnum{1,2,3,3,3,3} & \smallnum{0,0,0,-1,-2,-3} & no & \smallnum{7112304863367787048558233576052072215014675321363907} & \smallnum{4469185676428734639974452383618453753660435161/2} & \smallnum{7112261647151580030101624659619335519925711476341556} \\
\texttt{100001} & \smallnum{1,2,4,5,6,6} & \smallnum{0,0,1,1,1,0} & no & \smallnum{1524065290867530245621535300499514977054728315002104} & \smallnum{-3828932116026472559198788568433749976214041595/2} & \smallnum{1524017925592426999561326797446302255863899532741375} \\
\texttt{100010} & \smallnum{1,2,3,4,5,6} & \smallnum{0,0,0,0,0,0} & yes & \smallnum{0} & \smallnum{-2343399975092575105438479690358125499641790995/2} & \smallnum{-46622509032779111481622898773683378590496135429} \\
\texttt{100011} & \smallnum{1,2,4,6,8,9} & \smallnum{0,0,1,2,3,3} & no & \smallnum{11684500735970377861824508608459096601155417638607831} & \smallnum{-5181377842220618892778562069812336863375301013/2} & \smallnum{11684452694472411518691133315519133190671145275717393} \\
\texttt{100100} & \smallnum{2,3,5,5,5,6} & \smallnum{1,1,2,1,0,0} & no & \smallnum{7112304953291039626605244170687858693199131179307655/2} & \smallnum{-1642897833839748564407338933354584332285168248} & \smallnum{3556105382938640740730128774346067487699392629245648} \\
\texttt{100101} & \smallnum{2,3,6,7,8,9} & \smallnum{1,1,3,3,3,3} & no & \smallnum{38609654600028794974550092336842386566524088318341721/2} & \smallnum{-2856861750097112061173598401584372845141072342} & \smallnum{19304778992343602157339056091163863194573358342858587} \\
\texttt{100110} & \smallnum{2,3,5,6,7,9} & \smallnum{1,1,2,2,2,3} & no & \smallnum{23369001691353492518334739331847236795976805168320681/2} & \smallnum{-2887435656156262765869177658428740801511707513} & \smallnum{11684452507432044870080674893087031464931760397212896} \\
\texttt{100111} & \smallnum{2,3,6,8,10,12} & \smallnum{1,1,3,4,5,6} & no & \smallnum{89411831955032517195802001266025072347532491597294015/2} & \smallnum{-1371329609750044556718857840359533131297128238} & \smallnum{44705869155377603615032515010495767309917273826278838} \\
\texttt{101000} & \smallnum{1,3,4,5,5,6} & \smallnum{0,1,1,1,0,0} & no & \smallnum{3048130714821474844111052501795456588780621931229857/2} & \smallnum{-1416291236240951259705520293474649399774119141} & \smallnum{1524018490310455948280337641718506315425070516255856} \\
\texttt{101001} & \smallnum{1,3,5,7,8,9} & \smallnum{0,1,2,3,3,3} & no & \smallnum{32513393285487609174521013202387050799346841485639933/2} & \smallnum{-2910815701939971974622697752588051802093670789} & \smallnum{16256648281119057414464603074836844307305777973909246} \\
\texttt{101010} & \smallnum{1,3,4,6,7,9} & \smallnum{0,1,1,2,2,3} & no & \smallnum{19304827452883927735840547662954834691558295920242883/2} & \smallnum{-2988149699591832632959490425203579712125949189} & \smallnum{9652365287483219043263711968328063637883595158932321} \\
\texttt{101011} & \smallnum{1,3,5,8,10,12} & \smallnum{0,1,2,4,5,6} & no & \smallnum{83315570640491331395772922131569736580355244764592227/2} & \smallnum{-1752604202627271641960088598017985931375017468} & \smallnum{41657738116832417837790890202037341767875850332038714} \\
\texttt{101100} & \smallnum{2,4,6,7,7,9} & \smallnum{1,2,3,3,2,3} & no & \smallnum{18288783961608208431407906155925886948912865946459727} & \smallnum{-4084314154634278669949175910462295232693849119/2} & \smallnum{18288736468642085881444642277679003213449408924295236} \\
\texttt{101101} & \smallnum{2,4,7,9,10,12} & \smallnum{1,2,4,5,5,6} & no & \smallnum{54358329642810410073522008864851505228818401299543601} & \smallnum{-1203173126091839894142808813096473333170283847/2} & \smallnum{54358283590414801794778132889788170176265894039161746} \\
\texttt{101110} & \smallnum{2,4,6,8,9,12} & \smallnum{1,2,3,4,4,6} & no & \smallnum{41657785563038448371695634893137337266705215256836660} & \smallnum{-2904521074447760073967861550091889542548361589/2} & \smallnum{41657738659968865914991669005547633716444603307415934} \\
\texttt{101111} & \smallnum{2,4,7,10,12,15} & \smallnum{1,2,4,6,7,9} & no & \smallnum{95000071552711284678570944486074609385726115254890168} & \smallnum{5436759879421842113671664119871924723116170421/2} & \smallnum{95000028820282179156668072418247740817372636137735447} \\
\texttt{110000} & \smallnum{2,4,5,5,6,6} & \smallnum{1,2,2,1,1,0} & no & \smallnum{5588239752346762512942820663918461871318165870591515} & \smallnum{-1972916171753292355661563333421960762340587225/2} & \smallnum{5588193315079631403472713929477866656021944025057971} \\
\texttt{110001} & \smallnum{2,4,6,7,9,9} & \smallnum{1,2,3,3,4,3} & no & \smallnum{24385045189823071942422583338065564667048422933757369} & \smallnum{-4681404554002549632346815548109257693919923163/2} & \smallnum{24384997398311749708323838260998862108103735298555856} \\
\texttt{110010} & \smallnum{2,4,5,6,8,9} & \smallnum{1,2,2,2,3,3} & no & \smallnum{15748675197449610129145793971878043495418855194197242} & \smallnum{-5116633098989866786085864820203688652770927357/2} & \smallnum{15748627188324015401388472025286704889258688133493632} \\
\texttt{110011} & \smallnum{2,4,6,8,11,12} & \smallnum{1,2,3,4,6,6} & no & \smallnum{51818220943396554223386763530036800130264476902332730} & \smallnum{-2364981555911960650937958462292993218209296557/2} & \smallnum{51818174310096731034582509157398640479452027122444520} \\
\texttt{110100} & \smallnum{3,5,7,7,8,9} & \smallnum{2,3,4,3,3,3} & no & \smallnum{56898439608343667808794537993847642585237610204153173/2} & \smallnum{-1018830456563316929893399043458879954556380044} & \smallnum{28449173334532332108256410199865849329423009870456611} \\
\texttt{110101} & \smallnum{3,5,8,9,11,12} & \smallnum{2,3,5,5,6,6} & no & \smallnum{137165879275034555621572308059401490525942975482242553/2} & \smallnum{562020332282294534584335856993784931735286271} & \smallnum{68582894748728564860256759710377673752440578801167616} \\
\texttt{110110} & \smallnum{3,5,7,8,10,12} & \smallnum{2,3,4,4,5,6} & no & \smallnum{107700616963347390030046446923030328366246013483105467/2} & \smallnum{-428933916737463473860922474935610742346976297} & \smallnum{53850262601930733044735820696933760743196423719336505} \\
\texttt{110111} & \smallnum{3,5,8,10,13,15} & \smallnum{2,3,5,6,8,9} & no & \smallnum{222513537246979547172346630756607483985182108051134115/2} & \smallnum{3881986834771729696533391711816050327015173387} & \smallnum{111256727054667563125079083008036525304325540365500513} \\
\texttt{111000} & \smallnum{2,5,6,7,8,9} & \smallnum{1,3,3,3,3,3} & no & \smallnum{46738003969006594478846850452733369934740820044543769/2} & \smallnum{-1008039666527062494975081365476331380224187569} & \smallnum{23368955525654585479537001347626390986723189122844384} \\
\texttt{111001} & \smallnum{2,5,7,9,11,12} & \smallnum{1,3,4,5,6,6} & no & \smallnum{124973356559625861274089733052724284212687447738009159/2} & \smallnum{292250572876891751351735544092719616603831192} & \smallnum{62486633121254458281112688974438757694747499797595840} \\
\texttt{111010} & \smallnum{2,5,6,8,10,12} & \smallnum{1,3,3,4,5,6} & no & \smallnum{97540181324010316700098759381916055715749223323496063/2} & \smallnum{-745463767735576210734736203607836011140074605} & \smallnum{48770044465732345381649240052562895745722759846433495} \\
\texttt{111011} & \smallnum{2,5,7,10,13,15} & \smallnum{1,3,4,6,8,9} & no & \smallnum{210321014531570852824864055749930277671926580306900721/2} & \smallnum{3284896434331959741508659992260211168758427525} & \smallnum{105160465099872815511567840479966202591858618236637954} \\
\texttt{111100} & \smallnum{3,6,8,9,10,12} & \smallnum{2,4,5,5,5,6} & no & \smallnum{66550853027248106875485864955108345911608713966311834} & \smallnum{4422425587528389124175597442366164077864612819/2} & \smallnum{66550809787651855406856498139248138590374912223378312} \\
\texttt{111101} & \smallnum{3,6,9,11,13,15} & \smallnum{2,4,6,7,8,9} & no & \smallnum{127005443718426874750016428613733624225204453388923365} & \smallnum{12893196026276777822469953277096892775683318909/2} & \smallnum{127005404714215842655581410945051334269335000555342888} \\
\texttt{111110} & \smallnum{3,6,8,10,12,15} & \smallnum{2,4,5,6,7,9} & no & \smallnum{107192594937148981480534800576331450068516427921658401} & \smallnum{9271087391999643035145542389931421130881986373/2} & \smallnum{107192554121883632247532389245443716529911152687411656} \\
\texttt{111111} & \smallnum{3,6,9,12,15,18} & \smallnum{2,4,6,8,10,12} & no & \smallnum{184919925936798384019826571118968382221227531989239566} & \smallnum{23201997756075195145273056797260142194841659201/2} & \smallnum{184919892086988216824600214851837852346982788734829235} \\
\end{longtable}
\end{landscape}

\subsubsection{How to reproduce the numerical audit}\label{app:worked-repro}

The numerical appendix is self-contained and can be regenerated with exact integer and rational arithmetic from the following sequence.
\begin{enumerate}[label=\arabic*.]
\item Enter the incidence matrix with elements in rows and subsets in columns, then compute $d_{v,k}$ by cumulative sums down subset-column $v$.
\item Form $G,M,T$, then $P_{vh}$ and $\beta_v$ from \eqref{eq:P0}--\eqref{eq:Ph}.
\item Form $\lambda_v,L_v,K$, and $f=K\beta$; verify $f_v-f_{v+1}=2\sum_{u\le v}\beta_u$.
\item Generate the 24 atoms from the two displayed effective-ratio formulas, group them into bands, and subtract endpoints to obtain $\Delta_k$ and $W$.
\item Expand the early objective to obtain $a_v^0$ and $J_{uv}$, form $a_v^\star,D_v,\eta_v,e_v$, and verify \eqref{eq:worked-field-check}.
\item Form $U,H$, all tardiness ratios, the anchor data, $C_0,A,\Gamma$, and $B$.
\item Enumerate the $2^6=64$ logical vectors $x$; for each, compute $C_k(x)$, $\mathcal P(x)$, $R(2x-\mathbf1)$, and $A+\mathcal P+R-B$.
\item For the exact-cover state, sort the early jobs by increasing $r_j$ and the tardy jobs by decreasing $s_j$, then accumulate start times, completion times, deviations, and costs.
\end{enumerate}
All divisions occur in rational quantities with small fixed denominators; no floating-point approximation is necessary.

\clearpage
\subsection{A ``no'' instance}\label{app:worked-no}

\begin{howtoread}
\textbf{Theorem-verification lane.}  Read the modified incidence matrix, the exhaustive two-column source audit, the single failed prefix of the closest vector, the no-instance gap summary, and the three reverse gates.  The full 49-job catalogues may be skipped.  \textbf{Full arithmetic-audit lane.}  Recompute every hatted quantity from the same formulas used for the yes-instance and verify all 64 coded-state differences in \cref{tab:no-state-audit}.
\end{howtoread}

The second example is deliberately as close as possible to the first one.  Starting from the yes-instance in \cref{app:worked-source}, exchange \(e_3\) and \(e_4\) between subset-columns \(S_1\) and \(S_3\):
\[
 S_1=\{e_1,e_2,e_3\},\quad S_3=\{e_2,e_4,e_6\}
 \quad\longrightarrow\quad
 \widehat S_1=\{e_1,e_2,e_4\},\quad
 \widehat S_3=\{e_2,e_3,e_6\}.
\]
All other subset-columns are unchanged.  This is a degree-preserving \(2\)-switch: each of the two changed columns still contains three elements, and the exchanged elements still occur in three columns overall.  The example therefore remains a valid three-uniform, three-regular \RXC\ instance, but the unique cover of the first example is destroyed.

\begin{howtoread}
The phrases ``failure of the if route'' and ``failure of the only-if supposition'' below do not mean that either proved implication is false.  In a no-instance, the forward proof has no source witness from which to build a low-cost schedule, while a hypothetical target witness in the reverse proof is driven backward to an impossible exact cover.  The example traces both phenomena numerically.
\end{howtoread}

\subsubsection{The modified restricted exact-cover instance}\label{app:no-source}

Let \(\widehat X=\{e_1,\ldots,e_6\}\), and let
\[
\begin{aligned}
\widehat S_1&=\{e_1,e_2,e_4\}, &
\widehat S_2&=\{e_1,e_2,e_5\}, &
\widehat S_3&=\{e_2,e_3,e_6\},\\
\widehat S_4&=\{e_1,e_3,e_6\}, &
\widehat S_5&=\{e_4,e_5,e_6\}, &
\widehat S_6&=\{e_3,e_4,e_5\}.
\end{aligned}
\]
With elements in rows and subset-columns in columns, the incidence matrix is
\begin{equation}\label{eq:no-incidence}
\widehat A^{\mathrm{inc}}=
\begin{pmatrix}
1&1&0&1&0&0\\
1&1&1&0&0&0\\
0&0&1&1&0&1\\
1&0&0&0&1&1\\
0&1&0&0&1&1\\
0&0&1&1&1&0
\end{pmatrix}.
\end{equation}
Every row and every column contains three ones.  Its intersection-cardinality matrix is
\[
\begin{pmatrix}
3&2&1&1&1&1\\
2&3&1&1&1&1\\
1&1&3&2&1&1\\
1&1&2&3&1&1\\
1&1&1&1&3&2\\
1&1&1&1&2&3
\end{pmatrix}.
\]
Every off-diagonal entry is positive.  Hence no two subset-columns are disjoint.  An exact cover of six elements by \(3\)-sets would contain exactly two columns: if \(x\) were a cover, then summing the six row equations would give
\[
 6=\mathbf 1^\top\widehat A^{\mathrm{inc}}x
   =3\sum_{v=1}^6x_v,
 \qquad\text{so}\qquad \sum_vx_v=2.
\]
Those two selected columns would have to be disjoint, contrary to the displayed matrix.  Thus this is a no-instance of \RXC.

\begin{figure}[ht]
\centering
\begin{tikzpicture}[font=\small]
\matrix (M) [matrix of nodes,
  nodes={draw,minimum width=8mm,minimum height=7mm,anchor=center},
  column 1/.style={nodes={fill=figureopen}},
  column 5/.style={nodes={fill=figureopen}},
  column sep=-\pgflinewidth,row sep=-\pgflinewidth] {
1&1&0&1&0&0\\
1&1&1&0&0&0\\
0&0&1&1&0&1\\
1&0&0&0&1&1\\
0&1&0&0&1&1\\
0&0&1&1&1&0\\
};
\foreach \i in {1,...,6}{
  \node[left=4mm of M-\i-1] {$e_{\i}$};
  \node[above=3mm of M-1-\i] {$\widehat S_{\i}$};
}
\node[draw,rounded corners,fill=figurehard,align=left,text width=38mm]
  (miss) at ([xshift=31mm,yshift=10mm]M-3-6.east)
  {row $e_3$: selected sum $0$\\(uncovered)};
\node[draw,rounded corners,fill=figurehard,align=left,text width=38mm]
  (double) at ([xshift=31mm,yshift=-10mm]M-4-6.east)
  {row $e_4$: selected sum $2$\\(double-covered)};
\draw[-{Latex[length=2mm]},thin] (miss.west) -- ([xshift=1mm]M-3-6.east);
\draw[-{Latex[length=2mm]},thin] (double.west) -- ([xshift=1mm]M-4-6.east);
\node[draw,rounded corners,fill=figureopen,below=8mm of M]
  {the former cover columns $\widehat S_1,\widehat S_5$ now fail};
\end{tikzpicture}
\caption{The no-instance incidence matrix.  The highlighted pair was the exact cover in \cref{fig:worked-incidence}; after the \(2\)-switch it leaves \(e_3\) uncovered and covers \(e_4\) twice.}
\label{fig:no-incidence}
\end{figure}

\subsubsection{Prefix data and an exhaustive source-level audit}\label{app:no-rpe}

Define
\[
 \widehat d_{v,k}
 =|\widehat S_v\cap\{e_1,\ldots,e_k\}|,
 \qquad
 \widehat C_k(x)=\sum_{v=1}^6\widehat d_{v,k}x_v.
\]
The full prefix matrix is as follows.
\begin{table}[ht]
\centering
\caption{Prefix data for the no-instance.  Each column is the cumulative sum down the corresponding column of \eqref{eq:no-incidence}.}\label{tab:no-prefix}
\begin{tabular}{c c c c c c c}
\toprule
Prefix & $\widehat S_1$ & $\widehat S_2$ & $\widehat S_3$ & $\widehat S_4$ & $\widehat S_5$ & $\widehat S_6$\\
\midrule
$k=1$ & 1 & 1 & 0 & 1 & 0 & 0 \\
$k=2$ & 2 & 2 & 1 & 1 & 0 & 0 \\
$k=3$ & 2 & 2 & 2 & 2 & 0 & 1 \\
$k=4$ & 3 & 2 & 2 & 2 & 1 & 2 \\
$k=5$ & 3 & 3 & 2 & 2 & 2 & 3 \\
$k=6$ & 3 & 3 & 3 & 3 & 3 & 3 \\
\bottomrule
\end{tabular}
\end{table}

The six prefix equations are
\begin{align*}
\widehat C_1(x)&=x_1+x_2+x_4=1,\\
\widehat C_2(x)&=2x_1+2x_2+x_3+x_4=2,\\
\widehat C_3(x)&=2x_1+2x_2+2x_3+2x_4+x_6=3,\\
\widehat C_4(x)&=3x_1+2x_2+2x_3+2x_4+x_5+2x_6=4,\\
\widehat C_5(x)&=3x_1+3x_2+2x_3+2x_4+2x_5+3x_6=5,\\
\widehat C_6(x)&=3(x_1+x_2+x_3+x_4+x_5+x_6)=6.
\end{align*}
\begin{landscape}
\small
\setlength{\tabcolsep}{3pt}
\setlength{\LTleft}{0pt}
\setlength{\LTright}{0pt}
\tablecert{\textbf{Purpose.}  The table exhausts the only possible two-column exact-cover candidates in the modified source instance.  \textbf{Verification rule.}  Intersect each pair, identify uncovered and doubled rows, or equivalently subtract $(1,\ldots,6)$ from its prefix-total vector; no row may have zero error.}
\begin{longtable}{c c c c p{0.23\textheight} p{0.23\textheight}}
\caption{Exhaustive audit of the fifteen two-column candidates.}\label{tab:no-pair-audit}\\
\toprule
selected columns & intersection & uncovered & doubled & $\widehat C(x)$ & $\widehat C(x)-(1,\ldots,6)$\\
\midrule
\endfirsthead
\multicolumn{6}{c}{\tablename\ \thetable\ (continued)}\\
\toprule
selected columns & intersection & uncovered & doubled & $\widehat C(x)$ & $\widehat C(x)-(1,\ldots,6)$\\
\midrule
\endhead
\midrule
\multicolumn{6}{r}{Continued on next page}\\
\endfoot
\bottomrule
\endlastfoot
$\widehat S_1,\widehat S_2$ & $e_1$,$e_2$ & $e_3$,$e_6$ & $e_1$,$e_2$ & \smallnum{2,4,4,5,6,6} & \smallnum{1,2,1,1,1,0} \\
$\widehat S_1,\widehat S_3$ & $e_2$ & $e_5$ & $e_2$ & \smallnum{1,3,4,5,5,6} & \smallnum{0,1,1,1,0,0} \\
$\widehat S_1,\widehat S_4$ & $e_1$ & $e_5$ & $e_1$ & \smallnum{2,3,4,5,5,6} & \smallnum{1,1,1,1,0,0} \\
$\widehat S_1,\widehat S_5$ & $e_4$ & $e_3$ & $e_4$ & \smallnum{1,2,2,4,5,6} & \smallnum{0,0,-1,0,0,0} \\
$\widehat S_1,\widehat S_6$ & $e_4$ & $e_6$ & $e_4$ & \smallnum{1,2,3,5,6,6} & \smallnum{0,0,0,1,1,0} \\
$\widehat S_2,\widehat S_3$ & $e_2$ & $e_4$ & $e_2$ & \smallnum{1,3,4,4,5,6} & \smallnum{0,1,1,0,0,0} \\
$\widehat S_2,\widehat S_4$ & $e_1$ & $e_4$ & $e_1$ & \smallnum{2,3,4,4,5,6} & \smallnum{1,1,1,0,0,0} \\
$\widehat S_2,\widehat S_5$ & $e_5$ & $e_3$ & $e_5$ & \smallnum{1,2,2,3,5,6} & \smallnum{0,0,-1,-1,0,0} \\
$\widehat S_2,\widehat S_6$ & $e_5$ & $e_6$ & $e_5$ & \smallnum{1,2,3,4,6,6} & \smallnum{0,0,0,0,1,0} \\
$\widehat S_3,\widehat S_4$ & $e_3$,$e_6$ & $e_4$,$e_5$ & $e_3$,$e_6$ & \smallnum{1,2,4,4,4,6} & \smallnum{0,0,1,0,-1,0} \\
$\widehat S_3,\widehat S_5$ & $e_6$ & $e_1$ & $e_6$ & \smallnum{0,1,2,3,4,6} & \smallnum{-1,-1,-1,-1,-1,0} \\
$\widehat S_3,\widehat S_6$ & $e_3$ & $e_1$ & $e_3$ & \smallnum{0,1,3,4,5,6} & \smallnum{-1,-1,0,0,0,0} \\
$\widehat S_4,\widehat S_5$ & $e_6$ & $e_2$ & $e_6$ & \smallnum{1,1,2,3,4,6} & \smallnum{0,-1,-1,-1,-1,0} \\
$\widehat S_4,\widehat S_6$ & $e_3$ & $e_2$ & $e_3$ & \smallnum{1,1,3,4,5,6} & \smallnum{0,-1,0,0,0,0} \\
$\widehat S_5,\widehat S_6$ & $e_4$,$e_5$ & $e_1$,$e_2$ & $e_4$,$e_5$ & \smallnum{0,0,1,3,5,6} & \smallnum{-1,-2,-2,-1,0,0} \\
\end{longtable}
\end{landscape}

No row of \cref{tab:no-pair-audit} has a zero error vector, so the associated \RPE\ instance is also a no-instance.  The last prefix equation forces any candidate to select exactly two columns; the table therefore exhausts every possible prefix-exact candidate, while each of the other binary vectors fails already at $k=6$.  The closest visual continuation of the yes-example is the former choice
\[
 x^\dagger=(1,0,0,0,1,0).
\]
It gives
\[
 \widehat C(x^\dagger)=(1,2,2,4,5,6),
 \qquad
 \widehat C(x^\dagger)-(1,\ldots,6)=(0,0,-1,0,0,0).
\]
Thus the entire failure is localized at the third prefix.  Successive differences make the same defect pointwise:
\[
\begin{array}{lll}
\widehat C_1-\widehat C_0=x_1+x_2+x_4=1,&
\widehat C_2-\widehat C_1=x_1+x_2+x_3=1,&
\widehat C_3-\widehat C_2=x_3+x_4+x_6=0,\\
\widehat C_4-\widehat C_3=x_1+x_5+x_6=2,&
\widehat C_5-\widehat C_4=x_2+x_5+x_6=1,&
\widehat C_6-\widehat C_5=x_3+x_4+x_5=1.
\end{array}
\]
The zero in the \(e_3\) row and the two in the \(e_4\) row are exactly the uncovered and double-covered elements in \cref{fig:no-incidence}.

\subsubsection{The numerical compiler for the no-instance}\label{app:no-compiler}

The dimensions and global scales are unchanged:
\[
\begin{aligned}
 n&=6,\qquad N=7,\\
 G&=7^{20},\qquad M=7^{40},\qquad T=7^{50},\\
 (Q_0,Q_1,Q_2,Q_3)&=(10,100,1000,10000).
\end{aligned}
\]
Only two incidence positions have changed.  Consequently,
\[
 \widehat P_{1,3}=P_{1,3}-M,\qquad
 \widehat P_{3,2}=P_{3,2}+M,
\]
and
\[
 \widehat\beta_1=\beta_1-M,\qquad
 \widehat\beta_3=\beta_3+M.
\]
All base values and balanced loads are displayed in \cref{tab:no-P}; the unchanged rows are repeated so that this subsection independently specifies the no-instance.

\begin{landscape}
\scriptsize
\let\smallnum\numstr
\setlength{\tabcolsep}{1.5pt}
\setlength{\LTleft}{0pt}
\setlength{\LTright}{0pt}
\tablecert{\textbf{Purpose.}  The table certifies integral bases and balanced target loads for every no-instance block.  \textbf{Verification rule.}  Reuse $P_{vh}=T+M\kappa(k_{vh})+Q_h$ with the modified incidence data and verify both coded orientations have load $\widehat\beta_v$.}
\begin{longtable}{c c c p{0.16\textheight} p{0.16\textheight} p{0.16\textheight} p{0.16\textheight} p{0.18\textheight}}
\caption{Exact base processing values and balanced block loads for the no-instance.}\label{tab:no-P}\\
\toprule
$v$ & $\widehat S_v$ & $(\kappa_1,\kappa_2,\kappa_3)$ & $\widehat P_{v0}$ & $\widehat P_{v1}$ & $\widehat P_{v2}$ & $\widehat P_{v3}$ & $\widehat\beta_v$ \\
\midrule
\endfirsthead
\multicolumn{8}{c}{\tablename\ \thetable\ (continued)}\\
\toprule
$v$ & $\widehat S_v$ & $(\kappa_1,\kappa_2,\kappa_3)$ & $\widehat P_{v0}$ & $\widehat P_{v1}$ & $\widehat P_{v2}$ & $\widehat P_{v3}$ & $\widehat\beta_v$ \\
\midrule
\endhead
\midrule
\multicolumn{8}{r}{Continued on next page}\\
\endfoot
\bottomrule
\endlastfoot
1 & $\{1,2,4\}$ & $(6,5,3)$ & \smallnum{1798465042647412146620280340569649349251259} & \smallnum{1798465080848246712074448255018260184595355} & \smallnum{1798465074481440951165420269276825045372254} & \smallnum{1798465061747829429347364297793954766933252} & \smallnum{7193860259724929239207513162658689346152123} \\
2 & $\{1,2,5\}$ & $(6,5,2)$ & \smallnum{1798465042647412146620280340569649349251259} & \smallnum{1798465080848246712074448255018260184595355} & \smallnum{1798465074481440951165420269276825045372254} & \smallnum{1798465055381023668438336312052519627709251} & \smallnum{7193860253358123478298485176917254206928122} \\
3 & $\{2,3,6\}$ & $(5,4,1)$ & \smallnum{1798465042647412146620280340569649349251259} & \smallnum{1798465074481440951165420269276825045371354} & \smallnum{1798465068114635190256392283535389906148253} & \smallnum{1798465049014217907529308326311084488485250} & \smallnum{7193860234257706195571401219692948789256119} \\
4 & $\{1,3,6\}$ & $(6,4,1)$ & \smallnum{1798465042647412146620280340569649349251259} & \smallnum{1798465080848246712074448255018260184595355} & \smallnum{1798465068114635190256392283535389906148253} & \smallnum{1798465049014217907529308326311084488485250} & \smallnum{7193860240624511956480429205434383928480120} \\
5 & $\{4,5,6\}$ & $(3,2,1)$ & \smallnum{1798465042647412146620280340569649349251259} & \smallnum{1798465061747829429347364297793954766923352} & \smallnum{1798465055381023668438336312052519627700251} & \smallnum{1798465049014217907529308326311084488485250} & \smallnum{7193860208790483151935289276727208232360115} \\
6 & $\{3,4,5\}$ & $(4,3,2)$ & \smallnum{1798465042647412146620280340569649349251259} & \smallnum{1798465068114635190256392283535389906147353} & \smallnum{1798465061747829429347364297793954766924252} & \smallnum{1798465055381023668438336312052519627709251} & \smallnum{7193860227890900434662373233951513650032118} \\
\end{longtable}
\end{landscape}

\begin{landscape}
\scriptsize
\let\smallnum\numstr
\setlength{\tabcolsep}{1.5pt}
\setlength{\LTleft}{0pt}
\setlength{\LTright}{0pt}
\tablecert{\textbf{Purpose.}  The table certifies the hatted block loads, kernel levels, and affine fields.  \textbf{Verification rule.}  Compute one row from $\widehat f=\widehat K\widehat\beta$ and verify the recurrence $\widehat f_v-\widehat f_{v+1}=2\sum_{u\le v}\widehat\beta_u$.}
\begin{longtable}{c p{0.19\textheight} p{0.13\textheight} p{0.13\textheight} p{0.30\textheight}}
\caption{Block loads and minimum-kernel data for the no-instance.}\label{tab:no-kernel}\\
\toprule
$v$ & $\widehat\beta_v$ & $\lambda_v$ & $L_v=2\lambda_v$ & $\widehat f_v$ \\
\midrule
\endfirsthead
\multicolumn{5}{c}{\tablename\ \thetable\ (continued)}\\
\toprule
$v$ & $\widehat\beta_v$ & $\lambda_v$ & $L_v=2\lambda_v$ & $\widehat f_v$ \\
\midrule
\endhead
\midrule
\multicolumn{5}{r}{Continued on next page}\\
\endfoot
\bottomrule
\endlastfoot
1 & \smallnum{7193860259724929239207513162658689346152123} & \smallnum{79792266297612007} & \smallnum{159584532595224014} & \smallnum{6888172941284439601043621954748764781708332706313514011566498} \\
2 & \smallnum{7193860253358123478298485176917254206928122} & \smallnum{79792266297612006} & \smallnum{159584532595224012} & \smallnum{6888172941284439586655901435298906303293306380996135319262252} \\
3 & \smallnum{7193860234257706195571401219692948789256119} & \smallnum{79792266297612005} & \smallnum{159584532595224010} & \smallnum{6888172941284439557880460409132800868281309701844248213101762} \\
4 & \smallnum{7193860240624511956480429205434383928480120} & \smallnum{79792266297612004} & \smallnum{159584532595224008} & \smallnum{6888172941284439514717298914451283042126510583306463528429034} \\
5 & \smallnum{7193860208790483151935289276727208232360115} & \smallnum{79792266297612003} & \smallnum{159584532595224006} & \smallnum{6888172941284439457166416938520741303010853053899910986796066} \\
6 & \smallnum{7193860227890900434662373233951513650032118} & \smallnum{79792266297612002} & \smallnum{159584532595224004} & \smallnum{6888172941284439385227814545009233260024616971038941980442868} \\
\end{longtable}
\end{landscape}

The minimum-kernel data in \cref{tab:no-kernel} show that the transfer of one \(M\)-unit from \(\widehat\beta_1\) to \(\widehat\beta_3\) changes \(\widehat f_1\) and \(\widehat f_2\), but cancels from \(\widehat f_3,\ldots,\widehat f_6\), because rows \(3,\ldots,6\) of the minimum kernel assign the same coefficient to blocks \(1\) and \(3\).

The \(24\) effective atoms are obtained from the same two formulas
\[
 \widehat\rho_{vh}=\widehat f_v+\lambda_v(2\widehat P_{vh}+1)
 \quad(h=1,2,3),
 \qquad
 \widehat\rho_{v0}=\widehat f_v+\lambda_v(2\widehat P_{v0}+3).
\]
Taking exact minima and maxima within each element band gives \cref{tab:no-bands}.  The narrowest main gap is \(\widehat\Delta_3\); this is precisely the gap charged to the single failed prefix of \(x^\dagger\).

\begin{landscape}
\scriptsize
\let\smallnum\numstr
\setlength{\tabcolsep}{1.5pt}
\setlength{\LTleft}{0pt}
\setlength{\LTright}{0pt}
\tablecert{\textbf{Purpose.}  The table certifies the no-instance band endpoints, widths, and separating gaps.  \textbf{Verification rule.}  Group hatted atoms by element, take endpoint differences, and use the resulting signs and minimum gap in the hatted residual-dominance calculation.}
\begin{longtable}{c p{0.24\textheight} p{0.24\textheight} p{0.18\textheight} p{0.24\textheight}}
\caption{Effective bands, widths, and gaps for the no-instance.}\label{tab:no-bands}\\
\toprule
band & lower endpoint & upper endpoint & width & gap below \\
\midrule
\endfirsthead
\multicolumn{5}{c}{\tablename\ \thetable\ (continued)}\\
\toprule
band & lower endpoint & upper endpoint & width & gap below \\
\midrule
\endhead
\midrule
\multicolumn{5}{r}{Continued on next page}\\
\endfoot
\bottomrule
\endlastfoot
$\widehat{\mathcal R}_1$ & \smallnum{7175180150600438722954775936120363829879968940409256887323878} & \smallnum{7175180150600438820071889461507325841908480593525868478033475} & \smallnum{97117113525386962012028511653116611590709597} & \smallnum{1016043624362133281659083514184113553246143428037817} \\
$\widehat{\mathcal R}_2$ & \smallnum{7175180149584395048235620810919552597607007679507211777723307} & \smallnum{7175180149584395098592642654461280315695855387163113459286061} & \smallnum{50357021843541727718088847707655901681562754} & \smallnum{1016043721479246781578822482576513263397014461855010} \\
$\widehat{\mathcal R}_3$ & \smallnum{7175180148568351143312937756529351365235304358013215944716282} & \smallnum{7175180148568351326756374029340730115030494416110197315868297} & \smallnum{183443436272811378749795190058096981371152015} & \smallnum{1016043487678788716160162101963550374940926910024249} \\
$\widehat{\mathcal R}_4$ & \smallnum{7175180147552307421833691013151363448112959009064812318199878} & \smallnum{7175180147552307655634149040369189263271753983072289034692033} & \smallnum{233800458027217825815158794974007476716492155} & \smallnum{1016043505663439410040125526345672807857031951590608} \\
$\widehat{\mathcal R}_5$ & \smallnum{7175180146536263700354444269773375530991906294830430006115874} & \smallnum{7175180146536263916170251603111237921767286201207780366609270} & \smallnum{215815807333337862390775379906377350360493396} & \smallnum{1016043538035810585169112863689946286355926164549607} \\
$\widehat{\mathcal R}_6$ & \smallnum{7175180145520220054410730017935331471914270055285828661319569} & \smallnum{7175180145520220162318633684604262667301960008474503841566267} & \smallnum{107907903666668931195387689953188675180246698} & \smallnum{1016043563214321398704278787282416659435905480383424} \\
$\widehat{\mathcal R}_0$ & \smallnum{7175180144504176257395950783017399696745334254878857657299910} & \smallnum{7175180144504176491196408619231052684631853395849923180936145} & \smallnum{233800457836213652987886519140971065523636235} & -- \\
\end{longtable}
\end{landscape}

Thus
\[
 \widehat W=\readnum{1122242198505178340869120933393414062424292850,}
 \qquad
 \widehat\Delta_{\min}=\widehat\Delta_3=\readnum{1016043487678788716160162101963550374940926910024249.}
\]
The exact field-correction data are in \cref{tab:no-fields}.  As before, the columns scaled by four satisfy
\[
 4\widehat e_v=4\widehat a_v^0-\widehat\eta_v\widehat D_v-4\widehat a_v^\star,
 \qquad |4\widehat e_v|\le\widehat D_v/2.
\]

\begin{landscape}
\scriptsize
\let\smallnum\numstr
\setlength{\tabcolsep}{1.5pt}
\setlength{\LTleft}{0pt}
\setlength{\LTright}{0pt}
\tablecert{\textbf{Purpose.}  The table certifies every no-instance logical-field correction and residual field error.  \textbf{Verification rule.}  Check $4\widehat e_v=4\widehat a_v^0-\widehat\eta_v\widehat D_v-4\widehat a_v^\star$ and $|4\widehat e_v|\le\widehat D_v/2$.}
\begin{longtable}{c p{0.20\textheight} p{0.17\textheight} p{0.14\textheight} p{0.20\textheight} p{0.11\textheight}}
\caption{Exact logical-field correction for the no-instance.  All displayed field columns are multiplied by four.}\label{tab:no-fields}\\
\toprule
$v$ & $4\widehat a_v^0$ & $4\widehat a_v^\star$ & $\widehat D_v$ & $\widehat\eta_v$ & $4\widehat e_v$ \\
\midrule
\endfirsthead
\multicolumn{6}{c}{\tablename\ \thetable\ (continued)}\\
\toprule
$v$ & $4\widehat a_v^0$ & $4\widehat a_v^\star$ & $\widehat D_v$ & $\widehat\eta_v$ & $4\widehat e_v$ \\
\midrule
\endhead
\midrule
\multicolumn{6}{r}{Continued on next page}\\
\endfoot
\bottomrule
\endlastfoot
1 & \smallnum{-21679242848733936298983405246017092329267659591390358448222236329019234519752662372288414841668717901842} & \smallnum{56898438712708069579630563498852236715949337269836384} & \smallnum{178270561305452783600760183898294162} & \smallnum{-121608653105591756666661151680492411973085435526728815502014049158185} & \smallnum{-10584813658782193594642372387722256} \\
2 & \smallnum{-20646897884173472877318655818918277264987165404731754651778919036041064645669488345832858134583187226111} & \smallnum{52834264690054311939470061393469545484521209463473952} & \smallnum{165536949783634727629277313619846160} & \smallnum{-124726823293289057041984532247221190383216262808956521869559866252237} & \smallnum{-29942504744365429559070048954840143} \\
3 & \smallnum{-17549863027038471055466091560866847619476611577537027532384916321062722277173936121041174115963617454932} & \smallnum{44705915932554632168807768595682673972701406288978080} & \smallnum{127336115218180559714828702784502154} & \smallnum{-137823138368624970100157463649044675108826429553799575119536071567838} & \smallnum{-61452303140095695548705426838309960} \\
4 & \smallnum{-19614552835556321541507226869247048757100799992943056357209653512915544341699776347306690099041956617518} & \smallnum{45721959556916765450466852109866787525947549717015897} & \smallnum{140069726739998615686311573062950156} & \smallnum{-140034205049642230727565457899864029706785171275297714537692357870883} & \smallnum{-8202020797390013371609652760925667} \\
5 & \smallnum{-13420483252853301239858135473860595254009457826536166160026962231001451270052067637424124752151885357758} & \smallnum{32513393540869648668528648913365653964833688098760134} & \smallnum{76401669130908335828897221670710146} & \smallnum{-175656937937551388015534077772402687123062467012156260564467214760027} & \smallnum{-14450833146919828679376961919983950} \\
6 & \smallnum{-18582207673645379694385794039006192187845355764894663592240233032518093788195407259673517254903737108247} & \smallnum{44705915716738825383015201643088727752864227457943348} & \smallnum{114602503696362503743345832506054152} & \smallnum{-162144866598016401274065429136918646935339761336360158570722682893316} & \smallnum{4424375671750243478458522739796437} \\
\end{longtable}
\end{landscape}

For completeness, \cref{tab:no-couplings} gives all fifteen cross-block coefficients.  The third column is the main-gap coefficient
\[
 \widehat J_{uv}^{\rm main}
 =\frac14\sum_{k=1}^6\widehat\Delta_k
   \widehat d_{u,k}\widehat d_{v,k},
\]
and the last column is its exact within-band remainder.

\begin{landscape}
\scriptsize
\let\smallnum\numstr
\setlength{\tabcolsep}{1.5pt}
\setlength{\LTleft}{0pt}
\setlength{\LTright}{0pt}
\tablecert{\textbf{Purpose.}  The table certifies the hatted cross-block couplings and their main-gap/residual decomposition.  \textbf{Verification rule.}  Recompute the scaled atom-pair sum and main-gap sum for one block pair, then subtract; the global proof later uses only the aggregate residual bound.}
\begin{longtable}{c c p{0.28\textheight} p{0.28\textheight} p{0.28\textheight}}
\caption{Exact cross-block couplings for the no-instance.}\label{tab:no-couplings}\\
\toprule
$u$ & $v$ & $\widehat J_{uv}$ & $\widehat J_{uv}^{\rm main}$ & difference \\
\midrule
\endfirsthead
\multicolumn{5}{c}{\tablename\ \thetable\ (continued)}\\
\toprule
$u$ & $v$ & $\widehat J_{uv}$ & $\widehat J_{uv}^{\rm main}$ & difference \\
\midrule
\endhead
\midrule
\multicolumn{5}{r}{Continued on next page}\\
\endfoot
\bottomrule
\endlastfoot
1 & 2 & \smallnum{16764721350453620361955976127549538099518931384542881/2} & \smallnum{8382359351556524896929074967292417866466646357373945} & \smallnum{2647340570568097826192964702366585638669794991/2} \\
1 & 3 & \smallnum{13716590063720255620996128896018855810076799576420385/2} & \smallnum{13716587862399020493696116399381346263379317292049561/2} & \smallnum{1100660617563650006248318754773348741142185412} \\
1 & 4 & \smallnum{14224611809358115085391679702712052952007277982567969/2} & \smallnum{28449219349160174269051316312946806080004778012136939/4} & \smallnum{4269556055901732043092477299824009777952998999/4} \\
1 & 5 & \smallnum{9144392781121056696671452017998803005260243290711335/2} & \smallnum{9144391907067037164736781423359223038314901082760141/2} & \smallnum{437027009765967335297319789983472671103975597} \\
1 & 6 & \smallnum{6604283722014188207280658692342569512033784817161637} & \smallnum{26417131920589401747421602220752404109150530333989425/4} & \smallnum{2967467351081701032548617873938984608934657123/4} \\
2 & 3 & \smallnum{25401092810374571810178500053929059114534678333936147/4} & \smallnum{12700544356735581083655990873035673455522285340458953/2} & \smallnum{4096903409642866518307857712203490107653018241/4} \\
2 & 4 & \smallnum{6604284070916410017639215321642396633631151590795640} & \smallnum{26417132337833295448971065260255460464290714108955723/4} & \smallnum{3945832344621585796026314126070233892254226837/4} \\
2 & 5 & \smallnum{8636371071452498454498878276435697604202004154248331/2} & \smallnum{8636370154235317459716718660186386634386385106964837/2} & \smallnum{458608590497391079808124655484907809523641747} \\
2 & 6 & \smallnum{6096261951197817856131675711605105129130923791558520} & \smallnum{24385044909262522927341351168061058493436466430808209/4} & \smallnum{2895528748497185351678359362023087228735425871/4} \\
3 & 4 & \smallnum{22352961638742971154782657279719495023594793085094073/4} & \smallnum{11176478957960147107697466768057470537467851944981841/2} & \smallnum{3722822676939387723743604553948659089195130391/4} \\
3 & 5 & \smallnum{7620327533416687952098240304563833262881413307576725/2} & \smallnum{3810163308099753437273802898248220174015229471207615} & \smallnum{917217181077550634508067392914850954365161495/2} \\
3 & 6 & \smallnum{10668458525201782502746003514036419043660944096782783/2} & \smallnum{5334228573788772792958503144247804908592171984289847} & \smallnum{1377624236916828997225540809226476600128203089/2} \\
4 & 5 & \smallnum{15240655394154016938563652401259073180536669740444233/4} & \smallnum{3810163308099753437273802898248220174015229471207615} & \smallnum{2161755003189468440808266192484475751855613773/4} \\
4 & 6 & \smallnum{5334229332741028611787306294747707242733944414303280} & \smallnum{5334228573788772792958503144247804908592171984289847} & \smallnum{758952255818828803150499902334141772430013433} \\
5 & 6 & \smallnum{17272743038540597328342912800590940138111547715453003/4} & \smallnum{8636370154235317459716718660186386634386385106964837/2} & \smallnum{2730069962408909475480218166869338777501523329/4} \\
\end{longtable}
\end{landscape}

\subsubsection{Load lock, anchor, and the 49-job no-instance}\label{app:no-jobs}

The exact load-lock quantities are
\[
 \widehat U=\mathnum{583489201983731785266607839650588849977614462575543787905298044664841507417338080675476531688081516158715736685814354106598477105310962383846749076147397863,}
\]
\[
 \max_v|\widehat\eta_v|=\mednum{175656937937551388015534077772402687123062467012156260564467214760027,}
\]
and
\[
 \widehat H=\widehat U+2\max_v|\widehat\eta_v|+2=\mathnum{583489201983731785266607839650588849977614462575543787905298044664841507417338080675476883001957391261491767753969898911972723230244986696367878010576917919.}
\]
The strict dominance certificate is
\[
 \widehat H-\widehat U=\mednum{351313875875102776031068155544805374246124934024312521128934429520056}>0.
\]
Choose the all-rejected coded state \(\sigma=(-1,\ldots,-1)\) as the explicit benchmark used in the anchor construction.  Its cost is
\[
 \widehat C^\star=\mathnum{267448763683410661369552775798154563648899529608119085375293826634448762597664705230994596377703116027184387031023196308427965856403257033045344487944571498492190534308029665037069211634968042718586903984048367505303673231949601766248866470467640636208882322899.}
\]
The anchor data and common due date are
\[
\begin{aligned}
\widehat p_a&=\widehat H,\\
\widehat r_a&=1+\max_i\widehat r_i=\mednum{7031676545942439210557755708128045311808406689815824393605991,}\\
\widehat s_a&=1+\left\lceil\widehat C^\star/\widehat H^2\right\rceil=2,\\
\widehat d&=\mathnum{583489201983731785266607839650588849977614462575543787905298044664841507417338080675476883001957391261491767754056225234822016539157297678918642006883335353.}
\end{aligned}
\]
Its tardy self-cost is
\[
 \widehat p_a^2\widehat s_a=\mathnum{680919297663224297469522772039887703131818185701605521353164494680402051534719634615271115978950656748767054982151116098783494343969724943804661695902292754150322500412307014293134310459733809136889835337744566351718894261991040961104909767454545739800961024637863943069853588098070640688084360565913330526581122.}
\]
This value exceeds \(\widehat C^\star\) by
\[
 \mathnum{680919297663224297469522772039887703131818185701605253904400811269740681981943836460707467079421048629681679688324481650020896679264493949208283992786265569763291477215998586327277907202700763792401890766246074161184586232326003891893274799411827152896976976270358639396621638496304391821613892925277121644258223}>0.
\]
Therefore no optimal canonical assignment can place the anchor tardy.

The following two tables list every parameter of the no-instance.  The first gives the construction ratios; the second gives the actual integer input triples \((p_j,w_j^E,w_j^L)\).

\begin{landscape}
\tiny
\setlength{\tabcolsep}{1.5pt}
\setlength{\LTleft}{0pt}
\setlength{\LTright}{0pt}
\tablecert{\textbf{Purpose.}  The table records the ratio-level construction data for all 49 jobs of the no-instance.  \textbf{Verification rule.}  Compute each block ratio from the hatted affine formulas and check the anchor from $\widehat p_a=\widehat H$, $\widehat r_a=1+\max\widehat r_j$, and $\widehat s_a=2$.}
\begin{longtable}{c c c c c p{0.18\textheight} p{0.27\textheight} p{0.32\textheight}}
\caption{All jobs of the no-instance: metadata, processing times, and early/tardiness ratios.}\label{tab:no-job-ratios}\\
\toprule
job & block & pair & element & $\epsilon$ & $p_j$ & $r_j$ & $s_j$ \\
\midrule
\endfirsthead
\multicolumn{8}{c}{\tablename\ \thetable\ (continued)}\\
\toprule
job & block & pair & element & $\epsilon$ & $p_j$ & $r_j$ & $s_j$ \\
\midrule
\endhead
\midrule
\multicolumn{8}{r}{Continued on next page}\\
\endfoot
\bottomrule
\endlastfoot
$A_{1,0}$ & 1 & 0 & -- & -1 & \smallnum{1798465042647412146620280340569649349251262} & \smallnum{7031676542894308046120015286989908733170093170770118042669332} & \smallnum{93115851572934093495709703243946643628787667233869160597042200448109982827128673009132967825399307033039380053850525498671537237871130934491588011504169337408454433446548681} \\
$B_{1,0}$ & 1 & 0 & -- & 1 & \smallnum{1798465042647412146620280340569649349251259} & \smallnum{7031676542894308046120015286989908733170092931393319149833311} & \smallnum{93115851572934093495709703243946643628787667233869160597042200448109982827128673009132967825399307033039380053850525498671537237871130934491588011504169337408454433446548681} \\
$A_{1,1}$ & 1 & 1 & $e_{1}$ & 1 & \smallnum{1798465080848246712074448255018260184595356} & \smallnum{7031676545942439210557755708128045311808406689815824393605990} & \smallnum{93115851572934093495709703243946643628787667233869160597042200448109982827128673009132967825399307033039380053850525498671537237871130934491588011504169337408454433446548681} \\
$B_{1,1}$ & 1 & 1 & $e_{1}$ & -1 & \smallnum{1798465080848246712074448255018260184595355} & \smallnum{7031676545942439210557755708128045311808406610023558095993983} & \smallnum{93115851572934093495709703243946643628787667233869160597042200448109982827128673009132967825399307033039380053850525498671537237871130934491588011504169337408454433446548681} \\
$A_{1,2}$ & 1 & 2 & $e_{2}$ & 1 & \smallnum{1798465074481440951165420269276825045372255} & \smallnum{7031676545434417349818132304605022548702094086634446884232283} & \smallnum{93115851572934093495709703243946643628787667233869160597042200448109982827128673009132967825399307033039380053850525498671537237871130934491588011504169337408454433446548681} \\
\end{longtable}
\end{landscape}

\begin{landscape}
\tiny
\setlength{\tabcolsep}{1.5pt}
\setlength{\LTleft}{0pt}
\setlength{\LTright}{0pt}
\tablecert{\textbf{Purpose.}  The table records the physical positive-integer input triples for the 49-job no-instance.  \textbf{Verification rule.}  Multiply $p_j$ by the corresponding ratios in \cref{tab:no-job-ratios}; rowwise equality, not visual digit matching, is the certificate.}
\begin{longtable}{c p{0.24\textheight} p{0.32\textheight} p{0.32\textheight}}
\caption{The 49 actual AWET job triples for the no-instance.}\label{tab:no-job-weights}\\
\toprule
job & $p_j$ & $w_j^E=p_jr_j$ & $w_j^L=p_js_j$ \\
\midrule
\endfirsthead
\multicolumn{4}{c}{\tablename\ \thetable\ (continued)}\\
\toprule
job & $p_j$ & $w_j^E=p_jr_j$ & $w_j^L=p_js_j$ \\
\midrule
\endhead
\midrule
\multicolumn{4}{r}{Continued on next page}\\
\endfoot
\bottomrule
\endlastfoot
$A_{1,0}$ & \smallnum{1798465042647412146620280340569649349251262} & \smallnum{12646224453599219326700846974282670279597026129616521398630256192426199408965867282466660595702049696984} & \smallnum{167465603970267013873202147377146576014402913169162234755617016460017907651590877287008896920799377573858641721319379161878828010394234847246702602448825066513910743587139867247624655499460024842144900285105383685422} \\
$B_{1,0}$ & \smallnum{1798465042647412146620280340569649349251259} & \smallnum{12646224453599219326700846974282670279597004604076087886100782603200341715807813386906473045932506888549} & \smallnum{167465603970267013873202147377146576014402633821607515953336529330908175811659990924007195313317586447257297391370897775859800611490758649325603484308663514937414728975426253854821180735425512334132674921805044039379} \\
$A_{1,1}$ & \smallnum{1798465080848246712074448255018260184595356} & \smallnum{12646224727697089120677578014552695268877323897066776042430983603984228662055152765808373748581847782440} & \smallnum{167465607527370255232242716765981362961990560905569045390293717581356583938991096923989479673584100078457208783975562154009684909796375007001171708310556847415761192650060628043417002268608758225499185759020740525436} \\
$B_{1,1}$ & \smallnum{1798465080848246712074448255018260184595355} & \smallnum{12646224727697089120677578014552695268877316721886625441992163532094767154729310857327780222713369748965} & \smallnum{167465607527370255232242716765981362961990467789717472456200221871653339992347468136322245804423503036256760673992735025336675776828549607694138668930502996890262521112822756912482510680597254056161777304587293976755} \\
\end{longtable}
\end{landscape}

\subsubsection{Master identity and decision threshold for the no-instance}\label{app:no-threshold}

The master decomposition is unchanged in form.  For a coded state
\(x\in\{0,1\}^6\), with \(\sigma=2x-\mathbf1\), write
\[
 \widehat{\mathcal P}(x)
 =\sum_{k=1}^6\frac{\widehat\Delta_k}{2}
   \bigl(\widehat C_k(x)-k\bigr)^2
\]
and
\[
 \widehat R(\sigma)
 =\sum_{v=1}^6\widehat e_v\sigma_v
 +\sum_{1\le u<v\le6}
   \bigl(\widehat J_{uv}-\widehat J_{uv}^{\rm main}\bigr)
   \sigma_u\sigma_v.
\]
Then the exact identity is
\begin{equation}\label{eq:no-master}
 \widehat\Phi_{\rm can}(\sigma)
 =\widehat A+\widehat{\mathcal P}(x)+\widehat R(\sigma).
\end{equation}
There is no asymptotic term hidden in \eqref{eq:no-master}; all within-band and rounding effects are included in \(\widehat R\).
The exact objective constant is
\[
 \widehat C_0=\frac{\mathnum{534897527366821322739105551596309127297799059216238170750587653268897525195329410461989192755406232054368774062046392616855931712806514066090688975889142996984381068616059330074138423269936085437173807968096703513233692016041062418063754462248811152089393704635}}{2}.
\]
Moreover,
\[
 \sum_{v=1}^6\widehat D_v=\readnum{802217525874537526203420827542356930,}
\]
and therefore
\[
 \widehat\Gamma
 =\frac98(6^2)\widehat W+\frac18\sum_v\widehat D_v
 =\frac{\readnum{181803236158239999983734859972834788526506620165}}{4}.
\]
The shifted constant is
\[
 \widehat A=\frac{\mathnum{1069795054733642645478211103192618254595598118432476341501175306537795050390658820923978385510812464108737548124092785233711863425613028132181377951778285993968762137232118660148276846539872170874347615936193285101241564406205502437302657948362994797686001789057}}{4}.
\]
The decisive gap comparison is
\[
 \frac{\widehat\Delta_{\min}}2=\frac{\readnum{1016043487678788716160162101963550374940926910024249}}{2}
 >\frac{\readnum{181803236158239999983734859972834788526506620167}}{2}=2\widehat\Gamma+1,
\]
with positive margin
\[
 \frac{\widehat\Delta_{\min}}2-(2\widehat\Gamma+1)
 =\readnum{507930842221315238080089183551788770076200201702041.}
\]
Finally, the no-instance decision threshold is
\[
 \widehat B=\left\lfloor\widehat A+\widehat\Gamma\right\rfloor
 =\mathnum{267448763683410661369552775798154563648899529608119085375293826634448762597664705230994596377703116027184387031023196308427965856403257033045344487944571498492190534308029665037069211634968042718586903984048321275355841910590935609321598202083957396553127102305.}
\]
Thus a single nonzero integer prefix error costs more than the most favorable possible residual swing on both sides of the threshold comparison.  This is the numerical fact that will stop the attempted forward construction and make the reverse supposition contradictory.

\subsubsection{The ``if'' route: where the forward construction fails}\label{app:no-if}

The forward implication proved in \cref{lem:red2if} starts with an exact-cover vector.  The present source instance has no such vector, so there is no legitimate witness with which to begin that proof.  To see the obstruction at the closest possible point, nevertheless try the former yes-witness
\[
 x^\dagger=(1,0,0,0,1,0),
 \qquad
 \sigma^\dagger=(+1,-1,-1,-1,+1,-1).
\]
This produces the same formal early--tardy orientation pattern as in \cref{app:worked-if}: blocks \(1\) and \(5\) use their \(+1\) states and all other blocks use their \(-1\) states.  Every block is arithmetically synchronized and has its target tardy load.  The construction fails only when the logical prefix equations are evaluated.

For this attempted witness,
\[
 \widehat C(x^\dagger)=(1,2,2,4,5,6),
 \qquad
 \widehat C(x^\dagger)-(1,\ldots,6)=(0,0,-1,0,0,0).
\]
Therefore all terms in \(\widehat{\mathcal P}\) vanish except the third one, and
\[
 \widehat{\mathcal P}(x^\dagger)
 =\frac{\widehat\Delta_3}{2}
 =\frac{\readnum{1016043487678788716160162101963550374940926910024249}}{2}.
\]
The exact residual for this state is
\[
 \widehat R(\sigma^\dagger)=\frac{\readnum{-3294787992173267513354522350016919741241278067}}{4}.
\]
Substitution in \eqref{eq:no-master} gives the integer schedule cost
\[
 \widehat\Phi_{\dagger}
 =\widehat A+\widehat{\mathcal P}(x^\dagger)+\widehat R(\sigma^\dagger)
 =\mathnum{267448763683410661369552775798154563648899529608119085375293826634448762597664705230994596377703116027184387031023196308427965856403257033045344487944571498492190534308029665037069211634968042718586903984048321783331311243947690372528376838278431939949645139872.}
\]
It misses the required threshold by
\[
 \widehat\Phi_{\dagger}-\widehat B
 =\readnum{507975469333356754763206778636194474543396518037567}>0.
\]
This calculation is the direct numerical failure of the attempted ``if'' construction: the former cover still passes the block and load gadgets, but the uncovered element \(e_3\) creates one unit of third-prefix error, and the main gap converts that single unit into an insurmountable cost.

The former witness is not quite the least expensive coded state.  Exact enumeration shows that the minimum coded cost occurs at
\[
 \widehat x^{\rm best}=(0,1,0,0,0,1),
 \qquad
 \widehat C(\widehat x^{\rm best})=(1,2,3,4,6,6),
\]
so its sole error is at the fifth prefix.  For this state,
\[
 \widehat{\mathcal P}(\widehat x^{\rm best})
 =\frac{\readnum{1016043538035810585169112863689946286355926164549607}}{2},
\]
\[
 \widehat R(2\widehat x^{\rm best}-\mathbf1)
 =\frac{\readnum{-5625598711919279964842548215293427898608008223}}{4},
\]
and
\[
 \widehat\Phi_{\rm can}(2\widehat x^{\rm best}-\mathbf1)-\widehat B
 =\readnum{507974911809187752764569287492926111123856803617707}>0.
\]
Hence even the most favorable coded assignment is above the threshold.  \Cref{tab:no-state-audit} supplies the exhaustive \(2^6=64\)-state check.  The last column is the exact value of \(\widehat\Phi_{\rm can}-\widehat B\); every entry is positive.
\begin{landscape}
\tiny
\setlength{\tabcolsep}{1.5pt}
\setlength{\LTleft}{0pt}
\setlength{\LTright}{0pt}
\tablecert{\textbf{Purpose.}  The table exhausts all 64 no-instance coded states and certifies that none reaches the decision threshold.  \textbf{Verification rule.}  Compute the hatted prefix totals, penalty, residual, and $\widehat A+\widehat{\mathcal P}+\widehat R-\widehat B$; every final entry must be positive.}
\begin{longtable}{c p{0.15\textheight} p{0.17\textheight} c p{0.22\textheight} p{0.22\textheight} p{0.27\textheight}}
\caption{Exact audit of all 64 coded states for the no-instance.}\label{tab:no-state-audit}\\
\toprule
$x$ & $\widehat C(x)$ & $\widehat C(x)-(1,\ldots,6)$ & RPE? & $\widehat{\mathcal P}(x)$ & $\widehat R(\sigma)$ & $\widehat\Phi_{\rm can}-\widehat B$ \\
\midrule
\endfirsthead
\multicolumn{7}{c}{\tablename\ \thetable\ (continued)}\\
\toprule
$x$ & $\widehat C(x)$ & $\widehat C(x)-(1,\ldots,6)$ & RPE? & $\widehat{\mathcal P}(x)$ & $\widehat R(\sigma)$ & $\widehat\Phi_{\rm can}-\widehat B$ \\
\midrule
\endhead
\midrule
\multicolumn{7}{r}{Continued on next page}\\
\endfoot
\bottomrule
\endlastfoot
\texttt{000000} & \smallnum{0,0,0,0,0,0} & \smallnum{-1,-2,-3,-4,-5,-6} & no & \smallnum{92459962716614084396639698718108541805605806098669265/2} & \smallnum{47695293424111348295371716423956535537330164009/4} & \smallnum{46229947831321358666156927268268383683239655755220594} \\
\texttt{000001} & \smallnum{0,0,1,2,3,3} & \smallnum{-1,-2,-2,-2,-2,-3} & no & \smallnum{26417132704720167841790484498028594418372890703566529/2} & \smallnum{18927046305926658913194663649671590772513624417/4} & \smallnum{13208525633312620842559974613965216418387006853534328} \\
\texttt{000010} & \smallnum{0,0,0,1,2,3} & \smallnum{-1,-2,-3,-3,-3,-3} & no & \smallnum{41657785361610361398637486958024441764142315834388849/2} & \smallnum{27077689966468620937625082912483704892294337173/4} & \smallnum{20828853999418632756473981951567955794300249364123677} \\
\texttt{000011} & \smallnum{0,0,1,3,5,6} & \smallnum{-1,-2,-2,-1,0,0} & no & \smallnum{10160435966657714682655147378690040914454940867145461/2} & \smallnum{9229722697919569457368902805676115237483890897/4} & \smallnum{5080174839950492261219942097855728667559148177890414} \\
\texttt{000100} & \smallnum{1,1,2,2,2,3} & \smallnum{0,-1,-1,-2,-3,-3} & no & \smallnum{12192522571531490496380005774337010690946276991319485} & \smallnum{13423743216239965821231612833885425029852265211/4} & \smallnum{12192480476658254996371465148525225953605402827730747} \\
\texttt{000101} & \smallnum{1,1,3,4,5,6} & \smallnum{0,-1,0,0,0,0} & no & \smallnum{508021860739623390789411241288256631698507230927505} & \smallnum{-3201267808843462710537441503053251376084059453/4} & \smallnum{507975609613631619923737673212887659688531583257601} \\
\texttt{000110} & \smallnum{1,1,2,3,4,6} & \smallnum{0,-1,-1,-1,-1,0} & no & \smallnum{2032087126428642746474111487287841366275449744009737} & \smallnum{1453159771355112226718044092350497392238893467/4} & \smallnum{2032042038909546025252172233083871245202666177078063} \\
\texttt{000111} & \smallnum{1,1,3,5,7,9} & \smallnum{0,-1,0,1,2,3} & no & \smallnum{7620326724107410560316954274611860575800450197547431} & \smallnum{-4251571404092678403130137577110823903691337881/4} & \smallnum{7620280210405519977147357558362473089397342648057920} \\
\texttt{001000} & \smallnum{0,1,2,2,2,3} & \smallnum{-1,-1,-1,-2,-3,-3} & no & \smallnum{25401088767425114274419095062858134935138697410676787/2} & \smallnum{14071190638337216628067617653802726142846305193/4} & \smallnum{12700502450701177161513708614618487709553752790259651} \\
\texttt{001001} & \smallnum{0,1,3,4,5,6} & \smallnum{-1,-1,0,0,0,0} & no & \smallnum{2032087345841380063237905996760626816643157889892827/2} & \smallnum{-3676062584512840776305108646670405820944609687/4} & \smallnum{1015997303096004343408762988388158532022992082138951} \\
\texttt{001010} & \smallnum{0,1,2,3,4,6} & \smallnum{-1,-1,-1,-1,-1,0} & no & \smallnum{5080217877219418774607306488759796285797042916057291/2} & \smallnum{791324629314894346385523285648703132731770317/4} & \smallnum{2540063685631827156027243907045726346377173014316184} \\
\texttt{001011} & \smallnum{0,1,3,5,7,9} & \smallnum{-1,-1,0,1,2,3} & no & \smallnum{16256697072576954402292992063407834704847043823132679/2} & \smallnum{-6035648743899525156066330347347073721053051247/4} & \smallnum{8128301576567251666265211081406337306957960021648487} \\
\texttt{001100} & \smallnum{1,2,4,4,4,6} & \smallnum{0,0,1,0,-1,0} & no & \smallnum{1016043512857299650664637482826748330648426537286928} & \smallnum{-5309068861776614951098067720473748007851072041/4} & \smallnum{1015996734781044646510903774594825003514292947863877} \\
\texttt{001101} & \smallnum{1,2,5,6,7,9} & \smallnum{0,0,2,2,2,3} & no & \smallnum{10668457097220523716908055526769213905769344714054336} & \smallnum{-10913085991525411505062795583600611612761771993/4} & \smallnum{10668408918139986275555183327355324796919309896956297} \\
\texttt{001110} & \smallnum{1,2,4,5,6,9} & \smallnum{0,0,1,1,1,3} & no & \smallnum{6096261300153465649853954788770459702038517174807640} & \smallnum{-9941914858041063469547097318689868010543151825/4} & \smallnum{6096213363865711579588091468281136820874382912364643} \\
\texttt{001111} & \smallnum{1,2,5,7,9,12} & \smallnum{0,0,2,3,4,6} & no & \smallnum{33021415192987324635530810153085698545932205565504722} & \smallnum{-4625652138154222121590952514339376505447758461/4} & \smallnum{33021368585765250536975283821632576752390947576910066} \\
\texttt{010000} & \smallnum{1,2,2,2,3,3} & \smallnum{0,0,-1,-2,-2,-3} & no & \smallnum{18288783731404681285335624747647776686715908698035925/2} & \smallnum{11560533412276668010621067110899167951266827885/4} & \smallnum{9144349305026654151834819095375672859452810539069893} \\
\texttt{010001} & \smallnum{1,2,3,4,6,6} & \smallnum{0,0,0,0,1,0} & no & \smallnum{1016043538035810585169112863689946286355926164549607/2} & \smallnum{-5625598711919279964842548215293427898608008223/4} & \smallnum{507974911809187752764569287492926111123856803617707} \\
\texttt{010010} & \smallnum{1,2,2,3,5,6} & \smallnum{0,0,-1,-1,0,0} & no & \smallnum{2032086993342228126200287628309223182797958861614857/2} & \smallnum{-1719332597407802070195571912815137741390730999/4} & \smallnum{1015997616028925151149630331546640178917412456469638} \\
\texttt{010011} & \smallnum{1,2,3,5,8,9} & \smallnum{0,0,0,1,3,3} & no & \smallnum{19304827416914627264900650385096939319983516755987887/2} & \smallnum{-7985184871968112143738314571530378481259473791/4} & \smallnum{9652366261352056080422293324254833568700006436470455} \\
\texttt{010100} & \smallnum{2,3,4,4,5,6} & \smallnum{1,1,1,0,0,0} & no & \smallnum{1524065416760084389699034049362088595792042399958538} & \smallnum{-6927687417108371279413779974891006987194163565/4} & \smallnum{1524018234029190552606218262202101664343163974762606} \\
\texttt{010101} & \smallnum{2,3,5,6,8,9} & \smallnum{1,1,2,2,3,3} & no & \smallnum{14224609615230740211449790684374393029980739070374767} & \smallnum{-11970583448203058404469396863737334478188784745/4} & \smallnum{14224561171775838600685193633310183886949987896523540} \\
\texttt{010110} & \smallnum{2,3,4,5,7,9} & \smallnum{1,1,1,1,2,3} & no & \smallnum{8636370280127871559226577082685692539893985366578464} & \smallnum{-11560533414034967596997354228667409672429267357/4} & \smallnum{8636321939185478490484681899632142164344435632606584} \\
\texttt{010111} & \smallnum{2,3,5,7,10,12} & \smallnum{1,1,2,3,5,6} & no & \smallnum{38609654787069162300410771038070770242855452250924367} & \smallnum{-5683149595494016820132098450036382053417795221/4} & \smallnum{38609607915472723866906570071331164525062807269820521} \\
\texttt{011000} & \smallnum{1,3,4,4,5,6} & \smallnum{0,1,1,0,0,0} & no & \smallnum{2032087209158035497738984584540063638337941371879259/2} & \smallnum{-5675955734925997583451600810440681012604957967/4} & \smallnum{1015996734781044457370100495654836000301585908045097} \\
\texttt{011001} & \smallnum{1,3,5,6,8,9} & \smallnum{0,1,2,2,3,3} & no & \smallnum{27433175606099347141240497854564672506715334712711717/2} & \smallnum{-11841093963787313581110889662821464061454169363/4} & \smallnum{13716539391967143063791857715844927339294520366158477} \\
\texttt{011010} & \smallnum{1,3,4,5,7,9} & \smallnum{0,1,1,1,2,3} & no & \smallnum{16256696935893609836794070651187271526541827305119111/2} & \smallnum{-11618084295990062588203700690836179070341224891/4} & \smallnum{8128300112616691360881392340953469845529014440598292} \\
\texttt{011011} & \smallnum{1,3,5,7,10,12} & \smallnum{0,1,2,3,5,6} & no & \smallnum{76203265949776191319162458561957426932464761073810917/2} & \smallnum{-6862942675215740683942116875739607009184342971/4} & \smallnum{38101585808343387295646062361734501322633496614164675} \\
\texttt{011100} & \smallnum{2,4,6,6,7,9} & \smallnum{1,2,3,2,2,3} & no & \smallnum{15748675071557055711295647503923173146538762626843887} & \smallnum{-9272885856553485978512029680436219594285427853/4} & \smallnum{15748627302526552012924156942200759828786732428831883} \\
\texttt{011101} & \smallnum{2,4,7,8,10,12} & \smallnum{1,2,4,4,5,6} & no & \smallnum{49786133565182802704880416715926697215096147234419504} & \smallnum{-3294787992313541125763320095470734284254424321/4} & \smallnum{49786087290676765066495139341381680138715444544158383} \\
\texttt{011110} & \smallnum{2,4,6,7,9,12} & \smallnum{1,2,3,3,4,6} & no & \smallnum{38101633167323856629918402130239657786701623478294273} & \smallnum{-6567994404859677220031064790893814644599239685/4} & \smallnum{38101586074516215854999101188758466854550830701829311} \\
\texttt{011111} & \smallnum{2,4,7,9,12,15} & \smallnum{1,2,4,5,7,9} & no & \smallnum{89411831969420238542936608662615955124031778299799564} & \smallnum{10330383309015905534638517461549025775437857163/4} & \smallnum{89411789101207026236912996388530327302591090532608814} \\
\texttt{100000} & \smallnum{1,2,2,3,3,3} & \smallnum{0,0,-1,-1,-2,-3} & no & \smallnum{15240653214414363055215248168610758263144812843264101/2} & \smallnum{10330383309213982479634994521720348166130939217/4} & \smallnum{7620283738993969271103248059339016352962316327711814} \\
\texttt{100001} & \smallnum{1,2,3,5,6,6} & \smallnum{0,0,0,1,1,0} & no & \smallnum{2032087043699249995209238390035619094212958116140215/2} & \smallnum{-6567994404643902772347586756808658162946971883/4} & \smallnum{1015996429041984276628930174406127136244806694672096} \\
\texttt{100010} & \smallnum{1,2,2,4,5,6} & \smallnum{0,0,-1,0,0,0} & no & \smallnum{1016043487678788716160162101963550374940926910024249/2} & \smallnum{-3294787992173267513354522350016919741241278067/4} & \smallnum{507975469333356754763206778636194474543396518037567} \\
\texttt{100011} & \smallnum{1,2,3,6,8,9} & \smallnum{0,0,0,2,3,3} & no & \smallnum{22352957933904945495021026964133957743554612610759711/2} & \smallnum{-9272885856395514863416230961068570960313095851/4} & \smallnum{11176431197921969088631801694294245395937434600450852} \\
\texttt{100100} & \smallnum{2,3,4,5,5,6} & \smallnum{1,1,1,1,0,0} & no & \smallnum{2032087169591804094719096812534924999720558375753842} & \smallnum{-6862942675050471822135199869054723229534963585/4} & \smallnum{2032040003047095772101145345019964527342619365357905} \\
\texttt{100101} & \smallnum{2,3,5,7,8,9} & \smallnum{1,1,2,3,3,3} & no & \smallnum{16764718379389338736550104500238575049623318949351287} & \smallnum{-11618084295807096223709782710237461199732659757/4} & \smallnum{16764670024059225224776052639077904281560887389531307} \\
\texttt{100110} & \smallnum{2,3,4,6,7,9} & \smallnum{1,1,1,2,2,3} & no & \smallnum{10160435538623030674286765372204201751679533293964376} & \smallnum{-11841093963679848051891651970854088129484725777/4} & \smallnum{10160387127540500194324756465576215829460369296127891} \\
\texttt{100111} & \smallnum{2,3,5,8,10,12} & \smallnum{1,1,2,4,5,6} & no & \smallnum{42165807056891200235551210380280625070355064081491495} & \smallnum{-5675955734800834551545362144559470989676328633/4} & \smallnum{42165760187093226975342576560225095721790185035754296} \\
\texttt{101000} & \smallnum{1,3,4,5,5,6} & \smallnum{0,1,1,1,0,0} & no & \smallnum{3048130714821474907779110110885736446194973323469867/2} & \smallnum{-5683149595541749087696004172059881370077953007/4} & \smallnum{1524018485814299008452287197726831999430012515591641} \\
\texttt{101001} & \smallnum{1,3,5,7,8,9} & \smallnum{0,1,2,3,3,3} & no & \smallnum{32513393134416544191441125486293036546000494470664757/2} & \smallnum{-11560533414065002361874258976777074898130239395/4} & \smallnum{16256648226265879019469976340866780870034391076117489} \\
\texttt{101010} & \smallnum{1,3,4,6,7,9} & \smallnum{0,1,1,2,2,3} & no & \smallnum{19304827452883928066914447230224289950112923159890935/2} & \smallnum{-11970583448308594004620981900478341642528878331/4} & \smallnum{9652365282987062396308726526151676646773919321070844} \\
\texttt{101011} & \smallnum{1,3,5,8,10,12} & \smallnum{0,1,2,4,5,6} & no & \smallnum{83315570489420267189443337246377136587463984734945173/2} & \smallnum{-6927687417196209376878364037718180060575071403/4} & \smallnum{41657738061979239735669328469882565655489845597049695} \\
\texttt{101100} & \smallnum{2,4,6,7,7,9} & \smallnum{1,2,3,3,2,3} & no & \smallnum{18288783835715654236395961319787355166181342505820407} & \smallnum{-7985184872048652494491780347040316408963334245/4} & \smallnum{18288736388610396664232841763127275197405108638331805} \\
\texttt{101101} & \smallnum{2,4,7,9,10,12} & \smallnum{1,2,4,5,5,6} & no & \smallnum{54358329340668280050060981584482224850452791016577240} & \smallnum{-1719332597470644918262036714411241578135405705/4} & \smallnum{54358283460026091122399756085258053038945264856070773} \\
\texttt{101110} & \smallnum{2,4,6,8,9,12} & \smallnum{1,2,3,4,4,6} & no & \smallnum{41657785437145894565058841472449512614201235308861401} & \smallnum{-5625598712057623648183693305520873673991804477/4} & \smallnum{41657738579937176990652933492811193025285685184255241} \\
\texttt{101111} & \smallnum{2,4,7,10,12,15} & \smallnum{1,2,4,6,7,9} & no & \smallnum{95000071250569155298157299057517155567245454033547908} & \smallnum{11560533412156021829966922994585556266842217379/4} & \smallnum{95000028689893468777162760615532911004937389117447212} \\
\texttt{110000} & \smallnum{2,4,4,5,6,6} & \smallnum{1,2,1,1,1,0} & no & \smallnum{8128349041657159119343773936489336075988086301622321/2} & \smallnum{-4625652138075915195571937172404334310574036979/4} & \smallnum{4064127913606505480693092141545381728213333880646875} \\
\texttt{110001} & \smallnum{2,4,5,7,9,9} & \smallnum{1,2,2,3,4,3} & no & \smallnum{43689872689467091914020466494036313893929164436114853/2} & \smallnum{-9941914857945059040841081002840991724710244595/4} & \smallnum{21844888408445791910745477103032913028019519413841237} \\
\texttt{110010} & \smallnum{2,4,4,6,8,9} & \smallnum{1,2,1,2,3,3} & no & \smallnum{28449219931862854619155562510587674725329740796241817/2} & \smallnum{-10913085991504907911631459556383077265567986311/4} & \smallnum{14224561786850889873350807413713955058198422379469290} \\
\texttt{110011} & \smallnum{2,4,5,8,11,12} & \smallnum{1,2,2,4,6,6} & no & \smallnum{98556224196614057252699129708880199080816359358593697/2} & \smallnum{-5309068861738413854979730719342379569698100611/4} & \smallnum{49278065320230773631746105175792426496116155628116655} \\
\texttt{110100} & \smallnum{3,5,6,7,8,9} & \smallnum{2,3,3,3,3,3} & no & \smallnum{24893067421046497575754424956729674370432909213888675} & \smallnum{-6035648743854026313236875058898470137223032433/4} & \smallnum{24893020461325272052247850713795916437118243281475526} \\
\texttt{110101} & \smallnum{3,5,7,9,11,12} & \smallnum{2,3,4,5,6,6} & no & \smallnum{64010743540106555144926783812494382913772136218294329} & \smallnum{791324629378090691901979548011140807520974879/4} & \smallnum{64010698287128672929449460854274276707860206471883008} \\
\texttt{110110} & \smallnum{3,5,6,8,10,12} & \smallnum{2,3,3,4,5,6} & no & \smallnum{50294156098548359074755530836771724391164654346028883} & \smallnum{-3676062584525145266063332672939310084794526673/4} & \smallnum{50294109728723673383469218387223562947640001520742174} \\
\texttt{110111} & \smallnum{3,5,7,10,13,15} & \smallnum{2,3,4,6,8,9} & no & \smallnum{106684572526079051563361327012909206203276651564364211} & \smallnum{14071190638342609640996394601447655969955573955/4} & \smallnum{106684530593067671589013741328292863356493512426602659} \\
\texttt{111000} & \smallnum{2,5,6,7,8,9} & \smallnum{1,3,3,3,3,3} & no & \smallnum{46738003969006595306531599370907008081127388143663899/2} & \smallnum{-4251571404260180689671505017370603416170856239/4} & \smallnum{23368955470801407028220631333862256489215708402462849} \\
\texttt{111001} & \smallnum{2,5,7,9,11,12} & \smallnum{1,3,4,5,6,6} & no & \smallnum{124973356207126710444876317082436425167805842152475207/2} & \smallnum{1453159771205307442863677626004551970718560857/4} & \smallnum{62486633016044258463765023323422625876343782129222777} \\
\texttt{111010} & \smallnum{2,5,6,8,10,12} & \smallnum{1,3,3,4,5,6} & no & \smallnum{97540181324010318304533811130991108122590878407944315/2} & \smallnum{-3201267809068768329666488258030538736243513611/4} & \smallnum{48770044410879167325074827215158496344963623516438714} \\
\texttt{111011} & \smallnum{2,5,7,10,13,15} & \smallnum{1,3,4,6,8,9} & no & \smallnum{210321014179071703281745403483266071746814872844614971/2} & \smallnum{13423743216032357704789567052821971760651996801/4} & \smallnum{105160464994662616088962132005309805870203244958651645} \\
\texttt{111100} & \smallnum{3,6,8,9,10,12} & \smallnum{2,4,5,5,5,6} & no & \smallnum{66550852800641509884743271210053451447938264024873146} & \smallnum{9229722697719259087637975311929897113960669871/4} & \smallnum{66550809657263144754558047185832286221715410888385573} \\
\texttt{111101} & \smallnum{3,6,9,11,13,15} & \smallnum{2,4,6,7,8,9} & no & \smallnum{127005443214856658625749642642809379625646178966438188} & \smallnum{27077689966286008070581156392651320859730301895/4} & \smallnum{127005404533470110637251664354383484579779262272358621} \\
\texttt{111110} & \smallnum{3,6,8,10,12,15} & \smallnum{2,4,5,6,7,9} & no & \smallnum{107192594710542385132839588683088382164730927041843814} & \smallnum{18927046305668545210876056841207864801310467591/4} & \smallnum{107192553991494921989975895468387599257999995742805671} \\
\texttt{111111} & \smallnum{3,6,9,12,15,18} & \smallnum{2,4,6,8,10,12} & no & \smallnum{184919925433228168793279397436217083611211612197338530} & \smallnum{47695293423870932095740110589406643657086192931/4} & \smallnum{184919891906242485201012425437529737754175394842231722} \\
\end{longtable}
\end{landscape}

The table gives a stronger check than the reduction itself needs.  Source-level reasoning already ruled out every exact cover.  The numerical audit additionally verifies, without appealing to an asymptotic estimate, that no one of the 64 canonical schedules reaches the constructed bound.

\subsubsection{The ``only if'' route: the hypothetical target witness is impossible}\label{app:no-onlyif}

Now run the reverse implication in full.  Suppose, for contradiction, that the constructed no-instance admitted a schedule of cost at most \(\widehat B\).  Choose an optimal schedule and replace it, using \cref{lem:compact}, by an optimal canonical schedule.

\par\medskip\noindent\textbf{Gate 1: the anchor must be early in the selected canonical schedule.}\quad
The all-rejected coded state provides a feasible benchmark of cost \(\widehat C^\star\), and
\[
 \widehat C^\star-\widehat B=\readnum{46229947831321358666156927268268383683239655755220594}>0.
\]
If the anchor were tardy, its tardy self-cost alone would be
\[
 \widehat p_a^2\widehat s_a=\mathnum{680919297663224297469522772039887703131818185701605521353164494680402051534719634615271115978950656748767054982151116098783494343969724943804661695902292754150322500412307014293134310459733809136889835337744566351718894261991040961104909767454545739800961024637863943069853588098070640688084360565913330526581122.}
\]
This value exceeds \(\widehat C^\star\), and hence \(\widehat B\), by a vast margin.  Therefore the selected optimal canonical schedule keeps the anchor early.

\par\medskip\noindent\textbf{Gate 2: every block must have its target load and one of its two coded states.}\quad
If the vector of block-load errors is a nonzero integer vector, positive definiteness of the minimum kernel makes the load-lock loss at least \(\widehat H\).  All other assignment-dependent oscillation is bounded in absolute value by \(\widehat U\).  The exact slack is
\[
 \widehat H-\widehat U=\mednum{351313875875102776031068155544805374246124934024312521128934429520056}>0.
\]
Thus a low-cost schedule cannot carry any block-load error.  The arithmetic synchronization argument then leaves, in each of the six blocks, only the displayed state and its complement.  Consequently the whole schedule decodes to a binary vector \(x\in\{0,1\}^6\).

\par\medskip\noindent\textbf{Gate 3: every prefix equation would have to be exact.}\quad
Because the source audit in \cref{tab:no-pair-audit} proves that no binary vector is prefix-exact, the decoded vector has at least one nonzero integer error \(\widehat C_k(x)-k\).  Therefore
\[
 \widehat{\mathcal P}(x)\ge\frac{\widehat\Delta_{\min}}2.
\]
Using \(\widehat R(\sigma)\ge-\widehat\Gamma\) in \eqref{eq:no-master} yields
\[
\begin{aligned}
 \widehat\Phi_{\rm can}(\sigma)-\widehat B
 &\ge \widehat A+\frac{\widehat\Delta_{\min}}2
       -\widehat\Gamma-\widehat B\\
 &=
 \frac{\readnum{1015861684442630476160178367103577540152400403404085}}{2}>0.
\end{aligned}
\]
This contradicts the supposed cost bound.  Equivalently, gap dominance would force \(\widehat{\mathcal P}(x)=0\); then \(\widehat C_k(x)=k\) for every \(k\), and successive differences would give
\[
 \widehat A^{\mathrm{inc}}x=\mathbf1.
\]
But \(\widehat C_6(x)=6\) forces exactly two selected columns, whereas every pair of columns intersects.  The decoded exact cover demanded by the reverse implication therefore cannot exist.  This closes the contradiction from the hypothetical ``only if'' starting assumption.

\clearpage
\subsubsection{Side-by-side conclusion}\label{app:no-conclusion}

The two examples differ only by exchanging \(e_3\) and \(e_4\) between columns \(S_1\) and \(S_3\), yet the reduction records that local combinatorial change exactly.
\begin{table}[ht]
\centering
\caption{What changes between the closely related yes- and no-instances.}\label{tab:yes-no-summary}
\begin{tabular}{p{0.18\textwidth} p{0.34\textwidth} p{0.34\textwidth}}
\toprule
feature & yes-instance in \cref{app:worked-yes} & no-instance in \cref{app:worked-no}\\
\midrule
source structure & $S_1\cap S_5=\varnothing$ and $S_1\cup S_5=X$ & every pair of subset-columns intersects\\
closest vector & $x=(1,0,0,0,1,0)$ & $x^\dagger=(1,0,0,0,1,0)$\\
prefix totals & $(1,2,3,4,5,6)$ & $(1,2,2,4,5,6)$\\
main penalty & $\mathcal P(x)=0$ & $\widehat{\mathcal P}(x^\dagger)=\widehat\Delta_3/2$\\
threshold test & $\Phi_{\rm cover}-B<0$ & $\widehat\Phi_\dagger-\widehat B>0$, and all 64 coded states are above $\widehat B$\\
logical outcome & exact cover decodes to a low-cost schedule & a low-cost schedule would decode to a nonexistent exact cover\\
\bottomrule
\end{tabular}
\end{table}

In the yes-instance, the forward route begins with a genuine exact cover, annihilates every prefix square, and reaches the decision bound.  In the no-instance, the nearest analogous assignment incurs one third-prefix square; more generally, the reverse route shows that any schedule reaching the bound would have to annihilate all prefix squares and hence create an exact cover that the source incidence matrix forbids.  The two examples therefore exhibit both sides of the reduction, including the precise numerical mechanism by which the no-instance is rejected.

\clearpage

\clearpage
\setrunningunit{Glossary}
\setshortsectionhead{Glossary and notation}
\section{Glossary and Notation}
\label{app:glossary}
\begingroup
\small
\setstretch{1.05}
\setlength{\tabcolsep}{5pt}
\tablecert{\textbf{Purpose.}  The table fixes the principal terminology and symbol meanings used across the five Parts and appendices.  \textbf{Verification rule.}  Treat construction-specific symbols as local unless the scope column or an explicit notation reset states otherwise.}
\begin{longtable}{@{}p{0.23\textwidth}p{0.71\textwidth}@{}}
\toprule
Term or symbol & Meaning in this tutorial\\
\midrule
\endfirsthead
\toprule
Term or symbol & Meaning in this tutorial\\
\midrule
\endhead
\bottomrule
\endfoot
\AWET & Single-machine asymmetric weighted earliness--tardiness scheduling with common due date $d=\sum_jp_j$.\\
$p_j,w_j^E,w_j^L$ & Processing time, earliness weight, and tardiness weight of job $j$.\\
$r_j,s_j$ & Smith-type ratios $w_j^E/p_j$ and $w_j^L/p_j$ for the early and tardy arms.\\
$E,T$ & Early/on-time and tardy job sets in a canonical compact V-shaped schedule.\\
Canonical side objective & The exact quadratic set function in \eqref{eq:front-canonical}; it is the common starting point of all four technical Parts.\\
\RXC & Restricted Exact Cover by 3-Sets, the strong-hardness source in Part I.\\
\RPE & Regular prefix exactness, the intermediate source whose squared prefix errors are compiled into Part I's scheduling objective.\\
Anchored SDP & The Part II relaxation obtained after forcing one enumerated job to have early marginal one.\\
$e_i,t_i,e_{ij},t_{ij}$ & Part II pseudo-marginal early/tardy masses and pseudo-joint masses.\\
Minimum kernel & A matrix with entries $\min\{q_i,q_j\}$; its Gram representation is the positive-semidefinite engine of the tardy payment in Part II.\\
Strict antithetical & The class with $r_1<\cdots<r_n$ and $s_1>\cdots>s_n$ after one common indexing.\\
$Q,S,U,M,D$ in Part III & Total tagged load, selected early tagged load, residual budget $3Q^3$, dominant scale $8Q^3$, and due-date scale $M+Q$.\\
Ratio-order refinements & Fixed total orders $\prec_E,\prec_L$ refining nondecreasing $r_j$ and $s_j$; ties may be broken deterministically because equal-ratio swaps do not alter pair coefficients.\\
Ratio permutation & After indexing jobs by $\prec_E$, the permutation of their ranks in $\prec_L$.\\
Inversion graph $G_\times$ & The graph joining exactly those job pairs whose relative orders in $\prec_E$ and $\prec_L$ disagree.\\
Separable permutation & A permutation recursively constructed by direct and skew sums; equivalently, its inversion graph is a cograph.\\
Separating-tree block $J(M)$ & The exact leaf set below separating-tree node $M$; it is an interval in both ratio orders.\\
Ordered separating tree & A binary direct/skew decomposition tree for a separable ratio permutation, with left child earlier in $\prec_E$.\\
$z_M=(e_M,t_M,u_M,v_M)$ & Part IV's additive block signature: early processing load, tardy processing load, earliness weight of early jobs, and tardiness weight of tardy jobs.\\
$\mathcal D_M(z)$ & Minimum internal canonical cost of a side assignment on separating-tree block $J(M)$ having signature $z$; the exact Part IV dynamic-programming table.\\
Geometric box & A four-tuple of logarithmic coordinate buckets, with zero kept in a singleton bucket, used to trim the exact separating-tree table.\\
\end{longtable}
\endgroup

\clearpage
\setrunningunit{Citation Guide}
\section{Citing This Document}
\label{app:citing}
The title is \emph{Asymmetric Weighted Earliness Tardiness: Scheduling with a Nonrestrictive Common Due Date}, version dated \versiondate.  Citations to a particular result should include its Part-based theorem number and, where helpful, the descriptive Part title.  The front-matter provenance note states the tutorial's policy for historical priority claims.

\clearpage
\phantomsection
\addcontentsline{toc}{section}{References}
\label{app:references}
\setrunningunit{References}

\end{document}